\documentclass[11pt, letterpaper, reqno]{amsart}
\usepackage[utf8]{inputenc} 	
\usepackage{microtype} 			
\usepackage{geometry} 			
\usepackage{fancyhdr} 			
\usepackage{amsmath}			
\usepackage{amsthm}		 		
\usepackage{amssymb}	 		
\usepackage{esint}	 			
\usepackage{bm}					
\usepackage{mathrsfs}			
\usepackage{thmtools}
\usepackage{xcolor}				
\usepackage{tikz}				
\usetikzlibrary{patterns} 
\usepackage[all]{xy}			
\usepackage{graphicx}			
\usepackage{enumitem} 			
\usepackage{array}
\usepackage{lmodern}			
\usepackage[T1]{fontenc}		
\usepackage[colorlinks=true, citecolor=cyan, urlcolor=magenta]{hyperref}
\usepackage{cleveref}
\usepackage[backend=bibtex, style=alphabetic, backref=true, firstinits=true, isbn=false]{biblatex}
\makeatletter
\renewbibmacro{in:}{}
\DeclareFieldFormat{pages}{#1}
\renewcommand*{\bibnamedash}{%
	\leavevmode\raise +0.6ex\hbox to 5.5ex{\hrulefill}.\space\space}

\InitializeBibliographyStyle{\global\undef\bbx@lasthash}

\newbibmacro*{bbx:savehash}{\savefield{fullhash}{\bbx@lasthash}}

\renewbibmacro*{author}{%
	\ifboolexpr{
		test \ifuseauthor
		and
		not test {\ifnameundef{author}}
	}
	{%
		\iffieldequals{fullhash}{\bbx@lasthash}
		{\bibnamedash\addcomma\space}
		{\printnames{author}}%
		\usebibmacro{bbx:savehash}%
		\iffieldundef{authortype}
		{}
		{%
			\setunit{\addcomma\space}%
			\usebibmacro{authorstrg}%
		}%
	}
	{\global\undef\bbx@lasthash}%
}
\makeatother
\newtheorem{propositionx}{Proposition}[section]
\newenvironment{proposition}
{\pushQED{\qed}\propositionx}
{\popQED\endpropositionx}
\newenvironment{propositionp}
{\pushQED{\qed}\propositionx}
{\popQED\endpropositionx}
\newtheorem{theoremx}[propositionx]{Theorem}
\newenvironment{theorem}
{\pushQED{\qed}\theoremx}
{\popQED\endtheoremx}
\newenvironment{theoremp}
{\pushQED{\qed}\theoremx}
{\popQED\endtheoremx}
\newtheorem{corollaryx}[propositionx]{Corollary}
\newenvironment{corollary}
{\pushQED{\qed}\corollaryx}
{\popQED\endcorollaryx}
\newtheorem{lemmax}[propositionx]{Lemma}
\newenvironment{lemma}
{\pushQED{\qed}\lemmax}
{\popQED\endlemmax}

\theoremstyle{remark}
\newtheorem*{remark*}{Remark}
\newtheorem{remarkx}{Remark}
\newenvironment{remark}
{\pushQED{\qed}\remarkx}
{\popQED\endremarkx}

\newcommand{\dd}{\,\mathrm{d}}

\newcommand{\bbB}{\mathbb{B}}
\newcommand{\bbC}{\mathbb{C}}

\newcommand{\bbN}{\mathbb{N}}

\newcommand{\bbR}{\mathbb{R}}
\newcommand{\bbS}{\mathbb{S}}

\newcommand{\bbZ}{\mathbb{Z}}

\newcommand{\calA}{\mathcal{A}}

\newcommand{\calD}{\mathcal{D}}
\newcommand{\calE}{\mathcal{E}}
\newcommand{\calF}{\mathcal{F}}

\newcommand{\calI}{\mathcal{I}}

\newcommand{\calK}{\mathcal{K}}
\newcommand{\calL}{\mathcal{L}}

\newcommand{\calN}{\mathcal{N}}

\newcommand{\calP}{\mathcal{P}}

\newcommand{\calR}{\mathcal{R}}
\newcommand{\calS}{\mathcal{S}}

\newcommand{\calV}{\mathcal{V}}
\newcommand{\calW}{\mathcal{W}}
\newcommand{\calX}{\mathcal{X}}
\newcommand{\calY}{\mathcal{Y}}
\newcommand{\calZ}{\mathcal{Z}}

\newcommand{\bfx}{\mathbf{x}}

\title{Hydrogen-like Schr\"odinger operators at low energies}
\author{Ethan Sussman}
\date{February 12th, 2025 (Last Updated; typos fixed), April 18th, 2022 (Preprint)}
\email{ethanws@mit.edu}
\address{Department of Mathematics, Massachusetts Institute of Technology, Massachusetts 02139-4307, USA}

\makeatletter
\@namedef{subjclassname@2020}{\textup{2020} Mathematics Subject Classification}
\makeatother

\begin{document}
	
\begin{abstract}
	Consider a Schr\"odinger operator on an asymptotically Euclidean manifold $X$ of dimension at least two, and suppose that the potential is of attractive Coulomb-like type. 
	Using Vasy's second 2nd-microlocal approach, ``the Lagrangian approach,'' we analyze -- uniformly, all the way down to $E=0$ -- the output of the limiting resolvent $R(E\pm i 0) = \lim_{\epsilon \to 0^+} R(E\pm i \epsilon)$. The Coulomb potential causes the output of the low-energy resolvent to possess oscillatory asymptotics which differ substantially from the sorts of asymptotics observed in the short-range case by Guillarmou, Hassell, Sikora, and (more recently) Hintz and Vasy. Specifically, the compound asymptotics at low energy and large spatial scales are more delicate, and the resolvent output is smooth all the way down to $E=0$. In fact, we will construct a  compactification of $(0,1]_E\times X$ on which the resolvent output is given by a specified (and relatively complicated) function that oscillates as $r\to\infty$ times something polyhomogeneous. As a corollary, we get complete and compatible asymptotic expansions for solutions to the scattering problem as functions of both position and energy, with a transitional regime. So, in summary, we develop the low-energy scattering theory of attractive Coulombic potentials in the time-independent formalism, in the process studying delicate behavior at low energy and large scales.
\end{abstract}
	
\maketitle

\tableofcontents

\section{Introduction}
\label{sec:introduction}

In this work we are concerned with scattering theory in the presence of Coulomb-like potentials. The long-range nature of the Coulomb potential modifies the large-scale asymptotics of the solutions of the relevant PDEs \cite{Yafaev}\cite{MelroseSC, MelroseGeometric}\cite{VasyLA}, the most well-studied of which being the time-independent Schr\"odinger equation, the time-dependent Schr\"odinger equation, and the wave equation. It has been known for a long while that the standard definitions of the S-matrix fail for the exact Coulomb potential, in both the time-independent and time-dependent settings --- see e.g. \cite{Dollard}\cite[\S XI.9]{RS}. The culprit is the behavior at large-$r$. 

The reason is easiest to explain in the time-independent (i.e.\ spectral side) setting, in which case the relevant PDE is the Schr\"odinger--Helmholtz equation on $\bbR^n$. First, we recall the short-range case. When the potential is short-range, the relevant ansatz for an incoming spherical wave is $e^{i \sigma r}$, where $r$ is the distance from the origin and $E=\sigma^2$ is the energy of the wave. That is, $\sigma$ is the frequency. Scattering theory involves the study of solutions $u$ of the PDE of the form 
\begin{equation}
u=e^{i\sigma r} A + e^{-i \sigma r} B,  
\label{eq:1}
\end{equation}
where $A,B$ are some sufficiently nice functions. Specifically, we want $A,B$ to be non-oscillatory. Roughly, they should be smooth functions of $1/(1+r^2)^{1/2}$ and the angular coordinates $\theta\in \bbS^{n-1}$, times some polynomial weight (which is not important in this introductory discussion). Thus, a function such as \cref{eq:1} is the sum of an incoming spherical wave and an outgoing spherical wave. The ``incoming profile'' is the leading term $A_0(\theta)$ in the large-$r$ asymptotic expansion of $A$. Usually, it is possible to prove that there exists a unique solution $u$ of the form \cref{eq:1} for specified $A_0$. The all-important S-matrix is then the function $A_0\mapsto B_0$ that sends the incoming profile to the outgoing profile $B_0$, the leading term in the large-$r$ asymptotic expansion of $B_0$. Likewise, the \emph{Poisson map} is the map $A_0\mapsto u$ which takes the incoming profile and produces the unique solution to the PDE with that incoming data. This is how the story goes in the short-range case. One can find proofs of these results in the more general setting of scattering on asymptotically conic manifolds in \cite{MelroseSC, MelroseGeometric}, the perspective in which matches our own.

In the long range case, solutions no longer look like \cref{eq:1}. Instead, solutions look like
\begin{equation}
u= r^{i \mathsf{Z}/2\sigma } e^{i\sigma r} A + r^{-i \mathsf{Z}/2\sigma } e^{-i \sigma r} B,
\label{eq:2} 
\end{equation}
for non-oscillatory $A,B$, where $\mathsf{Z}$ is the strength of the Coulomb potential. In atomic physics, where the equation is modeling an electron orbiting a positively charged nucleus, $\mathsf{Z}$ is the atomic number, which is traditionally denoted with the letter ``$Z$.' So, $e^{i \sigma r}$ is no longer a sufficient ansatz for an incoming spherical wave --- rather $r^{i\mathsf{Z}/2\sigma} e^{i \sigma r}$ is. The outgoing oscillations are no longer at frequency $\sigma$. Rather, the phase receives a logarithmic correction $\mathsf{Z} \log (r) /2\sigma$, known in the literature as the \emph{Coulomb phase}. This is the reason why the traditional definitions of the S-matrix fail for Coulomb-like potentials: $r^{ i\mathsf{Z}/2\sigma}A$ is not non-oscillatory as $r\to\infty$. The story is similar in the time-dependent setting; the usual definitions of the S-matrix (or related objects like the wave/M{\o}ller operators) receive corrections from the Coulomb phase. The resulting objects are known as the Dollard-modified operators, after Dollard's original work on the subject \cite{Dollard}. With some abuse of terminology, we can refer to \cref{eq:2} as the Dollard-modified ansatz as well. 

So far, this has all been for positive energy/frequency, $\sigma>0$. The behavior near $\sigma=0$ is of interest in part because, in the passage from the time-independent picture to the time-dependent picture, it is necessary to take a Fourier transform in $\sigma$ or $E=\sigma^2$, and $\sigma=0$ is included in the domain of integration. 
For example, for the wave equation and the time-dependent Schr\"odinger equation with a short-range potential, the dominant contributions to the long-time behavior in spatially compact regions comes from low-energy behavior. Physically, this makes sense --- high-energy wavepackets leave spatially compact regions rapidly, whereas low-energy wavepackets can linger and therefore contribute to the long-time asymptotics. So, if we want to understand the asymptotics in spacetime of solutions of these sorts of PDE, it is necessary to understand the low-energy limit of scattering theory.

Unfortunately, now we see a defect in the Dollard ansatz: the Coulomb phase diverges as $\sigma\to 0^+$. Thus, while the ansatz \cref{eq:2} is sufficient for understanding the Schr\"odinger--Helmholtz equation at positive energy, if we want to understand the corresponding time-dependent PDE then something better is needed. Dollard managed to make do because the multiplier $r^{\pm i \mathsf{Z}/2\sigma}$, while highly oscillatory as $\sigma\to 0^+$, is a bounded operator on $L^2$. However, the extraction of long-time asymptotics is a much more delicate matter than the existence of a suitable S-matrix as an operator $L^2\to L^2$. It is for this reason that we find it necessary to go beyond Dollard's ansatz. Note that these difficulties do not arise in the short range case: $e^{i \sigma r}$ is  a well-behaved ansatz all the way down to and including $\sigma=0$. Interesting phenomena still occur in this regime --- see e.g.\ \cite{GH1, GH2, GHK}\cite{VasyN0, VasyN0L}\cite{HintzPrice} for works in the same amount of generality as that pursued here --- but these have to do with the asymptotics of the coefficients $A,B$ in \cref{eq:1}, not the oscillatory prefactor.

One way to see that \cref{eq:2} fails in the $\sigma\to 0^+$ limit is to consider exactly what happens exactly \emph{at} $\sigma=0$. At $\sigma=0$, it turns out that the solutions of the PDE are still oscillatory, at least if $\mathsf{Z}>0$, which is the case we treat here. Specifically, the solutions look like 
\begin{equation}
u = e^{2i \sqrt{\mathsf{Z}r } } A + e^{2i \sqrt{\mathsf{Z}r}} B 
\label{eq:3}
\end{equation}
for non-oscillatory $A,B$ --- we will essentially prove a precise version of this in \S\ref{sec:0_operator}. 
This fact depends crucially on the presence of a Coulomb potential and on its sign. In the short-range case ($\mathsf{Z}=0$), the PDE is essentially Laplace's equation, for which solutions look like polynomials, and in the $\mathsf{Z}<0$ case, a generic solution of the PDE is exponentially growing as $r\to\infty$ (though the physically relevant solutions are typically exponentially decaying), as can be seen by replacing $\mathsf{Z}\to -\mathsf{Z}$ in \cref{eq:3}. The $\sqrt{r}$-type oscillatory behavior seen in \cref{eq:3} is specific to the case of attractive Coulombic potentials. 
In the Dollard ansatz \cref{eq:2}, there is no hint of this sort of oscillatory behavior. 
This is why Dollard's ansatz must fail in the $\sigma \to 0^+$ limit. 
Somehow, the phases $\sigma r$ and $\mathsf{Z} \log(r)/2\sigma$, one of which goes to 0 and the other of which blows up as $\sigma\to 0^+$, have to give way to the intermediate $2\sqrt{\mathsf{Z} r}$. 

The problem before us is therefore this: to come up with an alternative to the Dollard ansatz 
\begin{equation} 
r^{\pm i \mathsf{Z}/2\sigma } e^{\pm i \sigma r}
\label{eq:5}
\end{equation} 
that correctly describes the behavior in the $\sigma\to 0^+$ limit, and which reduces to \cref{eq:3} exactly at $\sigma=0$. The answer we present is\footnote{Actually, this is only correct for potential scattering. When the metric is variable, the ansatz receives another correction coming from one particular subleading term in the metric. We call this term ``$a_{00}$'' below.}
\begin{equation}
\exp \Bigg[ \pm i\Bigg(  r \sqrt{\sigma^2 + \frac{\mathsf{Z}}{r} }  + \frac{\mathsf{Z}}{\sigma} \operatorname{arcsinh} \Bigg(\sigma \sqrt{\frac{r}{\mathsf{Z}} } \Bigg)\Bigg) \Bigg]. 
\label{eq:4} 
\end{equation}
(The apparent singularity of the $\operatorname{arcsinh}$ term  at $\sigma=0$ is removable because $\operatorname{arcsinh}$ is odd. Thus, \cref{eq:4} is better behaved than \cref{eq:5} as $\sigma\to 0^+$.) 
It is an elementary exercise to show that, for each individual $\sigma>0$, this reduces to the Dollard spherical wave 
modulo smooth factors which can be absorbed into $A,B$, and likewise that, for $\sigma=0$, this reduces to the correct $\exp(\pm 2 i \sqrt{\mathsf{Z} r})$, again modulo nice factors. For $\sigma>0$, the $r\sqrt{\sigma^2 + \mathsf{Z}/r}$ contributes the main term, while $\operatorname{arcsinh}$ term gives the Coulomb phase $\mathsf{Z}\log (r)/2\sigma$. On the other hand, when $\sigma=0$ exactly, both terms contribute equally, each giving one factor of $\sqrt{\mathsf{Z} r}$.
So, the ansatz \cref{eq:4} 
successfully interpolates between the positive energy and low energy behavior while still approximately solving the relevant PDE, and it is the essentially unique (meaning modulo continuous factors) ansatz that does so. This is made precise by \Cref{cor:scattering} below. 

For the reader wondering where the complicated formula \cref{eq:4} comes from, we explain in the next subsection how the low-energy problem has a semiclassical structure at large-radii; the ansatz above is precisely the Liouville--Green (a.k.a. WKB or JWKB) approximation that results. The essential uniqueness of \cref{eq:4} comes from the essential uniqueness of the Liouville--Green approximation. Usually, it is the high-energy regime in which one encounters semiclassical phenomena. The fact that the low-energy, large-radius hydrogen atom has a semiclassical structure is not as well-known as it perhaps should be. 

Our goal in this manuscript is to develop the low-energy scattering theory of attractive Coulomb-like potentials along these lines, taking \cref{eq:4} as the ansatz describing an incoming spherical wave. Our focus is on describing the behavior of the limiting resolvent $R(\sigma^2 \pm i0)$. The reason for this focus is that the S-matrix and Poisson operators can both be constructed in terms of the two limiting resolvents, as can the spectral projector via one of Stone's theorems. Thus, the limiting resolvents can be taken as the fundamental objects in time-independent scattering theory. This is why our main theorem, \Cref{thm:main}, only talks about the resolvents. However, the reader more interested in scattering theory than spectral theory might prefer to think of \Cref{cor:scattering} as the main result.

We say nothing about the time-dependent formalism except for a few brief sketches. While time-independent scattering theory is important in its own right, the most convincing motivation comes from understanding the corresponding PDE in spacetime. Passing to the time-dependent formalism is, in some sense, ``just'' a matter of taking the Fourier transform. However:
\begin{itemize}
	\item whether one takes the Fourier transform in $\sigma$ or $E$ depends on whether one wants to study the wave/Klein--Gordon equation or the time-dependent Schr\"odinger equation, and different asymptotic forms result,
	\item in order to control the Fourier transform, one needs also the high energy, $\sigma\to\infty$, asymptotics; this is actually already understood, but it is orthogonal to the low-energy phenomena that we are focused on, 
	\item producing full asymptotic expansions on the spacetime side can involve a nontrivial amount of work, even once full asymptotic expansions are known on the spectral side. This work should not be difficult --- it boils down to studying the asymptotics of some simple oscillatory integrals --- but it also should not be trivial.
\end{itemize}
Thus, we have chosen to only deal with time-independent scattering theory here. The time-dependent equations are well worth studying, of course, but it is our view that, once the spectral side is understood, understanding the time-dependent side is more or less straightforward. In the short-range case, this is well established, and the long-range case should not be different in this regard. We intend to tackle the time-dependent Schr\"odinger equation for ionized hydrogen in future work.

This manuscript has been written in the language of the Melrose school of geometric microlocal analysis. This language is an incisive tool for studying multifaceted asymptotic expansions like those we encounter below. It provides a precise notion of what it means to have complete understanding of the asymptotics. Unfortunately, this language can also make for difficult reading. Let us therefore point out that, in \S\ref{ap:model}, we consider the case of an exact Coulomb potential. Then, the solutions of the differential equation are special functions, namely Whittaker functions. In that appendix, we apply our main results to make concrete statements about the asymptotics of the Whittaker functions in a poorly understood asymptotic regime. There, we reproduce and refine classical results from the literature of special functions. Hopefully, that appendix provides an accessible introduction to the problem studied here. The results are illustrated with numerics, unlike elsewhere in this paper.

The next subsection is a technical introduction to the paper, for general asymptotically conic manifolds, this being the natural amount of generality in which our results hold. One of its purposes is to record the main results of the paper. We will repeat there some of the discussion here, but with much more detail. 
Our main result is \Cref{thm:main}, which expresses, in precise terms, the fact that the resolvent output admits full asymptotic expansions in terms of elementary functions on a particular manifold-with-corners compactifying the low-energy regime --- see \Cref{fig:single_space}. To make this more concrete, we present in \Cref{cor:main1}, \Cref{cor:main2}, \Cref{cor:main3} the content of this result vis-a-vis explicit asymptotic expansions. Roughly, our result that the resolvent output $R(\sigma^2 \pm i0)f$ is, for Schwartz $f$, of the form 

\begin{equation}
\exp \Bigg[ \pm i\Bigg(  r \sqrt{\sigma^2 + \mathsf{Z}{r} }  + \frac{\mathsf{Z}}{\sigma} \operatorname{arcsinh} \Bigg(\sigma \sqrt{\frac{r}{\mathsf{Z}} } \Bigg)\Bigg) \Bigg] A
\label{eq:6}
\end{equation}
for $A$ which is almost a smooth function of $\sigma^2$, $1/r$, and the angular variables $\theta$. There are two reasons why the ``almost'' in the previous sentence is necessary:
\begin{itemize}
	\item The functions are not smooth, but instead polyhomogeneous. Polyhomogeneity is the natural generalization of smoothness to allow any generalized Taylor series, which may include logarithmic terms. For example, a polyhomogeneous function on $[0,1)_x$ is a function which admits well-behaved asymptotic expansions in functions of the form $x^j (\log x)^k$. 
	\item The $A$ in \cref{eq:6} is not polyhomogeneous in $\sigma$ and $1/r$. Rather, it is polyhomogeneous on the manifold-with-corners in \Cref{fig:single_space}. Concretely, this will mean polyhomogeneity in $\sigma$ and $\hat{x} = 1/(\sigma^2 r)$ as $\hat{x}\to 0^+$ and polyhomogeneity in $x=1/r$ and $\hat{E} = \sigma^2 r$ as $\hat{E}\to 0^+$, as well as in the angular variables $\theta$. 
\end{itemize}
For the reader not familiar with the use of blowups for this sort of asymptotic analysis, it is worth keeping in mind the simplest example: on the two-dimensional quadrant $[0,\infty)^2\subset \bbR^2_{x,y}$, the polar angle $\theta = \operatorname{arctan}(y/x)$ is not polyhomogeneous. Rather, it is polyhomogeneous on the blowup 
\begin{equation}
[[0,\infty )^2 ; \text{origin}] \cong [0,\pi/2]_\theta\times [0,\infty)_r. 
\end{equation}
Thus, in this example, a blowup just means using polar coordinates, a coordinate system which is singular with respect to the smooth structure on the original manifold $[0,\infty)^2$.

For another example, consider $x/(x+y)$. When $y=0$, this is identically $1$. When $x=0$, this is identically $0$. The limits $\lim_{x\to 0^+},\lim_{y\to 0^+}$ do not commute:
\begin{equation}
\lim_{x\to 0^+} \lim_{y\to 0^+} \frac{x}{x+y} \neq \lim_{y\to 0^+} \lim_{x\to 0^+} \frac{x}{x+y}.
\end{equation}
This is exactly the sort of behavior that polyhomogeneity excludes. The function $x/(x+y)$ is not polyhomogeneous on the quadrant $[0,\infty)^2$, but rather on the blowup $[[0,\infty )^2 ; \text{origin}]$. Indeed, 
\begin{equation}
\frac{x}{x+y} = \frac{\cos(\theta)}{\cos(\theta)+\sin(\theta)}
\end{equation}
is a perfectly smooth function of $\theta\in [0,\pi/2]$. We will make use of more complicated manifolds-with-corners below, but the one $X^{\mathrm{sp}}_{\mathrm{res}}$ we use in \Cref{thm:main} to describe the resolvents is just  $[[0,\infty )^2 ; \text{origin}]$ times $\bbS^{n-1}$, where the Cartesian coordinates $x,y$ are to be replaced by $\sigma^2,1/r$. (The definition of $X^{\mathrm{sp}}_{\mathrm{res}}$ involves an additional change of smooth structure, but we can ignore that for now.) 

Let us emphasize that we do not allow our potential to be singular at $r=0$ (except in the appendix). Thus, the exact Coulomb potential $1/r$ is excluded. This is for reasons of simplicity. It is not expected to be difficult to extend the results below to handle such singularities. The low-energy asymptotics near them is straightforward since, as far as the $r\to 0$ limit is concerned, the differential equation does not degenerate at low-energy. Consequently, unlike in the $r\to\infty$ case where the $\sigma\to 0^+$ limit requires a transitional asymptotic regime, no such complexities are present near a singularity. Unfortunately, developing this requires a detailed discussion of regular singular PDE beyond the elementary results required here. The added burden was felt to be not worth it. The ODE case is simple enough, so the discussion in \S\ref{ap:model} includes a singularity at $r=0$. 

\subsection{Technical introduction}
In \cite{MelroseSC}, Melrose introduced the programme of understanding the limiting absorption principle on asymptotically conic manifolds, denoted $X$ below, from the microlocal point of view, initially for the Laplacian and then for Schr\"odinger operators more generally \cite[\S16]{MelroseSC}. The case of fixed energy $E>0$ was dealt with first by Melrose \cite{MelroseSC}, then by  Hassell \& Vasy \cite{HassellVasy} and -- using a more modern approach -- Vasy \cite{VasyLA} again, while uniform estimates in the high energy (a.k.a.\ ``semiclassical,'' $E\to\infty$) limit have been established by Vasy \& Zworski \cite{VasyZworski} and Vasy \cite[\S5]{VasyLA}. More recently, the low energy $E\to 0^+$ behavior has been understood to a highly satisfactory degree in the works of Guillarmou, Hassell, and Sikora \cite{GHK} and Vasy \cite{VasyN0}\cite{VasyN0L} --- see also Hintz \cite[\S2]{HintzPrice}. (Complementarily, Guillarmou and Hassell \cite{GH1}\cite{GH2} consider the $E\to 0^-$ limit of the resolvent kernel.) While the previous results applied to all Schr\"odinger operators, sometimes under standard dynamical assumptions, the low energy results proven so far in this level of generality  apply only to Schr\"odinger operators without Coulomb-like (a.k.a. long range) terms. Indeed, a Coulomb-like potential has a serious effect on the asymptotics of formal Schr\"odinger eigenfunctions in the low energy limit --- this is true in the context of Euclidean potential scattering, and the general case inherits that complexity. 

In this paper, we study the case of an \emph{attractive} Coulomb-like potential using the framework of \cite{VasyLA}\cite{VasyN0L}, Vasy's 2nd second microlocal approach (the ``Lagrangian approach''). 
This is in contrast to his 1st second microlocal approach to the low energy limit, pursued in \cite{VasyN0}, which utilized variable order Sobolev spaces (as previously used in e.g.\ \cite[Proposition 5.28]{VasyGrenoble} for the $E>0$ case). 
The use of the term ``Lagrangian'' belies the fact that the framework of Lagrangian  (more properly \emph{Legendrian}) distributions does not appear explicitly in this approach --- in fact this is precisely the point: instead of ``module regularity \cite[\S6]{HassellMelroseVasy}\cite{HassellETAL1} for a Lagrangian submanifold'' -- which implies extra regularity \emph{outside} of that submanifold via some elliptic estimates -- we need only consider b-Sobolev regularity. We use a conjugation to move what would otherwise have been the energy-dependent Legendrian submanifold for which module regularity (in this case essentially the Sommerfeld radiation condition) is established to the zero section of the scattering cotangent bundle (which can be blown up to the fibers of the b-cotangent bundle over the boundary). The upshot is that the Sommerfeld radiation condition, in one of its forms, is replaced by a condition regarding b-Sobolev regularity. This is convenient for getting estimates which are uniform in $E$ because, unlike the relevant notion of module regularity prior to conjugation (see e.g. \cite{HassellMelroseVasy}\cite{HassellETAL1}), in which the relevant test module depends on $E$, the relevant notion of b-regularity (as stated in \cite[Theorem 1.1]{VasyLA}) does not depend so much on $E$. (The sharpest form of \cite[Theorem 1.1]{VasyLA} is phrased using sc,b-Sobolev regularity, which is a form of module regularity --- the key point here is that the relevant test module is $E$-independent.)

We will apply similar considerations to the study of the  $E\to 0^+$ limit of the limiting resolvents ``$P((E\pm i0)^{1/2})^{-1}=(P(0) - E \mp i0 )^{-1}$'' of the Schr\"odinger operator 
\begin{equation}
	P(0)=\triangle_g - \mathsf{Z} x + V , 
	\label{eq:P(0)}
\end{equation}
where $V$ is short-range and $\mathsf{Z}>0$, so the total potential $W= - \mathsf{Z}x+V$ is of attractive Coulomb-like type. Here $x\in C^\infty(X;\bbR^{\geq 0})$ is a boundary defining function (bdf), e.g. $1/\langle r \rangle$ when $X$ is asymptotically Euclidean, $r$ denoting the Euclidean radial coordinate. 
(Note: we follow the convention of parametrizing the spectral family $\smash{\{P(0)-E\}_{E\geq 0}}$ of $P(0)$ in terms of $\sigma=E^{1/2}$, so we write $P(\sigma)=P(0)-\sigma^2$.)
Singular versions of the operator \cref{eq:P(0)} first appeared in Schr\"odinger's model of atomic hydrogen -- or more generally hydrogenic ions -- hence the operators we consider could also be called ``hydrogen-like.''
In \cref{eq:P(0)}, $\triangle_g\geq 0$ is the positive semidefinite Laplacian associated with an asymptotically conic metric. The conjugated perspective complicates the family of operators under consideration (see \S\ref{sec:operator}) -- more so than in the case $\mathsf{Z}=0$, when the total potential is short-range -- but it greatly facilitates the derivation of low energy asymptotics. See below for a heuristic discussion (and \S\ref{sec:mainproof} for details). It should be noted that although we work with general asymptotically conic manifolds, our results are new even on exact Euclidean space. The existing literature on low energy asymptotics in the presence of a Coulomb-like term is quite sparse --- the only previous treatments the author is aware of are  \cite{Yafaev}\cite{Nakamura}\cite{SkibstedFournais}\cite{DerezinskiSkibsted1}\cite{DerezinskiSkibsted2}\cite{Bouclet}\cite{Skibsted}. In comparison to these earlier works, we require more of our potential, but the payoff is a complete understanding of asymptotics at spatial infinity. 

We comment briefly on our focus on the case $E>0$ and $\mathsf{Z}>0$. 
Considering the spectral family of a not-necessarily-attractive Coulomb-like Schr\"odinger operator at energy $E\in \bbR\backslash \{0\}$ and with attractivity strength $\mathsf{Z}\in \bbR\backslash \{0\}$ (the atomic number for the case of an electron orbiting an atomic nucleus, using appropriately natural units): out of the four cases (I) $E>0,\mathsf{Z}>0$ (attractive, scattering near zero energy), (II) $E<0,\mathsf{Z}>0$ (attractive, ellipticity near zero energy), (III) $E>0,\mathsf{Z}<0$ (repulsive, scattering near zero energy), (IV) $E<0,\mathsf{Z}<0$ (repulsive, ellipticity near zero energy), the first and fourth are the most tractable, as evidenced by the state of the literature on similar problems in exact Euclidean space --- see \cite{Yafaev}\cite{Nakamura} for work on cases (I) and (III) and \cite{SkibstedFournais}\cite{Skibsted} for case (I). In the more difficult case (II), for example, one can encounter an infinite sequence of bound states, as in the hydrogen atom --- see \cite[Theorem 11.6.7]{Monster}. Case (III), which has been studied partially in \cite{Yafaev}, is expected to be intermediate between (II) and (I) in terms of difficulty. While our focus is on (I), the pseudodifferential technology developed yields an easy treatment of case (IV).

The jumping off point for our analysis is the proof of a symbolic estimate, \Cref{thm:symbolic_main}, of $u \in \calS'(X)$ in terms of $f=P(E^{1/2})u$ that is uniform down to $E =0$. The estimate is structurally similar to the combined radial point and propagation estimates proven for $E>0$ in \cite[\S8]{MelroseSC}. We will formulate the estimate in a ``second microlocal'' framework akin to that in \cite{VasyLA}, as this  dovetails with the conjugated perspective, but a similar estimate can be articulated using function spaces analogous to the somewhat more standard \emph{variable order} sc-Sobolev spaces \cite{VasyOrig}\cite{VasyGrenoble}\cite{HassellETAL1}. To illustrate the idea, we rewrite the operator $P(\sigma)$ in terms of the coordinate $\hat{x} = x/\sigma^2$, which is appropriate for homogenizing the spectral parameter and Coulomb potential. To the relevant order, $\sigma^{-2} P(\sigma)$ is given in terms of $\hat{x},\sigma>0$, the latter of which we suggestively rename `$h$,' by
\begin{equation}
	\hat{P}(h) = -h^2 (\hat{x}^2 \partial_{\hat{x}} )^2 + h^2 \hat{x}^2 \triangle_{\partial X}  - 1 - \mathsf{Z} \hat{x},
	\label{eq:Prescaled}
\end{equation}  
which we consider as a 1-parameter family of operators on the exact cone $[0,\infty)_{\hat{x}}\times \partial X$. Note that $\sigma=h$ appears on the right-hand side of \cref{eq:Prescaled} as an effective \emph{semiclassical} parameter (somewhat surprisingly, since semiclassical problems typically arise in the study of the high energy regime rather than the low energy regime of interest here), and so $\smash{\hat{P}}$ can be studied as a semiclassical operator on an exact cone, an approach that seems to have first been undertaken by Nakamura \cite{Nakamura}. 
This conic problem differs from the conic problem arising in \cite{WangCone}\cite{GH1}\cite{GHK}\cite{VasyN0L}, which has no semiclassical parameter. The qualitative features of the semiclassical family defined by \cref{eq:Prescaled} are as follows: 
\begin{itemize}
	\item At the ``large'' end of the cone, $\{\hat{x}=0\}$, $\smash{\hat{P}}(h)$ is of real principal type for each $h>0$. This holds regardless of the sign of $\mathsf{Z}$ and reflects the dynamical structure of $P(\sigma)$ for $\sigma>0$.
	\item At $h=0$, one has real principal type propagation from $\hat{x} = 0$ to $\hat{x}=\infty$ (at finite frequencies).
\end{itemize}
The second of these is of secondary importance below.

The situation changes completely if $\mathsf{Z}<0$ or if $E<0$. Then, \cref{eq:Prescaled} is to be replaced by 
\begin{equation}
	\hat{P} (h) = -h^2 (\hat{x}^2 \partial_{\hat{x}} )^2 + h^2 \hat{x}^2 \triangle_{\partial X}  \pm 1 \pm \mathsf{Z} \hat{x},
	\label{eq:Prescaledpm}
\end{equation}
where the first sign is that of $-E$ and the second sign is that of $-\mathsf{Z}$, each possible pair of signs corresponding to one of the four cases (I), (II), (III), (IV) above. If $E<0$, then \cref{eq:Prescaledpm} is elliptic at the large end of the cone, and if $\mathsf{Z}<0$ then \cref{eq:Prescaledpm} is semiclassically elliptic near the small end of the cone. Even for the case when $X$ is one-dimensional, one subtlety of the cases (II) $E<0,\mathsf{Z}>0$ and (III) $E>0,\mathsf{Z}<0$ is understanding precisely the transition from ellipticity to nonellipticity that occurs at $h=0$ and $\hat{x}=1/|\mathsf{Z}|$. This subtlety arises, for example, in physicists' treatment of the WKB approximation, where Airy functions are used to patch quasimodes in the classically allowed and classically forbidden regions --- see \cite{Sakurai} for a standard treatment at a physicist's level of rigor. In case (II), the full partial differential operator \cref{eq:Prescaledpm} has an additional subtlety: its Hamiltonian flow has closed loops corresponding to classical Keplerian orbits. See \cite[\S11.6]{Monster} for a discussion of the effects of this on eigenvalue counting. 

Zooming back from the boundary, we will analyze $P(\sigma)$ on a mwc (``mwc'' standing for 
``manifold-with-corners,'' in the sense of Melrose's school \cite{MelroseCorners}) which we will denote $X^{\mathrm{sp}}_{\mathrm{res}} = [ [0,\infty)_E\times X; \{0\}\times \partial X;\frac{1}{2}]$, depicted in \Cref{fig:single_space}. This is the result of modifying the smooth structure of 
\begin{equation} 
	X^{\mathrm{sp}}_{\mathrm{res},0} = [ [0,\infty)_E\times X; \{0\}\times \partial X]
\end{equation} 
at the front face of the blow-up so that $(E+x)^{1/2}$ becomes a bdf. 
One can analyze $P(\sigma)$ by quantizing the Lie algebra $\calV_{\mathrm{sc,leC}}(X)$ of smooth vector fields on $X^{\mathrm{sp}}_{\mathrm{res}}$ that
\begin{enumerate}[label=(\Roman*)]
	\item are tangent to the level sets of $E$, 
	\item lie in $\varrho_{\mathrm{bf}_{00}} \varrho_{\mathrm{tf}_{00}} C^\infty(X_{\mathrm{res}}^{\mathrm{sp}} )\otimes \calV_{\mathrm{b}}(X) \subset \calV_{\mathrm{E}}(X^{\mathrm{sp}}_{\mathrm{res}})$,
\end{enumerate}
where $\varrho_{\mathrm{bf}_{00}},\varrho_{\mathrm{tf}_{00}}$ are as in \Cref{fig:single_space} and $\calV_{\mathrm{E}}(X^{\mathrm{sp}}_{\mathrm{res}})$ is the Lie algebra of smooth vector fields on $X^{\mathrm{sp}}_{\mathrm{res}}$. The quotient algebra $\calV_{\mathrm{sc,leC}}(X)/ \varrho_{\mathrm{bf}_{00}} \varrho_{\mathrm{tf}_{00}}\calV_{\mathrm{sc,leC}}(X)$ is commutative, so the corresponding $\Psi$DO calculus (which is closely related to the leC-calculus discussed below) is under symbolic control. This allows us to prove a half-Fredholm estimate involving variable order ``sc,leC''-Sobolev spaces that is uniform down to zero energy. The second-microlocal estimate \Cref{thm:symbolic_main} is a sharper version of this.

Although it is not our most general result, as the main propositions of \S\ref{sec:symbolic}, \S\ref{sec:mainproof} require less smoothness of the coefficients of the PDE, we will prove the following main theorem. We state it using some standard or semistandard terminology, which, if not standard, is recalled in \S\ref{subsec:preliminaries} below. The special case of Euclidean potential scattering off of an attractive asymptotically Coulomb-like potential is partially stated in \Cref{cor:main1}, \Cref{cor:main2}, and \Cref{cor:main3}.

\begin{figure}
	\begin{tikzpicture}[scale=.75]
		\draw (4,0) to[out = 90, in=180] (6,2);
		\draw[->] (6,2) -- (6,4) node[pos = 1.1] {$x$};
		\draw[->] (4,0) -- (0,0) node[pos = 1.05] {$E$};
		\draw[->,red] (3.9,.1) --++ (-1,0) node[above] {$\sigma$}; 
		\draw[->,red] (3.9,.1) to[out = 90, in=245] (4.15,1.05) node[left] {$\varrho_{\mathrm{bf}_{00}}$};
		\draw[->,red] (5.9,2.1) --++ (0,1) node[left] {$x^{1/2}$};
		\draw[->,red] (5.9,2.1) to[out = 180, in=25] (5.05,1.9) node[above] {$\varrho_{\mathrm{zf}_{00}}\;$};
		\draw[->,red] (4.5,1.5) --++ (-1,1) node[above] {$\varrho_{\mathrm{tf}_{00}}$};
		\node (bf) at (2,-.25) {$\mathrm{bf}$};
		\node (tf) at (4.95,.95) {$\mathrm{tf}$};
		\node (zf) at (6.25,3) {$\;\mathrm{zf}$};
	\end{tikzpicture}
	\caption{The mwc $X^{\mathrm{sp}}_{\mathrm{res}} = [[0,\infty)_{E} \times X;\{0\}\times \partial X; 1/2]$, with bdfs $\varrho_{\mathrm{bf}_{00}} = \mathsf{Z} x/(\sigma^2+\mathsf{Z} x) $, $\varrho_{\mathrm{tf}_{00}} = (\sigma^2+\mathsf{Z}x)^{1/2}$, and $\varrho_{\mathrm{zf}_{00}} = \sigma^2 / (\sigma^2+\mathsf{Z}x)$ in terms of $E^{1/2}=\sigma$. This is $X^{\mathrm{sp}}_{\mathrm{res},0}=[[0,\infty)_{E} \times X;\{0\}\times \partial X]$ with the smooth structure at the front face of the blow-up modified. The interior of the mwc is to the upper-left of the drawn boundary. 
	(Degrees of freedom associated with $\partial X$ omitted from the diagram.)}
	\label{fig:single_space}
\end{figure}
\begin{theorem}[Main theorem, rough version]
	The limiting resolvent output is of exponential-polyhomogeneous type on $X^{\mathrm{sp}}_{\mathrm{res}}$. 
\end{theorem}
More precisely:
\begin{theorem}[Main theorem]	\label{thm:main}
	Suppose that $(X,\iota,g_0)$ is an exactly conic manifold of dimension $\dim X=n\geq 2$, and let $x \in C^\infty(X)$, $x:X\to [0,\infty)$, denote a compatible boundary-defining-function (bdf), so that, near $\partial X$,
	\begin{equation}
		g_0 = \frac{\mathrm{d}x^2}{x^4} + \frac{g_{\partial X}}{x^2} 
	\end{equation}
	for some Riemannian metric $g_{\partial X}$ on $\partial X$. Let $g$ denote a fully classical asymptotically conic metric on $X$, which in this context means a Riemannian metric on $X^\circ=X\backslash \partial X$ which near $\partial X$ has the form 
	\begin{equation}
		g = g_0 + a_{00} \frac{\mathrm{d}x^2}{x^3} + \frac{\Gamma_{1,\partial X}\odot\mathrm{d}x}{x^2} + \frac{h_{1,\partial X}}{x} +  x^2 C^\infty(X; {}^{\mathrm{sc}}\mathrm{Sym}^2 T^* X)
	\end{equation}
	for some $a_{00} \in \bbR$, $\Gamma_{1,\partial X} \in \Omega^1(\partial X)$, and $h_{1,\partial X} \in C^\infty(\partial X;\operatorname{Sym}^2 T^* \partial X)$. 
	Given $\mathsf{Z}>0$ and $V\in x^{2}C^\infty(X;\mathbb{R})$, consider the Schr\"odinger operator 
	\begin{equation}
	P(0) = \triangle_g - \mathsf{Z} x + V : \calS'(X)\to \calS'(X), 
	\end{equation}
	where $\triangle_g$ is the (positive semidefinite) Laplacian. 
	For each $E>0$, let 
	\begin{equation}
	R(E \pm i 0) = R(E \pm i 0;\mathsf{Z}) : \calS(X)\to \calS'(X)
	\end{equation}
	denote the limiting resolvent from above or below the spectrum --- cf.\  Melrose \cite[\S14]{MelroseSC}. (That is, for any $f\in \calS(X)$, $u_\pm=R(E\pm i0)f$ is the unique solution to $P(0)u=Eu+f$ satisfying the weak Sommerfeld radiation condition \cite[\S11]{MelroseSC}.)
	Set  
		\begin{equation}
		\Phi(x;\sigma) =    \frac{1}{x} \sqrt{\sigma^2+\mathsf{Z} x-\sigma^2 a_{00} x} + \frac{1}{\sigma}(\mathsf{Z} -\sigma^2 a_{00}) \operatorname{arcsinh} \Big( \frac{\sigma}{x^{1/2}} \frac{1}{(\mathsf{Z}  - \sigma^2 a_{00})^{1/2}} \Big)
		\label{eq:misc_p00}
	\end{equation}
	for all $\sigma>0$ such that $\sigma^2 a_{00}< \mathsf{Z}$. 
	
		Then, for any Schwartz function $f\in \calS(X)$, the function $u_{0,\pm}$ on $X^{\mathrm{sp}}_{\mathrm{res}} \cap \{\mathsf{Z} > E a_{00}\}$ defined by 
		\begin{equation}
			u_\pm  = e^{\pm i \Phi(x;E^{1/2})} x^{(n-1)/2} (E+\mathsf{Z} x)^{-1/4} u_{0,\pm}
			\label{eq:u_decomp}
		\end{equation}
		(for $E$ such that $\mathsf{Z}>E a_{00}$) is polyhomogeneous on $X^{\mathrm{sp}}_{\mathrm{res}} \cap \{\mathsf{Z} > E a_{00}\}$, conormal of order zero, and smooth at $\mathrm{zf}^\circ \cup \mathrm{bf}^\circ$. 
		
		Moreover,  $R(E=0;\mathsf{Z}\pm i0):\calS(X)\to \calS'(X)$ is well-defined (e.g. as a strong limit of $R(E=0;\mathsf{Z}\pm i\epsilon)$ as $\epsilon\to 0^+$), and we can write $
		u_\pm(-;0)  = R(E=0;\mathsf{Z} \pm i0)f$ as 
		\begin{equation}
			u_\pm(-;0) =e^{\pm i \Phi(x;0)} x^{(n-1)/2} (\mathsf{Z} x)^{-1/4} u_{0,\pm}(-;0) 
			\label{eq:u_decomp0}
		\end{equation}
		($\Phi(-;0)$ being defined by removing the removable singularity of \cref{eq:misc_p00} at $\sigma=0$),
		where 
		\begin{equation} 
			u_{0,\pm}(-;0) \in C^\infty(X_{1/2})
		\end{equation} 
		is the restriction of $u_{0,\pm}$ to $\mathrm{zf} = \operatorname{cl}\{\sigma=0,x>0\} \subset X^{\mathrm{sp}}_{\mathrm{res}}$. 
		Thus, defining $u_\pm(-;E^{1/2})$ either as $R(E\pm i0)f$ for $E>0$ or $R(E=0;\mathsf{Z}\pm i 0)f$ for $E=0$, the formula \cref{eq:u_decomp} holds for all $E\geq 0$ sufficiently small.

\end{theorem}

For $z>0$, by $\operatorname{arcsinh}(z)$ we mean $\log(z+(1+z^2)^{1/2})$. 

\begin{remark}
	\label{rem:nonconstant}
	\Cref{thm:main} applies equally well for $\sigma$-dependent forcing $f\in C^\infty([0,\infty)_{\sigma^2}; \calS(X))$ as can be proven using the argument in \S\ref{sec:mainproof}.
\end{remark}
\begin{remark}
	In terms of the notation for spaces of polyhomogeneous functions used below, $u_{0,\pm} \in \calA^{(0,0),\calE,(0,0)}(X_{\mathrm{res}}^{\mathrm{sp}})$ for some index set\footnote{We will only consider index sets with the property that $(k,\kappa)\in \calE \Rightarrow (k+1,\kappa)\in \calE$.} $\calE\subset \{(z,\kappa)\in \bbC\times \bbN: \Re z \geq 0\}$ at tf.
	In fact, we can take
	\begin{equation}
		\calE = \{(k,\kappa)\in \bbN\times \bbN : \kappa \leq \lfloor k/2 \rfloor \}.
	\end{equation}
	See \S\ref{subsec:smoothness2}. 
	
	One can be even more precise than this: as the proof of \Cref{prop:normal_mapping_1} shows, the terms in the polyhomogeneous expansion of $u_{0,\pm}$ at $\mathrm{tf}$ having many logs are proportional to many powers of $E$. Specifically, the coefficient $a\in C^\infty(\mathrm{tf})$ of $\varrho_{\mathrm{tf}_{00}}^k\log^\kappa(\varrho_{\mathrm{tf}_{00}})$ for $(k,\kappa)\in \calE$ is actually in $\varrho_{\mathrm{zf}_{00}}^\kappa C^\infty(\mathrm{tf})$. In order for $(k,\kappa)\in \calE$ to hold, $2\kappa \leq k$, so any logarithm $\log^\kappa(\varrho_{\mathrm{tf}_{00}})$ in the asymptotic expansion at $\mathrm{tf}$ is suppressed by a factor of $E^\kappa$. 
	This is why $u_{0,\pm}(-;0)$ can be smooth on $X_{1/2} = \mathrm{zf}$ even though $u_{0,\pm}(-;-)$ is not necessarily (claimed to be) smooth at $\mathrm{zf}\cap \mathrm{bf}$. The statement of the theorem above is therefore slightly nonoptimal, and we will not introduce the terminology needed to state an optimal version. 
	\label{rem:indexset}
\end{remark}

In a standard manner -- cf.\ \cite[\S15]{MelroseSC} -- we can deduce from the conjunction of (1) an asymptotic summation argument, (2) \Cref{thm:main} (strengthened slightly by the first remark), and (3) the other results in the body of the paper below the following: 
\begin{corollary}
	\label{cor:scattering}
	Consider the setup of \Cref{thm:main}, and fix $E_0>0$ such that $E_0a_{00}<\mathsf{Z}$. 
	If $\alpha \in C^\infty(\partial X)$, then there exist 
	\begin{equation} 
		A,B \in \calA_{\mathrm{loc}}^{(0,0),\calE,(0,0)}(X_{\mathrm{res}}^{\mathrm{sp}})
	\end{equation} 
	such that $A|_{\mathrm{tf}\cup \mathrm{bf}} = \alpha\circ \pi$ (where $\pi:\mathrm{tf}\cup \mathrm{bf} \to \partial X$ denotes the restriction to $\mathrm{tf}\cup \mathrm{bf}$ of the composition of the blowdown map $X^{\mathrm{sp}}_{\mathrm{res}}\to [0,\infty)_\sigma\times X$ and the projection $[0,\infty)_\sigma\times X\to X$) and 
	\begin{equation}
		u = e^{-i \Phi(x;E^{1/2})} x^{(n-1)/2}(E+\mathsf{Z} x)^{-1/4}A + e^{+i \Phi(x;E^{1/2})}(E+\mathsf{Z} x)^{-1/4}x^{(n-1)/2} B
	\end{equation}
	solves the Schr\"odinger--Helmholtz equation $\triangle_g u -\mathsf{Z}x u + Vu = Eu$ for all $E\in [0,E_0]$.
	Moreover, $u$ is the unique solution in $\{E\leq E_0\}$ with this property.
\end{corollary}
This shows that $\exp(\pm i \Phi(x;E^{1/2})) x^{(n-1)/2}(E+\mathsf{Z}x)^{-1/4}$ serves as a notion of ``incoming/outgoing spherical wave'' in the presence of an attractive Coulomb-like potential that makes sense all the way down to $E=0$. In \Cref{cor:scattering}, $\alpha$ serves as a notion of ``incoming data.'' The proof actually constructs the asymptotic expansion of $A$ at $\mathrm{tf}\cup \mathrm{bf}$. A natural question, which we do not investigate here, is whether or not  $B|_{\mathrm{tf}\cup\mathrm{bf}}=\beta\circ \mathrm{bdn}$ for some $\beta \in C^\infty([0,\infty)_\sigma \times \partial X)$, where $\mathrm{bdn}:X_{\mathrm{res}}^{\mathrm{sp}}\to [0,\infty)_\sigma \times X$ is the blowdown map. Physically, \Cref{cor:scattering} describes the scattering of nonrelativistic electrons off of a hydrogenic nucleus, or alternatively nonrelativistic Bhaba scattering. 

Compare the following corollary of \Cref{thm:main} (strengthened slightly by \Cref{rem:indexset}) with \cite[Cor. 1.5]{Nakamura}\cite[Eq. 1.2]{SkibstedFournais}: 
\begin{corollary}[Asymptotics at $\mathrm{zf}$]
	\label{cor:main1}
	For each $l\in \bbR$, let $S^l(\bbR^n) = S^l_{1,0}(\bbR^n)$ denote the space of symbols of order $l$. Fix $n\geq 2$. 
	Suppose that $W \in C^\infty(\bbR^n)$ is a classical symbol of order $-1$, so that there exist (unique) $W_0,W_1,W_2,W_3,\cdots \in C^\infty(\bbS^{n-1})$ such that 
	\begin{equation}
		W(\bfx) -  (1-\chi(r))\sum_{k=0}^K \frac{W_k(\bfx/r)}{r^{k+1}} \in S^{-K-2}(\bbR^n) , 
	\end{equation}
	for each $K\in \bbN$, 
	where $r = \lVert \bfx \rVert$ and $\chi \in C_{\mathrm{c}}^\infty(\bbR)$ is identically equal to one in some neighborhood of the origin. 
	Suppose further that $W$ is attractive and Coulomb-like, meaning that $W_0 = - \mathsf{Z}$ for some constant $\mathsf{Z}>0$. Consider the Schr\"odinger--Helmholtz operator 
	\begin{equation} 
		P(\sigma) = \triangle - \sigma^2 + W
	\end{equation} 
	for $\sigma\geq 0$, where $\smash{\triangle=-(\partial_{x_1}^2+\cdots + \partial_{x_n}^2)}$ is the positive semidefinite Laplacian.
	Given $f\in \mathcal{S}(\bbR^n)$ (where $\calS(\bbR^n)$ denotes the set of Schwartz functions) let 
	\begin{equation} 
		u_\pm(\bfx;\sigma) = R(\sigma^2 \pm i0) f(\bfx) 
	\end{equation}
	denote the output of the limiting resolvent $R(\sigma^2 \pm i0) = \operatorname{slim}_{\epsilon\to 0^+} R(\sigma^2 \pm i \epsilon)$ applied to $f$ for $\sigma>0$. 
	
	Then, there exist functions $w_{\pm,0},w_{\pm,1},w_{\pm,2},\cdots \in C^\infty(\bbR^n)$ such that, for each $K\in \bbN$ and $\bfx \in \bbR^n$, 
		\begin{equation}
			u_\pm(\bfx;\sigma) = \sum_{k=0}^K w_{\pm,k}(\bfx) \sigma^{2k} + O_{\bfx,f,K}(\sigma^{2K+2})
			\label{eq:potential_as_2}
		\end{equation} 
		as $\sigma\to 0^+$, 
		where the $O_{\bfx,f,K}(\sigma^{2K+2})$ term is uniformly bounded (i.e.\ by $\smash{C_{\calK,f,K}\sigma^{2K+2}}$) in compact subsets $\calK\Subset \bbR^n$ worth of $\bfx \in \bbR^n$. 
		
		In fact, we can take $w_{\pm,k} \in \langle r \rangle^k \calA^{\calE_k}(\mathrm{zf})$, and then the error term in \cref{eq:potential_as_2} is an $O_{f,K}(\sigma^{2K+2})$ family of elements of $\langle r \rangle^{K+1}\calA^{\calE_{K+1}}(\mathrm{zf})$, where $\calE_k$ is the index set 
		\begin{equation}
			\calE_k = \{ (\kappa,\varkappa) \in \bbN\times \bbN: \varkappa \leq \min\{k,\lfloor \kappa/2 \rfloor \} \} .
		\end{equation}
		
		Moreover, $w_{\pm,0}$ solves the PDE $P(0)w_{\pm,0} = f$.  
\end{corollary}

(Of course, the $O_{\bfx,f,K}(\sigma^{2k+2})$ error also depends on $W$, though we do not write this dependence explicitly. The same applies to the other errors in \Cref{cor:main2}, \Cref{cor:main3} below.)
\Cref{cor:main1} applies in particular to any $W\in C^\infty(\bbR^n)$ which is equal, outside of some compact set, to \begin{equation} 
	 -\mathsf{Z}/r + \sum_{j=2}^J r^{-j} W_j(\bfx/r) + W_\infty, 
	 \label{eq:special_potential}
\end{equation} 
for $W_j\in C^\infty(\bbS^{n-1})$, $J\in \bbN$, and $W_\infty \in \mathcal{S}(\bbR^n)$. 
Thus, while we impose significant restrictions on the radial behavior of the potential (and the potential's regularity) in order to get full asymptotic expansions, there are no symmetry requirements on $W_1,W_2,\cdots$. We do, however, require that $W_0=-\mathsf{Z}$ is constant, so the Coulomb-like part of $W$ is required to be spherically symmetric. We remark that \Cref{prop:main} (see also \Cref{rem:less_classicality}) allows us to study more general symbolic $W$, but then instead of full asymptotic expansions we get only partial asymptotic expansions together with symbolic estimates for the remainders. We also remark that the classical symbols of order $-2$ on $\bbR^n$ are precisely those of the form $V=\langle r \rangle^{-2} U$ for $U$ a smooth function on the radial compactification $\smash{\bbB^n = \overline{\bbR}^n}$ of $\bbR^n$, this being diffeomorphic to a closed ball, so \Cref{cor:main1} applies to many potentials that are not of the form \cref{eq:special_potential}, e.g. $W=-\langle r \rangle^{-1}$.

In a sense made precise by \Cref{thm:main}, the asymptotic expansion \cref{eq:potential_as_2} can be refined into an asymptotic expansion in powers of $\smash{\hat{E} = \sigma^2 / (\sigma^2+1/\langle r \rangle)}$ (with a complicated oscillatory prefactor) whose error terms are uniformly bounded without a loss of decay as $K$ increases. 
The non-uniform expansion \cref{eq:potential_as_2} already stands in stark contrast with the situation  when $W$ is short range, where instead one only has e.g. in the case $n=3$  (according to \cite[Theorem 3.1]{HintzPrice})
\begin{equation}
	u_\pm(\bfx;\sigma) = w_{\pm,0}(\bfx) + w_{\pm,1}(\bfx) \sigma + O_{\bfx,f}(\sigma^2 \log \sigma).
\end{equation} 
This is sharp --- the singular $\sigma^2 \log \sigma$ term is the source of the main term in Price's law.

Given the setup of \Cref{cor:main1}, the known large-$r$ asymptotic expansion of $u_\pm(\bfx;\sigma) = R(\sigma^2 \pm i0) f$ is \cite{MelroseSC}:
\begin{itemize}
	\item there exist functions $t_{\pm,0},t_{\pm,1},t_{\pm,2},\cdots \in C^\infty(\bbS^{n-1} \times \bbR^+)$ such that, for each $\sigma>0$ and nonzero $\bfx \in \bbR^n$,
	\begin{equation}
		u_\pm(\bfx; \sigma) = r^{-(n-1)/2} e^{\pm i \sigma r} r^{\pm i\mathsf{Z}/2\sigma} \Big[ \sum_{k=0}^K t_{\pm,k}(\bfx/r,\sigma) r^{-k} + O_{\sigma,f,K}(r^{-(K+1)}) \Big]
		\label{eq:potential_as_1}
	\end{equation}
	for each $K\in \bbN$, where the $O_{\sigma,f,K}(r^{-(K+1)})$ term is uniformly bounded in compact subsets $\calK\Subset \bbR^+$ worth of $\sigma$. 
\end{itemize}
See \S\ref{ap:model} for the Whittaker case, where the $t_\pm(\sigma)$ are written down explicitly. 
Note that the existence of the expansion \cref{eq:potential_as_1} is contained in \Cref{thm:main} (from the asymptotic expansion at $\mathrm{bf}^\circ\subset \bbB^{\mathrm{sp}}_{\mathrm{res}}$). Note that the $r^{\pm i \mathsf{Z}/2\sigma}$ term in \cref{eq:potential_as_1} is singular as $\sigma\to 0^+$, which renders \cref{eq:potential_as_1} unsuitable to study the low energy limit. The situation is ameliorated in the following way --- according to \Cref{thm:main}, \cref{eq:potential_as_1} admits a repackaging that is \emph{uniform} down to $\sigma=0$: 
\begin{corollary}[Asymptotics at $\mathrm{bf}$]\label{cor:main2}
	Consider the setup of \Cref{cor:main1}. 
	Let $\varrho = (\sigma^2 r+1)^{-1}$, so that $r=(1-\varrho) \varrho^{-1} \sigma^{-2}$.
	There exist functions $\tau_{\pm,0},\tau_{\pm,1},\tau_{\pm,2},\cdots \in \calA^{\calE}(\bbS^{n-1}\times [0,1) )$ such that
	 for any $K\in \bbN$ and $\sigma>0$, 
		\begin{multline}
			u_{\pm} ( \theta (1-\varrho)\varrho^{-1} \sigma^{-2} ;\sigma )=   \varrho^{(n-1)/2} \sigma^{(2n-3)/2} \exp \Big( \pm i  \Big(\frac{1-\varrho}{\sigma\varrho} \Big) \Big(1 + \frac{\mathsf{Z} \varrho}{1-\varrho}\Big)^{1/2} \Big) \\
			\times \Big(  \Big(\frac{1-\varrho}{\mathsf{Z}\varrho}\Big)^{1/2} + \Big(1+ \frac{1-\varrho}{\mathsf{Z}\varrho}\Big)^{1/2} \Big)^{\pm i \mathsf{Z}/\sigma}  \Big[ \sum_{k=0}^K \tau_{\pm,k}(\theta,\sigma) \varrho^k  + O_{f,K}(\varrho^{K+1})  \Big]
			\label{eq:potential_as_3}
		\end{multline}
		holds for every $\theta \in \bbS^{n-1}$ as $\varrho \to 0^+ $, where 
		\begin{equation} 
			O_{f,K}(\varrho^{K+1})|_{\sigma \leq \bar{\sigma}} \in  \varrho^{K+1} L^\infty(\bbR^n_\bfx \backslash \bbB^n; \calA^\calE_{\mathrm{loc}} [0,\bar{\sigma})_\sigma ) 
		\end{equation} 
		for every $\bar{\sigma}>0$.  
\end{corollary}
Note that, for $\sigma$ bounded away from zero, $\varrho \in C^\infty([0,1]_{1/r} )$ and 
\begin{equation}
	\exp \Big( \pm i \frac{1}{\sigma \varrho} (1-\varrho) \Big(1 + \frac{\mathsf{Z} \varrho}{1-\varrho}\Big)^{1/2} \Big) \Big(  \Big(\frac{1-\varrho}{\mathsf{Z}\varrho}\Big)^{1/2} + \Big(1+ \frac{1-\varrho}{\mathsf{Z}\varrho}\Big)^{1/2} \Big)^{\pm i \mathsf{Z}/\sigma} \in e^{\pm i \sigma r} r^{\pm i \mathsf{Z}/2\sigma} C^\infty([0,1]_{1/r}) 
\end{equation}
uniformly, so \Cref{cor:main2} does imply \cref{eq:potential_as_1}. 

Finally, we have the following result regarding joint asymptotics as $r\to\infty$, $\sigma^2 \to 0$ together, with the product $\varsigma^2=r\sigma^2$ fixed:
\begin{corollary}[Asymptotics at $\mathrm{tf}$]\label{cor:main3}
	Consider the setup of \Cref{cor:main1}. There exist
	functions 
	\begin{equation} 
		v_{\pm,k,\varkappa}\in \hat{E}^{\varkappa} C^\infty(\bbS^{n-1}\times [0,\infty)_{\hat{E}} )
	\end{equation} 
	such that, for every $K\in \bbN$ and $\varsigma> 0$, 
		\begin{multline}
			u_\pm(\bfx;\varsigma/r^{1/2}) = r^{-(2n-3)/4} e^{\pm i r^{1/2} \sqrt{\varsigma^2 + \mathsf{Z}  } }  \bigg( \frac{\varsigma}{\mathsf{Z}^{1/2}}  +\sqrt{1+ \frac{\varsigma^2 }{\mathsf{Z}} }\bigg)^{\pm i \mathsf{Z} r^{1/2}/\varsigma}\\ 
			\Big[ \sum_{k=0}^K\sum_{\varkappa=0}^{\lfloor k/2 \rfloor} \frac{v_{\pm,k,\varkappa}(  \bfx/r, \varsigma^2)}{r^{k/2}} \log^\varkappa (r) + O_{\varsigma,f,K}\Big( \frac{\log^{\lfloor (K+1)/2 \rfloor} (r)}{r^{(K+1)/2}} \Big)  \Big],
		\end{multline}
		as $r\to \infty$, where $O_{\varsigma,f,K}( r^{-(K+1)/2}\log^{\lfloor (K+1)/2 \rfloor}(r))$ is uniformly bounded by $r^{-(K+1)/2}\log^{\lfloor (K+1)/2 \rfloor} (r) $ in compact subsets worth of $\varsigma \in [0,\infty)$. 
		
		Moreover, $\rho^{k/2} v_{\pm,k,\varkappa}(\theta,1/\rho)$ is smooth in $\rho\in [0,\infty)$, for each $\theta \in \bbS^{n-1}$. Thus, for $\varsigma\geq 1$, $v_{\pm,k,\varkappa}(\theta,\varsigma^2) = O_{f,k,\varkappa}(\varsigma^k)$.
	
\end{corollary}
We emphasize that these compound asymptotics are in a different regime than the one relevant to the short-range case (this being the regime of $r\to\infty$ for $r\sigma$ fixed, not $r\sigma^2$ fixed) --- see \Cref{rem:scaling}. 

Statements analogous to \Cref{cor:main1}, \Cref{cor:main2}, and \Cref{cor:main3} apply to the $A,B$ in \Cref{cor:scattering}.

A few remarks regarding \Cref{thm:main}:
\begin{remark}
	The apparent singularity of $u_\pm$ in \Cref{thm:main} when $\mathsf{Z} = \sigma^2 a_{00}$ is fictitious --- we can write 
	\begin{equation}
	u_\pm = e^{\pm i \sigma/x} x^{(n-1)/2\mp i \mathsf{Z}/2\sigma} v_{0,\pm}
	\end{equation}
	for $v_{0,\pm}:(0,\infty)_\sigma\times X\to \bbC$ smooth. In terms of $v_{0,\pm}$, $u_{0,\pm} = e^{\mp i \Phi} (E+\mathsf{Z} x)^{1/4} e^{\pm i \sigma/x} x^{(n-1)/2\mp i \mathsf{Z}/2\sigma} v_{0,\pm}$. Thus, $u_{0,\pm}$ is singular when $\mathsf{Z} = \sigma^2 a_{00}$. The left-hand side of  \cref{eq:u_decomp} is smooth for all $\sigma>0$, but we have written it as a product of two functions which both have singularities when $\mathsf{Z} = \sigma^2 a_{00}$. 
	The particular form of $\Phi$ in \Cref{thm:main} is needed to get uniform estimates down to $\sigma=0$, but it introduces a fictitious singularity at some positive $\sigma$ when $a_{00}>0$. 
\end{remark}

\begin{remark}
	Some explicit bounds for $u_{0,\pm}$ in the b-Sobolev spaces 
	\begin{equation} 
		H_{\mathrm{b}}^{m,l}(X) = \{u\in \calS'(X): Lu \in L^2(X,g)\text{ for all }L\in \Psi_{\mathrm{b}}^{m,l}(X)\}
	\end{equation} 
	will be found in \S\ref{sec:mainproof}. We are indexing the b-Sobolev spaces such that $\smash{H_{\mathrm{b}}^{0,0}(X)}$ is equal to $L^2(X,g)$ rather than $L^2(X,g_{\mathrm{b}})$ for some b-metric $\smash{g_{\mathrm{b}}}$. This is the (somewhat nonstandard) convention followed in \cite{VasyLA}\cite{VasyN0L}, and it is convenient here for the same reasons. See \cite{APS}\cite{VasyGrenoble} for a pedagogical introduction to the b-calculus. 
	
	We write $X_{1/2}$ to denote the mwc canonically diffeomorphic to $X$ over the interior, with the smooth structure at the boundary modified so that $\smash{x_{1/2}=2^{-1/2} x^{1/2}}$ becomes a bdf (the factor of $2^{-1/2}$ being a convenient choice). 
	(Note that, in terms of the mwb $X$, the mwb $X_{1/2}$ is canonically defined without needing to fix a choice of bdf for $X$, unlike the case for what Wunsch calls $X_{\mathrm{q}}$ in \cite[\S4]{Wunsch}. However, since we are fixing a boundary-collar once and for all, this is not a crucial point.)
	The b-Sobolev spaces are convenient to work with in part because (except for indexing) they do not depend on whether we use $x_{1/2}$ or $x$ as a bdf --- that is, at the level of sets
	\begin{equation} 
		H_{\mathrm{b}}^{m,l}(X) = H_{\mathrm{b}}^{m,2l+n/2}(X_{1/2}), 
	\end{equation}
	with an equivalence at the level of Banach spaces. 
	A refined estimate of $u_{0,\pm}$ in terms of the ``leC-Sobolev spaces'' 
	\begin{equation} 
		H_{\mathrm{leC}}^{m,s,\varsigma,l,\ell}(X)=\smash{\{H_{\mathrm{leC}}^{m,s,\varsigma,l,\ell}(X)(\sigma)\}_{\sigma\geq 0}}
	\end{equation} 
	is given in \Cref{prop:misc_estimate}. The leC-Sobolev spaces play the same role here as the sc,b,res-Sobolev spaces in \cite[Theorem 1.1]{VasyN0L}.\footnote{We will omit the comma between ``sc,b.'}
	Although it may be natural  to develop \emph{doubly-} second-microlocalized Sobolev spaces $\smash{H_{\mathrm{sc},\mathrm{sc}_{1/2},\mathrm{b}}^{m,s,\varsigma,l}(X)}$, refining both $H_{\mathrm{scb}}^{m,s,l}(X)$ and $H_{\mathrm{scb}}^{m,\varsigma,2l}(X_{1/2})$, we will not do so here, since the leC-Sobolev spaces suffice for the proof of \Cref{thm:main}.
\end{remark}

\begin{remark}
	\label{rem:less_classicality}
	We also have analogues of \Cref{thm:main} dropping the classicality assumptions regarding the metric and the potential. If $g$ is a symbolically asymptotically conic metric in the sense below (precisely the sort of metric considered in \cite{VasyLA}) and $V \in S^{-3/2-\delta}(X) = x^{3/2+\delta} S^0(X)$ for $\delta>0$, the conclusion of the theorem holds except that ``$u_{0,\pm}\in \calA^{(0,0),\calE,(0,0)}(X^{\mathrm{sp}}_{\mathrm{res}})$'' is replaced by the weaker 
	\begin{align}
		\begin{split} 
		u_{0,\pm} &\in \calA^{0,0,0}(X^{\mathrm{sp}}_{\mathrm{res}}), \\
		\chi u_{0,\pm} &\in C^\infty([0,\infty)_E\times X) \text{ for all }\chi \in C_{\mathrm{c}}^\infty(X^\circ).
		\end{split}
	\end{align}
	Moreover, if $g$ is classical to $\alpha_1$st order, $\alpha_1>1$, and $V$ is classical to $\alpha_2>3/2$ order (so, for instance, $V\in x^2 C^\infty(X)+S^{-3/2-\delta_2}(X)$ is classical to $3/2+\delta_2$th order, and any sc-metric in the sense of \cite{MelroseSC} is classical to all orders, while the metrics considered by Vasy in \cite{VasyLA} are classical to $>1$ order), then, setting $\delta_1 = \min\{\alpha_1-1,\alpha_2-1\}$ and $\delta_0 = \min\{\alpha_1-1,\alpha_2-3/2\}$, we have
	\begin{equation} 
		u_{0,\pm} \in \calA^{(0,0),\calE,(0,0)}(X^{\mathrm{sp}}_{\mathrm{res}}) + \calA_{\mathrm{loc}}^{((0,0),\delta_1),2\delta_0-,(0,0)}(X^{\mathrm{sp}}_{\mathrm{res}}) \subset  \calA_{\mathrm{loc}}^{((0,0),\delta_1),(\calE,2\delta_0-),(0,0)}(X^{\mathrm{sp}}_{\mathrm{res}})
	\end{equation} 
	(assuming, for simplicity, that $\delta_0\notin \bbN$)
	where $\calA_{\mathrm{loc}}^{((0,0),\delta_1),(\calE,2\delta_0-),(0,0)}(X^{\mathrm{sp}}_{\mathrm{res}})$ is the Fr\'echet space of conormal distributions on $X^{\mathrm{sp}}_{\mathrm{res}}$ that 
	have partial polyhomogenous expansions at each of $\mathrm{bf},\mathrm{tf},\mathrm{zf}$, with merely conormal remainders at order $\delta_1$ at $\mathrm{bf}$ and at order $2\delta_0-\varepsilon$ at $\mathrm{tf}$ for every $\varepsilon>0$, 
	with a full expansion at $\mathrm{zf}$. Elements of this space are smooth at $\mathrm{zf}$ in the sense of defining an element 
	\begin{equation} 
		C^\infty([0,\infty)_{\varrho_{\mathrm{zf}_{00}}} ; \calA^{(\calE,2\delta_0-)} (X_{1/2})) = C^\infty([0,\infty)_{\varrho_{\mathrm{zf}_{00}}} ; \calA^\calE (X_{1/2})) + C^\infty([0,\infty)_{\varrho_{\mathrm{zf}_{00}}} ; \calA^{2\delta_0-} (X_{1/2}))
	\end{equation} 
	via restriction to a small neighborhood of $\mathrm{zf} \subset X_{\mathrm{res}}^{\mathrm{sp}}$. This implies smoothness at $\mathrm{zf}^\circ$.\footnote{In this paper, we use `$\circ$' and the term ``interior'' in the mwc-theoretic sense; for instance, for any boundary hypersurface $\mathrm{f}$, $\mathrm{f}^\circ$ is the collection of points in $\mathrm{f}$ that are not in any other boundary hypersurface.}

	We refer to \cite[Eq. 22]{MelroseCorners}\cite[Def. 2.13]{HintzPrice} for the definition of the spaces of conormal distributions with partial polyhomogeneous expansions. We follow the notational conventions in \cite{HintzPrice}, \emph{mutatis mutandis}.
	
\end{remark}

\begin{remark}
	\label{rem:scaling}
	Note that polyhomogeneity on the mwc $X^{\mathrm{sp}}_{\mathrm{res}}$ with index set $\calE$ at $\mathrm{tf}$ and smoothness elsewhere is equivalent to polyhomogeneity on 
	\begin{equation}
		X^{\mathrm{sp}}_{\mathrm{res},0} = [ [0,\infty)_E\times X; \{0\}\times \partial X]
	\end{equation} 
	with index set $\{(k/2,\kappa) : (k,\kappa)\in \calE\}$ at the front face $\mathrm{tf}$ of the blow-up and index set $\bbN$ at the faces $\mathrm{bf}$ and $\mathrm{zf}$. 
	So, we have a Taylor series in powers of $x$ for $\sigma>0$ and a Taylor series in powers of $x^{1/2}$ at $\sigma = 0$, with a Taylor series in powers of $E=\sigma^2$ for fixed $x>0$ (with some logs in the expansion at $\mathrm{tf}$, except exactly at $\mathrm{zf}$). 
	We remark that the case of nonconstant $a_{00}$ involves a more general polyhomogeneity statement on $X^{\mathrm{sp}}_{\mathrm{res},0}$; for fixed $E>0$, this is discussed already in \cite[\S14]{MelroseSC}.
	
	(The use of $(E+\mathsf{Z}x)^{1/2}=(\sigma^2+\mathsf{Z}x)^{1/2}$ as a bdf for the front face $\mathrm{tf}$ of $X^{\mathrm{sp}}_{\mathrm{res}}$ is a convenient but otherwise arbitrary convention. The function $\smash{(E+x)^{1/2}=(\sigma^2+x)^{1/2}}$ would also work, as already utilized in \Cref{cor:main1}, \Cref{cor:main2}, \Cref{cor:main3}, but it is less convenient.) Although $X^{\mathrm{sp}}_{\mathrm{res,0}}$ is diffeomorphic to the mwc used in \cite{VasyN0L}\cite{HintzPrice} to study the low energy resolvent in the short-range case, it differs from their mwc as a compactification of $(0,1]_\sigma\times X$. Indeed, besides having a different smooth structure at $\{0\}\times X^\circ$, our blow-up resolves the ratio $E/x$ rather than $\sigma/x$, the latter being the ratio parametrizing the front face of the blow-up in \cite{VasyN0L}\cite{HintzPrice}.
\end{remark}
\begin{remark} 
\label{rem:exact}
When 
\begin{itemize}
	\item $g=g_0$ is an exactly conic metric and 
	\item  $f,V \in C_{\mathrm{c}}^\infty(X^\circ)$, 
\end{itemize}
the radial dependence of $u_\pm$ near infinity can be solved  for (up to a multiplicative constant) exactly via separation of variables. When separating variables, we need only solve the ``radial ODE,'' which ends up being a Whittaker ODE for $\sigma>0$ and a Bessel ODE for $\sigma=0$.
This motivating example is discussed in Appendix \S\ref{ap:model}, where the consequences of \Cref{thm:main} for solutions to Whittaker's ODE are discussed. The necessity of the factor of $(\sigma^2+\mathsf{Z} x)^{-1/4}$ in  \cref{eq:u_decomp}  can be read off of the asymptotics of the Bessel functions (or can alternatively be deduced from the structure of the radial ODE, with slightly more work). See \cref{eq:misc_asz}. 
Even in this exactly conic case, which was studied in \cite{Yafaev}, \Cref{thm:main} (or at least the part specifying the existence of a full asymptotic expansion) seems to be novel.
\end{remark} 
\begin{remark}
	\label{rem:intuition}
At first, it may seem somewhat paradoxical that at zero energy an attractive (rather than repulsive) Coulomb-like Schr\"odinger operator has scattering behavior. The paradox is resolved if we realize that, if a classical particle traveling in an attractive Coulomb force field has zero energy, then it must have more kinetic energy than if the force field were absent (the point being that ``zero energy'' refers to energy as measured \emph{with} the potential present, not without). 

It may also seem somewhat surprising that the presence of the attractive Coulomb potential allows us to avoid the b-analysis of the Laplacian in \cite{VasyN0L}, but such considerations are expected to be necessary in order to understand the $E\to 0^-$ limit. Roughly speaking, the presence of an attractive Coulomb potential has moved the locus of b-analysis infinitesimally negativewards -- more specifically negativewards on the front face of $[\bbR_E\times X;\{0\}\times \partial X]$, to its intersection with the lift of $\{E=-\mathsf{Z} x\}$ -- where it is irrelevant to the $E\to 0^+$ limit studied here. 
\end{remark}

The proof of \Cref{thm:main}, which is a special case of \Cref{prop:main}, is spread throughout \S\ref{sec:mainproof}, with a key estimate proven in \S\ref{sec:symbolic} (and other lemmas appearing in \S\ref{sec:calculus}, \S\ref{sec:operator}, \S\ref{sec:0_operator}). 
The conormality component of the theorem and smoothness at $\mathrm{zf}^\circ$ are deduced in \Cref{prop:grace}. From there,
classicality is deduced via an inductive argument using an explicit parametrix for the ``b,leC-normal operator'' of the conjugated spectral family --- see \Cref{prop:inductive_smoothness}. Under only partial classicality assumptions, the inductive step of the argument can only be carried out finitely many times, and the upshot is \Cref{prop:main}. The analysis in \S\ref{sec:mainproof} is essentially the analysis of a family of ODEs and applies equally well in the $n=1$ case. The analysis in \S\ref{sec:symbolic} is a uniform version of the sc-analysis of the Laplacian carried out in \cite{MelroseSC}\cite{VasyLA}, where we rely on the presence of the attractive Coulomb potential to prevent a full degeneration (from the perspective of the b-calculus) of the spectral family to the Laplacian at zero energy. 

We note four deficiencies of our treatment, which can form the basis for further work: 
\begin{enumerate}
	\item while we discuss the output and mapping properties of the low energy resolvent, we do not discuss its Schwartz kernel (as Guillarmou, Hassell, and Sikora \cite{GHK} do in the case of a short-range potential),
	\item we only treat the case when the Coulomb-like potential has the form $\mathsf{Z} x$ for constant $\mathsf{Z}>0$, while more general $\mathsf{Z} \in C^\infty(\partial X;\bbR^+)$ (and $a_{00} \in C^\infty(\partial X;\bbR)$) may be of interest \cite[\S14]{MelroseSC} (if $\mathsf{Z} \in C^\infty(\partial X;\bbR^+)$, a change of coordinates from $x$ to $x_0=\mathsf{Z} x$ does not necessarily preserve the form of the metric we require, so we cannot just reduce the case of nonconstant $\mathsf{Z}$ to the constant case via a change of bdf),  
	\item it should be possible to study the resolvent for $\sigma$ in a cone $\{\sigma \in \bbC: \operatorname{arg}(\sigma) \in [0,\theta)\}$ for $\theta\in (0,\pi/2)$, in accordance with the work of Skibsted and Fournais \cite{SkibstedFournais}\cite{Skibsted} in the Euclidean case -- see \cite[Theorem 1.2]{Skibsted} in particular -- thus giving a version of the limiting absorption principle that is uniform down to zero energy (as opposed to the \emph{ad hoc} use of $\mathsf{Z}\pm i 0$ in \Cref{thm:main} to describe the zero energy limit), and 
	\item it should be possible to extend \Cref{thm:main} to the case where the Coulomb potential and manifold possess conic singularities, so as to handle an exact Coulomb potential on $[\bbR^d;\{0\}]$ as a special case. (Indeed, blowing up the origin of Euclidean space and using $r$ as a bdf, the Coulomb potential $\mathsf{Z}/r$ and spectral term $\sigma^2$ are both lower order than the Euclidean Laplacian with respect to the b-calculus at $r=0$ in terms of both decay and regularity.) The exact 1D model problem discussed in \S\ref{ap:model} is an example of this. 
\end{enumerate}

Along with some analysis of the $E<0$ case, the present work can serve as input to the study of the Klein--Gordon equation on (not necessarily axially symmetric) asymptotically Schwarzschild spacetimes (away from the event horizon) in the spirit of Hintz's recent treatment of Price's law \cite{HintzPrice} on asymptotically subextremal Kerr spacetimes. This problem was the original motivation for the present work.  
Some heuristic investigations in this direction have been undertaken by physicists \cite{Detweiler}\cite{HodTsvi}\cite{KoyamaTomimatsu}\cite{KoyamaTomimatsu2}\cite{BurkoKhanna}\cite{KonoplyaETAL}\cite{BarrancoETAL0}\cite{BarrancoETAL}, usually in the case of a spherically or axially symmetric black hole spacetime (e.g. Schwarzschild, Reissner--Nordstr\"om, Kerr, etc.). Indeed, \cite{KoyamaTomimatsu} treat their problem (pointwise temporal decay rates) using an uncontrolled approximation of the radial part of their PDE as a Whittaker ODE, this being equivalent to the radial part of the time-independent Schr\"odinger equation in an attractive Coulomb potential. 
This problem can be solved exactly in terms of special functions (see \S\ref{ap:model}), but we do not need to do so in order to understand the asymptotics. 

Moreover, while the asymptotics for fixed $E>0$ can be read off of the exact solution (using the large argument expansions of Whittaker functions), the exact solution does not even help much in understanding the $E\to 0$ limit. Instead of relying on the Whittaker approximation, it is possible to exactly solve the Klein--Gordon equation on the Schwarzschild exterior (at the level of individual spherical harmonics) using confluent Heun functions \cite[\S IV]{BarrancoETAL}. When this paper was first written, it was an open problem to rigorously establish temporal decay at a polynomial rate (as this involves knowing the asymptotics of the confluent Heun functions in some regime where the argument and a parameter are both taken to infinity). 
Since then, \cite{Moortel}\cite{Moortel2} proved the required result.
For more complicated spacetimes, no exact solution is possible. A more robust analysis is therefore necessary.

For other mathematical work on Price's law and similar problems, see Donninger--Schlag--Soffer \cite{DSS1}\cite{DSS2}, Metcalfe--Tataru--Tohaneanu \cite{TataruPrice}\cite{Tatarut3}, Morgan--Wunsch \cite{MorganThesis}\cite{MorganWunsch}, and Looi \cite{Looi}. The Klein--Gordon equation on Kerr-like spacetimes is subtle --- for instance, on subextremal Kerr, Shlapentokh-Rothman \cite{SR} has constructed finite energy solutions that grow exponentially in time  by exploiting superradiant instability.

We remark finally that there does not appear to be much overlap between the present work and similarly titled physics papers, e.g. \cite{Burke1}\cite{Burke2}\cite{Burke3}\cite{Burke4}\cite{Burke5}, because these concern the scattering of electrons off \emph{atomic} hydrogen, not ionic hydrogen. The former consists of one hydrogen nucleus and one electron (different from the incoming electron) which is tightly bound to the nucleus. Thus, atomic hydrogen is electrically neutral, and the incoming electron feels at most a dipole interaction, which is short-range.

\subsection{Preliminaries and Outline}
\label{subsec:preliminaries}

Recall that an exactly conic manifold $X$ (according to one of several essentially equivalent definitions) consists of the following data: an arbitrary (smooth, connected) compact manifold-with-boundary $X_0$, a Riemannian metric $g_{\partial X} \in C^\infty(X;\operatorname{Sym}^2 T^* \partial X_0)$ on the closed manifold $\partial X_0 $, and an embedding 
\begin{equation}
	\iota: \hat{X}= [0,\bar{x})_x\times \partial X_0 \to  X_0
\end{equation}
(the ``boundary-collar'') which is a diffeomorphism between the cylinder $[0,\bar{x})_x\times \partial X_0$ for some $\bar{x}>0$ and a neighborhood of $\partial X_0$. Note that $\iota$ satisfies $\mathrm{id}|_{\partial X_0} = \iota(0,-)$. We notationally conflate $X$ with the underlying $X_0$. We also conflate the image of $\iota$ with $[0,\bar{x})_x\times \partial X_0$. 
(This is an alternative, if a somewhat crude one, to Melrose's use in \cite{MelroseSC} of the inward pointing normal bundle of $\partial X_0$.)
Given an exactly conic manifold $X$, a Riemannian metric $g_0$ on $X^\circ$ is called an exactly conic metric (this notion being defined relative to $\iota$) if 
\begin{equation}
g_0 \in C^\infty(X; {}^{\mathrm{sc}}\mathrm{Sym}^2 T^* X), \qquad g_0 = x^{-4} \mathrm{dx}^2+x^{-2} g_{\partial X} \text{ near }\partial X.  
\end{equation}
Below, we write $x:X\to [0,\infty)$ to denote a globally defined boundary-defining function, compatible with the boundary-collar $\iota$ in the sense that $x(\iota(x_0,y))=x_0$ for all $x_0 \in [0,\bar{x} )$.  
(Euclidean space fits into this framework, where by ``Euclidean space'' we mean the radial compactification of $\bbR^d$, endowed with the induced metric. Here $x=1/\langle r\rangle$, $r$ being the Euclidean radial coordinate. See \cite[\S1]{MelroseSC}.) 
We will also consider (symbolically) asymptotically conic metrics, which -- according to our use of the term ``asymptotically'' -- are Riemannian metrics $g$ on $X^\circ$ of the form 
\begin{equation}
g - g_0 \in x C^\infty(X; {}^{\mathrm{sc}}\mathrm{Sym}^2 T^* X)+ S^{-1-\delta} (X; {}^{\mathrm{sc}}\mathrm{Sym}^2 T^* X)
\label{eq:g_ass}
\end{equation}
for some $\delta>0$. Hence, $g$ is conic up to a classical subleading term -- decaying faster than $g_0$ by one order of $x$ -- and a merely symbolic error decaying slightly faster than that.
More generally, for $\alpha_1>1$, we will say that the asymptotically conic metric $g$ is ``classical to $\alpha_1$th order'' if \cref{eq:g_ass} can be strengthened to 
\begin{equation}
	g - g_0 \in x C^\infty(X; {}^{\mathrm{sc}}\mathrm{Sym}^2 T^* X)+ S^{-\alpha_1} (X; {}^{\mathrm{sc}}\mathrm{Sym}^2 T^* X).
	\label{eq:g_ass2}
\end{equation}
We will call $g$ ``(fully) classical'' if the symbolic term in \cref{eq:g_ass} can be dropped entirely, i.e.\ if \cref{eq:g_ass2} holds for all $\alpha_1>1$. (Any sc-metric in the sense of \cite{MelroseSC} is fully classical in this sense.)
The term \emph{asymptotically conic manifold} will be used to refer to a conic manifold equipped with an asymptotically conic metric, so a triple $(X,\iota,g)$. 
This is a more general notion than that of an ``sc-manifold'' \cite{MelroseSC}\cite{MelroseGeometric}\cite{JoshiBarreto}\cite{HassellWunsch}, which are also sometimes called asymptotically conic --- besides allowing a symbolic term in \cref{eq:g_ass}, we also allow a classical $O(x^{-3})\mathrm{d}x^2$ term (see $a_{00}$ in \Cref{thm:main}). (Rewriting the metric in terms of the tortoise coordinate $x_*$, the $O(x^{-3})\dd x^2$ term can be removed, but $x_*$ is not a smooth function of $x$ and so $x_*$ is not a bdf on the mwb $X$.) We will assume that this term is constant on $\partial X$. Thus, the metrics we consider are Riemannian metrics on $X^\circ$ of the form 
\begin{equation}
	g = g_0 + a_{00} \frac{\mathrm{d}x^2}{x^3} + \frac{\Gamma_{1,\partial X}\odot\mathrm{d}x}{x^2} + \frac{h_{1,\partial X}}{x} +  x^2 C^\infty(X; {}^{\mathrm{sc}}\mathrm{Sym}^2 T^* X)+ S^{-\alpha_1} (X; {}^{\mathrm{sc}}\mathrm{Sym}^2 T^* X)
	\label{eq:g_ass3}
\end{equation}
for some $a_{00} \in \bbR$, 1-form $\Gamma_{1,\partial X} \in \Omega^1(\partial X)$, and symmetric 2-cotensor $h_{1,\partial X} \in C^\infty(\partial X;\operatorname{Sym}^2 T^* \partial X)$. 
A fully classical $g$ is then of the form 
\begin{equation}
	g = g_0 + a_{00} \frac{\mathrm{d}x^2}{x^3} + \frac{\Gamma_{1,\partial X}\odot\mathrm{d}x}{x^2} + \frac{h_{1,\partial X}}{x} +  x^2 C^\infty(X; {}^{\mathrm{sc}}\mathrm{Sym}^2 T^* X)
	\label{eq:g_ass4}
\end{equation}
near $\partial X$.

Before we specialize to the case of spectral families of Schr\"odinger operators for the proof of \Cref{thm:main}, 
we study in \S\ref{sec:operator}, \S\ref{sec:symbolic} smooth families $\{P(\sigma)\}_{\sigma \geq 0} \subset \operatorname{Diff}^{2}(X^\circ)$ of elliptic 2nd-order differential operators on $X^\circ$ of the form 
\begin{equation} 
	P(\sigma) = P_0(\sigma) + P_1(\sigma) + P_2(\sigma),
\end{equation} 
where,  for each $i=1,2,3$, 
$P_i = \{P_i(\sigma)\}_{\sigma\geq 0} \subset \operatorname{Diff}^2(X^\circ)$ is a family of differential operators  (depending smoothly on $\sigma^2\geq 0$, all the way down to $\sigma=0$), which near $\partial X$ can be written as 
\begin{itemize}
	\item 
	\begin{equation}
			P_{0}(\sigma) = -(1+x a_{00}(\sigma) )(x^2 \partial_x)^2 + x^2 \triangle_{\partial X}   + x^3 [ a(\sigma) +  n-1] \partial_x- \mathsf{Z} x  - \sigma^2
			\label{eq:P0}
	\end{equation}
	for 
	\begin{itemize}
		\item $\smash{a_{00}} \in C^\infty([0,\infty)_{\sigma^2};\bbR)$, $\mathsf{Z}>0$ satisfying the \emph{attractivity condition}
		\begin{equation} 
			\mathsf{Z} - \sigma^2 a_{00}(\sigma) > 0
			\label{eq:ac}
		\end{equation} 
		for each $\sigma\geq 0$, 
		\item $a \in C^\infty([0,\infty)_{\sigma^2})$, 
	\end{itemize} 
	
	\item 
	\begin{equation} 
		P_{1}(\sigma) = \textstyle{\sum_{j=1}^{J}}( x^4  P_{\perp,j}(\sigma) b_j(x;\sigma) \partial_x + x^3 b'_j(x;\sigma) P_{\partial X,j}(\sigma)  + x^2   b''_j(x;\sigma)Q_{\partial X,j}(\sigma))
	\label{eq:misc_000}
	\end{equation} 
	for some $J\in \bbN$,  where 
	\begin{itemize}
		\item $b_j(-;\sigma),b'_j(-;\sigma),b''_j(-;\sigma) \in S^0(X)$, more specifically  
		\begin{equation}
			b_j, b'_j, b''_j \in C^\infty([0,\infty)_{\sigma^2}; S^0(X)), 
			\label{eq:misc_bsz}
		\end{equation}
		\item  $P_{\perp,j}(\sigma),Q_{\partial X,j}(\sigma) \in \operatorname{Diff}^{1}(\partial X)$ and $P_{\partial X,j}(\sigma) \in \smash{\operatorname{Diff}^{2}(\partial X)}$ for each $\sigma\geq 0$, so that
		\begin{align}
			P_{\perp,j} , Q_{\partial X,j} &\in C^\infty([0,\infty)_{\sigma^2} ; \operatorname{Diff}^1(\partial X)),\\ 
			P_{\partial X,j} &\in C^\infty([0,\infty)_{\sigma^2} ; \operatorname{Diff}^2(\partial X)),
		\end{align}
		for each $j=1,\ldots,J$, 
	\end{itemize}
	\item $P_2 \in C^\infty([0,\infty)_{\sigma^2} ; S \operatorname{Diff}_{\mathrm{scb}}^{2,-1-\delta,-3/2-\delta}(X))$ for some $\delta>0$.
\end{itemize}
Here $\operatorname{Diff}_{\mathrm{scb}}^{m,s,l}(X)=\operatorname{Diff}_{\mathrm{b}}^{m,l}(X)\cap \operatorname{Diff}_{\mathrm{sc}}^{m,s}(X)$ and $S\operatorname{Diff}_{\mathrm{scb}}^{m,s,l}(X)=S\operatorname{Diff}_{\mathrm{b}}^{m,l}(X)\cap S\operatorname{Diff}_{\mathrm{sc}}^{m,s}(X)$. 
We may take $\delta<1/2$, as this will be useful in reducing casework later on.
Note that we are requiring smoothness with respect to $E=\sigma^2$, rather than with respect to $\sigma$, in compact subsets of $X^\circ$. This is not strictly necessary for the analysis in \S\ref{sec:symbolic}, for which even smoothness with respect to $\sigma$ is mostly unnecessary, but since we use smoothness in $E$ in \S\ref{sec:mainproof} we will work with it from the outset and develop $\Psi$DO calculi accordingly in \S\ref{sec:calculus}. We will say that $P_{1}$ is ``classical to $\beta_1$th order'' for $\beta_1>0$ if we can replace \cref{eq:misc_bsz} by 
\begin{equation}
	b_j, b'_j, b''_j \in C^\infty([0,\infty)_{\sigma^2}; C^\infty(X))+C^\infty([0,\infty)_{\sigma^2}; x^{\beta_1} S^0(X) ).
	\label{eq:misc_a1o}
\end{equation}
Similarly, we say that $P_2$ is ``classical to $(\beta_2,\beta_3)$th order'' for $\beta_2>\beta_3>0$ if 
\begin{align} 
	\begin{split} 
	P_2 &\in C^\infty([0,\infty)_{\sigma^2} ;  \operatorname{Diff}_{\mathrm{scb}}^{2,-2,-2}(X)+ S \operatorname{Diff}_{\mathrm{scb}}^{2,-1-\beta_2,-3/2-\beta_2}(X) + x^{3/2+\beta_3}S^0(X)) \\
	&= C^\infty([0,\infty)_{\sigma^2} ;  \operatorname{Diff}_{\mathrm{sc}}^{2,-2}(X)+ S \operatorname{Diff}_{\mathrm{scb}}^{2,-1-\beta_2,-3/2-\beta_2}(X)+ x^{3/2+\beta_3}S^0(X)) 
	\end{split} 
	\label{eq:misc_p1i}
\end{align} 
(note that $\operatorname{Diff}_{\mathrm{scb}}^{2,-2,-2}(X) = \operatorname{Diff}_{\mathrm{sc}}^{2,-2}(X)$). We could work with even finer measures of classicality, but for the proof of \Cref{thm:main} (and eventually \Cref{prop:main}) it suffices to keep track of $\beta_1,\beta_2,\beta_3$.

We restrict attention to the case of constant $\mathsf{Z}$,  though the case of variable $\mathsf{Z}$ may also be of interest.  
(The $x^3(n-1)\partial_x$ term in \cref{eq:misc_000} is merely conventional and can be parametrized away upon redefining $a$, but it is convenient because $x^2 \partial_x - x(n-1)/2$ is formally anti- self-adjoint with respect to the $L^2=L^2([0,\bar{x})\times \partial X,g_0)=L^2_{\mathrm{sc}}([0,\bar{x})\times \partial X)$ inner product.)
The ``main'' and ``subleading'' terms of $P$ are collected in $P_0$. Although subleading by only one order in the sense of decay, the terms in $P_1$ are all negligible in the analysis of \S\ref{sec:symbolic}, as are the terms in $P_2$ (which is why they are allowed to be symbolic). (The key point here is that the principal symbol of $P_1$ vanishes to one additional order at the radial sets of $P$'s Hamiltonian flow because of the tangential derivatives. On the other hand, $P_2$ is just sufficiently low order everywhere.)
Note that we are only requiring $P(\sigma),P_0(\sigma)$ to be elliptic in the traditional sense; the operators above can be considered as sc-operators, but they will be nonelliptic in the sc-calculus. 
The `$S$' in front of $S\operatorname{Diff}$ indicates that the differential operators therein have coefficients which are required only to be conormal functions on $X$, hence are ``symbolic.'' 

We do not concern ourselves with uniformity as $\sigma\to \infty$, this already being handled in \cite[\S5]{VasyLA} for a suitably large class of operators. 

By ``attractive Coulomb-like Schr\"odinger operator (on $X$),'' we mean an operator of the form 
\begin{equation}
	P(0) = \triangle_g -\mathsf{Z} x +V(x) 
\end{equation}
for some asymptotically conic metric $g$ on $X$, $\mathsf{Z}>0$, and real-valued $V \in S^{-3/2-\delta}(X)$ for $\delta>0$. The spectral family of $P(0)$ is the family $\{P(\sigma)=P(0)-\sigma^2\}_{\sigma\geq 0}$.
The spectral families of attractive Coulomb-like Schr\"odinger operators are therefore included in the class of families we study here if we restrict attention to sufficiently small $\sigma$ such that the attractivity condition \cref{eq:ac} is satisfied. See \Cref{prop:form}. 
From the explicit formula for the Laplacian in local coordinates, we see that $a_{00}$ is the restriction to $\partial X$ of $-x^{3}(\delta g)_{00}$, where $\delta g= g-g_0$.
Although it is only obvious upon changing coordinates from $x$ to $x/(1+xa_{00})$ -- see \S\ref{ap:model} --  a nonzero value of $a_{00}$ has essentially the same effect as a nonzero Coulomb potential on the asymptotics of solutions to Helmholtz's equation for positive $\sigma$, and it turns out that \cref{eq:ac} -- taking into account both the actual Coulomb potential and the effective Coulomb potential from $a_{00}$ -- is the correct notion of ``attractivity'' in this generality. Note that if $P$ is the spectral family of a Coulomb-like Schr\"odinger operator with $\mathsf{Z}>0$ on an asymptotically conic manifold, then we can modify $P$ by cutting off $a_{00}$ outside of some sufficiently small neighborhood of $\sigma=0$ such that the resulting family of operators satisfies the attractivity condition. We only care about behavior in a neighborhood of $\sigma=0$, so this does not restrict applicability.

For $\sigma \gg 0$, attractive Coulomb-like operators do not need to be distinguished from non-attractive Coulomb-like operators as far as the limiting absorption principle is concerned, except with regards to the precise logarithmic corrections to spherical waves. The difference matters only in the $\sigma\to 0$ limit. That this limit is delicate can be seen already in the case when $X=[0,1]_x$, equipped with an exactly conic metric (`$x$' now only denoting a bdf for one end). Then,  when the potential is an exact Coulomb potential near the $x=0$ boundary, $P(\sigma) u = 0$ is an ODE,
\begin{equation}
	- (x^2 \partial_x)^2 u -\sigma^2 u - \mathsf{Z} x u  = 0
	\label{eq:misc_aku}
\end{equation}
for $x$ sufficiently close to zero. This is a form of Whittaker's ODE for $\sigma>0$, degenerating to a Bessel ODE 
\begin{equation}
	- (x^2 \partial_x)^2 u - \mathsf{Z} x u  = 0
	\label{eq:misc_bku}
\end{equation}
at $\sigma=0$. Hence, smooth families $\{u(-;\sigma)\}_{\sigma\geq 0}$ of solutions $u=u(-;\sigma)$ to $P(\sigma)u(-;\sigma) = f$ for $f\in C_{\mathrm{c}}^\infty(X^\circ)$ are near $x=0$ linear combinations of Whittaker functions (evaluated along the imaginary axis) for $\sigma>0$ degenerating in some way to linear combinations of Bessel functions as $\sigma\to \smash{0^+}$.
See \Cref{prop:radialcont} for a precise statement.
An analysis of the low-energy limit of attractive Coulomb-like Schr\"odinger operators must therefore -- at the very least -- describe in appropriate asymptotic regimes the degeneration of  Whittaker functions to Bessel functions (a classical topic, but nonetheless delicate).  The former oscillate like 
\begin{equation} 
	\exp\Big(\pm i \frac{\sigma}{x}\Big) x^{\mp i \mathsf{Z}/2\sigma } = \exp\Big(\pm i\Big(    \frac{\sigma}{x} - \frac{\mathsf{Z}}{2\sigma} \log x\Big)\Big)
	\label{eq:incomplete}
\end{equation}  
as $x\to 0^+$, while the latter oscillate like  $\exp (\pm 2i \mathsf{Z}^{1/2} / x^{1/2})$. While the leading order term of the phase in \cref{eq:incomplete} is suppressed as $\sigma\to 0^+$, the logarithmic correction $\mp i (\mathsf{Z}/\sigma) \log x$ blows up as $\sigma\to 0^+$. This, along with $O_\sigma(1)$, $O_\sigma(x)$, etc.\ contributions which were omitted from \cref{eq:incomplete} (and include terms like $1/\sigma,1/\sigma^2,\cdots$), end up contributing as $\sigma\to 0^+$ to the phase $\smash{\pm 2i \mathsf{Z}^{1/2} / x^{1/2}}$ relevant at zero energy. Clearly, \cref{eq:incomplete} paints an incomplete picture of the oscillations of solutions to \cref{eq:misc_aku} in the $\sigma\to 0^+$ limit. This is the first indication that the analysis of the low energy limit will need to be done on a resolution of $[0,\infty)_\sigma\times X$. We refer the reader to \S\ref{ap:model} for a further discussion of the ODE case. 

The ``conjugated perspective'' involves studying the conjugated spectral family $\tilde{P} = \{\tilde{P}(\sigma)\}_{\sigma\geq 0} \subset \operatorname{Diff}^2(X^\circ)$, 
\begin{equation}
\tilde{P}(\sigma) = \exp(-i \Phi) P(\sigma) \exp(+i \Phi) : \calS'(X)\ni u \mapsto  \exp(-i \Phi) (P(\sigma)(\exp(+i \Phi) u)) ,
\label{eq:tildeP}
\end{equation} 
where the ``phase'' $\Phi= \Phi(-;\sigma)$ is an appropriate function on $[0,\infty)_\sigma\times X$  such that $\exp( + i \Phi)$ captures the asymptotics of outgoing solutions $u$ to the PDE $P u=f$ to some desired level of precision. (Thus, we deal with the `$+$' case of \Cref{thm:main}, the `$-$' case being analogous.) At the level of the phase space of the sc-calculus, this conjugation corresponds to a symplectomorphism moving one of the two sets of radial points (the ``selected radial set'') to the zero section of the sc-fibers, which upon second microlocalization gets blown up to the fibers of $\smash{{}^{\mathrm{b}}\overline{T}^*_{\partial X}X}$, the phase space of the b-calculus over $\partial X$ (see \cite{VasyLA}).  The following choice of $\Phi$ is sufficiently detailed:
\begin{equation}
	\Phi(x;\sigma) =    \frac{1}{x} \sqrt{\sigma^2+\mathsf{Z}x-\sigma^2 a_{00} x} + \frac{1}{\sigma}(\mathsf{Z}-\sigma^2 a_{00}) \operatorname{arcsinh} \Big( \frac{\sigma}{x^{1/2}} \frac{1}{(\mathsf{Z}  - \sigma^2 a_{00})^{1/2}} \Big)   - \frac{i}{2}a\log x , 
	\label{eq:Phi}
\end{equation}
(extended from $\sigma>0$ to $\sigma=0$ by continuity), hence the appearance of $\Phi$ in \Cref{thm:main} (for which we have $a=0$). The square roots in \cref{eq:Phi} are well-defined by the attractivity condition \cref{eq:ac}. 
See the beginning of \S\ref{sec:operator} for a motivation of \cref{eq:Phi}. 
Given compact $K\subset (0,\infty)$, we can write, for $\sigma \in K$,  
\begin{equation} 
	\Phi = \sigma x^{-1} - \frac{\mathsf{Z}-\sigma^2 a_{00}}{2\sigma} \log x - \frac{i}{2} a\log x + O_K(1)
	\label{eq:osc_+}
\end{equation} 
as $x\to 0^+$.  This is in accordance with \cref{eq:incomplete}, \cite[Theorem 1.1]{VasyLA}; $\Phi$ differs from Vasy's phase modulo a logarithmic plus smooth correction for each individual $\sigma>0$. The $O_K$ term in \cref{eq:osc_+}, when formally expanded around $\sigma = 0^+$, contains terms proportional to $\sigma^{-2N}$ for all $N\in \bbN^+$. The estimate \cref{eq:osc_+} is therefore not uniform as $\sigma\to 0$.  Similarly, while we can write $(\sigma^2+\mathsf{Z}x)^{1/2} = \mathsf{Z}^{1/2} x^{1/2}(1+\sigma^2/2 \mathsf{Z}x + O(\sigma^4/\mathsf{Z}^2x^2))$
for $x$ bounded away from zero, this obviously cannot be used to understand the asymptotics of outgoing solutions to $Pu=f$ for any $\sigma>0$. Hence, the precise form of $(\sigma^2+\mathsf{Z}x)^{1/2}$ in \cref{eq:Phi} (and the precise form of the other terms appearing in \cref{eq:Phi}) is important (at least modulo functions which are well-behaved on $X_{\mathrm{res}}^{\mathrm{sp}}$). 
This is the second indication that we will need to resolve the $x,\sigma \to 0^+$ regime, and apparently in doing so we had ought to resolve the ratio $x/\sigma^2$. 

This is accomplished by the mwc $X_{\mathrm{res}}^{\mathrm{sp}}$. While, as already remarked upon in \Cref{rem:scaling}, this resolution is reminiscent of that used in \cite{VasyN0L}\cite{HintzPrice} in order to resolve the low energy behavior of interest there, it performs a somewhat different role here; theirs was used to study the degeneration of $\triangle_g - \sigma^2$ as $\sigma\to 0^+$ to an elliptic element 
\begin{equation} 
	\triangle_g \in \operatorname{Diff}^{2,-2}_{\mathrm{b}}(X)
\end{equation} 
of the b-calculus, but 
\begin{equation}
	\triangle_g  - \mathsf{Z} x \in \operatorname{Diff}^{2,-1}_{\mathrm{b}}(X)
\end{equation}
is not elliptic in the b-calculus, hence the oscillating solutions to the ODE \cref{eq:misc_bku} noted above. 
Instead, we use the resolution to interpolate between Vasy's analysis in \cite{VasyLA} of $\triangle_g - \sigma^2 - \mathsf{Z} x$ performed in $\Psi_{\mathrm{scb}}(X)$ for $\sigma>0$ and the analysis of the ``zero-energy operator'' $\triangle_g - \mathsf{Z} x$ performed in $\Psi_{\mathrm{scb}}(X_{1/2})$ along \emph{similar lines} (see \S\ref{sec:0_operator}).
According to the previous paragraph, what makes the $\Phi$ in \Cref{thm:main} the ``right'' choice is the asymptotics of $\Phi$ near the boundary $\partial X^{\mathrm{sp}}_{\mathrm{res}}$ of this resolution (or more accurately near $\partial X^{\mathrm{sp}}_{\mathrm{res}}\backslash \mathrm{zf}^\circ = \mathrm{bf}\cup \mathrm{tf}$). 
Indeed, for $\sigma = 0$, we can write 
\begin{equation}
	\Phi = 2\sqrt{\frac{\mathsf{Z}}{x}}  - \frac{i}{2}a\log x.
	\label{eq:osc_0}
\end{equation}
Hence, \cref{eq:Phi} interpolates between the oscillations \cref{eq:osc_+} seen in solutions of the ODE \cref{eq:misc_aku} at positive energy and the oscillations \cref{eq:osc_0} seen in solutions of the ODE \cref{eq:misc_bku} at zero energy. 

We refer to \S\ref{sec:operator} for a further discussion of the conjugated operator family and \S\ref{sec:0_operator} for a discussion of the situation at zero energy. 

As our main technical tool, we situate the family $\tilde{P}=\{\tilde{P}(\sigma)\}_{\sigma\geq 0}$ in a pseudodifferential calculus, 
\begin{equation} 
\Psi_{\mathrm{leC}}(X) = \bigcup_{m,s,\varsigma,l,\ell \in \bbR} \Psi_{\mathrm{leC}}^{m,s,\varsigma,l,\ell}(X),
\label{eq:leC_calculus}
\end{equation}
which we will call the ``leC-calculus'' (``low energy Coulomb''-calculus, for lack of a better name), the elements of which can be interpreted as particular families of b-$\Psi$DOs on $X$. The calculus comes with a refined symbol calculus tailored to the problem at hand. Compared to the calculus of b-$\Psi$DOs with parameters (i.e.\  $\Psi_{\mathrm{b}}(X)$-valued symbols on some parameter space), the symbol calculus here is refined in the sense of being second-microlocalized \`a la Vasy (so as to keep track both of b-decay and sc-decay orders) and ``resolved at the corner'' (so as to keep track of the asymptotic regime when both $\sigma^2\to 0$ and $x\to 0$ at compatible rates). The corresponding symbols are conormal functions on the ``leC-phase space,'' an iterated blow-up  of $\smash{[0,\infty)_\sigma \times {}^{\mathrm{b}}\overline{T}^* X}$. This mwc has six boundary faces -- df, sf, ff, bf, tf, zf -- and is described in the next section. In \cref{eq:leC_calculus}, $m$ is the ``differential order'' (order at fiber infinity, df), $s$ is the sc-decay order at positive energy (order at sf), $\varsigma$ is the sc-decay order at zero energy (order at ff), $l$ is the b-decay order at positive energy (order at bf), and $\ell$ is the b-decay order at zero energy (order at tf).
We remark that the scattering calculus with respect to $x^2$ (rather than $x^{1/2}$) has been used by Wunsch \cite{Wunsch}\cite{HassellWunsch}, who called it the ``quadratic scattering calculus.''
The leC-calculus is discussed in \S\ref{sec:calculus}.

We now say a word about our usage of the symbol `$\preceq$' (the usage of `$\succeq$' being analogous). When stating a proposition involving an estimate, we will be explicit about the dependence of the constants involved on parameters. In order to avoid a proliferation of symbols denoting different but unimportant constants, when proving a proposition of the form 
\begin{itemize}
	\item for all $p_1 \in \calP_2,\cdots,p_M\in \calP_M$, there exists a constant $C(p_1,\ldots,p_M)>0$ such that $r_1(p_1,\cdots,p_N)\leq C(p_1,\cdots,p_M)r_2(p_1,\cdots,p_N)$ for all $p_{M+1} \in \calP_{M+1},\cdots ,p_N\in \calP_N$, 
\end{itemize}
(where $M,N\in \bbN$, $N\geq M$, $\calP_2,\cdots,\calP_N$ are some sets, $r_i:\calP_2\times \cdots \times \calP_N \to \bbR$ for $i=1,2$) 
we will write an intermediate estimate of the form 
\begin{equation} 
	r_3 (p_1,\ldots,p_N) \leq C'(p_1,\cdots,p_M) r_4(p_1,\ldots,p_N)
	\label{eq:misc_uqq}
\end{equation}  
as $r_3 \preceq r_4$, with the key point being that, according to \cref{eq:misc_uqq}, $C'$ depends only on the parameters that $C$ depends on (so that the estimate \cref{eq:misc_uqq} is ``uniform'' in $p_{M+1},\cdots,p_{N}$). All constants below depend on the geometric data in the setup of \Cref{thm:main}, so we will not be explicit about that dependence.
\section{The leC-calculus}
\label{sec:calculus}

We now turn to our discussion of the leC-calculus. This calculus is, in many ways, similar to Vasy's second microlocalized calculi $\Psi_{\mathrm{scb}}(X)=\cup_{m,s,l\in \bbR}\Psi_{\mathrm{scb}}^{m,s,l}(X)$ \cite[\S2]{VasyLA} and $\Psi_{\mathrm{scb,res}}(X)=\smash{\cup_{m,s,l,\ell\in \bbR}\Psi_{\mathrm{scb,res}}^{m,s,l,\ell}(X)}$ \cite[\S3]{VasyN0L},
with the main novel feature of the leC-calculus (besides the relatively unimportant alteration of the smooth structure at $\sigma=0$) being another resolution of the phase space (which we actually carry out before resolving the scattering face for $\sigma>0$). This can be seen at a glance, comparing \Cref{fig:LEC_phase_space} to \Cref{fig:vasy}, \cite[Figure 4]{VasyN0L}. The zero face $\mathrm{zf}$ of the leC-phase space is identifiable with the phase space $\smash{{}^{\mathrm{scb}} \overline{T}^* X_{1/2}}$ of the calculus $\Psi_{\mathrm{scb}}(X_{1/2})$, while for $\sigma_0>0$ the $\{\sigma=\sigma_0\}$ cross-section of the leC-phase space is identifiable with the phase space $\smash{{}^{\mathrm{scb}} \overline{T}^* X}$ of the scb-calculus. Thus, the leC-calculus interpolates between these two calculi as $\sigma\to 0^+$, as $\Psi_{\mathrm{scb,res}}(X)$ interpolates between $\Psi_{\mathrm{scb}}(X)$ and $\Psi_{\mathrm{b}}(X)$ in the same limit.

The leC-phase space ${}^{\mathrm{leC}}\overline{T}^* X$ is introduced in \S\ref{subsec:phasespace}, along with corresponding algebras of symbols. Calculi are discussed in \S\ref{subsec:calculi}, and the corresponding leC-Sobolev spaces (which are really families of scb-Sobolev spaces) are discussed in \S\ref{subsec:sobolev}. As many of the results in this section are either consequences of standard results for the b-calculus or derivable via similar arguments, some details are omitted. Still, we've made an effort to give a relatively complete list of the results needed later, and in the process we give the leC analogues of some standard arguments. 
In this section we mainly write $\lambda = \sigma^2$ in place of the parameter $E$ used in the introduction.

When we write `$\Psi_{\mathrm{b}}(X)$,' and more generally when we refer to ``b-$\Psi$DOs,'' we mean the \emph{conormal} (a.k.a.\ ``symbolic'') b-algebra \cite[Definition 5.15]{VasyGrenoble}, rather than the closely related, slightly smaller calculus defined in \cite[Definition 4.22]{APS}. The latter calculus has symbols that have some additional classicality at the boundary. Our convention follows \cite{VasyLA}. In \cite{VasyN0}, the notation `$\Psi_{\mathrm{bc}}(X)$' is used instead. 

\subsection{Phase Spaces and Symbols}
\label{subsec:phasespace}
Recall that we can identify ${}^{\mathrm{b}}T^* X$ over the boundary collar with $[0,\bar{x})_x\times \bbR_{\xi_{\mathrm{b}}}\times (T^* \partial X)_{\eta_{\mathrm{b}}}$ via the map 
\begin{equation}
	(0,\bar{x})_x\times \bbR_{\xi_{\mathrm{b}}}\times (T^* \partial X)_{\eta_{\mathrm{b}}}\ni (x,\xi_{\mathrm{b}},\eta_{\mathrm{b}}) \mapsto \xi_{\mathrm{b}} \frac{\dd x}{x} + \operatorname{plr}^*(\eta_{\mathrm{b}}), 
\end{equation}
where $\operatorname{plr} = \pi_{\mathrm{R}}\circ \iota^{-1} : \iota(\hat{X} ) \to \partial X$. This defines a diffeomorphism between $[0,\bar{x})_x\times \bbR_{\xi_{\mathrm{b}}}\times (T^* \partial X)_{\eta_{\mathrm{b}}}$ and ${}^{\mathrm{b}}T^*_{\iota(\hat{X} )}X$.

Let 
\begin{align}
	\begin{split} 
	{}^{\mathrm{b,sp}} \overline{T}^* X &= [ [0,\infty)_\lambda\times {}^{\mathrm{b}} \overline{T}^* X  ; \{ \lambda = x = 0\}]\\
	&= [[0,\infty)_\lambda\times {}^{\mathrm{b}} \overline{T}^* X ;\{0\} \times {}^{\mathrm{b}} \overline{T}^*_{\partial X} X ]
	\end{split} 
\end{align} 
denote the phase space of the resolved calculus of 1-parameter families of b-$\Psi$DOs (\cite[Figure 4]{VasyN0L}), and letting $\beta: \smash{{}^{\mathrm{b,sp}} \overline{T}^* X\to[0,\infty)_\lambda\times  {}^{\mathrm{b}} \overline{T}^* X }$ denote the blowdown map, let 
\begin{align}
	\begin{split} 
	\mathrm{zf}_{00} &= \operatorname{cl}\beta^{-1}(\{ \lambda=0, x>0\} ),\\
	\mathrm{tf}_{00} &= \beta^{-1}(\{\lambda=0=x\}),\\
	\mathrm{bf}_{00} &= \operatorname{cl} \beta^{-1}(\{\lambda>0,x=0\}),
	\end{split}
	\intertext{and} 
	\mathrm{df}_{00} &= \operatorname{cl}((\partial\, {}^{\mathrm{sp,b}} \overline{T}^* X )\backslash (\mathrm{zf}_{00}\cup \mathrm{tf}_{00} \cup \mathrm{bf}_{00}) )
\end{align}
denote its (closed) boundary faces. For each $\mathrm{f} \in \{\mathrm{zf}_{00},\mathrm{tf}_{00},\mathrm{bf}_{00},\mathrm{df}_{00}\}$, let $\varrho_{0,\mathrm{f}} \in C^\infty({}^{\mathrm{b,sp}} \overline{T}^* X ; [0,\infty))$ denote a bdf of the respective face, which we can take to be equal near $\{x=0\}$ to 
\begin{equation}
	\varrho_{0,\mathrm{zf}_{00}} = \frac{\lambda}{\lambda+x}, \qquad \varrho_{0,\mathrm{tf}_{00}} = \lambda + \mathsf{Z}x, \qquad \varrho_{0,\mathrm{bf}_{00}} = \frac{x}{\lambda+\mathsf{Z}x}, \qquad \varrho_{0,\mathrm{df}_{00}} = (1+\xi_{\mathrm{b}}^2 + \eta_{\mathrm{b}}^2)^{-1/2} 
	\label{eq:misc_k1a}
\end{equation}
(defined initially in the interior of ${}^{\mathrm{b,sp}} \overline{T}^* X$, these then extending to smooth functions on ${}^{\mathrm{b,sp}} \overline{T}^* X$). (Below, we conflate smooth functions on mwc with their restrictions to the interior when such a conflation does not cause trouble.) In \cref{eq:misc_k1a} and below, we write $\xi_{\mathrm{b}}$ for the b-cofiber coordinate dual to $x$, and $\smash{\eta_{\mathrm{b}}^2 = g_{\partial X}^{-1}(\eta_{\mathrm{b}},\eta_{\mathrm{b}})}$ for $\eta_{\mathrm{b}} \in T^* \partial X$.

There exists a unique mwc $\smash{{}^{\mathrm{b,leC}}\overline{T}^* X} =  [[0,\infty)_\lambda\times {}^{\mathrm{b}} \overline{T}^* X ;\{0\} \times {}^{\mathrm{b}} \overline{T}^*_{\partial X} X ;\frac{1}{2}]=[{}^{\mathrm{b,sp}} \overline{T}^* X ; \mathrm{tf}_{00} ; \frac{1}{2}]$, depicted in \Cref{fig:preliminary},  with the following properties: 
\begin{enumerate}
	\item as a set, $\smash{{}^{\mathrm{b,leC}}\overline{T}^* X}$ is equal to ${}^{\mathrm{b,sp}} \overline{T}^* X$ (a convenient convention),
	\item $\smash{{}^{\mathrm{b,leC}}\overline{T}^* X}$ has the same smooth structure as ${}^{\mathrm{b,sp}} \overline{T}^* X$ away from $ \mathrm{tf}_{00}$, so that if $\varphi \in C^\infty({}^{\mathrm{b,sp}} \overline{T}^* X)$ is supported away from $\mathrm{tf}_{00}$, then $\varphi \in C^\infty(\smash{{}^{\mathrm{b,leC}}\overline{T}^* X})$; moreover, 
	\item $\smash{{}^{\mathrm{b,leC}}\overline{T}^* X}$ has four faces, equal as sets to $\mathrm{zf}_{00},\mathrm{tf}_{00},\mathrm{bf}_{00},\mathrm{df}_{00}$, for which $\varrho_{0,\mathrm{zf}_{00}}$, $\varrho_{0,\mathrm{tf}_{00}}^{1/2}$, $\varrho_{0,\mathrm{bf}_{00}}$, $\varrho_{0,\mathrm{df}_{00}}$ serve as bdfs (respectively). 
\end{enumerate}
We will refer to $\smash{{}^{\mathrm{b,leC}}\overline{T}^* X}$ as the b,leC- phase space. See \Cref{fig:preliminary}. 

We will refer to the bdfs of $\smash{{}^{\mathrm{b,leC}}\overline{T}^* X}$ as $\smash{\varrho_{\mathrm{zf}_{00}}=\varrho_{0,\mathrm{zf}_{00}}}$, $\smash{\varrho_{\mathrm{tf}_{00}}=\varrho_{0,\mathrm{tf}_{00}}^{1/2}}$, $\varrho_{\mathrm{bf}_{00}}= \varrho_{0,\mathrm{bf}_{00}}$, and $\varrho_{\mathrm{df}_{00}}=\varrho_{0,\mathrm{df}_{00}}$. Thus, in terms of $\sigma = \lambda^{1/2}$, 
\begin{equation}
	\varrho_{\mathrm{tf}_{00}} = \sqrt{\sigma^2+\mathsf{Z}x}, \qquad \varrho_{\mathrm{bf}_{00}} = \frac{x}{\sigma^2+\mathsf{Z}x}, \qquad \varrho_{\mathrm{zf}_{00}} = \frac{\sigma^2}{\sigma^2+\mathsf{Z}x}. 
\end{equation}
Note that the notions of zeroth order conormality on $\smash{{}^{\mathrm{b,leC}}\overline{T}^* X}$  and ${}^{\mathrm{b,sp}} \overline{T}^* X$ agree, as does the notion of smoothness at $\mathrm{zf}_{00}^\circ$. For each $m,l,\ell\in \bbR$, we let 
\begin{align}
	S_{\mathrm{b,leC}}^{m,l,\ell}(X) &= \calA^{-m,-l,-\ell/2, (0,0)}_{\mathrm{loc}} ({}^{\mathrm{sp,b}}\overline{T}^* X) = \calA^{-m,-l,-\ell,(0,0)}_{\mathrm{loc}} ({}^{\mathrm{b,leC}}\overline{T}^* X), \label{eq:misc_kul}\\
	S_{\mathrm{b,leC}}(X) &= \cup_{m,l,\ell\in \bbR} S_{\mathrm{b,leC}}^{m,l,\ell}(X),
\end{align} 
where we are enumerating the faces of the b,leC-phase space in the order $\mathrm{df}_{00},\mathrm{bf}_{00},\mathrm{tf}_{00},\mathrm{zf}_{00}$. Thus, $m$ is the order at $\mathrm{df}_{00}$, $l$ is the order at $\mathrm{bf}_{00}$, $\ell$ is the order at $\mathrm{tf}_{00}$, and the order at $\mathrm{zf}_{00}$ is just zero (and we have a full Taylor series there, with the terms in the Taylor series elements of $\calA^{-m,-\ell/2}(\mathrm{zf}_{00})$). (The `loc' subscript in ``$\calA_{\mathrm{loc}}$'' refers to the fact that we do not require $L^\infty$ bounds in the $\sigma\to\infty$ direction. That is, we only have bounds in compact subsets (which can include boundary points) of the mwc ${}^{\mathrm{b,leC}}\overline{T}^* X$, which is only noncompact because of the $\sigma\to\infty$ direction.) We also define 
\begin{equation} 
	S_{\mathrm{cl,b,leC}}^{m,l,\ell}(X) = \varrho_{\mathrm{df}_{00}}^{-m}\varrho_{\mathrm{bf}_{00}}^{-l}\varrho_{\mathrm{tf}_{00}}^{-\ell}C^\infty({}^{\mathrm{b,sp}}\overline{T}^* X) \subset S_{\mathrm{b,leC}}^{m,l,\ell}(X).
\end{equation}
Given $a \in S_{\mathrm{b,leC}}^{m,l,\ell}(X)$, we may restrict $a$ to $\mathrm{zf}_{00}$, giving an element $a(-;0)\in S_{\mathrm{b}}^{m,\ell/2}(X) = S_{\mathrm{b}}^{m,\ell}(X_{1/2})$.

Note that the elements of $S_{\mathrm{b,leC}}(X)$ are all symbols on $[0,\infty)_\lambda\times {}^{\mathrm{b}} \overline{T}^* X$, i.e.\ elements of 
\begin{equation} 
	\bigcup_{m,l,\nu \in \bbR} x^{-l} \lambda^{-\nu} \varrho_{\mathrm{df}_{00}}^{-m} \calA^0([0,\infty)_\lambda\times {}^{\mathrm{b}} \overline{T}^* X).
\end{equation} 
Specifically, for all $m,l,\ell\in \bbR$, 
\begin{equation}
	S_{\mathrm{b,leC}}^{m,l,\ell}(X) \subset x^{-l_0} \lambda^{-\nu} \varrho_{\mathrm{df}_{00}}^{-m} \calA^0([0,\infty)_\lambda\times {}^{\mathrm{b}} \overline{T}^* X)
\end{equation}
whenever $l_0\geq l$ and $2\nu+2l_0 \geq \ell$. Consequently, if $l_0\geq l,\ell/2$, 
\begin{equation}
	S_{\mathrm{b,leC}}^{m,l,\ell}(X) \subseteq \calA^0([0,\infty)_\lambda ; S_{\mathrm{b}}^{m,l_0}(X) ). 
	\label{eq:base_inclusion}
\end{equation}

\begin{lemma}
	\label{lem:sigmaZx_conormal}
	For any $\ell\in \bbR$, $(\sigma^2+\mathsf{Z} x)^{-\ell/2} \in C^0([0,\infty)_\sigma ; S_{\mathrm{b}}^{0,\max\{0,\ell/2+\varepsilon\}}(X) )$ for any $\varepsilon>0$. 
\end{lemma}
\begin{proof}
	Continuity at $\sigma>0$ is clear. 
	Let $l = \max\{0,\ell/2+\varepsilon\}$. It suffices to restrict attention to $x<\bar{x}$, so we have to prove that, for each $k \in \bbN$, $x^{l}(x \partial_x)^k (\sigma^2+\mathsf{Z} x)^{-\ell/2} \to  x^{l}(x \partial_x)^k (\sigma^2+\mathsf{Z} x)^{-\ell/2}|_{\sigma=0}$ in $L^\infty[0,\bar{x}]$. We compute 
	\begin{equation}
		x^{l}(x \partial_x)^k (\sigma^2+\mathsf{Z} x)^{-\ell/2} = x^{l}\sum_{j=0}^k c_{j,k} x^j(\sigma^2+\mathsf{Z} x)^{-\ell/2-j} 
		\label{eq:misc_ku8}
	\end{equation}
	for some $c_{j,k}(\mathsf{Z})\in \bbR$. Observe that $x^{\max\{0,\ell/2\}} x^j(\sigma^2+\mathsf{Z} x)^{-\ell/2-j} \in L^\infty([0,\bar{x}]_x \times [0,1]_\sigma )$ for every $j\in \bbN$.  Consequently, if $\ell>0$, the extra factor of  $x^\varepsilon$ in \cref{eq:misc_ku8} in conjunction with the uniform convergence as $\sigma\to 0^+$ of $x^{j+\ell/2} (\sigma^2+\mathsf{Z} x)^{-\ell/2-j} \to \mathsf{Z}^{-\ell/2-j}$ in compact subsets of $(0,\bar{x}]_x$ implies that 
		\begin{equation}
			x^{l}\sum_{j=0}^k c_{j,k} x^j(\sigma^2+\mathsf{Z} x)^{-\ell/2-j}  \to x^{\varepsilon}\sum_{j=0}^k c_{j,k} \mathsf{Z}^{-\ell/2-j}
			\label{eq:misc_086}
		\end{equation}
		uniformly in all of $[0,\bar{x}]$ as $\sigma\to 0^+$.  
		
		If $\ell \leq 0$, then we can instead use that $x^j (\sigma^2+\mathsf{Z} x)^{-\ell/2-j} \to (\mathsf{Z} x)^{-\ell/2} \mathsf{Z}^{-j}$ uniformly in $[0,\bar{x}]$.
\end{proof}

\begin{proposition}
	\label{prop:continuity}
	If  $a \in S_{\mathrm{b,leC}}^{m,l,\ell}(X)$ and $l_0\in \bbR$ satisfies $l_0\geq l$ and $l_0>\ell/2$, then $\{a(-;\sigma)\}_{\sigma\geq  0} \in C^0([0,\infty)_\sigma ; S_{\mathrm{b}}^{m,l_0}(X) )$. 
\end{proposition}
\begin{proof}
	We first reduce to the case $m,l,\ell = 0$:
	\begin{itemize}
		\item For any $m,l,\ell \in \bbR$, 
		\begin{equation} 
			x^{-l} (\sigma^2+\mathsf{Z} x)^{l-\ell/2} \varrho_{\mathrm{df}_{00}}^{-m} \in \smash{C^0([0,\infty)_\sigma ; S_{\mathrm{b}}^{m,l_1}(X) )}
		\end{equation} 
		if $l_1\geq l$, $l_1>\ell/2$ by the previous lemma. 
		
		Any $a \in S_{\mathrm{b,leC}}^{m,l,\ell}(X)$ can be written as 
		$a = x^{-l} (\sigma^2+\mathsf{Z} x)^{l-\ell/2} \varrho_{\mathrm{df}_{00}}^{-m}  a_0$
		for $a_0 \in S_{\mathrm{b,leC}}^{0,0,0}(X)$, so if we know that $a_0 \in C^0([0,\infty)_\sigma ; S_{\mathrm{b}}^{0,\varepsilon}(X) )$ for  $\varepsilon\geq 0$, then we can conclude that \begin{equation} 
			a \in C^0([0,\infty)_\sigma ; S_{\mathrm{b}}^{m,l_1}(X) )C^0([0,\infty)_\sigma ; S_{\mathrm{b}}^{0,\varepsilon}(X)) = C^0([0,\infty)_\sigma ; S_{\mathrm{b}}^{m,l_1+\varepsilon}(X)).
		\end{equation}  
		
		Taking $l_1<l_0$ and $\varepsilon \in (0,l_0-l_1)$, we get $a\in  C^0([0,\infty)_\sigma ; S_{\mathrm{b}}^{m,l_0}(X))$.

	\end{itemize}
	To prove the proposition in the case $m,l,\ell=0$: 
	\begin{itemize}
		\item Continuity at $\sigma>0$ is clear, so we only need to check that, for $a\in S_{\mathrm{b,leC}}^{0,0,0}(X)$, $a(-;\sigma) \to a(-;0)$ in $\smash{S_{\mathrm{b}}^{0,\varepsilon}(X)}$ for every $\varepsilon>0$. Note that $L a \in L^\infty ([0,1]_\sigma\times [0,\bar{x}]_x\times \partial X)$ for any $L\in \operatorname{Diff}_{\mathrm{b}}(X)$. Since, as $\sigma\to 0^+$, $La \to La|_{\sigma=0}$ uniformly in compact subsets of $(0,\bar{x} ]_x\times \partial X$, we can conclude that 
		\begin{equation}
		x^\varepsilon La \to x^\varepsilon La|_{\sigma=0}  
		\end{equation}
		uniformly in all of $[0,\bar{x}]\times \partial X$. Thus, $a(-;\sigma)\to a(-;0)$ in $S_{\mathrm{b}}^{0,\varepsilon}(X)$.
	\end{itemize}
	 
\end{proof}

\begin{figure}
	\begin{center} 
	\begin{tikzpicture}[scale=.75]
		\draw (0,4,5) -- (4,4,5) -- (5,4,4) -- (5,4,0);
		\draw[dashed] (1,0,5) -- (4,0,5) -- (5,0,4) -- (5,0,1);
		\draw[dashed] (0,4,5) -- (0,1,5);
		\draw (4,4,5) -- (4,0,5);
		\draw (5,4,4) -- (5,0,4);
		\draw[dashed] (5,1,0) -- (5,4,0);
		\draw[dashed] (0,4,5) -- (0,4,2);
		\draw[dashed] (5,4,0) -- (2,4,0);
		\node (df) at (2.5,4,2.5) {$\mathrm{df}_{00}$};
		\node (bf) at (2,2,5) {$\mathrm{bf}_{00}$};
		\node (zf) at (5,2,2) {$\mathrm{zf}_{00}$};
		\node (tf) at (4.5,2,4.5) {$\mathrm{tf}_{00}$};
	\end{tikzpicture}
	\qquad\qquad 
	\begin{tikzpicture}[scale=.75]
		\coordinate (ul) at (0,4,5) {};
		\coordinate (ulc) at (3,4,5) {};
		\coordinate (urc) at (5,4,3) {};
		\coordinate (ur) at (5,4,0) {};
		\coordinate (lc) at (4,3,5) {};
		\coordinate (rc) at (5,3,4) {};
		\coordinate (ld) at (4,0,5) {};
		\coordinate (rd) at (5,0,4) {};
		\draw (ld)--(lc)--(ulc)--(ul);
		\draw (ulc)--(urc)--(rc)--(lc);
		\draw (rd)--(rc)--(urc)--(ur);
		\draw[dashed] (ld)--(rd); 
		\draw[dashed] (ld) --++ (-3,0,0);
		\draw[dashed] (rd) --++ (0,0,-3);
		\draw[dashed] (ul) --++ (0,-3,0);
		\draw[dashed] (ul) --++ (0,0,-3);
		\draw[dashed] (ur) --++ (-3,0,0);
		\draw[dashed] (ur) --++ (0,-3,0);
		\node (df) at (2.5,4,2.5) {$\mathrm{df}_{0}$};
		\node (bf) at (2,2,5) {$\mathrm{bf}_{0}$};
		\node (zf) at (5,2,2) {$\mathrm{zf}_{0}$};
		\node (tf) at (4.5,2,4.5) {$\mathrm{tf}_{0}$};
		\node (ff) at (4.25,3.5,4.25) {$\mathrm{ff}_{0}$};
	\end{tikzpicture}
	\end{center} 
	\caption{The phase spaces ${}^{\mathrm{b,leC}} \overline{T}^* X$ (left, cf.\ \cite[Fig. 1]{VasyN0L}) and $\smash{[{}^{\mathrm{b,leC}} \overline{T}^* X ; \mathrm{df}_{00} \cap \mathrm{tf}_{00} ]}$ (right), with the degrees of freedom associated with $T^* \partial X$ omitted. (In other words, if we were to consider the case $\dim X=1$, then the figures above would depict the phases spaces.) For simplicity, we only depict the $\xi_{\mathrm{b}}>0$ half of phase space.}
	\label{fig:preliminary}
\end{figure}

We now introduce the full leC- phase space ${}^{\mathrm{leC}}\overline{T}^* X$. This is the mwc gotten from ${}^{\mathrm{b,leC}}\overline{T}^* X$ by first blowing up the edge $\mathrm{df}_{00} \cap \mathrm{tf}_{00}$, resulting in a mwc with five faces -- $\mathrm{df}_0,\mathrm{ff}_0,\mathrm{bf}_0,\mathrm{tf}_0,\mathrm{zf}_0$ (\Cref{fig:preliminary}, right), where $\mathrm{ff}_0$ is the front face of the blow up -- with bdfs 
\begin{equation}
	\varrho_{\mathrm{df}_0} = \frac{\varrho_{\mathrm{df}_{00}}}{\varrho_{\mathrm{df}_{00}}+\varrho_{\mathrm{tf}_{00}}} , \qquad \varrho_{\mathrm{ff}_0} = \varrho_{\mathrm{df}_{00}} + \varrho_{\mathrm{tf}_{00}}, \qquad \varrho_{\mathrm{tf}_0} = \frac{\varrho_{\mathrm{tf}_{00}}}{\varrho_{\mathrm{df}_{00}}+\varrho_{\mathrm{tf}_{00}}}, 
\end{equation}
$\varrho_{\mathrm{zf}_0} = \varrho_{\mathrm{zf}_{00}}$, $\varrho_{\mathrm{bf}_0} = \varrho_{\mathrm{bf}_{00}}$, and then performing a polar blowup of the edge $\mathrm{df}_0 \cap \mathrm{bf}_0$, resulting in a mwc with six faces -- $\mathrm{df},\mathrm{sf},\mathrm{ff},\mathrm{bf},\mathrm{tf},\mathrm{zf}$, where $\mathrm{sf}$ is front face of the second blowup -- with bdfs 
\begin{align}
	\begin{split} 
		\varrho_{\mathrm{df}} &= \frac{\varrho_{\mathrm{df}_{0}}}{\varrho_{\mathrm{df}_{0}}+\varrho_{\mathrm{bf}_{0}}} = \frac{\varrho_{\mathrm{df}_{00}}}{\varrho_{\mathrm{df}_{00}}+\varrho_{\mathrm{bf}_{00}}(\varrho_{\mathrm{df}_{00}}+\varrho_{\mathrm{tf}_{00}} )}  ,\\  
		\varrho_{\mathrm{sf}} &= \varrho_{\mathrm{df}_{0}} + \varrho_{\mathrm{bf}_0}
		= \frac{\varrho_{\mathrm{df}_{00}}}{\varrho_{\mathrm{df}_{00}}+\varrho_{\mathrm{tf}_{00}}} + \varrho_{\mathrm{bf}_{00}}, \\
		\varrho_{\mathrm{bf}} &= \frac{\varrho_{\mathrm{bf}_{0}}}{\varrho_{\mathrm{df}_{0}}+\varrho_{\mathrm{bf}_{0}}} = \frac{\varrho_{\mathrm{bf}_{00}} (\varrho_{\mathrm{df}_{00}}+\varrho_{\mathrm{tf}_{00}})}{\varrho_{\mathrm{df}_{00}} +  \varrho_{\mathrm{bf}_{00}}(\varrho_{\mathrm{df}_{00}}+\varrho_{\mathrm{tf}_{00}}) },
	\end{split} 
	\label{eq:bdfs}
\end{align}
$\varrho_{\mathrm{tf}} = \varrho_{\mathrm{tf}_0}$, $\varrho_{\mathrm{ff}}= \varrho_{\mathrm{ff}_0}$, $\varrho_{\mathrm{zf}} = \varrho_{\mathrm{zf}_0}$. See \Cref{fig:LEC_phase_space}.
Now let, for each $m,s,\varsigma,l,\ell \in \bbR$,  
\begin{align}
	\begin{split} 
		S_{\mathrm{leC}}^{m,s,\varsigma,l,\ell}(X) &= \calA^{-m,-s,-\varsigma,-l,-\ell,(0,0)}({}^{\mathrm{leC}}\overline{T}^* X) = \varrho_{\mathrm{df}}^{-m}\varrho_{\mathrm{sf}}^{-s}\varrho_{\mathrm{ff}}^{-\varsigma}\varrho_{\mathrm{bf}}^{-l}\varrho_{\mathrm{tf}}^{-\ell} S^{0,0,0}_{\mathrm{b,leC}}(X), \\
		S_{\mathrm{leC}}(X)&=\cup_{m,s,\varsigma,l,\ell\in \bbR} S_{\mathrm{leC}}^{m,s,\varsigma,l,\ell}(X), \\
		S_{\mathrm{cl,leC}}^{m,s,\varsigma,l,\ell}(X) &= \varrho_{\mathrm{df}}^{-m}\varrho_{\mathrm{sf}}^{-s}\varrho_{\mathrm{ff}}^{-\varsigma}\varrho_{\mathrm{bf}}^{-l}\varrho_{\mathrm{tf}}^{-\ell}C^\infty ({}^{\mathrm{leC}}\overline{T}^* X) \subset S_{\mathrm{leC}}^{m,s,\varsigma,l,\ell}(X), \\
		S_{\mathrm{cl,leC}}(X)&=\cup_{m,s,\varsigma,l,\ell\in \bbR} S_{\mathrm{cl,leC}}^{m,s,\varsigma,l,\ell}(X).
	\end{split} 
\end{align}
At the level of sets (and at the level of $\bbC$-algebras), $S_{\mathrm{leC}}(X)$ is equal to $S_{\mathrm{b,leC}}(X) = \cup_{m,l,\ell\in \bbR} S_{\mathrm{b,leC}}^{m,l,\ell}(X)$, but the filtration above presents $S_{\mathrm{leC}}(X)$ as a multigraded $\bbC$-algebra. The isomorphism 
\begin{equation} 
	\times \varrho_{\mathrm{df}}^{-m}\varrho_{\mathrm{sf}}^{-s}\varrho_{\mathrm{ff}}^{-\varsigma}\varrho_{\mathrm{bf}}^{-l}\varrho_{\mathrm{tf}}^{-\ell} : S^{0,0,0}_{\mathrm{b,leC}}(X) \to S_{\mathrm{leC}}^{m,s,\varsigma,l,\ell}(X)
\end{equation} 
of vector spaces 
allows us to consider each $\smash{S_{\mathrm{leC}}^{m,s,\varsigma,l,\ell}(X)}$ as a Fr\'echet space, and likewise for $\smash{S_{\mathrm{cl,leC}}^{m,s,\varsigma,l,\ell}(X)}$. The $\bbC$-algebras $S_{\mathrm{leC}}(X)$ and $S_{\mathrm{cl,leC}}(X)$ are then multigraded Fr\'echet algebras, as pointwise multiplication of symbols is jointly continuous with respect to the relevant topologies.

Observe that if $a\in S_{\mathrm{leC}}^{m,s,\varsigma,l,\ell}(X)$ then $a(-;0)\in S_{\mathrm{scb}}^{m,\varsigma,\ell}(X_{1/2})$. For $\sigma>0$,   $a(-;\sigma)\in S_{\mathrm{scb}}^{m,s,l}(X)$.

\begin{figure}[t]
	\begin{center} 
	\begin{tikzpicture}[scale=.75]
		\coordinate (ul) at (0,3.5,5) {};
		\coordinate (uul) at (0,4,4.5) {};
		\coordinate (uulc) at (3,4,4.5) {};
		\coordinate (ulc) at (3,3.5,5) {};
		\coordinate (urc) at (5,4,3) {};
		\coordinate (ur) at (5,4,0) {};
		\coordinate (lc) at (4,3,5) {};
		\coordinate (rc) at (5,3,4) {};
		\coordinate (ld) at (4,0,5) {};
		\coordinate (rd) at (5,0,4) {};
		\draw (ld)--(lc); 
		\draw[red] (lc)--(ulc);
		\draw (ulc)--(ul);
		\draw (uul)--(uulc)--(urc)--(rc)--(lc);
		\draw (rd)--(rc)--(urc)--(ur);
		\draw (ulc)--(uulc);
		\draw[dashed] (uul)--(ul);
		\draw[dashed] (ld)--(rd); 
		\draw[dashed] (ld) --++ (-3,0,0);
		\draw[dashed] (rd) --++ (0,0,-3);
		\draw[dashed] (ul) --++ (0,-3,0);
		\draw[dashed] (uul) --++ (0,0,-3);
		\draw[dashed] (ur) --++ (-3,0,0);
		\draw[dashed] (ur) --++ (0,-3,0);
		\node (df) at (2.5,4,2.5) {$\mathrm{df}$};
		\node (bf) at (2,2,5) {$\mathrm{bf}$};
		\node (zf) at (5,2,2) {$\mathrm{zf}$};
		\node (tf) at (4.5,2,4.5) {$\mathrm{tf}$};
		\node (ff) at (4.25,3.5,4.25) {$\mathrm{ff}$};
		\node (sf) at (1.5,3.75,4.75) {$\mathrm{sf}$};
	\end{tikzpicture}
	\end{center}
	\caption{The leC-phase space ${}^{\mathrm{leC}}\overline{T}^* X$, with the degrees of freedom associated with $T^* \partial X$ omitted.  
	The edge $\mathrm{ff}\cap \mathrm{bf}$ has been highlighted in red.} 
	\label{fig:LEC_phase_space}
\end{figure}

\begin{figure}[t]
	\begin{tikzpicture}[scale=.75]
		\coordinate (ul) at (0,3.5,5) {};
		\coordinate (uul) at (0,4,4.5) {};
		\coordinate (ulc) at (4,3.5,5) {};
		\coordinate (uulc) at (4,4,4.5) {};
		\coordinate (urc) at (5,4,4) {};
		\coordinate (ur) at (5,4,0) {};
		\coordinate (rc) at (5,4,4) {};
		\coordinate (ld) at (4,0,5) {};
		\coordinate (rd) at (5,0,4) {};
		\draw (ld)--(ulc)--(ulc)--(ul);
		\draw (uulc)--(urc)--(rc);
		\draw (rd)--(rc)--(urc)--(ur);
		\draw (uulc)--(uul);
		\draw (uulc)--(ulc);
		\draw[dashed] (uul)--(ul);
		\draw[dashed] (ld)--(rd); 
		\draw[dashed] (ld) --++ (-3,0,0);
		\draw[dashed] (rd) --++ (0,0,-3);
		\draw[dashed] (ul) --++ (0,-3,0);
		\draw[dashed] (uul) --++ (0,0,-3);
		\draw[dashed] (ur) --++ (-3,0,0);
		\draw[dashed] (ur) --++ (0,-3,0);
		\node (df) at (2.5,4,2.5) {$\mathrm{df}$};
		\node (bf) at (2,2,5) {$\mathrm{bf}$};
		\node (zf) at (5,2,2) {$\mathrm{zf}$};
		\node (tf) at (4.5,2,4.5) {$\mathrm{tf}$};
		\node (sf) at (2,3.75,4.75) {$\mathrm{sf}$};
	\end{tikzpicture}
	\caption{The phase space of Vasy's resolved calculus. Cf.\ \cite[Figure 4]{VasyN0L} (which also depicts the $\xi<0$ half of this phase space). }
	\label{fig:vasy}
\end{figure}

\subsection{Calculi}
\label{subsec:calculi}
After recalling some preliminary notions in \S\ref{subsubsec:Vasy}, we discuss the b,leC-calculus in \S\ref{subsubsec:bleC} and the full leC-calculus  in \S\ref{subsubsec:leC}. Since the b,leC-calculus is essentially $\Psi_{\mathrm{b,sp,res}}(X)$, except that we enforce classicality at $\mathrm{zf}_{00}$ (which introduces no complications), and since all references to the b,leC-calculus in the rest of the paper (i.e.\ \S\ref{sec:operator}, \S\ref{sec:symbolic}, \S\ref{sec:mainproof}) could be replaced by references to the leC-calculus with only notational complications, we will only sketch the arguments in \S\ref{subsubsec:bleC}. (However, it is important to understand the residual operators $\smash{\Psi_{\mathrm{b,leC}}^{-\infty,l,\ell}(X)}$, as these are the ``non-symbolic'' parts of leC-operators. But -- once again -- this  is essentially $\smash{\Psi_{\mathrm{b,sp,res}}^{-\infty,l,\ell}(X)}$ except for additional classicality at the ``zero face'' of Vasy's double space.)

\subsubsection{$\Psi_{\mathrm{b,sp,res}}(X)$}
\label{subsubsec:Vasy}

We now recall the notion of the b-calculus with conormal dependence on parameters: for any (compact) mwb $X$ and a (connected) mwc $M$ (the ``parameter space''), we have a multigraded $\bbC$-algebra 
\begin{align}
	\Psi_{\mathrm{b};M}(X)  &= \bigcup_{m,l \in \bbR} \Psi_{\mathrm{b};M}^{m,l}(X),\\ \Psi_{\mathrm{b};M}^{m,l}(X) &= \calA^0_{\mathrm{loc}}(M;  \Psi^{m,l}_{\mathrm{b}}(X)),
\end{align}
the members of which are the families $\{A_\lambda\}_{\lambda\in M^\circ}$ of b-$\Psi$DOs on $X$ depending conormally on a parameter $\lambda \in M^\circ$ (as above,  the `loc' refers to the fact that we only require uniform bounds in compact subsets of $M$). 
So, letting $\smash{C^\infty \cap L^\infty_{\mathrm{loc}}(M ;\Psi^{m,l}_{\mathrm{b}}(X))}$
denote the Fr\'echet space of smooth maps $M^\circ\to \smash{\Psi^{m,l}_{\mathrm{b}}(X)}$ that are (but whose derivatives are not necessarily) uniformly bounded with respect to each Fr\'echet seminorm of $\smash{\Psi^{m,l}_{\mathrm{b}}(X)}$ in every compact subset of $M$, 
\begin{equation}
	\Psi_{\mathrm{b};M}^{m,l}(X) = \{A_\bullet \in  C^\infty \cap L^\infty_{\mathrm{loc}}(M ;\Psi^{m,l}_{\mathrm{b}}(X))  : [L A_\bullet]_{\mathrm{a.e.}} \in L^\infty_{\mathrm{loc}}(M ;\Psi^{m,l}_{\mathrm{b}}(X)) \text{ $\forall$}L\in \operatorname{Diff}_{\mathrm{b}}(M) \}. 
\end{equation}
Note that each $\smash{\Psi_{\mathrm{b};M}^{m,l}(X)}$ is a Fr\'echet space, and $\Psi_{\mathrm{b};M}(X)$ can be regarded as a multigraded Fr\'echet algebra. 
Relevant to the study of spectral families of operators is the case $M=[0,\infty)_\lambda$. In this case, we write ``,\,sp'' in place of ``$;M$'' in the notation. 

In \cite[\S3]{VasyN0L}, Vasy (using slightly different notation) defines a particular ``refinement'' of $\Psi_{\mathrm{b,sp}}(X) = \Psi_{\mathrm{b};M}(X)$, 
\begin{equation}
	\Psi_{\mathrm{b},\mathrm{sp,res}}(X) = \Psi_{\mathrm{b,sp}}(X), \qquad \Psi_{\mathrm{b},\mathrm{sp,res}}(X)  = \bigcup_{m,l,\ell \in \bbR} \Psi_{\mathrm{b},\mathrm{sp,res}}^{m,l,\ell}(X),
\end{equation}
a multigraded $\bbC$-algebra 
which is equal, at the level of $\bbC$-algebras, to $\Psi_{\mathrm{b,sp}}(X)$, but with a 3-parameter multigrading (and associated symbol calculus) such that 
\begin{itemize}
	\item $\smash{\Psi_{\mathrm{b},\mathrm{sp,res}}^{m,l,\ell}(X) \subset \Psi_{\mathrm{b,sp}}^{m, \max\{l,\ell\}}(X)}$
	\item and $(\lambda+x)^{-l}$, considered as a multiplication operator, is in $\smash{\Psi_{\mathrm{b},\mathrm{sp,res}}^{0,0,l}(X)}$.
\end{itemize}
The three indices $m,l,\ell$ in $\smash{\Psi_{\mathrm{b},\mathrm{sp,res}}^{m,l,\ell}(X)}$ keep track of three notions of order, roughly  the ``differential order'' $m$, the b-decay order away from zero energy $l$ -- that is at $\mathrm{bf}_{00}$ -- and the b-decay order $\ell$ at $\mathrm{tf}_{00}$ (with respect to $\varrho_{0,\mathrm{tf}_{00}}$). 
In \cite{VasyN0L}, Vasy keeps track of an additional order, the order at $\mathrm{zf}_{00}$, which for our purposes can be set to zero throughout (as in \cref{eq:misc_kul}). 

Define $\sigma_{\mathrm{b,sp,res}}^{m,l,\ell}: \Psi_{\mathrm{b,sp,res}}^{m,l,\ell}(X) \to \calA_{\mathrm{loc}}^{-m,-l,-\ell,0}({}^{\mathrm{b,sp}}\overline{T}^* X )/\calA_{\mathrm{loc}}^{-m+1,-l,-\ell,0}({}^{\mathrm{b,sp}}\overline{T}^* X)$ by 
\begin{equation}
	\sigma_{\mathrm{b,sp,res}}^{m,l,\ell}(a)  = (\lambda+x)^{l-\ell} \sigma_{\mathrm{b,sp}}^{m,l}( (\lambda+x)^{\ell-l} a) 
\end{equation}
for $a\in S_{\mathrm{b,sp,res}}^{m,l,\ell}(X)$, where $\sigma_{\mathrm{b,sp}}(a)$ denotes the b-principal symbol map applied $\lambda$-wise to $a$ considered as an element of the family b-algebra $\Psi_{\mathrm{b,sp}}(X)$. Then, for all $m,l,\ell\in \bbR$, 
\begin{equation}
	0 \to \Psi_{\mathrm{b,sp,res}}^{m-1,l,\ell}(X)\hookrightarrow \Psi_{\mathrm{b,sp,res}}^{m,l,\ell}(X) \overset{\sigma_{\mathrm{b,sp,res}}}{\longrightarrow}\calA_{\mathrm{loc}}^{-m,-l,-\ell,0}({}^{\mathrm{sp,b}}\overline{T}^* X)/ \calA_{\mathrm{loc}}^{-m+1,-l,-\ell,0}({}^{\mathrm{sp,b}}\overline{T}^* X)  \to 0
	\label{eq:ses}
\end{equation}
is a short exact sequence and, for all $m',l',\ell'\in \bbR$,  
\begin{align} 
	\sigma_{\mathrm{b,sp,res}}^{m,l,\ell}(A)\sigma_{\mathrm{b,sp,res}}^{m',l',\ell'}(B) &= \smash{\sigma_{\mathrm{b,sp,res}}^{m+m',l+l',\ell+\ell'}(AB)}
	\label{eq:hom}
	\\
	\{\sigma_{\mathrm{b,sp,res}}^{m,l,\ell}(A),\sigma_{\mathrm{b,sp,res}}^{m',l',\ell'}(B)\} &= -i\sigma_{\mathrm{b,sp,res}}^{m+m'-1,l+l',\ell+\ell'}([A,B]).
	\label{eq:com}
\end{align} 
for all $\smash{A \in \Psi_{\mathrm{b,sp,res}}^{m,l,\ell}, B \in \smash{\Psi_{\mathrm{b,sp,res}}^{m',l',\ell'}}}$. 
We will compute Poisson brackets using the convention that momentum derivatives of the first entry have positive sign. (The sign in \cref{eq:com} depends on the choice of sign used in the Fourier transform used in defining the calculus.)
The bdfs $\smash{(\sigma^2+\mathsf{Z}x)^{1/2}}$ and $\smash{x/(\sigma^2+\mathsf{Z}x)}$ of $X^{\mathrm{sp}}_{\mathrm{res}}$, considered as multiplication operators, are representatives of their own principal symbols:
\begin{align}
	\sigma_{\mathrm{b,sp,res}}^{0,l,0}( x^{-l} (\sigma^2 +\mathsf{Z} x)^{l} ) &= x^{-l} (\sigma^2 +\mathsf{Z} x)^{l} \mod S_{\mathrm{b,sp,res}}^{-1,l,0}(X) 
	\label{eq:misc_54a} \\
	\sigma_{\mathrm{b,sp,res}}^{0,0,\ell}(  (\sigma^2 +\mathsf{Z} x)^{-\ell} ) &=  (\sigma^2 +\mathsf{Z} x)^{-\ell} \mod S_{\mathrm{b,sp,res}}^{-1,0,\ell}(X). \label{eq:misc_54b}
\end{align} 
More generally, if $a \in S_{\mathrm{b,sp,res}}^{m,l,\ell}(X)$, then $a \in \sigma_{\mathrm{b,sp,res}}^{m,l,\ell}(a)$. 

It is very convenient to make use of a ``(left) quantization'' map (right quantization working equally well):
\begin{align}
	\operatorname{Op}&: S_{\mathrm{b}}(X)\to \Psi_{\mathrm{b}}(X) \\ 
	&: S_{\mathrm{b}}^{m,l}(X) \to \Psi_{\mathrm{b}}^{m,l}(X),
	\label{eq:misc_msl}
\end{align} 
discussed e.g. in \cite{VasyGrenoble} among other places, 
given by the left quantization of symbols in local coordinates. 
This will be noncanonical, depending on a choice of atlas on $X$, among other things. While not surjective (missing out on the remainder term $R'$ in \cite[Definition 5.15]{VasyGrenoble}), it will be modulo $\Psi_{\mathrm{b}}^{-\infty,l}(X)$.

Applied $\lambda$-wise to an element of $\calA^0([0,\infty)_\lambda;S_{\mathrm{b}}^{m,l}(X))$, the result is an element of $\Psi_{\mathrm{b,sp}}^{m,l}(X)= \calA^0([0,\infty)_\lambda;\Psi_{\mathrm{b}}^{m,l}(X))$. Some elementary properties of $\operatorname{Op}$ which we can arrange are: 
\begin{itemize}
	\item for any $f \in \cup_{l\in \bbR}\calA^l (X)$,
	\begin{equation} 
		\operatorname{Op}(f)= f 
		\label{eq:misc_l12}
	\end{equation} 
	(this property distinguishing left quantization from right),
	where the $f$ on the right-hand side denotes the multiplication operator $u\mapsto fu$,
	\item $\operatorname{Op}$ is $\bbC$-linear, 
	\item \cref{eq:misc_msl} is continuous for any $m,l\in \bbR$, 
	\item $\sigma_{\mathrm{b}}^{m,l}(\operatorname{Op}(a)) = a \bmod S_{\mathrm{b}}^{m-1,l}(X)$ for all $a\in  S_{\mathrm{b}}^{m,l}(X)$.
\end{itemize}
(\Cref{eq:misc_l12} holds for the Kohn--Nirenberg calculus $\Psi_\infty$ under left quantization. Since the Schwartz kernel of the multiplication operator $u\mapsto fu$ is supported on the diagonal, it is unaltered by the cutoff $\psi(t-t')$ in \cite[Definition 5.15]{VasyGrenoble}. As a consequence, \cref{eq:misc_l12} holds also for $\Psi_{\mathrm{b}}$.)
\subsubsection{$\Psi_{\mathrm{b,leC}}(X)$}
\label{subsubsec:bleC}

Let $\Psi_{\mathrm{b,leC}}^{-\infty,l,\ell}(X)$ denote the elements of $\Psi_{\mathrm{b,sp,res}}^{-\infty,l,\ell/2}(X)$ whose Schwartz kernels are smooth at the face $\operatorname{cl}\{\sigma=0,x'>0,x>0\}$ of the double space $X^{2\mathrm{b,sp,res}}$ \cite[Figure 2]{VasyN0L} (with the terms in the Taylor series being elements of $\smash{\Psi_{\mathrm{b}}^{-\infty,\ell}(X_{1/2})}$): 
\begin{equation}
	\operatorname{SK}\Psi_{\mathrm{b,leC}}^{-\infty,l,\ell}(X) = \calA_{\mathrm{loc}}^{-l,-\ell/2,-\infty,-\infty,(0,0)}(X^{2\mathrm{b,sp,res}}), 
	\label{eq:misc_ska}
\end{equation}
where we are listing the boundary faces of $X^{2\mathrm{b,sp,res}}$ in the order $\mathrm{bf}=\operatorname{cl}\{x,x'=0, \sigma>0\}$, $\mathrm{tf}=\{x,x'=0, \sigma=0\}$, $\mathrm{lb} = \operatorname{cl}\{x=0,x'>0,\sigma>0\}$, $\mathrm{rb} = \operatorname{cl}\{x'=0,x>0,\sigma>0\}$, and $\mathrm{zf} = \operatorname{cl}\{\sigma=0,x>0,x'>0\}$. Thus, $\Psi_{\mathrm{b,leC}}^{-\infty,l,\ell}(X)$ inherits from \cref{eq:misc_ska} a Fr\'echet space structure, and it can be shown that 
\begin{equation} 
	\Psi_{\mathrm{b,leC}}^{-\infty,\infty,\infty}(X) = \cup_{l,\ell\in \bbR} \Psi_{\mathrm{b,leC}}^{-\infty,l,\ell}(X)
\end{equation} 
is then a multigraded Fr\'echet algebra. (So operator composition defines a jointly continuous map $\Psi_{\mathrm{b,leC}}^{-\infty,l,\ell}(X)\times \Psi_{\mathrm{b,leC}}^{-\infty,l',\ell'}(X) \to \Psi_{\mathrm{b,leC}}^{-\infty,l+l',\ell+\ell'}(X)$ for all $l,\ell,l',\ell'\in \bbR$.)

An argument similar to that used to prove \Cref{prop:continuity} yields:
\begin{propositionp}
	\label{prop:residual_continuity}
	Given an element $K \in \operatorname{SK}\Psi_{\mathrm{b,leC}}^{-\infty,l,\ell}(X)$, if $l_0\geq l$ and $l_0>\ell/2$, then it is the case that $\{K(-;\sigma)\}_{\sigma \geq 0} \in C^0([0,\infty)_\sigma ; \operatorname{SK} \Psi_{\mathrm{b}}^{m,l_0}(X))$ for any $m\in \bbR$. 
\end{propositionp}

Consequently, elements of $\Psi_{\mathrm{b,leC}}^{-\infty,\infty,\infty}(X)$ can be considered as continuous families of b-$\Psi$DOs indexed either by $\bbR^+_\sigma$ or by $[0,\infty)_\sigma $. 
We now define, for each $m,l,\ell\in \bbR$, 
\begin{equation}
	\Psi_{\mathrm{b,leC}}^{m,l,\ell}(X) = \operatorname{Op}(S_{\mathrm{b,leC}}^{m,l,\ell}(X)) +  \Psi_{\mathrm{b,leC}}^{-\infty,l,\ell}(X).
\end{equation} 
Thus, by \cref{eq:base_inclusion}, if $l_0\geq l,\ell/2$, then $\Psi_{\mathrm{b,leC}}^{m,l,\ell}(X) \subseteq \Psi_{\mathrm{b,sp}}^{m,l_0}(X)$.
In addition: 
\begin{proposition}
	\label{prop:operator_cont}
	If $a = \{a(-;\sigma)\}_{\sigma \geq 0} \in S_{\mathrm{b,leC}}^{m,l,\ell}(X)$, then, if $l_0\geq l$ and $l_0 > \ell/2$,  
	\begin{equation}
		\{\operatorname{Op}(a(-;\sigma))\}_{\sigma \geq 0} \in C^0([0,\infty)_\sigma ; \Psi_{\mathrm{b}}^{m,l_0}).
	\end{equation}
	Consequently, if $A = \{A(\sigma)\}_{\sigma>0} \in \Psi_{\mathrm{b,leC}}^{m,l,\ell}(X)$, then there exists some $A(0) \in \Psi_{\mathrm{b}}^{m,l_0}(X)$ such that $\{A(\sigma)\}_{\sigma\geq 0}  \in C^0([0,\infty)_\sigma ; \Psi_{\mathrm{b}}^{m,l_0} )$.
\end{proposition}
\begin{proof}
	Using the continuity of $\operatorname{Op}:S_{\mathrm{b}}^{m,l_0}(X) \to \Psi_{\mathrm{b}}^{m,l_0}(X)$, the first statement follows from \Cref{prop:continuity}. The second statement follows from the first in conjunction with \Cref{prop:residual_continuity}. 
\end{proof}

Since $\operatorname{Op}$ is linear, $\smash{\Psi_{\mathrm{b,leC}}^{m,l,\ell}(X)}$ is a vector space, and it inherits a Fr\'echet space structure from $S_{\mathrm{b,leC}}^{m,l,\ell}(X)$ and $\Psi_{\mathrm{b,leC}}^{-\infty,l,\ell}(X)$, more specifically the quotient topology associated to the definitional surjection 
\begin{equation} 
	S_{\mathrm{b,leC}}^{m,l,\ell}(X) \times \Psi_{\mathrm{b,leC}}^{-\infty,l,\ell}(X) \to \Psi_{\mathrm{b,leC}}^{m,l,\ell}(X).
\end{equation} 
From the definition of $\smash{\Psi_{\mathrm{b,sp,res}}^{m,l,\ell/2}(X)}$ given in \cite[\S3]{VasyN0L}, $\smash{\Psi_{\mathrm{b,leC}}^{m,l,\ell}(X)  \subset \Psi_{\mathrm{b,sp,res}}^{m,l,\ell/2}(X)}$.

Just as the set of classical Kohn--Nirenberg $\Psi$DOs is a subalgebra of the calculus of all Kohn--Nirenberg $\Psi$DOs, 
\begin{equation}
	\Psi_{\mathrm{b,leC}}(X) = \bigcup_{m,l,\ell\in \bbR} \Psi_{\mathrm{b,leC}}^{m,l,\ell}(X) 
\end{equation}
is a subalgebra of $\Psi_{\mathrm{b,sp,res}}(X)$, with composition of $\Psi$DOs defining jointly continuous products 
\begin{equation} 
	\Psi_{\mathrm{b,leC}}^{m,l,\ell}(X)\times \Psi_{\mathrm{b,leC}}^{m',l',\ell'}(X) \to
	\Psi_{\mathrm{b,leC}}^{m+m',l+l',\ell+ \ell'}(X)
\end{equation} 
for all $m,l,\ell,m',l',\ell' \in \bbR$. The key observation here, in addition to the continuity of $	\Psi_{\mathrm{b,leC}}^{m,l,\ell}(X)\times \smash{\Psi_{\mathrm{b,leC}}^{-\infty,l',\ell'}(X)} \to
\smash{\Psi_{\mathrm{b,leC}}^{-\infty,l+l',\ell+ \ell'}(X)}$ for all $m,l,\ell,l',\ell' \in \bbR$,  is that the reduction formula for full symbols in local coordinates respects classicality at $\mathrm{zf}_{00}$. 

From $\sigma_{\mathrm{b,sp,res}}$, we get a set $\{\sigma_{\mathrm{b,leC}}^{m,l,\ell}\}_{m,l,\ell\in\bbR}$ of maps $\sigma_{\mathrm{b,leC}}^{m,l,\ell}: \Psi_{\mathrm{b,leC}}^{m,l,\ell} \to S_{\mathrm{b,leC}}^{m,l,\ell}(X) / S_{\mathrm{b,leC}}^{m-1,l,\ell}(X)$
such that  
\begin{equation}
	0 \to \Psi_{\mathrm{b,leC}}^{m-1,l,\ell}(X)\hookrightarrow \Psi_{\mathrm{b,leC}}^{m,l,\ell}(X) \overset{\sigma_{\mathrm{b,leC}}}{\longrightarrow} S^{m,l,\ell}_{\mathrm{b,leC}}(X)/S^{m-1,l,\ell}_{\mathrm{b,leC}}(X)\to 0
	\label{eq:ses2}
\end{equation}
is a short exact sequence and such that 
\begin{align} 
\sigma_{\mathrm{b,leC}}^{m,l,\ell}(A)\sigma_{\mathrm{b,sp,res}}^{m',l',\ell'}(B) &= \sigma_{\mathrm{b,leC}}^{m+m',l+l',\ell+\ell'}(AB)
\label{eq:hom2} \\
\{\sigma_{\mathrm{b,leC}}^{m,l,\ell}(A),\sigma_{\mathrm{b,leC}}^{m',l',\ell'}(B)\} &= -i\sigma_{\mathrm{b,leC}}^{m+m'-1,l+l',\ell+\ell'}([A,B]).
\label{eq:com2}
\end{align} 
for all $A \in \Psi_{\mathrm{b,leC}}^{m,l,\ell}, B \in \smash{\Psi_{\mathrm{b,leC}}^{m',l',\ell'}}$, with each of \cref{eq:ses2}, \cref{eq:hom2}, \cref{eq:com2} following  from  each of \cref{eq:ses}, \cref{eq:hom}, \cref{eq:com} respectively. 

Let $\calA^{0,0,(0,0)}  (X^{\mathrm{sp}}_{\mathrm{res}})$ denote the set of distributions on $X^{\mathrm{sp}}_{\mathrm{res}}$ which are conormal to all boundaries and smooth at $\mathrm{zf}$ (in particular smooth everywhere except possibly at $\mathrm{tf} ,\mathrm{bf}$).
\begin{proposition}
	\label{prop:multiplication}
	For any $f\in \calA^{0,0,(0,0)}  (X^{\mathrm{sp}}_{\mathrm{res}})$ and $l,\ell\in \bbR$, the multiplication operator given by multiplication by $x^{-l} (\sigma^2+\mathsf{Z}x)^{l-\ell/2} f(x;\sigma)$ defines an element of $\Psi_{\mathrm{b,leC}}^{0,l,\ell}(X)$.
\end{proposition}
\begin{proof}
	The given multiplication operator $M=\{M(\sigma)\}_{\sigma>0}$ is given by $M(\sigma)=\operatorname{Op}(x^{-l} (\sigma^2+\mathsf{Z} x)^{l-\ell/2} f(x;\sigma) )$ (using \cref{eq:misc_l12} for each individual $\sigma>0$), so the proposition follows from $f\in S_{\mathrm{b,leC}}^{0,0,0}(X)$. 
\end{proof}

\subsubsection{$\Psi_{\mathrm{leC}}(X)$}
\label{subsubsec:leC}

In order to define the full leC-calculus, we will use the following properties of $\operatorname{Op}$:
\begin{itemize} 
	\item $\operatorname{Op}(a) \in \Psi_{\mathrm{b,leC}}^{m,l,\ell}(X)$ (if and) only if $a \in S_{\mathrm{b,leC}}^{m,l,\ell}(X)$. 
	\item \begin{equation} 
		\sigma_{\mathrm{b,leC}}^{m,l,\ell}(\operatorname{Op}(a)) = a \bmod S_{\mathrm{b,leC}}^{m-1,l,\ell}(X)
		\label{eq:misc_84h}
	\end{equation} 
	whenever $\operatorname{Op}(a) \in \Psi_{\mathrm{b,leC}}^{m,l,\ell}(X)$, 
	\item 
	there exists a function $\sharp : S_{\mathrm{b,leC}}(X)^2 \to S_{\mathrm{b,leC}}(X)$ (given by the ``reduction formula'' for $\Psi_\infty$, related to $\Psi_{\mathrm{b}}$ via \cite[\S6]{VasyGrenoble}) such that, 
	for any $m,m',s,s',\varsigma,\varsigma',l,l',\ell,\ell'\in \bbR$, $a\in S_{\mathrm{leC}}^{m,s,\varsigma,l,\ell}(X)$ and $b \in S_{\mathrm{leC}}^{m',s',\varsigma', l',\ell'}(X)$, 
	\begin{align}
		a\sharp b &\in S_{\mathrm{leC}}^{m+m',s+s',\varsigma+\varsigma',l+l',\ell+\ell'}(X), \\
		\operatorname{esssupp}_{\mathrm{leC}}(a\sharp b)&\subseteq \operatorname{esssupp}_{\mathrm{leC}}(a)\cap \operatorname{esssupp}_{\mathrm{leC}}( b), \label{eq:misc_k45} \\
		\operatorname{Op}(a)\operatorname{Op}(b) &= \operatorname{Op}(a\sharp b) +E,
	\end{align}
	for some $E \in \smash{\Psi_{\mathrm{b,leC}}^{-\infty,l+l',\ell+\ell'}(X)}$ which depends continuously on $a,b$. Moreover,
	\begin{align}
		a\sharp b &= a b \bmod S_{\mathrm{leC}}^{m+m'-1,s+s'-1,\varsigma+\varsigma'-1,l+l',\ell+\ell'}(X), \label{eq:misc_k46} \\
		a \sharp b - b \sharp a &= i \{a,b\} \bmod S_{\mathrm{leC}}^{m+m'-2,s+s'-2,\varsigma+\varsigma'-2,l+l',\ell+\ell'}(X) \label{eq:misc_k47}
	\end{align} 
	for all such $a,b$, 
	\item there exists another continuous ($\bbC$-antilinear) function $\flat:S_{\mathrm{b,leC}}(X)\to S_{\mathrm{b,leC}}(X)$ (which can also be written in local coordinates in terms of the reduction formula) such that for all $a\in \smash{S_{\mathrm{leC}}^{m,s,\varsigma,l,\ell}(X)}$,
	\begin{equation} 
		\operatorname{Op}(\flat a) = \operatorname{Op}(a)^* + E
	\end{equation} 
	for some $E\in \Psi_{\mathrm{b,leC}}^{-\infty,l,\ell}(X)$ which depends continuously on $a$, 
	where the asterisk denotes an $L^2(X,g_0)$-based adjoint (for arbitrary exactly conic $g_0$), and 
	\begin{align}
		\flat a &= a^* \bmod S_{\mathrm{leC}}^{m-1,s-1,\varsigma-1,l,\ell}(X)  \\
		\operatorname{esssupp}_{\mathrm{leC}}(\flat a) &= \operatorname{esssupp}_{\mathrm{leC}}(a).
	\end{align}
	
\end{itemize}
Here, for $s\in S_{\mathrm{b,leC}}(X)$, $\operatorname{esssupp}_{\mathrm{leC}}(s)$  consists of those points in $\mathrm{df}\cup \mathrm{sf} \cup \mathrm{ff}$ failing to possess a neighborhood in which $s$ vanishes to infinite order at the boundary of the leC- phase space. 

As in \cite{VasyLA}\cite{VasyN0L}, these properties follow from the relation between $\Psi_{\mathrm{b}}(X)$ and $\Psi_\infty(\bbR^n)$, as explained in \cite[\S6]{VasyGrenoble}, and the basic properties of $\Psi_\infty(\bbR^n)$ (in particular the reduction formula), for which the standard reference is \cite{Hormander}.
It is crucial for us that $\sharp$ satisfies the equations \cref{eq:misc_k45}, \cref{eq:misc_k46}, \cref{eq:misc_k47} above and not just the weaker b,leC- analogues. This fundamental fact can be read off of the reduction formula for full symbols in local coordinates, in terms of which $\sharp$ can be written.

We can now define, for each $m,s,\varsigma,l,\ell\in \bbR$, 
\begin{equation}
	\Psi_{\mathrm{leC}}^{m,s,\varsigma,l,\ell}(X) = \operatorname{Op}(S_{\mathrm{leC}}^{m,s,\varsigma,l,\ell}(X)) + \Psi_{\mathrm{b,leC}}^{-\infty,l,\ell} \subseteq \Psi_{\mathrm{b,leC}}^{\infty,l,\ell}. 
	\label{eq:misc_frc}
\end{equation}
Evidently, \cref{eq:misc_frc} endows $\Psi_{\mathrm{b,leC}}^{m,s,\varsigma,l,\ell}(X)$ with a topology, so that it becomes a Fr\'echet space. Consider the graded vector space 
\begin{equation}
	\Psi_{\mathrm{leC}}(X) = \bigcup_{m,s,\varsigma,l,\ell\in \bbR} \Psi_{\mathrm{leC}}^{m,s,\varsigma,l,\ell}(X). 
\end{equation}
At the level of vector spaces, this is just $\Psi_{\mathrm{b,leC}}(X)$. 

Moreover, since $S_{\mathrm{leC}}^{m,m+l,m+\ell,l,\ell}(X) = S_{\mathrm{b,leC}}^{m,l,\ell}(X)$ at the level of sets, 
$\Psi_{\mathrm{leC}}^{m,m+l,m+\ell,l,\ell}(X) = \Psi_{\mathrm{b,leC}}^{m,l,\ell}(X)$ for all $m,l,\ell\in \bbR$.

\begin{proposition}
	\label{prop:algebra}
	$\Psi_{\mathrm{leC}}(X)$ is a multigraded $\bbC$-algebra:  for any $m,m',s,s',\varsigma,\varsigma',l,l',\ell,\ell'\in \bbR$, $A\in \Psi_{\mathrm{leC}}^{m,s,\varsigma,l,\ell}(X)$ and $B \in \Psi_{\mathrm{leC}}^{m',s',\varsigma', l',\ell'}(X)$, 
	\begin{equation}
		AB \in \Psi_{\mathrm{leC}}^{m+m',s+s',\varsigma+\varsigma',l+l',\ell+\ell'}(X). 
		\label{eq:misc_ab1}
	\end{equation}
\end{proposition}
\begin{proof}
	We can write $A=\operatorname{Op}(a) + E$ and $B=\operatorname{Op}(b) + F$ for $a \in S_{\mathrm{leC}}^{m,s,\varsigma,l,\ell}$, $b\in S_{\mathrm{leC}}^{m',s',\varsigma', l',\ell'}$, $E\in \Psi_{\mathrm{b,leC}}^{-\infty,l,\ell}$, $F\in \smash{\Psi_{\mathrm{b,leC}}^{-\infty,l',\ell'}}$. Thus, 
	\begin{align}
		AB &= \operatorname{Op}(a) \operatorname{Op}(b) + E \operatorname{Op}(b) + \operatorname{Op}(a) F + EF \\
		&= \operatorname{Op}(a\sharp b) + E \operatorname{Op}(b) + \operatorname{Op}(a) F + EF +G
	\end{align}
	for some $\smash{G\in \Psi_{\mathrm{b,leC}}^{-\infty,l+l',\ell+\ell'}}$. Since  $a \in S_{\mathrm{b,leC}}^{M,l,\ell}$, $b\in S_{\mathrm{b,leC}}^{M', l',\ell'}$ for $M=\max\{m,s-l,\varsigma-\ell\}$ and $M'=\max\{m',s'-l',\varsigma'-\ell'\}$, 
	\begin{equation}
		\operatorname{Op}(a) \in \Psi_{\mathrm{b,leC}}^{M,l,\ell}(X),
		\qquad 
		\operatorname{Op}(b) \in \Psi_{\mathrm{b,leC}}^{M',l',\ell'}(X), 
	\end{equation}
	which implies that $E \operatorname{Op}(b),\operatorname{Op}(a) F \in \Psi_{\mathrm{b,leC}}^{-\infty,l+l',\ell+\ell'}$, and likewise $E F \in \Psi_{\mathrm{b,leC}}^{-\infty,l+l',\ell+\ell'}$. 
	
	Since $a\sharp b \in S_{\mathrm{leC}}^{m+m',s+s',\varsigma+\varsigma',l+l',\ell+\ell'}$, we deduce that \cref{eq:misc_ab1} holds. 
\end{proof}
	In fact (as can be proven using finite-order truncations of the reduction formula), operator composition defines a jointly continuous map 
	\begin{equation} 
	\Psi_{\mathrm{leC}}^{m,s,\varsigma,l,\ell}(X)\times \Psi_{\mathrm{leC}}^{m',s',\varsigma', l',\ell'}(X) \to \Psi_{\mathrm{leC}}^{m+m',s+s',\varsigma+\varsigma',l+l',\ell+\ell'}(X), 
	\end{equation} 
	so we can say that the leC-calculus is a multigraded Fr\'echet algebra. 
\begin{lemma}
	Suppose that $A \in \smash{\Psi_{\mathrm{leC}}^{m,s,\varsigma,l,\ell}(X)}$ can be written either as $A=\operatorname{Op}(a_1)+E_1$ or $A=\operatorname{Op}(a_2) + E_2$ for some symbols $a_1,a_2 \in \smash{S_{\mathrm{leC}}^{m,s,\varsigma, l,\ell}}(X)$ and $E_1,E_2 \in \Psi_{\mathrm{b,leC}}^{-\infty,l,\ell}$. 
	
	Then $a_1-a_2 \in S_{\mathrm{b,leC}}^{-\infty,l,\ell}(X) $. 
\end{lemma} 
\begin{proof}
	By the linearity of $\operatorname{Op}$, $\operatorname{Op}(a_1-a_2) \in \smash{\Psi_{\mathrm{b,leC}}^{-\infty,l,\ell}(X)}$. Then, using the $\sigma_{\mathrm{b,leC}}$-short exact sequence and the property \cref{eq:misc_84h} of $\operatorname{Op}$: for all $N\in \bbN$,
	\begin{equation} 
		0=\sigma_{\mathrm{b,leC}}^{-N,l,\ell}(\operatorname{Op}(a_1-a_2)) = a_1 - a_2 \bmod S_{\mathrm{b,leC}}^{-N-1,l,\ell}(X), 
	\end{equation} 
	which means that $a_1-a_2 \in S_{\mathrm{b,leC}}^{-N-1,l,\ell}(X)$.
\end{proof}

Thus, for any $A \in \Psi_{\mathrm{leC}}^{m,s,\varsigma,l,\ell}(X)$, we get notions 
\begin{align}
	\operatorname{Ell}_{\mathrm{leC}}^{m,s,\varsigma,l,\ell}(A) &= \operatorname{ell}_{\mathrm{leC}}^{m,s,\varsigma,l,\ell}(a) \subset \mathrm{df} \cup \mathrm{sf} \cup \mathrm{ff}, \\
	\operatorname{Char}_{\mathrm{leC}}^{m,s,\varsigma,l,\ell}(A) &= \operatorname{char}_{\mathrm{leC}}^{m,s,\varsigma,l,\ell}(a) = \mathrm{df}\cup \mathrm{sf} \cup \mathrm{ff} \backslash \operatorname{ell}_{\mathrm{leC}}^{m,s,\varsigma,l,\ell}(a), \\
	\operatorname{WF}_{\mathrm{leC}}^{\prime l,\ell}(A) &= \operatorname{esssupp}_{\mathrm{leC}}(a),
\end{align}
for any $a \in S_{\mathrm{leC}}^{m,s,\varsigma,l,\ell}$ with $A=\operatorname{Op}(a) + E$ for $E \in \Psi_{\mathrm{b,leC}}^{-\infty,l,\ell}(X)$. All three are subsets of $\mathrm{df}\cup \mathrm{sf}\cup \mathrm{ff}$. 
Another useful notion, which we will only apply to $A \in \cap_{l,\ell\in \bbR}\smash{\Psi_{\mathrm{leC}}^{m,s,\varsigma,l,\ell}(X)} = \smash{\Psi_{\mathrm{leC}}^{m,s,\varsigma,-\infty,-\infty}}(X)$, is 
\begin{equation}
	\operatorname{WF}'_{\mathrm{leC}}(A)  = \cap_{l,\ell\in \bbR} \operatorname{WF}_{\mathrm{leC}}^{\prime l,\ell}(A). 
\end{equation}
Observe that if $a \in S_{\mathrm{leC}}^{m,s,\varsigma,-\infty,-\infty}(X)$, then $\operatorname{Op}(A) \in \Psi_{\mathrm{leC}}^{m,s,\varsigma,-\infty,-\infty}(X)$ (tautologically) and 
\begin{equation} 
	\operatorname{WF}'_{\mathrm{leC}}(A) = \operatorname{esssupp}_{\mathrm{leC}}(a).
\end{equation} 
The $a \in S_{\mathrm{leC}}^{m,s,\varsigma,-\infty,-\infty}(X)$ that we consider are all supported away from $\mathrm{bf}\cup \mathrm{tf}$, in which case $\operatorname{WF}'_{\mathrm{leC}}(A)$ is disjoint from $\mathrm{bf}\cup \mathrm{tf}$. 

In addition,  we get the leC- principal symbol maps 
\[ 
\{\sigma_{\mathrm{leC}}^{m,s,\varsigma,l,\ell}\}_{m,s,\varsigma,l,\ell\in \bbR}, \qquad \sigma_{\mathrm{leC}}^{m,s,\varsigma,l,\ell} : \Psi_{\mathrm{leC}}^{m,s,\varsigma,l,\ell}(X) \to S_{\mathrm{leC}}^{m,s,\varsigma,l,\ell}(X) / S_{\mathrm{leC}}^{m-1,s-1,\varsigma-1,l,\ell}(X), 
\]
$\sigma_{\mathrm{leC}}^{m,s,\varsigma,l,\ell}(A) = a \bmod  S_{\mathrm{leC}}^{m-1,s-1,\varsigma-1,l,\ell}$ for any $a \in S_{\mathrm{leC}}^{m,s,\varsigma,l,\ell}$ with $A=\operatorname{Op}(a) + E$ for $E \in \smash{\Psi_{\mathrm{b,leC}}^{-\infty,l,\ell}}$. 
Unlike $\operatorname{Op}$, which is not canonical and depends on a particular choice of local coordinate charts, these notions are all canonical.

\begin{proposition}
	For every $m,s,\varsigma,l,\ell\in \bbR$, we have a short exact sequence 
	\begin{equation}
		0 \to \Psi^{m-1,s-1,\varsigma-1,l,\ell}_{\mathrm{leC}}(X)\hookrightarrow \Psi^{m,s,\varsigma,l,\ell}_{\mathrm{leC}}(X) \to S^{m,s,\varsigma,l,\ell}_{\mathrm{leC}}(X) / S^{m-1,s-1,\varsigma-1,l,\ell}_{\mathrm{leC}}(X)  \to 0, 
	\end{equation}
	where the second-to-last map is $\sigma_{\mathrm{leC}}^{m,s,\varsigma,l,\ell}$. 
\end{proposition}
\begin{proof}
	The surjectivity of $\sigma_{\mathrm{leC}}^{m,s,\varsigma,l,\ell}$ follows from the properties of $\operatorname{Op}$ listed above.
	
	If, on the other hand, $A \in\Psi^{m,s,\varsigma,l,\ell}_{\mathrm{leC}}(X)$ satisfies   $\sigma_{\mathrm{leC}}^{m,s,\varsigma,l,\ell}(A) = 0$, then 
	$A = \operatorname{Op}(a) + E$ for $a\in S_{\mathrm{leC}}^{m-1,s-1,\varsigma-1,l,\ell}$ and $E \in \Psi_{\mathrm{b,leC}}^{-\infty,l,\ell}$. Then, by the definition of $\Psi^{m-1,s-1,\varsigma-1,l,\ell}_{\mathrm{leC}}(X)$, $A \in \Psi^{m-1,s-1,\varsigma-1,l,\ell}_{\mathrm{leC}}(X)$. 
\end{proof}

\begin{proposition}
	\label{prop:basic_symbology}
	For every $m,s,\varsigma,l,\ell,m',s',\varsigma',l',\ell' \in \bbR$ and pair of $A \in \Psi_{\mathrm{leC}}^{m,s,\varsigma,l,\ell}(X)$ and $B \in \Psi_{\mathrm{leC}}^{m',s',\varsigma',l',\ell'}(X)$,
	\begin{equation} 
		\sigma_{\mathrm{leC}}^{m,s,\varsigma,l,\ell}(A)\sigma_{\mathrm{leC}}^{m',s',\varsigma',l',\ell'}(B) = \sigma_{\mathrm{leC}}^{m+m',s+s',\varsigma+\varsigma',l+l',\ell+\ell'}(AB).
	\end{equation} 
	Moreover,  $\sigma_{\mathrm{leC}}^{m+m'-1,s+s'-1,\varsigma+\varsigma'-1,l+\ell',\ell+\ell'}([A,B])$ is equal to 
	\begin{equation}
		 i\{a,b\} \bmod S_{\mathrm{leC}}^{m+m'-2,s+s'-2,\varsigma+\varsigma'-2,l+\ell',\ell+\ell'}(X), 
		 \label{eq:misc_cid}
	\end{equation} 
	where $a,b$ are any representatives of 	$\sigma_{\mathrm{leC}}^{m,s,\varsigma,l,\ell}(A)$ and $\sigma_{\mathrm{leC}}^{m',s',\varsigma',l',\ell'}(B)$.
\end{proposition}
\begin{proof}
	Write $A = \operatorname{Op}(a)+E$ and $B=\operatorname{Op}(b)+F$ for $a \in S_{\mathrm{leC}}^{m,s,\varsigma,l,\ell}(X)$, $b \in S_{\mathrm{leC}}^{m',s',\varsigma',l',\ell'}(X)$, $E \in \Psi_{\mathrm{leC}}^{-\infty,l,\ell}(X)$, $F\in \Psi_{\mathrm{leC}}^{-\infty,l',\ell'}$. Then 
	\begin{align}
		\begin{split} 
		\sigma_{\mathrm{leC}}^{m+m',s+s',\varsigma+\varsigma',l+\ell',\ell+\ell'}(AB) &= \sigma_{\mathrm{leC}}^{m+m',s+s',\varsigma+\varsigma',l+\ell',\ell+\ell'}(\operatorname{Op}(a)\operatorname{Op}(b)) \\
		&= \sigma_{\mathrm{leC}}^{m+m',s+s',\varsigma+\varsigma',l+\ell',\ell+\ell'}(\operatorname{Op}(a\sharp b)) \\
		&= a \sharp b \bmod S_{\mathrm{leC}}^{m+m'-1,s+s'-1,\varsigma+\varsigma'-1,l+l',\ell+\ell'}(X) \\
		&= a b \bmod S_{\mathrm{leC}}^{m+m'-1,s+s'-1,\varsigma+\varsigma'-1,l+l',\ell+\ell'}(X) \\
		&= 	\sigma_{\mathrm{leC}}^{m,s,\varsigma,l,\ell}(A)\sigma_{\mathrm{leC}}^{m',s',\varsigma',l',\ell'}(B).
		\end{split} 
	\end{align}

	The proof of \cref{eq:misc_cid} is similar. 
\end{proof}

\begin{proposition}
	\label{prop:diff->PsiDO} 
	Suppose that $L \in S\operatorname{Diff}_{\mathrm{scb}}^{m,s,l}(X)$, $m\in \bbN$, $s,l,\in \bbR$. Then, the constant family $\{L(\sigma)\}_{\sigma>0}, L(\sigma)=L$, (which we conflate with $L$) defines an element of $\Psi_{\mathrm{leC}}^{m,s,s+l,l,2l}(X)$. 
	
	Given any  $f\in  \calA^{0,0,(0,0)} (X^{\mathrm{sp}}_{\mathrm{res}})$  and $l_0,\ell_0 \in \bbR$, $ (x/(\sigma^2+\mathsf{Z}x))^{l_0} (\sigma^2+\mathsf{Z}x)^{\ell_0/2} f L$ defines an element of  $\Psi_{\mathrm{leC}}^{m,s-l_0,s+l-\ell_0 ,l-l_0,2l-\ell_0}(X)$.
\end{proposition}
\begin{proof}
	First consider the case when $s=m+l$, so that $L\in \smash{S\operatorname{Diff}_{\mathrm{b}}^{m,l}(X)}$. Then, 
	\begin{equation} 
		L\in \Psi_{\mathrm{b,leC}}^{m,l,2l}(X) = \Psi_{\mathrm{leC}}^{m,m+l,m+2l,l,2l}(X)= \Psi_{\mathrm{leC}}^{m,s,s+l,l,2l}(X).
	\end{equation}  

	To handle the general case, we note that we may write any $L\in S\operatorname{Diff}_{\mathrm{scb}}^{m,s,l}(X)$ -- where now $m,s,l\in \bbR$ are arbitrary -- as 
	\begin{equation} 
		L = \sum_{j=0}^{m} L_j
	\end{equation} 
	for $\smash{L_j \in S\operatorname{Diff}_{\mathrm{b}}^{j,\min\{s-j,l\}}(X) \subset S \operatorname{Diff}_{\mathrm{scb}}^{j, \min\{s,l+j\},\min\{s-j,l\}}(X)\subset S \operatorname{Diff}_{\mathrm{scb}}^{m,s,l}(X) }$. 
	(Indeed, it suffices to construct such a decomposition on $[0,\bar{x})_x\times \partial X$, where $L_j$ can be written as a linear combination of elements of $\smash{S(X)\partial_x^{j_1}\operatorname{Diff}^{j_2}(\partial X)}$ for $j_1+j_2= j$, $j_1,j_2 \in\bbN$.)
	Since 
	\begin{align}
		\begin{split} 
		L_j &\in \Psi_{\mathrm{b,leC}}^{j,\min\{s-j,l\},2\min\{s-j,l\}}(X)  \\
		&= \Psi_{\mathrm{leC}}^{j,\min\{s,l+j\},\min\{2s-j,j+2l\},\min\{s-j,l\},2\min\{s-j,l\}}(X), 
		\end{split}
		\label{eq:misc_ljs}
		\intertext{we deduce that}
		L &\in \Psi_{\mathrm{leC}}^{m,\min\{s,l+m\},s+l, \min\{s,l\}, \min\{2s,2l\}}(X) \subset \Psi_{\mathrm{leC}}^{m,s,s+l,l,2l}(X). 
	\end{align}

From \Cref{prop:multiplication}, for any $f\in \calA^{0,0,(0,0)} (X^{\mathrm{sp}}_{\mathrm{res}}) (X^{\mathrm{sp}}_{\mathrm{res}})$ and $l_0,\ell_0 \in \bbR$, the multiplication operator $x^{l_0} (\sigma^2+\mathsf{Z}x)^{-l_0+\ell_0/2} f(x;\sigma)$ defines an element of $\Psi_{\mathrm{b,leC}}^{0,-l_0,-\ell_0}(X) = \Psi_{\mathrm{leC}}^{0,-l_0,-\ell_0,-l_0,-\ell_0}(X)$. 
The second statement of the proposition therefore follows from the first via an application of \Cref{prop:algebra}. 
\end{proof}

For each $m,s,\varsigma,l,\ell \in \bbR$, we let $S\operatorname{Diff}_{\mathrm{leC}}^{m,s,\varsigma,l,\ell}(X)$ denote the set of elements of $\Psi_{\mathrm{leC}}^{m,s,\varsigma,l,\ell}(X)$ which are families of differential operators (all of which arise from the construction in \Cref{prop:diff->PsiDO}). Likewise, we let 
\begin{equation} 
\operatorname{Diff}_{\mathrm{leC}}^{m,s,\varsigma,l,\ell}(X) \subset S\operatorname{Diff}_{\mathrm{leC}}^{m,s,\varsigma,l,\ell}(X)
\end{equation}
denote the subset of families of differential operators
which can be written as a linear combination of elements of $\operatorname{Diff}_{\mathrm{b}}(X)$ times classical symbols on $X^{\mathrm{sp}}_{\mathrm{res}}$.

The elliptic parametrix construction, applied to the leC-calculus, yields the following: for any $m,s,\varsigma,l,\ell \in \bbR$ and (totally) elliptic $A \in \Psi_{\mathrm{leC}}^{m,s,\varsigma,l,\ell}(X)$ -- that is $A$ with 
\begin{equation} 
	\operatorname{Ell}_{\mathrm{leC}}^{m,s,\varsigma,\ell}(A) = \mathrm{df}\cup \mathrm{sf}\cup \mathrm{ff}
\end{equation} -- there exists, for each $N\in \bbN$, some $B \in  \smash{\Psi_{\mathrm{leC}}^{-m,-s,-\varsigma,-l,-\ell}}(X)$ such that $AB-1,BA-1 \in \smash{\Psi_{\mathrm{b,leC}}^{-N,0,0}(X)}$. (To elaborate, the leC-symbol calculus analogue of the construction of left and right parametrices via Neumann series yields $B_{\mathrm{L}},B_{\mathrm{R}}$ such that $
AB_{\mathrm{R}}-1,B_{\mathrm{L}}A - 1 \in \Psi_{\mathrm{b,leC}}^{-N,0,0}(X)$. 
Then, setting $E_{\mathrm{L}} = B_{\mathrm{L}}A-1 $ and $E_{\mathrm{R}} = AB_{\mathrm{R}}-1$, 
\begin{equation} 
	B_{\mathrm{L}} = B_{\mathrm{L}} (AB_{\mathrm{R}}-E_{\mathrm{R}}) = (E_{\mathrm{L}}+1)B_{\mathrm{R}}- B_{\mathrm{L}} E_{\mathrm{R}} = B_{\mathrm{R}} + (E_{\mathrm{L}}B_{\mathrm{R}} - B_{\mathrm{L}} E_{\mathrm{R}}), 
\end{equation} 
so taking either $B=B_{\mathrm{L}}$ or $B=B_{\mathrm{R}}$, both of $AB-1,BA-1 \in \Psi_{\mathrm{b,leC}}^{-N,0,0}(X)$ hold. Hence we do not need to distinguish left vs. right parametrices.)
\begin{lemma}
	Given any $m,s,\varsigma,l,\ell\in \bbR$, $A \in \Psi_{\mathrm{leC}}^{m,s,\varsigma,l,\ell}(X)$ and $\alpha \in \mathrm{df} \cup \mathrm{sf} \cup \mathrm{ff}$, $\alpha \notin \operatorname{WF}_{\mathrm{leC}}^{\prime l,\ell}(A)$ only if there exists some $B\in \Psi^{0,0,0,0,0}_{\mathrm{leC}}(X)$ that is elliptic at $\alpha$ and satisfies $AB,BA \in \Psi_{\mathrm{leC}}^{-\infty,-\infty,-\infty,l,\ell}(X)$. 
	
	The same statement applies if we replace $\alpha$ with the intersection of a finite union of closed balls within $\mathrm{df}\cup \mathrm{sf}\cup \mathrm{ff}$. 
\end{lemma}
\begin{proof}
	We write $A = \operatorname{Op}(a)$ for $a \in S_{\mathrm{leC}}^{\infty,\infty,\infty,l,\ell}(X)$, so that $\operatorname{WF}_{\mathrm{leC}}^{\prime l,\ell}(A) = \operatorname{esssupp}_{\mathrm{leC}}(a)$.
	
	Given  $B\in \Psi^{0,0,0,0,0}_{\mathrm{leC}}(X)$, $B=\operatorname{Op}(b)$, $b\in S_{\mathrm{leC}}^{0,0,0,0,0}(X)$, define $E \in \Psi_{\mathrm{leC}}^{-\infty,-\infty,-\infty,l,\ell}(X)$ by
	\begin{equation} 
		AB = \operatorname{Op}(a\sharp b) + E.
	\end{equation} 
	If $\alpha \notin \operatorname{esssupp}_{\mathrm{leC}}(a)$, then (since essential supports are closed) we can choose $b\in \smash{S_{\mathrm{cl,leC}}^{0,0,0,0,0}(X)}$ that is identically equal to one in a neighborhood of $\alpha$ but supported away from $\operatorname{esssupp}_{\mathrm{leC}}(a)$, so that 
	\[\operatorname{esssupp}_{\mathrm{leC}}(a)\cap \operatorname{esssupp}_{\mathrm{leC}}(b) = \varnothing.
	\] 
	Then, by \cref{eq:misc_k45}, $a\sharp b$ has empty essential support, which implies that $AB \in \Psi_{\mathrm{leC}}^{-\infty,-\infty,-\infty,l,\ell}(X)$. 
	
	That handles $AB$, and $BA$ is analogous. If we replace $\alpha$ in the previous argument with  the intersection of a finite union of closed balls within $\mathrm{df}\cup \mathrm{sf}\cup \mathrm{ff}$, the argument goes through the same. 
\end{proof}

We briefly discuss uniform families of leC-$\Psi$DOs (i.e.\ one-parameter families of b-$\Psi$DOs which are uniformly bounded in a sense appropriate for the leC-calculus). 
For each $m,s,\varsigma,l,\ell \in \bbR\cup \{-\infty\}$, 
\begin{equation}
	\Psi^{m,s,\varsigma,l,\ell}_{\mathrm{leC}}(X) = \bigcap_{\substack{m'\geq m, \cdots , \ell'\geq \ell \\ m',\cdots,\ell'\in \bbR}} \Psi^{m',s',\varsigma',l',\ell'}_{\mathrm{leC}}(X)
\end{equation} 
is a Fr\'echet space, so for any nonempty interval $I\subset \bbR$, we have a Fr\'echet space $\smash{L^\infty(I;\Psi^{m,s,\varsigma,l,\ell}_{\mathrm{leC}}(X))}$ whose elements are a.e.\ equivalence classes of measurable functions $I\to \Psi^{m,s,\varsigma,l,\ell}_{\mathrm{leC}}(X)$ which are uniformly bounded with respect to each of our countably many Fr\'echet seminorms on the codomain. For $I$ closed, we can safely conflate an element of $\calA^0(I;\Psi^{m,s,\varsigma,l,\ell}_{\mathrm{leC}}(X))$, which we can consider as a smooth family $\{A_t\}_{t \in I^\circ}$, with the corresponding element of $\smash{L^\infty(I;\Psi^{m,s,\varsigma,l,\ell}_{\mathrm{leC}}(X))}$. 

For $A=\{A_t\}_{t\in I} \in \smash{L^\infty(I;\Psi^{m,s,\varsigma,l,\ell}_{\mathrm{leC}}(X))}$, we define subsets $\operatorname{WF}'_{L^\infty,\mathrm{leC}}(A), \operatorname{WF}_{L^\infty,\mathrm{leC}}^{\prime l,\ell}(A) \subset \mathrm{df}\cup \mathrm{sf}\cup \mathrm{ff}$ by stipulating that a point $\alpha \in \mathrm{df}\cup \mathrm{sf} \cup \mathrm{ff}$  does not lie in 
\begin{itemize}
	\item $\operatorname{WF}'_{L^\infty,\mathrm{leC}}(A)$ if and only if there exists some $B\in \Psi_{\mathrm{leC}}^{0,0,0,0,0}(X)$ that is elliptic at $\alpha$ and satisfies $\{BA_t\}_{t\in I} \in L^\infty(I;\smash{\Psi_{\mathrm{leC}}^{-\infty,-\infty,-\infty,-\infty,-\infty}(X)})$, 
	\item $\operatorname{WF}^{\prime l,\ell}_{L^\infty,\mathrm{leC}}(A)$ if and only if there exists some $B\in \Psi_{\mathrm{leC}}^{0,0,0,0,0}(X)$ that is elliptic at $\alpha$ and satisfies $\{BA_t\}_{t\in I} \in L^\infty(I;\smash{\Psi_{\mathrm{leC}}^{-\infty,-\infty,-\infty,l,\ell}(X)})$. 
\end{itemize}
A uniform version of the elliptic parametrix construction goes through. 
We likewise have a notion of $\operatorname{esssupp}_{L^\infty,\mathrm{leC}}(a)$ for $a \in L^\infty(I;S_{\mathrm{leC}}^{m,s,\varsigma,l,\ell}(X))$. One definition is that $\alpha \in \mathrm{df}\cup \mathrm{sf}\cup \mathrm{ff}$ is not in $\operatorname{esssupp}_{L^\infty,\mathrm{leC}}(a)$ if and only if there exists some $b\in \smash{S^{0,0,0,0,0}_{\mathrm{leC}}(X)}$ that is elliptic at $\alpha$ and satisfies $ab \in \smash{S_{\mathrm{leC}}^{-\infty,-\infty,-\infty,-\infty}(X)}$. Quantizing:
for any interval $I\subset \bbR$ and any $m,s,\varsigma,l,\ell\in \bbR$, given $a \in L^\infty(I;S_{\mathrm{leC}}^{m,s,\varsigma,l,\ell}(X))$, 
letting $A=\{\operatorname{Op}(A_t)\}_{t\in I}$, 
\begin{equation} 
	\operatorname{WF}^{\prime l,\ell}_{L^\infty, \mathrm{leC}}(A) = \operatorname{esssupp}_{L^\infty,\mathrm{leC}}(a).
\end{equation}

\subsection{Sobolev Spaces}
\label{subsec:sobolev}
For each $m,s,l\in \bbR$, let $H_{\mathrm{scb}}^{m,s,l}(X)$ denote the Sobolev space of differential order $m$, sc-decay order $s$, and b-decay order $l$ associated to the scb-calculus in \cite{VasyLA}. 
Associated to the leC-calculus is a five-parameter family 
\[
\{H_{\mathrm{leC}}^{m,s,\varsigma,l,\ell}(X)\}_{m,s,\varsigma,l,\ell\in \bbR}
\] of families 
\begin{equation} 
	H_{\mathrm{leC}}^{m,s,\varsigma,l,\ell}(X)= \{H_{\mathrm{leC}}^{m,s,\varsigma,l,\ell}(X)(\sigma) \}_{\sigma\geq 0}
\end{equation} 
of ``leC-based Sobolev spaces,'' where 
\begin{itemize}
	\item  for each $\sigma>0$, 
	$\smash{H_{\mathrm{leC}}^{m,s,\varsigma,l,\ell}}(X)(\sigma)$ is a Hilbertizable Banach space equal to
	$\smash{H_{\mathrm{scb}}^{m,s,l}(X)} \subset \calS'(X)$ at the level of TVSs (i.e.\ equivalent at the level of Banach spaces), 
	\item $H_{\mathrm{leC}}^{m,s,\varsigma,l,\ell}(X)(0) = H_{\mathrm{scb}}^{m,\varsigma+n/2,\ell+n/2}(X_{1/2})$ (at the level of TVSs). 
\end{itemize}
The family $H_{\mathrm{leC}}^{m,s,\varsigma,l,\ell}(X)$ had ought to be thought of as interpolating between the Sobolev spaces $\smash{H_{\mathrm{scb}}^{m,s,l}(X)}$ and $H_{\mathrm{scb}}^{m,\varsigma+n/2,\ell+n/2}(X_{1/2})$ as $\sigma\to 0^+$.

Besides using the leC-Sobolev spaces to relate $H^{\bullet,\bullet,\bullet}_{\mathrm{scb}}(X)$ and $H^{\bullet,\bullet,\bullet}_{\mathrm{scb}}(X_{1/2})$ we can, more crudely, observe that 
\begin{equation}
	H_{\mathrm{scb}}^{m,m+l,l}(X)= H_{\mathrm{b}}^{m,l}(X) =H_{\mathrm{b}}^{m,2l+n/2}(X_{1/2}) = H_{\mathrm{scb}}^{m,m+2l+n/2,2l+n/2}(X_{1/2}).
\end{equation}
The spaces $H_{\mathrm{scb}}^{\bullet,\bullet,\bullet}(X)$ are more refined than $H_{\mathrm{b}}^{\bullet,\bullet}(X)$, in that the decay rate of terms like $\exp(i/x)$ need not be treated the same as the decay rate of the constant function as measured by the former. Similarly, the spaces $H_{\mathrm{scb}}^{\bullet,\bullet,\bullet}(X_{1/2})$ are more refined than $H_{\mathrm{b}}^{\bullet,\bullet}(X_{1/2})$ in that the decay rate of terms like $\smash{\exp(i/x^{1/2})}$ need not be the same as the decay rate of the constant function as measured by the former. As shown in \cite{VasyLA}, the scb-Sobolev spaces are well-suited for formulating a conjugated version of the Sommerfeld radiation condition, and for similar reasons the leC-Sobolev spaces are well-suited to the study of attractive Coulomb-like Schr\"odinger operators down to zero energy. 

In order to define the $\sigma$-dependent norm on $\smash{H_{\mathrm{leC}}^{m,s,\varsigma,l,\ell}}(X)(\sigma)$, we first note:

\begin{lemma}
	\label{lem:bsc}
	For any two elliptic $A,B \in \Psi_{\mathrm{leC}}^{m,s,\varsigma,l,\ell}$ and $N,M\in \bbZ$ with $N,M\geq -\min\{m,s-l,\varsigma-\ell\}$, and for any $\Sigma>0$, there exist $c,C>0$ such that 
	\begin{equation}
		c(\lVert B(\sigma)u \rVert_{ L^2(X)} + \lVert u \rVert_{H_{\mathrm{b,leC}}^{-M,l,\ell}(\sigma)}) \leq 	\lVert A(\sigma)u \rVert_{ L^2} + \lVert u \rVert_{H_{\mathrm{b,leC}}^{-N,l,\ell}(\sigma)} \leq C(\lVert B(\sigma)u \rVert_{ L^2} + \lVert u \rVert_{H_{\mathrm{b,leC}}^{-M,l,\ell}(\sigma)})
	\end{equation}
	holds for all $u \in \calS'(X)$ and $\sigma \in [0,\Sigma]$.  
\end{lemma}
\begin{proof}
	It suffices to prove $\lVert A(\sigma)u \rVert_{ L^2} + \lVert u \rVert_{H_{\mathrm{b,leC}}^{-N,l,\ell}(\sigma)} \leq C(\lVert B(\sigma)u \rVert_{ L^2} + \lVert u \rVert_{H_{\mathrm{b,leC}}^{-M,l,\ell}(\sigma)})$, the other inequality following by symmetry. 
	
	By the elliptic parametrix construction, for arbitrary $N_0\in \bbN$, we can find $\Lambda \in \Psi_{\mathrm{leC}}^{-m,-s,-\varsigma,-l,-\ell}$ and $R\in \Psi_{\mathrm{b,leC}}^{-N_0,0,0}$ with $\Lambda B = 1+R$. Now, setting  $m_0 = \min\{m,s-l,\varsigma-\ell\}$,  
	\begin{align}
		\begin{split} 
			\lVert u \rVert_{H_{\mathrm{b,leC}}^{-N,l,\ell}(\sigma)} &= \lVert (\Lambda(\sigma) B(\sigma) - R(\sigma)) u \rVert_{H_{\mathrm{b,leC}}^{-N,l,\ell}(\sigma)}  \\ & \leq 
			\lVert \Lambda(\sigma) B(\sigma) u \rVert_{H_{\mathrm{b,leC}}^{-N,l,\ell}(\sigma)} + \lVert  R(\sigma) u \rVert_{H_{\mathrm{b,leC}}^{-N,l,\ell}(\sigma)} \\
			& \preceq 
			\lVert \Lambda(\sigma) B(\sigma) u \rVert_{H_{\mathrm{b,leC}}^{-N,l,\ell}(\sigma)} + \lVert   u \rVert_{H_{\mathrm{b,leC}}^{-M,l,\ell}(\sigma)} \\
			&\preceq 
			\lVert B(\sigma) u \rVert_{H_{\mathrm{b,leC}}^{-N-m_0,0,0}(\sigma)} + \lVert   u \rVert_{H_{\mathrm{b,leC}}^{-M,l,\ell}(\sigma)} 
			\preceq \lVert B(\sigma) u \rVert_{L^2} +  \lVert   u \rVert_{H_{\mathrm{b,leC}}^{-M,l,\ell}(\sigma)}
		\end{split} 
		\label{eq:misc_kkk}
	\end{align}
	for sufficiently large $N_0$, 
	where we used that $\Lambda \in \Psi_{\mathrm{b,leC}}^{-m_0,-l,-\ell}$. 
	
	Similarly, for sufficiently large $N_0$,
	\begin{align}
		\begin{split} 
			\lVert A(\sigma) u \rVert_{L^2} = \lVert  A(\sigma) (\Lambda(\sigma) B(\sigma)-R(\sigma)) u \rVert_{L^2}  &\leq \lVert A(\sigma)\Lambda(\sigma) B(\sigma)  u \rVert_{L^2}  + \lVert  A(\sigma)R(\sigma) u \rVert_{L^2} \\
			&\preceq \lVert A(\sigma) \Lambda(\sigma)B(\sigma) u \rVert_{L^2}  + \lVert  u \rVert_{H_{\mathrm{b,leC}}^{-M,l,\ell}(\sigma)} \\
			&\preceq \lVert B(\sigma) u \rVert_{L^2} + \lVert  u \rVert_{H_{\mathrm{b,leC}}^{-M,l,\ell}(\sigma)} .
		\end{split} 
	\label{eq:misc_k34}
	\end{align}
	Combining \cref{eq:misc_kkk} with \cref{eq:misc_k34} yields the desired inequality. 
\end{proof}
We can now define a Hilbertizable norm on $H_{\mathrm{leC}}^{m,s,\varsigma,l,\ell}(X)(\sigma)$, for each $\sigma\geq 0$, by writing 
\begin{equation}
	\lVert u \rVert_{H_{\mathrm{leC}}^{m,s,\varsigma,l,\ell}(X)(\sigma)}  = \lVert \Lambda (\sigma) u \rVert_{L^2} + \lVert u \rVert_{H_{\mathrm{b,leC}}^{-N,l,\ell}(\sigma)}
	\label{eq:misc_sim}
\end{equation}
for $N\in \bbN$ with $N\geq - \min\{m,s-l,\varsigma-\ell\}$ and arbitrary elliptic $\Lambda \in \Psi_{\mathrm{leC}}^{m,s,\varsigma,l,\ell}(X)$. (Such $\Lambda$ can be constructed by the symbol calculus.) The previous lemma suffices to guarantee that the estimates we prove do not depend on the particular choice of $\Lambda$ and $N$ used in defining the norm \cref{eq:misc_sim} except with regards to the particular constants involved (which we do not keep track of anyways). 
However, it will be convenient to fix  
\begin{equation}
	\Lambda_{m,s,\varsigma,l,\ell} = (1/2)( \operatorname{Op}(\varrho_{\mathrm{df}}^{-m}\varrho_{\mathrm{sf}}^{-s}\varrho_{\mathrm{ff}}^{-\varsigma}\varrho_{\mathrm{bf}}^{-l}\varrho_{\mathrm{tf}}^{-\ell} ) +  \operatorname{Op}(\varrho_{\mathrm{df}}^{-m}\varrho_{\mathrm{sf}}^{-s}\varrho_{\mathrm{ff}}^{-\varsigma}\varrho_{\mathrm{bf}}^{-l}\varrho_{\mathrm{tf}}^{-\ell} )^*),
\end{equation}
which is certainly an elliptic element of $\Psi_{\mathrm{leC}}^{m,s,\varsigma,l,\ell}(X)$. 
Since $\Lambda_{m,s,\varsigma,l,\ell}(\sigma)$ is an elliptic element of $\Psi_{\mathrm{scb}}^{m,s,l}(X)$ for $\sigma>0$ and $\Psi_{\mathrm{scb}}^{m,\varsigma,\ell}(X_{1/2})$ for $\sigma = 0$, we see that $H_{\mathrm{leC}}^{m,s,\varsigma,l,\ell}(X)(\sigma)$ is indeed equivalent to 
\begin{itemize}
	\item $H_{\mathrm{scb}}^{m,s,l}(X)$ for $\sigma>0$ and
	\item $H_{\mathrm{scb}}^{m,\varsigma+n/2,\ell+n/2}(X_{1/2})$ for $\sigma = 0 $. (The shift of $n/2$ orders occurs because a sc-density on $X$ is a weighted sc-density on $X_{1/2}$, and vice versa; see \Cref{lem:218}.) 
\end{itemize}
On the other hand:
\begin{lemma}
	\label{lem:bleC_equiv}
	If $m,s,\varsigma,l,\ell\in \bbR$ satisfy $\varsigma=m+\ell$ and $s=m+l$, for each $\Sigma>0$ there exist constants $c(m,s,\varsigma,l,\ell,\Sigma),C(m,s,\varsigma,l,\ell,\Sigma)>0$ such that  
	\begin{equation}
		c\lVert u \rVert_{H_{\mathrm{leC}}^{m,s,\varsigma,l,\ell}(X)(\sigma)} \leq \lVert u \rVert_{H_{\mathrm{b,leC}}^{m,l,\ell}(X)(\sigma)} \leq C\lVert u \rVert_{H_{\mathrm{leC}}^{m,s,\varsigma,l,\ell}(X)(\sigma)}
	\end{equation}
	for all $u\in \calS'(X)$ and $\sigma \in [0,\Sigma]$. In particular, if $\ell=2l$,  $c'\lVert u \rVert_{H_{\mathrm{leC}}^{m,s,\varsigma,l,\ell}(X)(\sigma)} \leq \lVert u \rVert_{H_{\mathrm{b}}^{m,l}(X)} \leq C'\lVert u \rVert_{H_{\mathrm{leC}}^{m,s,\varsigma,l,\ell}(X)(\sigma)}$
	for some other constants constants $c'(m,s,\varsigma,l,\ell,\Sigma),C'(m,s,\varsigma,l,\ell,\Sigma)>0$.
\end{lemma}
\begin{proof}
	We have that $\Lambda_{m,s,\varsigma,l,\ell} \in \Psi_{\mathrm{b,leC}}^{m,l,\ell}(X)$, so 
	\begin{equation}
		\lVert u \rVert_{H_{\mathrm{leC}}^{m,m+l,m+\ell,l,\ell}(X)(\sigma)}  = \lVert \Lambda_{m,m+l,m+\ell,l,\ell}(\sigma) u \rVert_{L^2} + \lVert u \rVert_{H_{\mathrm{b,leC}}^{-N,l,\ell}(\sigma)} \preceq  \lVert u \rVert_{H_{\mathrm{b,leC}}^{m,l,\ell}(X)(\sigma)}
	\end{equation}
	for sufficiently large $N$ (by the boundedness of the elements of the resolved family b-calculus). 
	
	The reverse inequality follows from the ellipticity of $\Lambda_{m,m+l,m+\ell,l,\ell}$ as an element of $\Psi_{\mathrm{b,leC}}^{m,l,\ell}(X)$. 
\end{proof}

We now check the boundedness of leC-$\Psi$DOs acting on leC-Sobolev spaces (from the $L^2$-boundedness of zeroth order b-$\Psi$DOs):
\begin{proposition}
	For any $m,s,\varsigma,l,\ell,m_0,s_0,\varsigma_0,l_0,\ell_0 \in \bbR$, $A \in \Psi_{\mathrm{leC}}^{m,s,\varsigma,l,\ell}(X)$, and  $\Sigma>0$, there exists some constant $C = C(m,s,\varsigma,l,\ell,m_0,s_0,\varsigma_0,l_0,\ell_0 ,A,\Sigma) > 0$
	such that 
	\begin{equation}
		\lVert A(\sigma) u \rVert_{H_{\mathrm{leC}}^{m_0,s_0,\varsigma_0,l_0,\ell_0}(X)(\sigma)} \leq C \lVert u \rVert_{H_{\mathrm{leC}}^{m+m_0,s+s_0,\varsigma+\varsigma_0,l+l_0,\ell+\ell_0}(X)(\sigma)}
		\label{eq:misc_9kj}
	\end{equation}
	for all $u\in \calS'(X)$ and $\sigma \in [0,\Sigma]$. 
\end{proposition}
\begin{proof}
	Pick arbitrary elliptic $\Lambda_0 \in \smash{\Psi_{\mathrm{leC}}^{m_0,s_0,\varsigma_0,l_0,\ell_0}(X)}$. Then 
	\begin{align}
		\lVert A(\sigma) u \rVert_{H_{\mathrm{leC}}^{m_0,s_0,\varsigma_0,l_0,\ell_0}(X)(\sigma)} &\preceq \lVert \Lambda_0(\sigma)  A(\sigma) u \rVert_{L^2} + \lVert A(\sigma) u \rVert_{H_{\mathrm{b,leC}}^{-N,l_0,\ell_0}(X)(\sigma)}	 \\ 
		\lVert u \rVert_{H_{\mathrm{leC}}^{m+m_0,s+s_0,\varsigma+\varsigma_0,l+l_0,\ell+\ell_0}(X)(\sigma)} &\succeq \lVert \Lambda_1(\sigma)   u \rVert_{L^2} + \lVert u \rVert_{H_{\mathrm{b,leC}}^{-M,l+l_0,\ell+\ell_0}(X)(\sigma)}	
	\end{align}
	for  $N,M$ not too negative and arbitrary  elliptic 
	$\Lambda_1 \in \Psi_{\mathrm{leC}}^{m+m_0,s+s_0,\varsigma+\varsigma_0,l+l_0,\ell+\ell_0}(X)$.
	
	For $N$ sufficiently large compared to $M$ (dependent on $m,\cdots,\ell_0$), $\lVert A(\sigma) u \rVert_{H_{\mathrm{b,leC}}^{-N,l_0,\ell_0}(X)(\sigma)}\preceq \lVert u \rVert_{H_{\mathrm{b,leC}}^{-M,l+l_0,\ell+\ell_0}(X)(\sigma)}$.

	For each $M'\in \bbR$, an elliptic parametrix for $\Lambda_1$  can be used to construct $A_0 \in \Psi_{\mathrm{b,leC}}^{0,0,0}(X) \subset L^\infty([0,\infty); \Psi_{\mathrm{b}}^0(X) )$ and $R \in \Psi_{\mathrm{b,leC}}^{-M',l+l_0,\ell+\ell_0}(X)$ such that $\Lambda_0 A = A_0 \Lambda_1 +R$. Then, if $M'$ is sufficiently large, 
	\begin{equation}
		\lVert \Lambda_0  A(\sigma) u \rVert_{L^2} \leq \lVert A_0 \Lambda_1 u \rVert_{L^2} + \lVert R u \rVert_{L^2} \preceq \lVert \Lambda_1 u \rVert_{L^2} + \lVert u \rVert_{H_{\mathrm{b,leC}}^{-M,l+l_0,\ell+\ell_0}(X)(\sigma)},
	\end{equation}
	where we have used the uniform $L^2$-boundedness of $A_0$. 
	Combining the estimates above, we deduce \cref{eq:misc_9kj}. 
\end{proof}

In the rest of the paper, we will abbreviate 
\begin{equation}
	H_{\mathrm{leC}}^{m,s,\varsigma,l,\ell}(X)(\sigma) = H_{\mathrm{leC}}^{m,s,\varsigma,l,\ell}(X) = H_{\mathrm{leC}}^{m,s,\varsigma,l,\ell},
\end{equation} 
leaving the dependence on $\sigma$ implicit
(and likewise for other $\sigma$-dependent notions). 
In particular, what we call ``estimates'' will really be 1-parameter families of estimates with multiplicative constants that -- for any fixed $\Sigma> 0$ -- are uniform for $\sigma \in [0,\Sigma]$.

The following two lemmas will be used mostly without comment in \S\ref{sec:symbolic}:
\begin{lemma}
	\label{lem:duality}
	For any $m,s,\varsigma,l,\ell \in \bbR$, there exists a constant $C=C(m,s,\varsigma,l,\ell)$ such that, for all $u,v\in \calS'(X)$ and $\overline{\delta}>0$,  
	\begin{equation}
		|\langle u,v \rangle_{L^2} | \leq C\cdot ( \overline{\delta}^{-1}\lVert u \rVert^2_{H_{\mathrm{leC}}^{m,s,\varsigma,l,\ell}} + \overline{\delta}\lVert v \rVert^2_{H_{\mathrm{leC}}^{-m,-s,-\varsigma,-l,-\ell}})
		\label{eq:misc_mmm}
	\end{equation}
	for all $\sigma\geq 0$ for which the right-hand side is finite (in the strong sense that if the right-hand side is finite, then the left-hand side makes sense using the duality pairing for scb-Sobolev spaces and obeys the stated inequality). 
\end{lemma}
\begin{proof}
	By the parametrix construction, we can find, for each $N_0\in \bbN$, elliptic $V=V_{N_0}\in \Psi_{\mathrm{leC}}^{-m,-s,-\varsigma,-l,-\ell}(X)$ such that $V^*\Lambda_{m,s,\varsigma,l,\ell} = 1+Y$ for $Y=Y_{N_0} \in \Psi_{\mathrm{leC}}^{-N_0,-N_0,-N_0,0,0}(X)$. Then, 
	\begin{align}
		\begin{split} 
		2|\langle u,v \rangle_{L^2}| &\leq  2|\langle \Lambda_{m,s,\varsigma,l,\ell} u,Vv \rangle_{L^2}| + 2|\langle Yu,v\rangle_{L^2}| \\
		&\leq \overline{\delta}^{-1}\lVert \Lambda_{m,s,\varsigma,l,\ell}u \rVert_{L^2}^2 + \overline{\delta}\lVert Vv \rVert_{L^2}^2 + 2|\langle Yu,v\rangle_{L^2}| \\
		&\preceq \overline{\delta}^{-1}\lVert u \rVert^2_{H_{\mathrm{leC}}^{m,s,\varsigma,l,\ell}} + \overline{\delta} \lVert v \rVert^2_{H_{\mathrm{leC}}^{-m,-s,-\varsigma,-l,-\ell}}+ 2|\langle Yu,v\rangle_{L^2}|. 
		\end{split} 
		\label{eq:misc_m61}
	\end{align}
	On the other hand, for any $N_1 \in \bbN$, 
	\begin{align}
		\begin{split} 
		2|\langle Yu,v\rangle_{L^2}| &= 2|\langle x^{-l} (\sigma^2+\mathsf{Z}x)^{-\ell/2+l}Yu,x^l (\sigma^2+\mathsf{Z}x)^{\ell/2-l} v\rangle_{L^2}| \\
		&\leq \overline{\delta}^{-1} \lVert  x^{-l} (\sigma^2+\mathsf{Z}x)^{-\ell/2+l}Yu \rVert_{H_{\mathrm{b}}^{N_1,0} }^2 + \overline{\delta}\lVert x^l (\sigma^2+\mathsf{Z}x)^{\ell/2-l} v \rVert^2_{H_{\mathrm{b}}^{-N_1,0} } \\
		&\preceq \overline{\delta}^{-1}\lVert  x^{-l} (\sigma^2+\mathsf{Z}x)^{-\ell/2+l}Yu \rVert_{H_{\mathrm{b,leC}}^{N_1,0,0} }^2 + \overline{\delta}\lVert x^l (\sigma^2+\mathsf{Z}x)^{\ell/2-l} v \rVert^2_{H_{\mathrm{b,leC}}^{-N_1,0,0} } \\
		&\preceq \overline{\delta}^{-1}\lVert  Yu \rVert_{H_{\mathrm{b,leC}}^{N_1,l,\ell} }^2 + \overline{\delta}\lVert v \rVert^2_{H_{\mathrm{b,leC}}^{N_1,-l,-\ell} } \\
		&\preceq \overline{\delta}^{-1}\lVert  Yu \rVert_{H_{\mathrm{leC}}^{N_1,N_1+l,N_1+\ell,l,\ell} }^2 + \overline{\delta}\lVert v \rVert^2_{H_{\mathrm{leC}}^{-N_1,-N_1-l,-N_1-\ell,-l,-\ell} } \\
		&\preceq \overline{\delta}^{-1}\lVert  u \rVert_{H_{\mathrm{leC}}^{N_1-N_0,N_1+l-N_0,N_1+\ell-N_0,l,\ell} }^2 + \overline{\delta}\lVert v \rVert^2_{H_{\mathrm{leC}}^{-N_1,-N_1-l,-N_1-\ell,-l,-\ell} }. 
		\end{split} 
	\end{align}
	Taking $N_1$ sufficiently large, and then taking $N_0$ sufficiently large relative to that, we get $|\langle Yu,v\rangle_{L^2}|\preceq \overline{\delta}^{-1}\lVert u \rVert^2_{H_{\mathrm{leC}}^{m,s,\varsigma,l,\ell}} + \overline{\delta}\lVert v \rVert^2_{H_{\mathrm{leC}}^{-m,-s,-\varsigma,-l,-\ell}}$. Combining this with \cref{eq:misc_m61}, we get \cref{eq:misc_mmm}. 
\end{proof}
\begin{lemma}
	\label{lem:combination} 
	Let $m,s,\varsigma,l,\ell,m_0,s_0,\varsigma_0,l_0,\ell_0\in \bbR$. 
	Suppose that $A \in \Psi_{\mathrm{leC}}^{m,s,\varsigma,l,\ell}(X)$ and that we have some $J\in \bbN$ and $G_1,\ldots,G_J \in \Psi_{\mathrm{leC}}^{0,0,0,0,0}(X)$ such that 
	\begin{equation}
		\operatorname{WF}^{\prime l,\ell}_{\mathrm{leC}}(A) \subseteq \bigcup_{j=1}^J  \operatorname{Ell}_{\mathrm{leC}}^{0,0,0,0,0}(G_j). 
	\end{equation}
	Then, for each $\Sigma>0$ and $N\in \bbN$, there exists some $C=C(A,G_1,\ldots,G_J,\Sigma,N)>0$ such that 
	\begin{equation}
		\lVert A  u \rVert_{H_{\mathrm{leC}}^{m_0,s_0,\varsigma_0,l_0,\ell_0} } \leq C \Big[ \lVert u \rVert_{H_{\mathrm{leC}}^{-N,-N,-N,l+l_0,\ell+\ell_0} } + \sum_{j=1}^J \lVert G_j u \rVert_{H_{\mathrm{leC}}^{m+m_0,\cdots,\ell+\ell_0}} \Big]
		\label{eq:misc_est}
	\end{equation}
	for all $\sigma \in [0,\Sigma]$ and $u\in \calS'(X)$. 
\end{lemma}
\begin{proof}
	It suffices to consider the case of $u\in \calS(X)$, the general estimate \cref{eq:misc_est} following from this case via continuity. 
	
	Via quantizing some explicit symbols, there exist some $\bar{G}_1,\ldots,\bar{G}_J \in \Psi_{\mathrm{leC}}^{0,0,0,0,0}(X)$ such that $\operatorname{WF}^{\prime0,0}(\bar{G}_j) \subset  \operatorname{Ell}_{\mathrm{leC}}^{0,0,0,0,0}(G_j)$ and 
	\begin{equation}
	\operatorname{WF}^{\prime l,\ell}_{\mathrm{leC}}(A) \subseteq \bigcup_{j=1}^J  \operatorname{Ell}_{\mathrm{leC}}^{0,0,0,0,0}(\bar{G}_j). 
	\end{equation}
	
	We now apply the leC-analogue of the G{\aa}rding's inequality- type argument. We can choose $E\in \Psi_{\mathrm{leC}}^{0,0,0,0,0}(X)$ such that 
	\begin{align}
		\operatorname{Ell}_{\mathrm{leC}}^{0,0,0,0,0}(E) \cup \bigcup_{j=1}^J \operatorname{Ell}_{\mathrm{leC}}^{0,0,0,0,0}(\bar{G}_j) &= \mathrm{df} \cup \mathrm{sf} \cup \mathrm{ff} \\ 
		\operatorname{WF}_{\mathrm{leC}}'(E) \cap \operatorname{WF}^{\prime l,\ell}_{\mathrm{leC}}(A) &= \varnothing. 
	\end{align} 
	Thus, every representative of $\sigma_{\mathrm{leC}}^{0,0,0,0,0}(E^* E + \bar{G}_1^*\bar{G}_1+\cdots +\bar{G}_J^* \bar{G}_J)$ is nonvanishing on $\mathrm{df}\cup \mathrm{sf} \cup \mathrm{ff}$. We may assume without loss of generality that $E,\bar{G}_1,\ldots,\bar{G}_J$ are constant for $\sigma\geq 2\Sigma$, in which case (since $(\mathrm{df} \cup \mathrm{sf} \cup \mathrm{ff} )\cap \{\sigma \leq 2\Sigma\}$ is compact)
	there exists some $c>0$ such that $\sigma_{\mathrm{leC}}^{0,0,0,0,0}(E^* E + \bar{G}_1^*\bar{G}_1+\cdots +\bar{G}_J^* \bar{G}_J) \geq 2c$
	in some neighborhood of $\mathrm{df}\cup \mathrm{sf} \cup \mathrm{ff}$ (in the sense that every representative of the principal symbol has this property, for different neighborhoods).
	Via an iterative symbolic construction: for each $N_0\in \bbN$ there exists some elliptic $B=B_{N_0,c}\in \Psi_{\mathrm{leC}}^{0,0,0,0,0}(X)$ such that 
	\begin{equation} 
		E^* E + \bar{G}_1^*\bar{G}_1+\cdots +\bar{G}_J^* \bar{G}_J - c - B^* B \in \Psi_{\mathrm{leC}}^{-N_0,-N_0,-N_0,0,0}(X).
	\end{equation} 
	Then, for $\calX = H_{\mathrm{leC}}^{m_0,s_0,\varsigma_0,l_0,\ell_0}(X)$, 
	\begin{align} 
		c\lVert A u \rVert^2_{\calX }  &\leq   - \lVert B A u \rVert^2_{\calX} + \lVert E A u \rVert^2_{\calX } + \sum_{j=1}^J \lVert \bar{G}_j A u \rVert^2_{\calX } 
		\leq  \lVert E A u \rVert^2_{\calX} + \sum_{j=1}^J \lVert \bar{G}_j A u \rVert^2_{\calX} \\
		&\preceq \lVert u \rVert_{H_{\mathrm{leC}}^{-N,-N,-N,l+l_0,\ell+\ell_0} }^2 + \sum_{j=1}^J \lVert G_j u \rVert_{H_{\mathrm{leC}}^{m+m_0,\cdots,\ell+\ell_0}}^2
	\end{align} 
	for $N_0$ sufficiently large, 
	from which \cref{eq:misc_est} follows. 
\end{proof}

\begin{lemma}
	\label{prop:interpolation}
	If $m,s,\varsigma,l,\ell,m_0,s_0,\varsigma_0,l_0,\ell_0,m_1,s_1,\varsigma_1,l_1,\ell_1 \in \bbR$ satisfy $m_1>m>m_0,\cdots,\ell_1>\ell>\ell_0$,
	then, for each $\epsilon>0$  and $\Sigma>0$, there exists a $C(\epsilon) = C(\epsilon,m,\cdots,\ell_1,\Sigma)>0$ such that 
	\begin{equation}
		\lVert u \rVert_{H_{\mathrm{leC}}^{m,s,\varsigma,l,\ell}} \leq \epsilon \lVert u \rVert_{H_{\mathrm{leC}}^{m_1,s_1,\varsigma_1,l_1,\ell_1}} + C(\epsilon) \lVert u \rVert_{H_{\mathrm{leC}}^{m_0,s_0,\varsigma_0,l_0,\ell_0}}
		\label{eq:interpolation}
	\end{equation}
	for all $u\in \calS'(X)$ and $\sigma \in [0,\Sigma]$. 
\end{lemma}
\begin{proof}
	Fix $\Sigma>0$. 
	Suppose, to the contrary, that there exists some $\epsilon>0$, $\{\sigma_k\}_{k\in \bbN} \subset [0,\Sigma]$, and $\{u_k\}_{k\in \bbN}\subset \calS'(X)$ such that 
	\begin{equation}
		1 = \lVert u_k \rVert_{H_{\mathrm{leC}}^{m,s,\varsigma,l,\ell}(X)(\sigma_k) } \geq \epsilon \lVert u_k \rVert_{H_{\mathrm{leC}}^{m_1,s_1,\varsigma_1,l_1,\ell_1}(X)(\sigma_k)} + k \lVert u_k \rVert_{H_{\mathrm{leC}}^{m_0,s_0,\varsigma_0,l_0,\ell_0}(X)(\sigma_k)}, 
		\label{eq:misc_o99}
	\end{equation}
	i.e., for sufficiently large $N,M\in \bbN$, 
	\begin{multline}
			1 = \lVert \Lambda_{m,s,\varsigma,l,\ell} (\sigma_k) u_k \rVert_{L^2} + \lVert x^{-l} (\sigma^2_k+\mathsf{Z} x)^{l-\ell/2} u_k \rVert_{H_{\mathrm{b}}^{-M}}  \\ \geq \epsilon \lVert \Lambda_{m_1,s_1,\varsigma_1,l_1,\ell_1} (\sigma_k) u_k \rVert_{L^2} + \epsilon \lVert x^{-l_1} (\sigma^2_k+\mathsf{Z} x)^{l_1-\ell_1/2} u_k \rVert_{H_{\mathrm{b}}^{-N}}  \\ + k \lVert \Lambda_{m_0,s_0,\varsigma_0,l_0,\ell_0}(\sigma_k) u_k \rVert_{L^2} + k \lVert x^{-l_0} (\sigma^2_k+\mathsf{Z} x)^{l_0-\ell_0/2} u_k \rVert_{H_{\mathrm{b}}^{-M}}. 
			\label{eq:misc_log}
	\end{multline}
	Passing to a subsequence if necessary, we may assume without loss of generality that $\sigma_k\to\sigma_\infty$ for some $\sigma_\infty \in [0,\Sigma]$.

	By the Banach--Alaoglu theorem, by passing to a further subsequence if necessary, we can arrange that $\Lambda_{m_i,s_i,\varsigma_i,l_i,\ell_i} (\sigma_k) u_k  \to v_i$ weakly in $L^2$ for some $v_0,v_1\in L^2(X)$ and that 
	\begin{equation} 
	x^{-l_i} (\sigma_k^2+\mathsf{Z} x)^{l_i-\ell_i/2} u_k \to w_i
	\end{equation} 
	weakly in $H_{\mathrm{b}}^{-M}(X),H_{\mathrm{b}}^{-N}(X)$ for some $w_0,w_1 \in H_{\mathrm{b}}^{-M}(X),H_{\mathrm{b}}^{-N}(X)$, respectively. 
	Since 
	\begin{equation} 
	\Lambda_{m-m_1,s-s_1,\varsigma-\varsigma_1,l-l_1,\ell-\ell_1}\Lambda_{m_1,s_1,\varsigma_1,l_1,\ell_1} \in \Psi_{\mathrm{leC}}^{m,s,\varsigma,l,\ell}(X) 
	\end{equation} 
	is elliptic, \Cref{lem:bsc} yields 
	\begin{multline}
		1 = \lVert u_k \rVert_{H_{\mathrm{leC}}^{m,s,\varsigma,l,\ell}(X)(\sigma_k) } \preceq \\ \lVert \Lambda_{m-m_1,s-s_1,\varsigma-\varsigma_1,l-l_1,\ell-\ell_1}(\sigma_k)\Lambda_{m_1,s_1,\varsigma_1,l_1,\ell_1} (\sigma_k) u_k \rVert_{L^2} +\lVert x^{-l} (\sigma^2_k+\mathsf{Z} x)^{l-\ell/2} u_k \rVert_{H_{\mathrm{b}}^{-M}}.
		\label{eq:misc_r32}
	\end{multline}
	Taking $M>N$, we will produce a contradiction by showing that both terms on the right-hand side of \cref{eq:misc_r32} converge to zero as $k\to\infty$. 
	\begin{enumerate}
		\item For $M>N$,  
		\begin{multline}
			x^{-l_0} (\sigma_k^2+\mathsf{Z} x)^{l_0-\ell_0/2} u_k - x^{l_1-l_0} (\sigma^2_k+\mathsf{Z} x)^{l_0-l_1+\ell_1/2-\ell_0/2} w_1 \\ = x^{l_1-l_0} (\sigma^2_k+\mathsf{Z} x)^{l_0-l_1+\ell_1/2-\ell_0/2} (x^{-l_1} (\sigma_k^2+\mathsf{Z} x)^{l_1-\ell_1/2} u_k - w_1)  \to 0 
		\end{multline} 
		strongly in $H_{\mathrm{b}}^{-M}(X)$. 
		It follows that $x^{l_1-l_0} (\sigma^2_k+\mathsf{Z} x)^{l_0-l_1+\ell_1/2-\ell_0/2} w_1 \to w_0$ weakly in $H_{\mathrm{b}}^{-M}(X)$, and thus in $\calS'(X)$. But, it is also the case that 
		\begin{equation} 
			x^{l_1-l_0} (\sigma^2_k+\mathsf{Z} x)^{l_0-l_1+\ell_1/2-\ell_0/2} w_1 \to x^{l_1-l_0} (\sigma^2_\infty+\mathsf{Z} x)^{l_0-l_1+\ell_1/2-\ell_0/2} w_1
		\end{equation} 
		in $\calS'(X)$. 
		So, in fact, 
		\begin{equation}
			w_0 = x^{l_1-l_0} (\sigma^2_\infty +\mathsf{Z} x)^{\ell_0-l_1+\ell_1/2-\ell_0/2} w_1, 
			\label{eq:misc_i6o}
		\end{equation}
		and 
		\begin{equation} 
			x^{-l_0} (\sigma_k^2+\mathsf{Z} x)^{l_0-\ell_0/2} u_k \to w_0
		\end{equation} 
		strongly in $H_{\mathrm{b}}^{-M}(X)$. We can therefore deduce from the family of inequalities \cref{eq:misc_log} that $\lVert w_0 \rVert_{H_{\mathrm{b}}^{-M}} = 0$, i.e.\ $w_0 = 0$, from which and \cref{eq:misc_i6o} it follows that $w_1=0$. 
		
		But then, by the joint continuity of the multiplication operator
		\begin{equation} 
			x^{l_0} (\sigma^2 + \mathsf{Z} x)^{\ell_0/2-l_0} : [0,\Sigma]_\sigma \times \calS'(X)\to \calS'(X),
		\end{equation} 
		$u_k \to x^{l_0}(\sigma_\infty^2+\mathsf{Z} x)^{\ell_0/2-l_0} w_0 = 0$ in $\calS'(X)$.  We conclude that $v_0,v_1=0$.

		Via the same initial argument, we deduce that $ x^{-l} (\sigma^2_k+\mathsf{Z} x)^{l-\ell/2} u_k$ converges strongly in $H_{\mathrm{b}}^{-M}$ to $x^{l_1-l} (\sigma^2_\infty +\mathsf{Z} x)^{l-l_1+\ell_1/2-\ell/2} w_1 = 0$. Thus, $\lVert x^{-l} (\sigma^2_k+\mathsf{Z} x)^{l-\ell/2} u_k \rVert_{H_{\mathrm{b}}^{-M}}  \to 0$ as $k\to\infty$. 
		\item 
		
		We now consider 
		\begin{equation} 
			\Lambda =\Lambda_{m-m_1,s-s_1,\varsigma-\varsigma_1,l-l_1,\ell-\ell_1} \in C^0([0,\infty)_\sigma ; \Psi_{\mathrm{b}}^{-\varepsilon,-\varepsilon}(X) ) , 
		\end{equation}
		where $\varepsilon>0$ is sufficiently small. 
		Thus, we can write $\Lambda$ as the composition of a fixed compact operator on $L^2(X)$ and a continuous family of bounded operators on $L^2(X)$. 
		It follows (since $\Lambda_{m_1,s_1,\varsigma_1,l_1,\ell_1} (\sigma_k) u_k  \to v_1$ weakly in $L^2$) that 
		\begin{equation}
			\Lambda (\sigma_k)\Lambda_{m_1,s_1,\varsigma_1,l_1,\ell_1} (\sigma_k) u_k - \Lambda(\sigma_k) v_1 \to 0 
		\end{equation}
		strongly in $L^2(X)$. But $v_1 = 0$, as proven above, so 
		\begin{equation} 
			\lVert \Lambda_{m-m_1,s-s_1,\varsigma-\varsigma_1,l-l_1,\ell-\ell_1}(\sigma_k)\Lambda_{m_1,s_1,\varsigma_1,l_1,\ell_1} (\sigma_k) u_k \rVert_{L^2} \to 0
		\end{equation}  
		as $k\to\infty$.
	\end{enumerate}
\end{proof}

\begin{lemma}
	\label{lem:X12_conv}
	The $L^2_{\mathrm{sc}}(X),L^2_{\mathrm{sc}}(X_{1/2})$-based $\mathrm{b}$-Sobolev spaces on $X$ and $X_{1/2}$ are related by 
	\begin{equation} 
		H_{\mathrm{b}}^{m,l}(X) = H_{\mathrm{b}}^{m,2l+n/2}(X_{1/2}).
	\end{equation} 
	\label{lem:218}
\end{lemma}
\begin{proof}
	For $f,g \in L^2_{\mathrm{sc}}([0,\bar{x} )_x\times \partial X)$, 
	\begin{align}
		\begin{split} 
		\langle f,g \rangle_{L^2_{\mathrm{sc}}([0,\bar{x})_x\times \partial X)} &= \int_{\partial X} \Big( \int_{0}^{\bar{x}} f^*(x) g(x)  \frac{\dd x }{x^{n+1}}\Big) \dd\! \operatorname{Vol}_{g_{\partial X}}(y) \\
		&= 2\int_{\partial X} \Big( \int_{0}^{\bar{x}^{1/2}} f^*(\rho^2) g(\rho^2)  \frac{\dd \rho }{\rho^{2n+1}}\Big) \dd\! \operatorname{Vol}_{g_{\partial X}}(y) \\
		&= 2\int_{\partial X} \Big( \int_{0}^{\bar{x}^{1/2}} \frac{f^*(\rho^2)}{\rho^{n/2}} \frac{g(\rho^2)}{\rho^{n/2}}  \frac{\dd \rho }{\rho^{n+1}}\Big) \dd\! \operatorname{Vol}_{g_{\partial X}}(y).
		\end{split}
	\end{align}	
	Thus, $L^2_{\mathrm{sc}}([0,\bar{x})_x\times \partial X)\ni f(x)\mapsto \sqrt{2}f(\rho^2) \in \rho^{n/2} L^2_{\mathrm{sc}}([0,\bar{x}^{1/2})_\rho\times \partial X)$ defines an
	isomorphism of Hilbert spaces. 
	
	This implies that 
	\begin{equation}
		L^2_{\mathrm{sc}}(X) = x^{n/4} L^2_{\mathrm{sc}}(X_{1/2}). 
		\label{eq:misc_iuq}
	\end{equation}
	As follows from the definition \cite[Definition 5.15]{VasyGrenoble} (see also \cite[Definition 4.22]{APS} for the case of classicality at the front face of the b-double space),  $\smash{\Psi_{\mathrm{b}}^{m,l}(X) = \Psi_{\mathrm{b}}^{m,2l}(X_{1/2})}$ for all $m,l\in \bbR$. In conjunction with \cref{eq:misc_iuq}, this implies that 
	\begin{equation}
		H_{\mathrm{b}}^{m,l}(X) = \Psi_{\mathrm{b}}^{-m,-l}(X) L^2(X) =  \Psi_{\mathrm{b}}^{-m,-2l-n/2}(X_{1/2})  L^2(X_{1/2})  = H_{\mathrm{b}}^{m,2l+n/2}(X_{1/2}). 
	\end{equation}
\end{proof}

For use in \S\ref{sec:mainproof}, we briefly recall 
the connection between spaces of conormal distributions on $X_{\mathrm{res}}^{\mathrm{sp}}$, which are defined using $L^\infty$-based spaces, and the b-Sobolev spaces, which are defined using $L^2$. We have 
\begin{align}
	\calA_{\mathrm{loc}}^{\alpha,\beta,0}(X^{\mathrm{sp}}_{\mathrm{res}}) &= x^{\alpha} (\sigma^2+\mathsf{Z}x)^{-\alpha+\beta/2} \calA^{0,0}_{\mathrm{loc}}([0,\infty)\times X),
	\intertext{and}  
	\begin{split} 
	\calA^{0,0,(0,0)}_{\mathrm{loc}}(X^{\mathrm{sp}}_{\mathrm{res}}) &= \{u \in  \calA^0(X^{\mathrm{sp}}_{\mathrm{res}}) : [\hat{E} \mapsto u|_{E/x=\hat{E}}] \in C^\infty([0,\infty)_{\hat{E}}; \calA^0(X_{1/2}) ) \}, \\
	\calA^{\alpha,\beta,(0,0)}_{\mathrm{loc}}(X^{\mathrm{sp}}_{\mathrm{res}}) &= x^{\alpha} (\sigma^2+\mathsf{Z}x)^{-\alpha+\beta/2} \calA^{0,0,(0,0)}_{\mathrm{loc}}(X^{\mathrm{sp}}_{\mathrm{res}}),  
\end{split} 
\end{align}
where $\hat{E}=E/x$ and we are identifying level sets of $\hat{E}$  with $\mathrm{zf}\cong X_{1/2}$. Note that greater indices $\alpha,\beta$ means greater decay, the convention opposite of that used for symbols and $\Psi$DOs. We also use 
\begin{align} 
	\calA^{\alpha-,\beta-,0}_{\mathrm{loc}}(X^{\mathrm{sp}}_{\mathrm{res}})&= \cap_{\alpha'<\alpha,\beta'<\beta}\calA^{\alpha',\beta',0}_{\mathrm{loc}}(X^{\mathrm{sp}}_{\mathrm{res}}) \\
	\calA^{\alpha-,\beta-,(0,0)}_{\mathrm{loc}}(X^{\mathrm{sp}}_{\mathrm{res}})&= \cap_{\alpha'<\alpha,\beta'<\beta}\calA^{\alpha',\beta',(0,0)}_{\mathrm{loc}}(X^{\mathrm{sp}}_{\mathrm{res}}).
\end{align}
for $\alpha,\beta  \in \bbR$. These are all Fr\'echet spaces  
of (locally) conormal distributions on $X^{\mathrm{sp}}_{\mathrm{res}}$ that are smooth at $\mathrm{zf}^\circ$ (where, to reiterate, ``local'' just means that we do not require uniformity as $\sigma\to\infty$, only as $\sigma\to 0^+$). Via Sobolev embedding, 
\begin{equation}
	\cup_{l'>l}H_{\mathrm{b}}^{\infty, l'}(X)=H_{\mathrm{b}}^{\infty, l+}(X) \subseteq \calA^{l+n/2}(X) \subseteq H_{\mathrm{b}}^{\infty, l-}(X) = \cap_{l'<l} H_{\mathrm{b}}^{\infty, l'}(X),
\end{equation}
and, for each $\epsilon>0$, each seminorm of $\calA^{l+n/2}(X)$ can be controlled using only finitely many of the norms in the family $\{\lVert - \rVert_{H_{\mathrm{b}}^{m,l+\epsilon}}\}_{m\in \bbR}$.

Since, for each $\alpha\geq 0$, $\calA^0_{\mathrm{loc}}([0,\infty)_\sigma ; \calA^{\alpha-}(X) ) \subseteq \calA^{\alpha-,2\alpha-,0}_{\mathrm{loc}}(X^{\mathrm{sp}}_{\mathrm{res}})$, the preceding observation implies that
\begin{propositionp}
	\label{prop:conormality_inclusion}
	For any function $m_0:\bbR\to \bbR$, 
	\begin{equation}
		\bigcap_{l<-1/2} \bigcap_{m>m_0(l)} \calA^0_{\mathrm{loc}}([0,\infty)_\sigma ; H_{\mathrm{b}}^{m,l+\alpha}(X) ) \subseteq \calA_{\mathrm{loc}}^{(\alpha+(n-1)/2)-,(2\alpha+(n-1))-,0}(X^{\mathrm{sp}}_{\mathrm{res}}) 
	\end{equation}
	holds for each $\alpha \in \bbR$.
\end{propositionp}
\begin{remark*}
	Implicit in the statements above is the identification of elements of $\calA^0_{\mathrm{loc}}([0,\infty)_\sigma ; \calA^{\alpha-}(X) )$ with extendable distributions on $X^{\mathrm{sp}}_{\mathrm{res}}$, which occurs via 
	\begin{multline}
		\calA^0_{\mathrm{loc}}([0,\infty)_\sigma ; \calA^{\alpha-}(X) )\ni \{u(-;\sigma)\}_{\sigma> 0} \mapsto \\ \Big[ \dot{C}^\infty(X^{\mathrm{sp}}_{\mathrm{res}} ) \ni \chi \mapsto \int_{0}^\infty \langle u(-;\sigma), \chi(-;\sigma) \rangle \dd \sigma \Big] \in  \calD'(X^{\mathrm{sp}}_{\mathrm{res}}). 
	\end{multline}
	We defer to \cite{MelroseCorners}\cite{APS} for more about conormal distributions. See also \cite{HintzPrice} for the particular case of $X^{\mathrm{sp}}_{\mathrm{res},0}$ (where the notation $X^+_{\mathrm{res}}$ is used instead). 
\end{remark*}

\begin{proposition}
	\label{prop:smoothness_improver}
	Fix $\chi \in C_{\mathrm{c}}^\infty(X^{\mathrm{sp}}_{\mathrm{res}})$ supported away from $\mathrm{bf}$ and nonvanishing near $\mathrm{zf}$. 
	Then, if $(\chi x\partial_{E})^k u\in \calA_{\mathrm{loc}}^{0-,0-,0}(X_{\mathrm{res}}^{\mathrm{sp}} )$ for all $k\in \bbN$,
	then $u\in \calA_{\mathrm{loc}}^{0-,0-,(0,0)}(X_{\mathrm{res}}^{\mathrm{sp}})$.
\end{proposition}
\begin{proof}
	We want to show that $[\hat{E} \mapsto u|_{E/x=\hat{E}}] \in C^\infty([0,\infty)_{\hat{E}}; \calA^0(X_{1/2}))$. Observe that $x \partial_E = \partial_{\hat{E}}$ away from $\mathrm{bf}$, where the partial derivative on the right-hand side is taken with $x$ held constant.

	Consequently,  
	the $k=0,\ldots,K+1$ cases of $(\chi x\partial_{E})^k u\in \calA_{\mathrm{loc}}^{0-,0-,0}(X_{\mathrm{res}}^{\mathrm{sp}})$ together show that $[\hat{E} \mapsto u|_{\hat{E}}] \in C^K([0,\infty)_{\hat{E}}; \calA^{0-}(X_{1/2}))$. Here we are using that, given $v\in \calA_{\mathrm{loc}}^{0-,0-,0}(X_{\mathrm{res}}^{\mathrm{sp}})$,
	\begin{equation}
		[\hat{E} \mapsto v|_{E/x=\hat{E}} ] \in \calA^0([0,1)_{\hat{E}} ; \calA^{0-}(X_{1/2}) ) = \calA^0([0,1)_{\hat{E}} ; \calA^{0-}(X) ).
	\end{equation} 
	Since $K$ can be taken arbitrarily large, we conclude the claim.
\end{proof}
\begin{remark*}
	As the argument shows, each Fr\'echet seminorm of $\smash{u\in \calA_{\mathrm{loc}}^{0-,0-,(0,0)}(X_{\mathrm{res}}^{\mathrm{sp}})}$ is controlled by finitely many Fr\'echet seminorms of $(\chi x\partial_{E})^k u\in \calA_{\mathrm{loc}}^{0-,0-,0}(X_{\mathrm{res}}^{\mathrm{sp}})$ for finitely many $k$. In other words, the map 
	\begin{multline}
		\Big(\prod_{k=0}^\infty \calA_{\mathrm{loc}}^{0-,0-,0}(X_{\mathrm{res}}^{\mathrm{sp}})\Big) \cap \{\{ (\chi x \partial_E)^k u \}_{k=0}^\infty : u \in  \calA_{\mathrm{loc}}^{0-,0-,0}(X_{\mathrm{res}}^{\mathrm{sp}}) \} \ni \{ (\chi x \partial_E)^k u \}_{k=0}^\infty  \\ 
		\mapsto u \in \calA_{\mathrm{loc}}^{0-,0-,(0,0)}(X_{\mathrm{res}}^{\mathrm{sp}})
	\end{multline}
	is continuous when we endow the domain with the topology of $	\prod_{k=0}^\infty \calA_{\mathrm{loc}}^{0-,0-,0}(X_{\mathrm{res}}^{\mathrm{sp}})$. 
\end{remark*}

\section{The conjugated perspective}
\label{sec:operator}

We now construct the ``conjugated'' operator $\tilde{P} = \{\tilde{P}(\sigma)\}_{\sigma\geq 0}$. Given some 
\[
\{f(-;\sigma)\}_{\sigma\geq 0}, \{g(-;\sigma)\}_{\sigma\geq 0} \subset C^\infty(X^\circ)
\] 
and a family of differential operators $\{D(\sigma)\}_{\sigma\geq 0} \subset \operatorname{Diff}(X^\circ)$, we use the somewhat abusive notation $fDg = \{f(-;\sigma)D(\sigma)g(-;\sigma)\}_{\sigma\geq 0} \subset \operatorname{Diff}(X^\circ)$ to denote the family of differential operators 
\begin{equation}
	M_f D M_g = \{M_{f(-;\sigma)} D(\sigma) M_{g(-;\sigma)} \}_{\sigma\geq 0}, 
\end{equation}
where for $h\in C^\infty(X^\circ)$, $M_h: C_{\mathrm{c}}^\infty(X^\circ)\to C_{\mathrm{c}}^\infty(X^\circ)$ denotes the multiplication operator $C_{\mathrm{c}}^\infty(X^\circ)\ni \varphi \mapsto h \varphi$. For us, $f,g$ will be of the form $f(-;\sigma) = \exp(- i \Phi(-;\sigma))$ and $g(-;\sigma) = \exp(+ i \Phi(-;\sigma))$ for some $\Phi = \{\Phi(-;\sigma)\}_{\sigma\geq 0} \subset C^\infty(X^\circ)$. The conjugated operator $\tilde{P}$ is defined by 
\begin{equation}
	\tilde{P} = M_{\exp(-i \Phi)} P M_{\exp(+i \Phi)} = e^{-i \Phi} P e^{+i \Phi}
\end{equation}
for our eventual choice of $\Phi$.

As discussed in the introduction, $\Phi$ had ought to be determined to an order or two (including logarithmic terms) on $X^{\mathrm{sp}}_{\mathrm{res}}$ by the actual asymptotics of solutions of $P(\sigma)u = f$ for $f\in \calS(X)$. Rather than determine what $\Phi$ should be in this manner (that is by solving the PDE, or a model thereof, to a sufficient degree of accuracy), it is actually easier to work backwards, meaning to find the asymptotics of solutions to the given PDE by first finding a choice of $\Phi$ for which $\tilde{P}$ has a workable form, where ``workable'' roughly means qualitatively similar to the conjugated operator in \cite{VasyLA}. (We will then have to actually show that this choice describes the asymptotics of solutions to the PDE.) Such a choice has already been stated in the introduction, \S\ref{sec:introduction}, \cref{eq:Phi}, and the resultant conjugated operator is computed below. 
We now motivate that choice. 
Consider the model operator $P_{\mathrm{Model}} \in \operatorname{Diff}_{\mathrm{sc}}^2([0,\infty)_x)$ given by 
\begin{equation}
	P_{\mathrm{Model}}(\sigma) = - (1+x a_{00}) (x^2 \partial_x)^2 + (a + n-1) x^3 \partial_x - \sigma^2 -\mathsf{Z} x,
	\label{eq:pmodel}
\end{equation}
where $a_{00} ,a,\mathsf{Z} \in \bbR$, $\mathsf{Z}>0$, and where we require $a_{00}<0$ for simplicity. 
This captures the leading and subleading terms of $P(\sigma)$ in $\Psi_{\mathrm{scb}}(X)$, modulo the terms involving nonradial derivatives (which can be expected to be unimportant -- i.e.\ under symbolic control -- based on considerations similar to those in \cite{VasyLA}). 
For $\Phi(-;\sigma) \in C^\infty(\bbR^+_x\times \bbR^+_\sigma)$, let 
\begin{equation} 
	\tilde{P}_{\mathrm{Model}} = M_{\exp(-i \Phi)} P_{\mathrm{Model}} M_{\exp(+i \Phi)} = \exp(-i \Phi) P_{\mathrm{Model}} \exp(+i \Phi). 
\end{equation} 
This will be qualitatively  similar (at least with regards to the sc-calculus) to Vasy's conjugated operator family (with $\Im \alpha_\pm = 0$) if 
$\tilde{P}_{\mathrm{Model}}$ is equal to 
\begin{multline}
	P_{\mathrm{Goal}} = - (1+x a_{00}) \Big(x^2 \partial_x - \frac{x(n-1)}{2}\Big)^2   + 2i (1+x a_{00})  \sqrt{\sigma^2+\mathsf{Z} x - \sigma^2 a_{00} x} (x^2\partial_x) \\ + 2i (1+x a_{00}) \Big(  - \frac{(n-1)}{2}x \sqrt{\sigma^2+\mathsf{Z} x} + \frac{\mathsf{Z} x^2}{4} \frac{1}{\sqrt{\sigma^2+\sf\mathsf{Z}  x}} \Big)
	\label{eq:misc_pgo}
\end{multline}
modulo terms which are two sc-decay orders subleading at sf and ff (in particular in $x^2 \operatorname{Diff}_{\mathrm{sc}}^{*,0}(X)$ for $\sigma>0$ and $x^2 \operatorname{Diff}_{\mathrm{sc}}^{*,0}(X_{1/2})$ for $\sigma=0$, where $X=[0,\infty)_x$). 
The specific form of the lower order terms in \cref{eq:misc_pgo} bears comment: the operators $\smash{x^2 \partial_x - x(n-1)/2}$ and 
\begin{equation} 
(\sigma^2+\mathsf{Z} x)^{1/2} x^2 \partial_x - x(\sigma^2+\mathsf{Z} x)^{1/2} \frac{(n-1)}{2} + \frac{\mathsf{Z} x^2}{4} \frac{1}{(\sigma^2+\mathsf{Z} x)^{1/2}}
\end{equation} 
are both formally anti- self-adjoint with respect to the  $L^2([0,1), x^{-(n+1)} \dd x) = L^2_{\mathrm{sc}} [0,1)$ inner product. 
Indeed, for any $f,g \in \dot{\calS}([0,\infty))$, 
\begin{equation}
	\int_0^\infty f^*(x) \partial_x g(x) x^{-(n+1)} \dd x = - \int_0^\infty (\partial_x f)^* g(x) x^{-(n+1)} \dd x  + (n+1) \int_0^\infty \Big( \frac{f(x)}{x} \Big)^* g(x) x^{-(n+1)} \dd x,
\end{equation}
so, as bilinear forms  $\dot{\calS}([0,\infty))^2 \to \bbC$, 
\begin{align} 
	\partial_x^* &= - \partial_x + (n+1)/x,\\ 
	(x^2 \partial_x)^* &= x^2 \partial_x^* +[\partial_x^*,x^2] = - x^2 \partial_x  + x(n-1) 
	\intertext{and}
	\begin{split} 
		(\sqrt{\sigma^2+\mathsf{Z} x} x^2 \partial_x)^* &= \sqrt{\sigma^2+\mathsf{Z} x}(x^2 \partial_x)^* + [(x^2 \partial_x)^*, \sqrt{\sigma^2+\mathsf{Z} x}] \\
		&= - \sqrt{\sigma^2+\mathsf{Z} x} x^2 \partial_x  + x(n-1)\sqrt{\sigma^2+\mathsf{Z} x} - \frac{\mathsf{Z} x^2}{2} \frac{1}{\sqrt{\sigma^2+\mathsf{Z} x}},
	\end{split} 
\end{align}
which implies the claimed anti- self-adjointness. The $(1+x a_{00})$ terms in \cref{eq:misc_pgo}, along with the $-\sigma^2 a_{00} x$ under the square root, spoil anti- self-adjointness or self-adjointness, but only negligibly. Thus, the terms in \cref{eq:misc_pgo} have definite adjointness modulo negligible errors. 

As a preliminary step towards $P_{\mathrm{Goal}}$, we can conjugate away the $(a+n-1)x^3 \partial_x$ term in \cref{eq:pmodel}, getting 
\begin{equation}
	x^{-(a+n-1)/2} P_{\mathrm{Model}} x^{(a+n-1)/2} = - (1+x a_{00}) (x^2 \partial_x)^2 - \sigma^2 -\mathsf{Z} x \bmod x^2 \operatorname{Diff}_{\mathrm{sc}}^{1,0}(X). 
	\label{eq:misc_loq}
\end{equation}
The $x^2 \operatorname{Diff}^{1,0}_{\mathrm{sc}}$ remainder may contain some terms in \cref{eq:misc_pgo}, for example the square of the $x(n-1)/2$ terms, but we can conjugate these back in at the end of the computation.
Conjugating \cref{eq:misc_loq} by $\exp(+i\varphi)$ for a to-be-decided  $\varphi(x;\sigma)\in C^\infty(\bbR^+_x\times \bbR^+_\sigma)$, we get 
\begin{equation}
	e^{-i\varphi} x^{-(a+n-1)/2} P_{\mathrm{Model}} x^{(a+n-1)/2} e^{+i \varphi} = - (1+x a_{00}) (x^2 \partial_x + i x^2 \varphi')^2 - \sigma^2 -\mathsf{Z} x \bmod x^2 \operatorname{Diff}_{\mathrm{sc}}^{1,0}(X), 
	\label{eq:misc_072}
\end{equation}
assuming the contribution from conjugating the $x^2\operatorname{Diff}_{\mathrm{sc}}^{1,0}$ in \cref{eq:misc_loq} is negligible.
It is not unreasonable to expect (perhaps based on the $\mathsf{Z} =0$ case) that, for our eventual choice of $\varphi$, the leading order new contribution to \cref{eq:misc_072} is $x^4 (\varphi')^2$.  

Since the Coulomb term in \cref{eq:misc_072} is subleading order relative to $\sigma^2$ in the sense of decay, we should really be keeping track of the new contributions to one subleading order. To this order, the new contribution to the effective potential is $x^4(1+x a_{00}) (\varphi')^2$ (assuming the terms with second derivatives of $\varphi$ are negligible). Thus, we seek to arrange
\begin{equation}
	x^4 (1+x a_{00}) (\varphi')^2 = \sigma^2 + \mathsf{Z} x  \bmod x^2 C^\infty(X) .
	\label{eq:misc_j64}
\end{equation}
Multiplying through by $(1+x a_{00})^{-1} = 1 - x a_{00} \bmod x^2 C^\infty(X)$, this suggests setting $x^4 (\varphi')^2 = \sigma^2 + \mathsf{Z} x - \sigma^2 a_{00} x $, the solution of which (up to an arbitrary additive constant and conventional choice of sign) is
\begin{equation} 
	\varphi(x;\sigma) = \frac{1}{x} \sqrt{\sigma^2+\mathsf{Z} x - \sigma^2 a_{00} x } + \frac{1}{\sigma} (\mathsf{Z} - \sigma^2 a_{00} ) \operatorname{arcsinh}\Big( \frac{\sigma}{x^{1/2}} \frac{1}{(\mathsf{Z} -\sigma^2 a_{00})^{1/2}} \Big). 
	\label{eq:misc_616}
\end{equation} 
Recall that $\operatorname{arcsinh}(z) = \log(z+(1+z^2)^{1/2})$ for all $z\geq 0$.  Expanding $\operatorname{arcsinh}(z)$ in Taylor series around $z=0$, we see that the apparent singularity in \cref{eq:misc_616} at $\sigma = 0$ is removable (to all orders), and hence $\varphi(x;\sigma)$ defines a smooth function on $\bbR^+_x\times \bbR_\sigma$, and it is even in $\sigma$.

We observe that, given \cref{eq:misc_j64}, the $x^4 \varphi''$ term in \cref{eq:misc_072} is indeed negligible except at ff, where there is one non-negligible contribution, and it is precisely the final term in \cref{eq:misc_pgo} (modulo negligible terms). 

Given this definition of $\varphi$, $e^{-i\varphi} x^{-(a+n-1)/2} P_{\mathrm{Model}} x^{(a+n-1)/2} e^{+i \varphi}$ is given, modulo $x^2 \operatorname{Diff}^{1,0}_{\mathrm{sc}}(X)$, by 
\begin{multline}
	- (1+x a_{00})((x^2 \partial_x)^2  + 2 i x^4 \varphi' \partial_x -  x^4 \varphi' \varphi' + i(x^2 \partial_x)^2 \varphi) - \sigma^2 - \mathsf{Z} x  = - (1+x a_{00})((x^2 \partial_x)^2 \\ - 2 i x^2 \sqrt{\sigma^2+\mathsf{Z} x-\sigma^2 a_{00} x} \partial_x + i(x^2 \partial_x)^2 \varphi)  \bmod C^\infty(\bbR_{\sigma^2};x^2 C^\infty(X)),
	\label{eq:misc_oii}
\end{multline}
plus the term in $\smash{\operatorname{Diff}_{\mathrm{leC}}^{0,-2,-5,-2,-5}}$ that results from applying the first order operator in the $x^2 \operatorname{Diff}^{1,0}_{\mathrm{sc}}(X)$ remainder in \cref{eq:misc_072} to $\varphi$, which we will work out later.

We can now add back in the $x(n-1)/2$ terms: 
\begin{equation}
	e^{-i\varphi} x^{-a/2} P_{\mathrm{Model}} e^{+i \varphi} x^{a/2}  = P_{\mathrm{Goal}} \bmod \operatorname{Diff}_{\mathrm{leC}}^{1,-2,-4,-2,-4}.
\end{equation}
So, at least in this model case, conjugation by $e^{+i\Phi}=e^{i \varphi + (1/2)a \log x }$ has the required properties.

Returning to the full problem, we consider the family $\Phi = \{\Phi(-;\sigma)\}_{\sigma\geq 0}$ of $\Phi(-;\sigma) \in C^\infty( X^\circ\times \bbR^+_\sigma)$ given by 
\begin{multline}
	\Phi(x;\sigma) =    \frac{1}{x} \sqrt{\sigma^2+\mathsf{Z} x-\sigma^2 a_{00} x} + \frac{1}{\sigma}(\mathsf{Z} -\sigma^2 a_{00}) \operatorname{arcsinh} \Big( \frac{\sigma}{x^{1/2}} \frac{1}{(\mathsf{Z}  - \sigma^2 a_{00})^{1/2}} \Big)   -\frac{i}{2} a \log x .
\end{multline}
Observe that (after removing the removable singularity at $\sigma=0$) \cref{eq:osc_0} holds. 
We then define, for each $\sigma \geq 0$, 
\begin{equation} 
	\tilde{P}(\sigma) = M_{e^{-i \Phi(-;\sigma)}} P(\sigma) M_{e^{+i \Phi(-;\sigma)}},
\end{equation} 
i.e.\ $\tilde{P} = \exp(-i \Phi) P \exp(+i \Phi)$. 
Thus, $\tilde{P}(\sigma) \in S\operatorname{Diff}_{\mathrm{sc}}(X)$ for each $\sigma\geq 0$, and the coefficients all depend smoothly on $\sigma$ all the way down to $\sigma=0$ in compact subsets of $X^\circ$. Of course, this does not mean that $ [0,\infty)_\sigma \ni \sigma  \mapsto \tilde{P}(\sigma)\in S \operatorname{Diff}^{2,0}_{\mathrm{sc}}(X)$ is smooth all the way down to $\sigma=0$. This map is continuous for $\sigma>0$ but discontinuous at $\sigma=0$, as can be verified by computing sc-principal symbols. 

Let $\tilde{P}_0 = e^{-i\Phi}  P_0 e^{+i\Phi}$, $\tilde{P}_1 = e^{-i\Phi}  P_1 e^{+i\Phi}$, and $\tilde{P}_{2} = e^{-i\Phi}  P_{2} e^{+i\Phi}$, where $P_0,P_1,P_2$ are as in \S\ref{sec:introduction}. We have:

\begin{proposition}
	\label{prop:P00tilde_comp}
	For each $\sigma\geq 0$, $\tilde{P}_0(\sigma)$ is given with respect to the boundary-collar $\iota$ by 
	\begin{equation}
		\tilde{P}_0(\sigma) = - (1+xa_{00}) (x^2 \partial_x)^2 + x^2 \triangle_{\partial X} + (n-1) x^3 \partial_x+ L(\sigma) +V_{\mathrm{eff}}(x;\sigma), 
		\label{eq:misc_837}
	\end{equation}
	where
	\begin{equation}
			L(\sigma) = 2i x (1+xa_{00}) \sqrt{\sigma^2+ \mathsf{Z} x - \sigma^2 a_{00} x }\Big(x \partial_x - \frac{n-1}{2} + \frac{\mathsf{Z}}{4}  \frac{x}{\sigma^2+\mathsf{Z} x} \Big) - a  a _{00} x^4 \partial_x  \label{eq:misc_lk3} 
	\end{equation}
	and $V_{\mathrm{eff}} \in x^2 C^\infty(X^{\mathrm{sp}}_{\mathrm{res}})$. 
\end{proposition}
\begin{proof}
	We work on $\hat{X}=[0,\bar{x})_x\times \partial X$. 
	
	We may write $\tilde{P}_0 = P_0 + e^{-i \Phi}[P_0, e^{+i \Phi}]$. Observe that 
	\begin{align}
		[P_{0}, e^{+i \Phi}] &= - (1+xa_{00})[  (x^2 \partial_x)^2,e^{i \Phi} ] +  x^3 (a + n-1) [\partial_x, e^{i \Phi}] + x^2 [\triangle_{\partial X}, e^{i \Phi}] \\
		e^{-i\Phi}[  (x^2 \partial_x)^2,e^{i \Phi} ] &=   2i  x^4 \Phi' \partial_x -  x^4 \Phi' \Phi' + i (x^2 \partial_x)^2 \Phi  \\
		e^{-i\Phi}[\partial_x ,e^{i\Phi}] &=i \Phi'  \\
		e^{-i\Phi}[\triangle_{\partial X}, e^{i\Phi}] &= 0, \label{eq:misc_a0c}
	\end{align} 
	where the primes denote differentiation in $x$. (If $a,a_{00}$ were nonconstant functions on $\partial X$, then \cref{eq:misc_a0c} would not hold. This is ultimately the reason for assuming that $a,a_{00}$ are constant.) 
	Thus, if we set 
	\begin{align}
		L(\sigma) &= a x^3 \partial_x - 2i x^4 (1+xa_{00}) \Phi' \partial_x +V_{\mathrm{SA}}  ,  \label{eq:misc_j31} \\
		\begin{split} 
			V_{\mathrm{eff}}(x;\sigma) &= -(1+x a_{00})( - x^4 \Phi'\Phi' + i (x^2 \partial_x)^2 \Phi) + i x^3 (a+n-1)  \Phi' - V_{\mathrm{SA}} - \sigma^2-\mathsf{Z}  x
		\end{split} \label{eq:Veffd}
	\end{align} 
	for 
	\begin{equation} 
		V_{\mathrm{SA}} = -2 i x (1+x a_{00}) \sqrt{\sigma^2 + \mathsf{Z} x- \sigma^2 a _{00} x}\Big( \frac{n-1}{2} - \frac{\mathsf{Z}}{4} \frac{x}{\sigma^2+\mathsf{Z} x}\Big),
	\end{equation}  
	then \cref{eq:misc_837} holds, and it only remains to verify \cref{eq:misc_lk3} and the fact that 
	$V_{\mathrm{eff}} \in x^2 C^\infty(X^{\mathrm{sp}}_{\mathrm{res}})$. 
	
	We compute that  
	\begin{align}
		\begin{split} 
		+x^4 \Phi' \Phi'  &= x^4 \Big(  \frac{1}{x^2} \sqrt{\sigma^2+\mathsf{Z} x - \sigma^2a_{00} x}   + \frac{ia}{2x} \Big)^2  
		\\
		&=  \sigma^2+\mathsf{Z} x - \sigma^2a_{00} x + i a x \sqrt{\sigma^2+\mathsf{Z} x - \sigma^2a_{00} x}  - \frac{ x^2 a^2}{4} ,  \end{split} \label{eq:vc1}  \\
		-i (x^2 \partial_x)^2 \Phi &=  - \frac{ax^2}{2} + \frac{i}{2} x^2 \frac{(\mathsf{Z} -\sigma^2 a_{00})}{ (\sigma^2+\mathsf{Z} x- \sigma^2 a_{00} x)^{1/2}} ,   \label{eq:vc2}\\
		+i x^3 \Phi' 
		&= \frac{ax^2}{2} - ix  \sqrt{\sigma^2+\mathsf{Z} x - \sigma^2 a_{00} }  .\label{eq:vc3} 
	\end{align}

	From \cref{eq:vc3} and \cref{eq:misc_j31}, we get \cref{eq:misc_lk3}. 
	
	When adding up the various contributions to $V_{\mathrm{eff}}$, as written in \cref{eq:Veffd}, a few key cancellations happen by design: 
	\begin{enumerate}[label=(\Roman*)]
		\item the first $\sigma^2$ in \cref{eq:Veffd} (coming from $x^4 \Phi'\Phi'$, \cref{eq:vc1}) cancels with the last $-\sigma^2$ in \cref{eq:Veffd}, 
		\item the first $\mathsf{Z}x$ term, also coming from $x^4 \Phi'\Phi'$, cancels with the $-\mathsf{Z}x$ in \cref{eq:Veffd},  so that the original Coulomb-like term has been ``conjugated away,''
		\item multiplying  \cref{eq:vc1} by $(1+xa_{00})$ the terms in $(1+xa_{00})(\sigma^2-\sigma^2 a_{00} x)$ linear in $a_{00}$ cancel per difference-of-squares, 
		\item the $ i a x (\sigma^2+\mathsf{Z} x- \sigma^2 a_{00} x)^{1/2}$ term in \cref{eq:vc1} cancels with the term in $ia x^3 \Phi'$ coming from the last term in \cref{eq:vc3}. 
		\item The $i x (\sigma^2 + \mathsf{Z} x - \sigma^2 a_{00} x)^{1/2} (n-1)$ term in $V_{\mathrm{eff}}$ coming from $V_{\mathrm{SA}}$ cancels with the identical term in $i(n+1) x^3 \Phi'$ coming from the last term in \cref{eq:vc3}. 
	\end{enumerate}
	All in all, $V_{\mathrm{eff}}$ is given by 
	\begin{multline}
		a_{00} (\mathsf{Z} - \sigma^2 a_{00})x^2  + i a_{00}(a+n-1) x^2 \sqrt{\sigma^2+\mathsf{Z} x - \sigma^2 a_{00} x } -  \Big[ (1+x a_{00}) \Big( \frac{a^2}{4} + \frac{a}{2} \Big) - \frac{a}{2} (a+n-1) \Big]x^2   \\
		+ \frac{ix}{2} (1+ x a_{00}) \Big[  \frac{\mathsf{Z}x -\sigma^2 a_{00}x}{ (\sigma^2+\mathsf{Z} x- \sigma^2 a_{00} x)^{1/2}}  -   (\sigma^2 + \mathsf{Z} x - \sigma^2 a_{00} x )^{1/2} \frac{ \mathsf{Z} x}{\sigma^2+\mathsf{Z} x}\Big]. 
		\label{eq:misc_613}
	\end{multline}
	It is clear that the first line of \cref{eq:misc_613} defines an element of $x^2 C^\infty(X^{\mathrm{sp}}_{\mathrm{res}})$. On the other hand, $-(ix/2)(1+xa_{00} )\sigma^2 a_{00} x/(\sigma^2+\mathsf{Z} x - \sigma^2 a_{00} x)^{1/2}$ is in $x^2 C^\infty(X^{\mathrm{sp}}_{\mathrm{res}})$ as well. 
	Thus, it remains to verify that 
	\begin{equation}
		 \frac{\mathsf{Z}x}{ (\sigma^2+\mathsf{Z} x)^{1/2}}  -   (\sigma^2 + \mathsf{Z} x - \sigma^2 a_{00} x )^{1/2} \frac{ \mathsf{Z} x}{\sigma^2+\mathsf{Z} x} \in x C^\infty(X^{\mathrm{sp}}_{\mathrm{res}}),
	\end{equation}
	i.e.\ that 
	\begin{equation}
		1 - \Big(1 - \frac{\sigma^2 a_{00} x}{\sigma^2+\mathsf{Z} x} \Big)^{1/2} \in (\sigma^2+\mathsf{Z} x)^{1/2} C^\infty(X^{\mathrm{sp}}_{\mathrm{res}}). 
		\label{eq:misc_841} 
	\end{equation}
	This is of course not true for each term on the left-hand side individually, but we can expand 
	\begin{equation}
		\Big(1 - \frac{\sigma^2 a_{00} x}{\sigma^2 + \mathsf{Z} x}\Big)^{1/2} = 1 \bmod  \frac{\sigma^2 a_{00} x}{\sigma^2 + \mathsf{Z} x} C^\infty(X^{\mathrm{sp}}_{\mathrm{res}})  =1 \bmod x C^\infty(X^{\mathrm{sp}}_{\mathrm{res}})
		\label{eq:misc_991}
	\end{equation} 
	(since $f(\zeta) = \zeta^{-1} ((1-\zeta)^{1/2}-1) \in C^\infty(-\infty,1)_\zeta$, $f(\sigma^2 a_{00} x / (\sigma^2+\mathsf{Z} x)) \in C^\infty(X^{\mathrm{sp}}_{\mathrm{res}})$, which implies \cref{eq:misc_991}), 
	so in fact \cref{eq:misc_841} is true with some room to spare. 
	We can then conclude that $V_{\mathrm{eff}}$ is in $x^2 C^\infty(X^{\mathrm{sp}}_{\mathrm{res}})$.
\end{proof}

\begin{proposition}
	\label{prop:Lcomp}
	The family $L=\{L(\sigma)\}_{\sigma\geq 0} \in \operatorname{Diff}_{\mathrm{leC}}^{1,0,-2,-1,-3}(X)$ satisfies 
	\begin{multline}
		L = 2 i (1+xa_{00}) \Big( 1 - \frac{\sigma^2 a_{00} x}{2 }\frac{1}{\sigma^2+\mathsf{Z}x} \Big) x \sqrt{\sigma^2 + \mathsf{Z} x }  \Big( x \partial_x - \frac{n-1}{2} + \frac{\mathsf{Z}}{4} \frac{x}{\sigma^2+\mathsf{Z}x}\Big) \\ \bmod \operatorname{Diff}_{\mathrm{leC}}^{1,-2,-5,-3,-6}(X) 
		\label{eq:misc_iyt}
	\end{multline}
	near $\partial X$. 
\end{proposition}
\begin{proof}
	It suffices to restrict attention to $\hat{X} = [0,\bar{x})\times \partial X$. 
	
	We refer to \cref{eq:misc_lk3}. 
	Expanding $(1 - \sigma^2 a_{00} x / (\sigma^2+\mathsf{Z}x))^{1/2} = 1 - (1/2) \sigma^2 a_{00} x / (\sigma^2+\mathsf{Z}x) + O(\sigma^4 x^2 / (\sigma^2+\mathsf{Z} x)^2) $ in Taylor series, we deduce that 
	\begin{multline}
		2 i x (1+xa_{00}) \Big(\Big( 1 - \frac{\sigma^2 a_{00} x}{2 (\sigma^2+\mathsf{Z}x)} \Big)  \sqrt{\sigma^2 + \mathsf{Z} x } -\sqrt{\sigma^2 + \mathsf{Z} x - \sigma^2 a_{00} x }\Big) \Big( x \partial_x - \frac{n-1}{2} + \frac{\mathsf{Z}}{4} \frac{x}{\sigma^2+\mathsf{Z}x}\Big) \\ \in \operatorname{Diff}_{\mathrm{leC}}^{1,-2,-6,-3,-7}(X). 
	\end{multline}
	Thus,
	\begin{multline}
		L = 2 i (1+xa_{00}) \Big( 1 - \frac{\sigma^2 a_{00} x}{2 }\frac{1}{\sigma^2+\mathsf{Z}x} \Big) x \sqrt{\sigma^2 + \mathsf{Z} x }  \Big( x \partial_x - \frac{n-1}{2} + \frac{\mathsf{Z}}{4} \frac{x}{\sigma^2+\mathsf{Z}x}\Big) - a a_{00} x^4 \partial_x 
		\\ \bmod \operatorname{Diff}_{\mathrm{leC}}^{1,-2,-6,-3,-7}(X). 
	\end{multline}
	On the other hand, $a a_{00} x^4 \partial_x \in \operatorname{Diff}_{\mathrm{leC}}^{1,-2,-5,-3,-6}(X)$.
	We conclude \cref{eq:misc_iyt} from the above.
\end{proof}

\begin{proposition}
	\label{prop:tildeP0_orders}
	$\tilde{P}_0 \in \operatorname{Diff}_{\mathrm{leC}}^{2,0,-2,-1,-3}(X)$, and 
	\begin{multline}
		\tilde{P}_0 = - (x^2 \partial_x)^2 + x^2 \triangle_{\partial X}+2ix \sqrt{\sigma^2+\mathsf{Z} x} \Big( x \partial_x - \frac{n-1}{2} + \frac{\mathsf{Z}}{4} \frac{x}{\sigma^2+\mathsf{Z} x} \Big) \\ \bmod \operatorname{Diff}_{\mathrm{leC}}^{2,-1,-3,-2,-4}(X). 
		\label{eq:misc_v9o}
	\end{multline}
\end{proposition}
\begin{proof}
	We have $L(\sigma) \in \operatorname{Diff}_{\mathrm{leC}}^{1,0,-2,-1,-3}(X)$ and 
	\begin{align}
		\begin{split} 
		- (x^2 \partial_x)^2 + x^2 \triangle_{\partial X} &\in \operatorname{Diff}_{\mathrm{leC}}^{2,0,-2,-2,-4}(X)  \\
		x a_{00}(x^2 \partial_x)^2 &\in \operatorname{Diff}_{\mathrm{leC}}^{2,-1,-4,-3,-6}(X)\\
		(n-1)x^3 \partial_x &\in \operatorname{Diff}_{\mathrm{leC}}^{1,-1,-3,-2,-4}(X) \\
		V_{\mathrm{eff}} &\in \operatorname{Diff}_{\mathrm{leC}}^{0,-2,-4,-2,-4}(X).
		\end{split} 
		\label{eq:misc_58a}
	\end{align}
	Thus, by \Cref{prop:P00tilde_comp}, $\tilde{P}_0 \in \operatorname{Diff}_{\mathrm{leC}}^{2,0,-2,-1,-3}(X)$, and 
	\begin{equation}
		\tilde{P}_0 = - (x^2 \partial_x)^2 + x^2 \triangle_{\partial X} + L \bmod \operatorname{Diff}_{\mathrm{leC}}^{2,-1,-3,-2,-4}(X). 
	\end{equation} 
	Simplifying $L$ modulo  $\operatorname{Diff}_{\mathrm{leC}}^{2,-1,-3,-2,-4}(X)$ using \Cref{prop:Lcomp}, we get \cref{eq:misc_v9o}. 
\end{proof}

Recall that $b_j,P_{\perp,j}$ were defined in \cref{eq:misc_000}.
\begin{proposition}
	\label{prop:P1tilde_comp_new}
	For some $\Upsilon_1,\ldots,\Upsilon_J \in S\operatorname{Diff}_{\mathrm{leC}}^{1,-1,-4,-2,-5}(X)$ which near $\partial X$ are given by $\Upsilon_j = i x^4 \Phi' b_j P_{\perp,j}$, we have 
	\begin{equation}
		\tilde{P}_1(\sigma) = P_1(\sigma) + \sum_{j=1}^J \Upsilon_j + R   
		\label{eq:misc_lk7}
	\end{equation}
	for some $R \in C^\infty([0,\infty)_{\sigma^2} ; \operatorname{Diff}^2(X))$ which is supported outside of some neighborhood $U\subset X$ of $\partial X$. 
	Thus, $\tilde{P}_1 = \{\tilde{P}_1(\sigma)\}_{\sigma\geq 0} \in S\operatorname{Diff}_{\mathrm{leC}}^{2,-1,-3,-2,-4}(X)$. If $P_1$ is classical to order $\beta_1>0$, then 
	\begin{equation}
		\tilde{P}_1 \in \operatorname{Diff}_{\mathrm{leC}}^{2,-1,-3,-2,-4}(X) + S\operatorname{Diff}_{\mathrm{leC}}^{2,-1-\beta_1,-3-2\beta_1,-2-\beta_1,-4-2\beta_1}(X). 
	\end{equation}
\end{proposition}
\begin{proof}
	First observe that $\chi \tilde{P}_1 \in C^\infty([0,\infty)_{\sigma^2}; \operatorname{Diff}^2(X^\circ))$ for any $\chi \in C_{\mathrm{c}}^\infty (X^\circ)$.
	Now let $P_{j,\mathrm{ext}} \in C^\infty([0,\infty)_{\sigma^2}; S\operatorname{Diff}_{\mathrm{scb}}^{2,-1,-3}(X))$ be equal to $x^4 P_{\perp,j} b_j \partial_x$ near $\partial X$. Now define 
	\begin{equation}
		\Upsilon_j = e^{-i\Phi} [P_{j,\mathrm{ext}},e^{+i\Phi} ] = e^{-i\Phi} [ (1-\chi)P_{j,\mathrm{ext}},e^{+i\Phi} ] + e^{-i\Phi} [ \chi P_{j,\mathrm{ext}},e^{+i\Phi} ],
	\end{equation}
	where $\chi$ is identically equal to one in a sufficiently large open set such that $1-\chi$ is supported in a neighborhood for which \cref{eq:misc_000} applies.
	Evidently, we have $e^{-i\Phi} [ \chi P_{j,\mathrm{ext}},e^{+i\Phi} ] \in C^\infty([0,\infty)_{\sigma^2}; \operatorname{Diff}^2(X\backslash U))$ for some neighborhood $U \subset X$ of $\partial X$. On the other hand, $e^{-i\Phi} [ (1-\chi)P_{j,\mathrm{ext}},e^{+i\Phi} ] = i (1-\chi) x^4 \Phi' b_j P_{\perp,j}$, so $\Upsilon_j = e^{-i\Phi} [ P_{j,\mathrm{ext}},e^{+i\Phi} ]$ near $\partial X$. 
	
	We now write 
	\begin{align}
		\begin{split} 
		\tilde{P}_1 &= P_1 + e^{-i\Phi} [ (1-\chi) P_1,e^{+i\Phi}] + e^{-i\Phi} [ \chi P_1,e^{+i\Phi}]
		\\ 
		&= P_1 + \sum_{j=1}^J \Upsilon_j + e^{-i\Phi} [ \chi P_1,e^{+i\Phi}]. 
		\end{split} 
		\label{eq:misc_a12}
	\end{align}
	Set $R=e^{-i\Phi} [ \chi P_1,e^{+i\Phi}]$. Then \cref{eq:misc_lk7} holds, and $R \in C^\infty([0,\infty)_{\sigma^2} ; \operatorname{Diff}^2(X))$ is supported outside of some neighborhood $U\subset X$ of $\partial X$.
	
	We observe, from \cref{eq:misc_000} (and \Cref{prop:diff->PsiDO}), that 
	\begin{equation} 
		P_1 \in S\operatorname{Diff}_{\mathrm{leC}}^{2,-1,-4,-3,-6}(X) + S\operatorname{Diff}_{\mathrm{leC}}^{1,-1,-3,-2,-4}(X) \subseteq S\operatorname{Diff}_{\mathrm{leC}}^{2,-1,-3,-2,-4}(X).
	\end{equation} 
	Since $\Upsilon_1,\ldots,\Upsilon_J \in S\operatorname{Diff}_{\mathrm{leC}}^{1,-1,-4,-2,-5}(X)$ (and the same holds for $R$, trivially), we conclude that 
	\begin{equation} 
		\tilde{P}_1 \in S\operatorname{Diff}_{\mathrm{leC}}^{2,-1,-3,-2,-4}(X),
		\label{eq:misc_a10}
	\end{equation} 
	as claimed.  
	
	If, $b_j,b_j',b_{j}''$ satisfy \cref{eq:misc_a1o}, then the conclusion is similar, except now we can write $P_1 \in \operatorname{Diff}_{\mathrm{leC}}^{2,-1,-3,-2,-4}(X) + x^{\beta_1} S\operatorname{Diff}_{\mathrm{leC}}^{2,-1,-3,-2,-4}(X)$ and 
	\begin{equation} 
		\Upsilon_1,\ldots,\Upsilon_J \in \operatorname{Diff}_{\mathrm{leC}}^{1,-1,-4,-2,-5}(X) + x^{\beta_1} S\operatorname{Diff}_{\mathrm{leC}}^{1,-1,-4,-2,-5}(X),
	\end{equation} 
	which leads to the conclusion \cref{eq:misc_a12} as a strengthening of \cref{eq:misc_a10}. 
\end{proof}

\begin{proposition} 
	\label{prop:P2tilde_comp}
	Given $\delta>0$ such that $P_2 \in C^\infty([0,\infty)_{\sigma^2} ; S \operatorname{Diff}_{\mathrm{scb}}^{2,-1-\delta,-3/2-\delta}(X))$,
	\begin{equation} 
	\tilde{P}_2=\{\tilde{P}_2 (\sigma) \}_{\sigma\geq 0} \in S\operatorname{Diff}_{\mathrm{leC}}^{2,-1-\delta,-3-2\delta,-1-\delta,-3-2\delta}(X) .
	\label{eq:misc_p0o}
	\end{equation} 
	If $P_2$ is classical to order $(\beta_2,\beta_3)$, then we can write 
	\begin{equation} 
		\tilde{P}_2=\{\tilde{P}_2 (\sigma) \}_{\sigma\geq 0} \in \operatorname{Diff}_{\mathrm{leC}}^{2,-2,-4,-2,-4}(X) + S\operatorname{Diff}_{\mathrm{leC}}^{2,-1-\beta_2,-3-2\beta_2,-1-\beta_2,-3-2\beta_2}(X) + x^{3/2+\beta_3}S^0(X).
		\label{eq:misc_p02}
	\end{equation} 
\end{proposition}
\begin{proof}
	We restrict attention to $\hat{X}=[0,\bar{x})_x\times \partial X$, which suffices by an argument similar to that in the proof of \Cref{prop:P1tilde_comp_new}.  
	
	Let $y=(y_1,\ldots,y_{n-1})$ denote local coordinates on  $\partial X$. In terms of these, we can write
	\begin{multline}
		P_2(\sigma) = x^{\delta}\Big[x^5 c \partial_x^2  +x^4 \sum_{j=1}^{n-1} c_{j}  \partial_x \partial_{y_j} + x^3 \sum_{j,k=1}^{n-1} c_{j,k} \partial_{y_j} \partial_{y_k}  +d  x^{3}\partial_x + \sum_{j=1}^{n-1} d_j x^{2} \partial_{y_j} + x^{3/2}e \Big]
		\label{eq:misc_p21}
	\end{multline} 
	where $\{c,d,e\} \cup \{c_j,d_j,c_{j,k}\}_{j,k=1}^{n-1}\subset C^\infty([0,\infty)_{\sigma^2};S^0(X))$ and $c_{j,k}=c_{k,j}$.  
	We then have 
	\begin{multline} 
		P_2(\sigma)-\tilde{P}_2(\sigma) = x^\delta \Big[ 2i x^5 c \Phi' \partial_x+ ix^4 \sum_{j=1}^{n-1} c_j ( \Phi' \partial_{y_j} + \partial_{y_j} \Phi \partial_x) + 2ix^3 \sum_{j,k=1}^{n-1} c_{j,k}\partial_{y_j} \Phi \partial_{y_k}   - x^5 c \Phi'\Phi'   \Big] 
		\\ +x^\delta  \Big[  i x^5c \Phi'' + x^4  \sum_{j=1}^{n-1}  c_j (i \partial_{y_j} \Phi' - \Phi' \partial_{y_j} \Phi) +x^3 \sum_{j,k=1}^{n-1} c_{jk} (i\partial_{y_j}\partial_{y_k} \Phi - \partial_{y_j} \Phi \partial_{y_k} \Phi )  \Big] \\
		+ x^\delta \Big[ idx^3 \Phi' + \sum_{j=1}^{n-1} id_j x^2 \partial_{y_j} \Phi \Big], 
		\label{eq:misc_kjh}
	\end{multline} 
	i.e., since $\Phi$ does not depend on tangential coordinates,
	\begin{equation}
		P_2(\sigma) - \tilde{P}_2(\sigma) = x^\delta \Big[ 2 i x^5 c \Phi' \partial_x+ i x^4 \sum_{j=1}^{n-1} c_j  \Phi' \partial_{y_j}     - x^5 c \Phi'\Phi' + i x^5 c \Phi'' + i d x^3 \Phi'  \Big].
	\end{equation}
	It follows from \cref{eq:vc3} that
	\begin{equation} 
		x^{3+\delta} d \Phi' \in  \Psi_{\mathrm{leC}}^{0,-1-\delta,-3-2\delta ,-1-\delta,-3-2\delta}(X),
		\label{eq:misc_nbe}
	\end{equation} 
	$x^{5+\delta} c \Phi'\Phi' \in  \Psi_{\mathrm{leC}}^{0,-1-\delta,-4-2\delta ,-1-\delta,-4-2\delta}(X)$ by \cref{eq:vc1}, and $x^{5+\delta} c \Phi'' \in  \Psi_{\mathrm{leC}}^{0,-2-\delta,-5-2\delta,-2-\delta,-5-2\delta}(X)$ by \cref{eq:vc3}. 
	On the other hand, 
	\begin{equation} 
		x^{5+\delta} \Phi' \partial_x ,x^{4+\delta} \Phi' \partial_{y_j} \in \Psi_{\mathrm{leC}}^{1,-1-\delta,-4-2\delta,-2-\delta,-5-2\delta}(X)
		\label{eq:misc_nbf}
	\end{equation} 
	by \cref{eq:vc3}. Combining the observations above, in particular \cref{eq:misc_nbe} and \cref{eq:misc_nbf}, we conclude that  
	\begin{equation} 
		P_2- \tilde{P}_2 \in S\operatorname{Diff}_{\mathrm{leC}}^{1,-1-\delta,-3-2\delta ,-1-\delta,-3-2\delta}(X).
		\label{eq:misc_po1}
	\end{equation}  
	
	By \cref{eq:misc_p21} (and \Cref{prop:diff->PsiDO}, \cref{eq:misc_ljs}), 
	\begin{multline} 
		P_2 \in S\operatorname{Diff}_{\mathrm{leC}}^{2,-1-\delta,-3-2\delta,-2-\delta,-4-2\delta}(X) + S\operatorname{Diff}_{\mathrm{leC}}^{0,-3/2-\delta,-3-2\delta,-3/2-\delta,-3-2\delta}(X) \\ \subset S\operatorname{Diff}_{\mathrm{leC}}^{2,-1-\delta,-3-2\delta ,-3/2-\delta,-3-2\delta}(X).
		\label{eq:misc_p11}
	\end{multline} 
	We conclude \cref{eq:misc_p0o} from \cref{eq:misc_po1} and \cref{eq:misc_p11}. 
	
	If \cref{eq:misc_p1i} holds, then instead of \cref{eq:misc_po1} we conclude 
	\begin{equation} 
		P_2- \tilde{P}_2 \in \operatorname{Diff}_{\mathrm{leC}}^{1,-2,-5 ,-2,-5}(X) + S\operatorname{Diff}_{\mathrm{leC}}^{1,-1-\beta_2,-3-2\beta_2 ,-1-\beta_2,-3-2\beta_2}(X), 
		\label{eq:misc_pom}
	\end{equation}  
	and instead of \cref{eq:misc_p11} we have 
	\begin{equation}
		P_2 \in \operatorname{Diff}_{\mathrm{leC}}^{2,-2,-4,-2,-4}(X) + S\operatorname{Diff}_{\mathrm{leC}}^{2,-1-\beta_2,-3-2\beta_2,-3/2-\beta_2,-3-2\beta_2}(X) + x^{3/2+\beta_3}S^0(X).
		\label{eq:misc_pos}
	\end{equation}
	\Cref{eq:misc_p02} follows from \cref{eq:misc_pom} and \cref{eq:misc_pos}.
\end{proof}

\begin{proposition} 
\label{prop:tilde_P_inc} 
$\tilde{P} \in \operatorname{Diff}_{\mathrm{leC}}^{2,0,-2,-1,-3}(X) + S\operatorname{Diff}_{\mathrm{leC}}^{2,-1,-3-2\delta,-1-\delta ,-3-2\delta}(X)$, with  
$\tilde{P} = \tilde{P}_0 \bmod S\operatorname{Diff}_{\mathrm{leC}}^{2,-1,-3-2\delta,-1-\delta ,-3-2\delta}(X)$. 
If $P_1,P_2$ are classical to orders $\beta_1$ and $(\beta_2,\beta_3)$ respectively, then
\begin{equation}
	\tilde{P} \in \operatorname{Diff}_{\mathrm{leC}}^{2,0,-2,-1,-3}(X) + x^{\beta_1}S\operatorname{Diff}_{\mathrm{leC}}^{2,-1,-3,-2 ,-4}(X) + x^{\beta_2} S\operatorname{Diff}_{\mathrm{leC}}^{2,-1,-3,-1,-3}(X) + x^{3/2+\beta_3} S^0(X).
	\label{eq:misc_l18}
\end{equation}
Moreover,
\begin{multline}
	\tilde{P} = - (x^2 \partial_x)^2 + x^2 \triangle_{\partial X} +2ix \sqrt{\sigma^2+\mathsf{Z} x} \Big( x \partial_x - \frac{n-1}{2} + \frac{\mathsf{Z}}{4} \frac{x}{\sigma^2+\mathsf{Z} x} \Big) \\ \bmod \operatorname{Diff}_{\mathrm{leC}}^{2,-1,-3,-2,-4}(X) + S \operatorname{Diff}_{\mathrm{leC}}^{2,-1,-3-2\delta,-1-\delta,-3-2\delta}(X). 
	\label{eq:misc_v9m}
\end{multline}
Thus, 
\begin{multline}
	\tilde{P} = - (x^2 \partial_x)^2 + x^2 \triangle_{\partial X} +2ix (\sigma^2+\mathsf{Z} x)^{1/2} x \partial_x \\ \bmod \operatorname{Diff}_{\mathrm{leC}}^{2,-1,-3,-1,-3}(X) + S \operatorname{Diff}_{\mathrm{leC}}^{2,-1,-3-2\delta,-1-\delta,-3-2\delta}(X). 
	\label{eq:misc_v9v}
\end{multline}
\end{proposition}
\begin{proof}
	We have $\tilde{P}(\sigma) =   \tilde{P}_0 + \tilde{P}_1+\tilde{P}_2$.
	As seen in \Cref{prop:P1tilde_comp_new} and \Cref{prop:P2tilde_comp}, $\tilde{P}_1 \in S\operatorname{Diff}_{\mathrm{leC}}^{2,-1,-3,-2,-4}(X)$ and $\tilde{P}_2 \in S\operatorname{Diff}_{\mathrm{leC}}^{2,-1-\delta,-3-2\delta,-1-\delta,-3-2\delta}(X)$, so 
	\begin{equation}
		\tilde{P}_1 + \tilde{P}_2  \in S\operatorname{Diff}_{\mathrm{leC}}^{2,-1,-3-2\delta,-1-\delta ,-3-2\delta}(X)
		\label{eq:misc_9ko}
	\end{equation}
	(where we are using $\delta<1/2$). Likewise, if $P_1,P_2$ are classical to orders $\beta_1$ and $(\beta_2,\beta_3)$ then \Cref{prop:P1tilde_comp_new} and \Cref{prop:P2tilde_comp} yield 
	\begin{multline}
		\tilde{P}_1 + \tilde{P}_2  \in \operatorname{Diff}_{\mathrm{leC}}^{2,-1,-3,-2,-4}(X)+ x^{\beta_1}S\operatorname{Diff}_{\mathrm{leC}}^{2,-1,-3,-2 ,-4}(X) + x^{\beta_2} S\operatorname{Diff}_{\mathrm{leC}}^{2,-1,-3,-1,-3}(X) \\ + x^{3/2+\beta_3} S^0(X).
	\end{multline}
	By \Cref{prop:tildeP0_orders}, $\tilde{P}_0 \in \operatorname{Diff}_{\mathrm{leC}}^{2,0,-2,-1,-3}(X)$, so $\tilde{P}$ is in the claimed spaces.
	
	Furthermore, by \cref{eq:misc_9ko}, $\tilde{P}=\tilde{P}_0 \bmod S\operatorname{Diff}_{\mathrm{leC}}^{2,-1,-3-2\delta,-1-\delta ,-3-2\delta}(X)$.  
	Combining with \Cref{prop:tildeP0_orders}, we get \cref{eq:misc_v9m}. 
\end{proof}

To conclude this discussion, we let 
\begin{equation}
	\operatorname{N}(\tilde{P}) = \{\operatorname{N}(\tilde{P})(\sigma)\}_{\sigma\geq 0}, \qquad \operatorname{N}(\tilde{P})(\sigma) = 2i x \sqrt{\sigma^2 +\mathsf{Z} x} \Big( x \partial_x - \frac{n-1}{2} + \frac{\mathsf{Z}}{4} \frac{x}{\sigma^2+\mathsf{Z}x}\Big)
	\label{eq:Ndef}
\end{equation}
denote the \emph{leC-normal operator}, defined initially near $\partial X$. To avoid technicalities, we extend $\operatorname{N}(\tilde{P})(\sigma)$ to a differential operator on $X^\circ$ such that, in any compact subset of $X^\circ$, $\smash{\operatorname{N}(\tilde{P})(\sigma)}$ depends smoothly on $E=\sigma^2$, all the way down to $\sigma = 0$. Thus:
\begin{propositionp}
	$\operatorname{N}(\tilde{P}) \in \operatorname{Diff}_{\mathrm{b,leC}}^{1,-1,-3}(X)$.
	\label{prop:normal_operator_order}
\end{propositionp}
The following proposition justifies the term ``leC-normal operator:''  
\begin{proposition}
	\label{prop:normal_error}
	$\operatorname{N}(\tilde{P}) - \tilde{P} \in S\operatorname{Diff}_{\mathrm{leC}}^{2,0,-2,-1-\delta,-3-2\delta}(X) \subseteq S\operatorname{Diff}_{\mathrm{b,leC}}^{2,-1-\delta,-3-2\delta}(X)$. 
	If $P_1,P_2$ are classical to orders $\beta_1$ and $(\beta_2,\beta_3)$ respectively, then 
	\begin{multline} 
		\operatorname{N}(\tilde{P}) - \tilde{P} \in \operatorname{Diff}_{\mathrm{leC}}^{2,0,-2,-2,-4}(X) +  x^{\beta_1}S\operatorname{Diff}_{\mathrm{leC}}^{2,-1,-3,-2 ,-4}(X) + x^{\beta_2} S\operatorname{Diff}_{\mathrm{leC}}^{2,-1,-3,-1,-3}(X) \\ + x^{3/2+\beta_3}S^0(X).
		\label{eq:misc_nep}
	\end{multline}  
\end{proposition}
\begin{proof}
	We have $\operatorname{N}(\tilde{P})-\tilde{P}(\sigma) =   (\operatorname{N}(\tilde{P})-\tilde{P}_0) - \tilde{P}_1-\tilde{P}_2$. We first check that $\operatorname{N}(\tilde{P}) - \tilde{P} \in S\operatorname{Diff}_{\mathrm{leC}}^{2,0,-2,-1-\delta,-3-2\delta}(X)$. 
	\begin{itemize}
		\item By \cref{eq:misc_9ko},  $\tilde{P}_1+\tilde{P}_2 \in S\operatorname{Diff}_{\mathrm{leC}}^{2,0,-2,-1-\delta,-3-2\delta}(X)$, and by \cref{eq:misc_58a} the same holds for $\tilde{P}_0 -L$, so it suffices to check that 
		\begin{equation} 
			\operatorname{N}(\tilde{P}) - L \in  \operatorname{Diff}_{\mathrm{leC}}^{2,0,-2,-1-\delta,-3-2\delta}(X). 
		\end{equation} 
		Indeed, by \Cref{prop:Lcomp}, 
		\begin{align}
			\begin{split} 
				\operatorname{N}(\tilde{P}) - L &\in \operatorname{Diff}_{\mathrm{leC}}^{1,-2,-5,-3,-6}(X) + x \operatorname{Diff}_{\mathrm{leC}}^{1,0,-2,-1,-3}(X) + \sigma^2 x /(\sigma^2+\mathsf{Z}x) \operatorname{Diff}_{\mathrm{leC}}^{1,0,-2,-1,-3}(X)\\
				&\subseteq \operatorname{Diff}_{\mathrm{leC}}^{1,-2,-5,-3,-6}(X) +  \operatorname{Diff}_{\mathrm{leC}}^{1,-1,-4,-2,-5}(X) +   \operatorname{Diff}_{\mathrm{leC}}^{1,-1,-4,-2,-5}(X) \\
				&= \operatorname{Diff}_{\mathrm{leC}}^{1,-1,-4,-2,-5}(X)  \\ 
				&\subset \operatorname{Diff}_{\mathrm{leC}}^{2,0,-2,-1-\delta,-3-2\delta}(X). 	
				\label{eq:misc_60x}
			\end{split} 
		\end{align}
	\end{itemize}
	If $P_1,P_2$ are classical to orders $\beta_1,(\beta_2,\beta_3)$, then we instead get \cref{eq:misc_nep}.
\end{proof}

We now consider the $L^2_{\mathrm{sc}}(X) = L^2(X,g_0)$-based adjoint $\tilde{P}^*$, defined such that 
\begin{equation}
	\int_X f^* \tilde{P} g \dd \mathrm{Vol}_{g_0} = \int_X(\tilde{P}^* f)^* g \dd \mathrm{Vol}_{g_0} 
\end{equation}
for all $f,g\in \calS(X)$. 
This is a family of differential operators and, by \Cref{prop:diff->PsiDO}, an element of  $S\operatorname{Diff}_{\mathrm{leC}}^{2,0,-2,-1,-3}(X)\subset \Psi_{\mathrm{leC}}^{2,0,-2,-1,-3}(X)$. We form the differential operators 
\begin{equation}
	\Re \tilde{P} = \frac{1}{2} ( \tilde{P}+\tilde{P}^*),\qquad 
	\Im \tilde{P} = \frac{1}{2i} (\tilde{P} - \tilde{P}^*) ,
\end{equation}
the self-adjoint and anti- self-adjoint parts of $\tilde{P}$. 
\begin{proposition} 
	\label{prop:imaginary_comp}
	For any exactly conic metric $g_0$, there exists a differential operator $R=R_{g_0} \in S \operatorname{Diff}_{\mathrm{leC}}^{2,-1-\delta,-3-2\delta,-1-\delta,-3-2\delta}(X)$ such that  
	\begin{equation} 
		\tilde{P}^*= \tilde{P}+ (P_1^*-P_1) + \sum_{j=1}^J (\Upsilon^*_j- \Upsilon_j) + R,
		\label{eq:misc_msc}
	\end{equation}
	and, near $\partial X$, 
	\begin{align}
		\Upsilon^*_j &\in x^2 \varrho_{\mathrm{tf}} S^0(X^{\mathrm{sp}}_{\mathrm{res}}) \operatorname{Diff}^1(\partial X) \subset S\operatorname{Diff}_{\mathrm{leC}}^{1,-1,-4,-2,-5}(X) \label{eq:misc_1mn} \\
		P_1^* &= \textstyle{\sum_{j=1}^J }\big[- x^2  P_{\perp,j}^*( b^*_jx^2 \partial_x + xb_{*,j})+ x^3 b_j^{\prime *}P_{\partial X,j}^* + x^2 b^{\prime\prime *}_j Q^*_{\partial X,j}\big]  \in S\operatorname{Diff}_{\mathrm{leC}}^{2,-1,-3,-2,-4}(X)
		\label{eq:misc_uwu}
	\end{align}
	for some $b_{*,1},\cdots,  b_{*,J}\in S^0(X)$.
\end{proposition}
\begin{proof}
	It clearly suffices to restrict attention to $\hat{X}=[0,\bar{x})_x\times \partial X$, that is to compute the formal adjoint of $\tilde{P}$ with respect to 
	\begin{equation}
		L^2_{\mathrm{sc}}(\hat{X}) = L^2([0,\bar{x})_x\times \partial X_y, x^{-(n+1)} \dd x\! \dd \mathrm{Vol}_{g_{\partial X}}(y) )
	\end{equation}
	up to the required order. 
	
	Begin with $\tilde{P}_0$, which we rewrite as 
	\begin{equation}
		\tilde{P}_0 = -(1+ x a_{00}) \Big( x^2 \partial_x - \frac{x(n-1)}{2} \Big)^2 + x^2 \triangle_{\partial X} - x^4(n-1) a_{00} \partial_x + L(\sigma) + W
	\end{equation}	
	for $W \in x^2 C^\infty(X^{\mathrm{sp}}_{\mathrm{res}})$.
	\begin{itemize}
		\item We see that $x^2 \triangle_{\partial X}=x^2 \triangle_{g_{\partial X}}$ is formally self-adjoint on $L^2_{\mathrm{sc}}([0,\bar{x})\times \partial X)$, and 
		\item the adjoint of $W$ is its complex conjugate $W^* \in x^2 C^\infty(X^{\mathrm{sp}}_{\mathrm{res}})$. 
		\item We also have $a_{00} x^4 \partial_x \in S\operatorname{Diff}_{\mathrm{leC}}^{1,-2,-5,-3,-6}(X)$, thus $(a_{00} x^4 \partial_x)^* \in S\operatorname{Diff}_{\mathrm{leC}}^{1,-2,-5,-3,-6}(X)$. 
		\item 
		On the other hand, $x^2 \partial_x - x(n-1)/2$ is formally anti- self-adjoint on $L^2_{\mathrm{sc}}([0,\bar{x})\times \partial X)$, so 
		\begin{align*}
			\Big[(1+ x a_{00}) \Big( x^2 \partial_x - \frac{x(n-1)}{2} \Big)^2 \Big]^* &= \Big( x^2 \partial_x - \frac{x(n-1)}{2} \Big)^2 (1+ x a_{00})\\ 
			&= (1+ x a_{00}) \Big( x^2 \partial_x - \frac{x(n-1)}{2} \Big)^2 + a_{00} \Big[\Big( x^2 \partial_x - \frac{x(n-1)}{2} \Big)^2, x \Big] \\ 
			&= (1+ x a_{00}) \Big( x^2 \partial_x - \frac{x(n-1)}{2} \Big)^2 +2 a_{00} x^4 \partial_x + (3-n) a_{00} x^3 . 
		\end{align*}
		\item 
		By the same computation opening this subsection, $\operatorname{N}(\tilde{P})$ is formally self-adjoint. So, by \Cref{prop:Lcomp}, 
		\begin{align}
			\begin{split} 
				L^* &= L + \Big[\operatorname{N}(\tilde{P}), (1+x a_{00}) \Big(1 - \frac{\sigma^2 a_{00} x}{2(\sigma^2+\mathsf{Z} x)}\Big) \Big] \bmod \operatorname{Diff}_{\mathrm{leC}}^{1,-2,-5,-3,-6}(X) \\
				&= L + 2ix \sqrt{\sigma^2+\mathsf{Z}x}\Big[  x\partial_x, (1+x a_{00}) \Big(1 - \frac{\sigma^2 a_{00} x}{2(\sigma^2+\mathsf{Z} x)}\Big) \Big] \bmod \operatorname{Diff}_{\mathrm{leC}}^{1,-2,-5,-3,-6}(X) \\
				&= L \bmod  \operatorname{Diff}_{\mathrm{leC}}^{1,-2,-5,-2,-5}(X) = L \bmod  \operatorname{Diff}_{\mathrm{leC}}^{1,-2,-4,-2,-4}(X).
			\end{split} 
		\end{align} 
	\end{itemize}
	So, 
	\begin{equation} 
		\tilde{P}_0^* = \tilde{P}_0 \bmod S\operatorname{Diff}_{\mathrm{leC}}^{1,-2,-4,-2,-4}(X).
		\label{eq:misc_876}
	\end{equation}  
	On the other hand, we trivially have from \Cref{prop:P2tilde_comp} that 
	\begin{equation} 
		\tilde{P}_2^* \in  S\operatorname{Diff}_{\mathrm{leC}}^{2,-1-\delta,-3-2\delta,-1-\delta,-3-2\delta}(X).
		\label{eq:misc_877}
	\end{equation}  

	Now define $R = \tilde{P}^* - \tilde{P}- ( P_1^* - P_1) - \sum_{j=1}^J (\Upsilon^*_j- \Upsilon_j)$, so that \cref{eq:misc_msc} holds by construction. Rearranging this definition,
	\begin{align}
		\begin{split} 
		R &= (\tilde{P}_0^* - \tilde{P}_0) + (\tilde{P}_1^* - \tilde{P}_1)+(\tilde{P}_2^*-\tilde{P}_2)- ( P_1^* - P_1) - \sum_{j=1}^J (\Upsilon^*_j- \Upsilon_j) \\
		&= (\tilde{P}_0^* - \tilde{P}_0) + (\tilde{P}_2^*-\tilde{P}_2) + \Big(\Big[ \tilde{P}_1 - P_1 - \sum_{j=1}^J \Upsilon_j \Big]^* - \Big[ \tilde{P}_1 - P_1 - \sum_{j=1}^J \Upsilon_j \Big] \Big) .
		\end{split}
	\end{align}
	Thus, using \cref{eq:misc_876}, \cref{eq:misc_877}, \Cref{prop:P2tilde_comp}, and \Cref{prop:P1tilde_comp_new}, 
	\begin{multline}
		R\in S\operatorname{Diff}_{\mathrm{leC}}^{1,-2,-4,-2,-4}(X) + S\operatorname{Diff}_{\mathrm{leC}}^{2,-1-\delta,-3-2\delta,-1-\delta,-3-2\delta}(X) \\
		= S\operatorname{Diff}_{\mathrm{leC}}^{2,-1-\delta,-3-2\delta,-1-\delta,-3-2\delta}(X), 
	\end{multline}
	as claimed.

	\Cref{eq:misc_1mn} follows from the observation that  $\Upsilon_j^* = - i x^4 \Phi' b_j^* P_{\perp,j}^*$ near $\partial X$, where $P_{\perp,j}^*$ is computed using the $L^2(\partial X,g_{\partial X})$-inner product. On the other hand, 
	\begin{align}
		\begin{split} 
		P_1^* &= \sum_{j=1}^J \Big[ P_{\perp,j}^* (x^2 b_j^* (x^2\partial_x)^*+[(x^2\partial_x)^* ,x^2 b^*_j]) + x^3 b^{\prime *}_jP_{\partial X,j}^* + x^2 b^{\prime\prime*}_j Q_{\partial X,j}^* \Big] \\
		&=  \sum_{j=1}^J \Big[ P_{\perp,j}^* (x^2 b^*_j (-x^2\partial_x + x (n-1))-x[x\partial_x ,x^2 b^*_j]) + x^3 b^{\prime *}_jP_{\partial X,j}^* + x^2 b^{\prime\prime*}_j Q_{\partial X,j}^* \Big] \\
		&=  \sum_{j=1}^J \Big[ P_{\perp,j}^* (x^2 b_j^* (-x^2\partial_x + x (n-1))-2x^3 b^*_j - x^3 (x\partial_x b^*_j)) + x^3 b^{\prime *}_jP_{\partial X,j}^* + x^2 b^{\prime\prime*}_j Q_{\partial X,j}^* \Big]
		\end{split}
	\end{align}
	near $\partial X$, where $b_j,b_j',b_j''$ are as in \cref{eq:misc_000}.
	\Cref{eq:misc_uwu} follows from this, for some choice of $b_{*,j}$. 
\end{proof}
We see from the above that 
\begin{multline}
	\Im \tilde{P} \in S \operatorname{Diff}_{\mathrm{leC}}^{1,-1,-4,-2,-5}(X) + S \operatorname{Diff}_{\mathrm{leC}}^{2,-1,-3,-2,-4}(X) + S \operatorname{Diff}_{\mathrm{leC}}^{2,-1-\delta,-3-2\delta,-1-\delta,-3-2\delta}(X) \\
	\subset S \operatorname{Diff}_{\mathrm{leC}}^{2,-1,-3,-1-\delta,-3-2\delta}(X).
\end{multline}
Thus, $\Im \tilde{P}$ is one order lower than $ \tilde{P}$ at sf and ff and slightly lower order at bf and tf, and we have been entirely explicit about the leading terms of $\Im \tilde{P}$ (namely $\Im P_1$ and $\Im \Upsilon_j$) at sf and ff, the remainder $(2i)^{-1} R$ being slightly more than one order lower than $\tilde{P}$ at both faces.

\section{The situation at zero energy}
\label{sec:0_operator}

In this section we consider $P(0)$ and $\tilde{P}(0)$ in some detail.
Specifically, we apply \cite[Theorem 1.1]{VasyLA} to study the strong limit 
\begin{equation} 
	R(E=0; \mathsf{Z}\pm i 0) = \operatorname{slim}_{\epsilon\to 0^+} R(E=0; \mathsf{Z}\pm i \epsilon)
\end{equation} 
used in the statement of \Cref{thm:main} to characterize the resolvent output at zero energy. The mapping properties of this operator will be used in \S\ref{sec:mainproof} in order to prove the smoothness of the output of the conjugated resolvent at positive energy  all the way down to zero energy (as used e.g. in \Cref{cor:main1}).

Recall that $x_{1/2}=2^{-1/2} x^{1/2}$. Then, from the form \cref{eq:P0} of $P_0$,  
\begin{align}
	\begin{split} 
	P_{0}(0) &= -(1+2x_{\frac{1}{2}}^2a_{00}(0)) (x_{\frac{1}{2}}^3 \partial_{x_{\frac{1}{2}}} )^2 + 4x_{\frac{1}{2}}^4 \triangle_{\partial X} + 2x_{\frac{1}{2}}^5  [ a(0)+n-1] \partial_{x_{\frac{1}{2}}} - 2\mathsf{Z} x_{\frac{1}{2}}^2 \\
	&= - x_{\frac{1}{2}}^2 (1+2x_{\frac{1}{2}}^2 a_{00}(0)) (x_{\frac{1}{2}}^2 \partial_{x_{\frac{1}{2}}} )^2 +4 x_{\frac{1}{2}}^4 \triangle_{\partial X} + 2x_{\frac{1}{2}}^5 \Big[ a(0)+n- \frac{3}{2} - x_{\frac{1}{2}}^2 a_{00}(0)\Big] \partial_{x_{\frac{1}{2}}} - 2\mathsf{Z} x_{\frac{1}{2}}^2.
	\end{split}
	\label{eq:misc_kuy}
\end{align} 
(The extra $-x_{1/2}^3\partial_{x_{1/2}}$ in \cref{eq:misc_kuy} is the source of the $\smash{(\sigma^2+\mathsf{Z} x)^{-1/4}}$ term in \cref{eq:u_decomp}.) Therefore 
\begin{multline} 
	x_{1/2}^{-1-n/2} P_{0}(0) x_{1/2}^{(n-2)/2} = - (1+ 2x_{1/2}^2 a_{00}(0)) (x_{\frac{1}{2}}^2 \partial_{x_{\frac{1}{2}}})^2 + 4x_{\frac{1}{2}}^2 \triangle_{\partial X} - 2\mathsf{Z} \\  + x_{\frac{1}{2}}^3 \Big[ 2a(0)  + n-1 - 2 (n-1) x_{\frac{1}{2}}^2 a_{00}(0) \Big] \partial_{x_{\frac{1}{2}}}     +(n-2)\Big[ a(0) + \frac{3n}{4} - \frac{3}{2} - \frac{n+2}{2} x_{\frac{1}{2}}^2 a_{00}(0) \Big] x_{\frac{1}{2}}^2.  
	\label{eq:misc_kyy}
\end{multline}
Thus,  $\smash{x_{1/2}^{-1-n/2} \tilde{P}_0(0)x_{1/2}^{-1+n/2}  \in \operatorname{Diff}_{\mathrm{b}}(X_{1/2})}$
has the same form as the conjugated spectral family at positive energy, except with respect to $x_{1/2}$ instead of $x$ (and with an extra short-range potential in \cref{eq:misc_kyy}). 

We generalize this observation:
\begin{proposition}
	\label{prop:0energy1}
	If $g$ is an asymptotically conic metric on $X$, and if we set $g_{1/2}  = x_{1/2}^{2} g$, then there exists some $V_{\mathrm{eff}} \in x^2 S^0(X)$ such that 
	\begin{equation}
		x_{1/2}^{1-n/2} \triangle_{g} x_{1/2}^{-1+n/2} = x_{1/2}^2 \triangle_{g_{1/2}} + V_{\mathrm{eff}}. 
	\end{equation}
	holds. 
\end{proposition}
\begin{proof}
	We first want to show that $x_{1/2}^{1-n/2} \triangle_{g} x_{1/2}^{-1+n/2} - x_{1/2}^2 \triangle_{g_{1/2}}$ is zeroth order (and therefore a function on $X^\circ$). Indeed, $\triangle_g = x_{1/2}^2 \triangle_{g_{1/2}} + (n-2) x_{1/2} \nabla_{g_{1/2}} x_{1/2}$, so 
	\begin{equation}
		x_{1/2}^{1-n/2} \triangle_{g} x_{1/2}^{-1+n/2} = x_{1/2}^2 \triangle_{g_{1/2}} + x_{1/2}^{3-n/2} \triangle_{g_{1/2}} x_{1/2}^{-1+n/2} + (n-2) x_{1/2} g_{1/2}(\mathrm{d} x_{1/2} , \mathrm{d} x_{1/2}^{(n-2)/2}). 
	\end{equation}
	It is therefore the case that $x_{1/2}^{1-n/2} \triangle_{g} x_{1/2}^{-1+n/2} - x_{1/2}^2 \triangle_{g_{1/2}} = V_{\mathrm{eff}}$ for 
	\begin{equation}
		V_{\mathrm{eff}} = x_{1/2}^{3-n/2} \triangle_{g_{1/2}} x_{1/2}^{-1+n/2} + (n-2) x_{1/2} g_{1/2}(\mathrm{d} x_{1/2} , \mathrm{d} x_{1/2}^{(n-2)/2}). 
		\label{eq:misc_vef}
	\end{equation}
	Since $g_{1/2}$ is an asymptotically conic metric on $X_{1/2}$ (see below), the fact that $V_{\mathrm{eff}} \in x^2 S^0(X)$ can be read off \cref{eq:misc_vef}, but we check in local coordinates. 
	
	It suffices to restrict attention to a neighborhood of $\partial X$. Let $y_1,\ldots,y_{n-1}$ denote a local system of coordinates on $\partial X$. Then, using $g=x^{-2}_{1/2} g_{1/2}$, 
	\begin{align}
		\begin{split} 
		\triangle_g &= -\frac{x^{n}_{1/2}}{ |g_{1/2}|^{1/2} } \partial_i (x_{1/2}^{2-n} |g_{1/2}|^{1/2} g_{1/2}^{ij} \partial_j) \\
		&=  -\frac{x^{2}_{1/2}}{ |g_{1/2}|^{1/2} } \partial_i ( |g_{1/2}|^{1/2} g_{1/2}^{ij} \partial_j) + (n-2) x_{1/2} g_{1/2}^{0i}\partial_i,
		\end{split} 
		\label{eq:misc_rh5}
	\end{align}
	where $\partial_0 = \partial_{x_{1/2}}$ and $\partial_i = \partial_{y_i}$ for $i=1,\ldots,n-1$. 
	
	Conjugating the right-hand side of \cref{eq:misc_rh5} by $x_{1/2}^{(n-2)/2}$, we see that 
	\begin{equation}
		x_{1/2}^{1-n/2} \triangle_{g} x_{1/2}^{-1+n/2}
		= x_{1/2}^2 \triangle_{g_{1/2}} -  \frac{x_{1/2}}{|g_{1/2}|^{1/2}} \frac{n-2}{2} \partial_i(|g_{1/2}|^{1/2} g_{1/2}^{i0}) + 4^{-1}n(n-2) g_{1/2}^{00}.
	\end{equation}
	Since $g_{1/2}^{00} \in x_{1/2}^4 S^0(X_{1/2}) = x^2 S^0(x)$, $ 4^{-1}n(n-2) g_{1/2}^{00}\in x_{1/2}^4 S^0(X_{1/2})$. 
	
	Likewise, we see that $|g_{1/2}|^{-1/2} \partial_0 |g_{1/2}|^{1/2} \in x_{1/2}^{-1} S^0(X_{1/2})$, $|g_{1/2}|^{-1/2} \partial_i |g_{1/2}|^{1/2} \in  S^0(X_{1/2})$ for $i\neq 0$, and 
	\begin{equation} 
		\partial_0 g^{00}_{1/2},\partial_i g^{i0}_{1/2}  \in x_{1/2}^3 S^0(X_{1/2})
	\end{equation} 
	for $i\neq 0$. So,
	\begin{equation}
		V_{\mathrm{eff}} = -  \frac{x_{1/2}}{|g_{1/2}|^{1/2}} \frac{n-2}{2} \partial_i(|g_{1/2}|^{1/2} g_{1/2}^{i0}) + 4^{-1}n(n-2) g_{1/2}^{00}\in x_{1/2}^4 S^0(X_{1/2}) =  x^2 S^0(X) . 
	\end{equation} 
\end{proof}

Observe that if $g$ is an asymptotically conic metric on $X$, then $g_{1/2}$ is an asymptotically conic metric on $X_{1/2}$. Indeed, $g_{1/2}$ is certainly a Riemannian metric on $\smash{X^\circ=X_{1/2}^\circ}$, and it is a sum of $\smash{x_{1/2}^2 g_0}$, which is exactly conic on $X_{1/2}$, and terms in 
\begin{align}
	x^2 C^\infty(X; {}^{\mathrm{sc}}\!\operatorname{Sym}^2 T^* X) &\subset x_{1/2}^2C^\infty(X; {}^{\mathrm{sc}}\! \operatorname{Sym}^2 T^* X_{1/2}), \\
	x^{2+\delta}S^0(X; {}^{\mathrm{sc}}\!\operatorname{Sym}^2 T^* X) &\subset x_{1/2}^{2+2\delta}S^0(X; {}^{\mathrm{sc}}\! \operatorname{Sym}^2 T^* X_{1/2}).
\end{align} 
Consequently, by \Cref{prop:0energy1}, if $P$ is the spectral family of an attractive Coulomb-like Schr\"odinger operator, then 
\begin{equation} 
	x^{-1-n/2}_{1/2} P(0) x^{-1+n/2}_{1/2} = x_{1/2}^{-2} (x^{-1+n/2}_{1/2})^{-1} P(0) x^{-1+n/2}_{1/2}
\end{equation} 
is a member $P_{\mathrm{zero}}(2\mathsf{Z})=\triangle_{g_{1/2}} - 2\mathsf{Z} + W$ of the spectral family $\{P_{\mathrm{zero}}(\zeta)=P_{\mathrm{zero}}(0)-\zeta\}_{\zeta\geq 0}$ of a Schr\"odinger operator 
\begin{equation}
	P_{\mathrm{zero}}(0)=\triangle_{g_{1/2}} + W
\end{equation}
on $X_{1/2}$, where the potential $W \in x S^0(X) = x_{1/2}^2 S^0(X_{1/2})$ is short-range.

Thus, $P_{\mathrm{zero}}(2\mathsf{Z})$ satisfies the hypotheses of \cite[\S3]{VasyLA}, with $\smash{2^{1/2}\mathsf{Z}^{1/2}}$ in place of $\sigma$ and $X_{1/2}$ in place of $X$. 
Moreover, as seen from \cref{eq:osc_0} with $a=0$, the phase $\Phi(-;0)$ is just that used by Vasy's in his conjugation. Thus, $\tilde{P}(0)$ has the form of Vasy's conjugated operator (with $2^{1/2}\mathsf{Z}^{1/2}$ in place of $\sigma$ and $X_{1/2}$ in place of $X$).
In order to denote the $\mathsf{Z}$ dependence of $P(0)$ and $\tilde{P}(0)$, we write 
\begin{equation} 
	P(0)=P(0;\mathsf{Z}) = x_{1/2}^{1+n/2}P_{\mathrm{zero}}(2\mathsf{Z}) x_{1/2}^{1-n/2}
\end{equation} 
and $\tilde{P}(0)= \tilde{P}(0;\mathsf{Z})$. For $\epsilon>0$, let $P(0;\mathsf{Z}+i\epsilon)$ denote $P(0)$ with $\mathsf{Z}$ replaced by $\mathsf{Z}+i\epsilon$. 
Since $\calS(X_{1/2})=\calS(X)$ and $\calS'(X_{1/2})=\calS'(X)$, the limiting absorption principle (as in \cite{MelroseSC}) applies in the following form: 
	for $\epsilon>0$, the resolvent
	\begin{equation} 
		R_0(2\mathsf{Z}+2i\epsilon)= x_{1/2}^{-1-n/2} R(0;\mathsf{Z}+i\epsilon) x_{1/2}^{(n-2)/2} : \calS(X)\to \calS'(X) 
	\end{equation} 
	of $P_{\mathrm{zero}}(0)$ (evaluated at ``energy'' $\zeta = 2 \mathsf{Z}+2i\epsilon$) -- 
	defined e.g. via the functional calculus --
	admits a strong limit $\smash{x_{1/2}^{-1-n/2} R(0;\mathsf{Z}+i0) x_{1/2}^{-1+n/2}}: \calS(X)\to \calS'(X)$. Using \cite[Theorem 1.1]{VasyLA}, we can construct this resolvent as a map between suitable Sobolev spaces using the conjugated perspective:

\begin{proposition}
	\label{prop:LA_at_0}
	If $P(\sigma)$ is the spectral family of an attractive Coulomb-like Schr\"odinger operator on $X$, then, for any $m,\varsigma,\ell \in \bbR$ satisfying $\ell<-3/2<\varsigma$, 
	\begin{multline} 
		\tilde{P}(0;\mathsf{Z}) :\{ u \in H_{\mathrm{scb}}^{m,\varsigma+n/2,\ell+n/2}(X_{1/2}) :  \tilde{P}(0)   u \in H_{\mathrm{scb}}^{m-2,\varsigma+3+n/2, \ell+3+n/2}(X_{1/2}) \}   \\ \to H_{\mathrm{scb}}^{m-2,\varsigma+3+n/2, \ell+3+n/2}(X_{1/2})
	\end{multline} 
	is invertible, and the inverse 
	\begin{equation} 
		\tilde{R}_+(0;\mathsf{Z}):H_{\mathrm{scb}}^{m-2,\varsigma+3+n/2, \ell+3+n/2}(X_{1/2}) \to H_{\mathrm{scb}}^{m,\varsigma+n/2,\ell+n/2}(X_{1/2})
	\end{equation} 
	is related to $R(0;\mathsf{Z}+i0)$ by the formula $e^{+i\Phi(-;0)} \tilde{R}_+(0;\mathsf{Z}) e^{-i\Phi(-;0)} f = R(0;\mathsf{Z}+i0) f$, which holds for all $f\in \calS(X)$. 
\end{proposition}
\begin{proof}
	Suppose that we are given $m,\varsigma_{1/2},\ell_{1/2}\in \bbR$ with $\ell_{1/2}<-1/2<\varsigma_{1/2}$. Then, combining  the observations above and \cite[Theorem 1.1]{VasyLA}, 
	\begin{multline}
		e^{-i\Phi(-;0)} P_{\mathrm{zero}}(2\mathsf{Z}) e^{+i\Phi(-;0)} = x_{1/2}^{-1-n/2} \tilde{P}(0;\mathsf{Z})x_{1/2}^{-1+n/2}  : \\ \{ u \in H_{\mathrm{scb}}^{m,\varsigma_{1/2},\ell_{1/2}}(X_{1/2}) : x_{1/2}^{-1-n/2} \tilde{P}(0;\mathsf{Z})x_{1/2}^{-1+n/2}  u \in H_{\mathrm{scb}}^{m-2,\varsigma_{1/2}+1, \ell_{1/2}+1}(X_{1/2}) \} \\ \to H_{\mathrm{scb}}^{m-2,\varsigma_{1/2}+1, \ell_{1/2}+1}(X_{1/2})
		\label{eq:misc_tpo}
	\end{multline}
	is invertible, defining a continuous linear map 
	\begin{equation}
		\tilde{R}_0(\mathsf{Z}): H_{\mathrm{scb}}^{m-2,\varsigma_{1/2}+1, \ell_{1/2}+1}(X_{1/2}) \to H_{\mathrm{scb}}^{m,\varsigma_{1/2},\ell_{1/2}}(X_{1/2}), 
	\end{equation}
	and this is related to the limiting resolvent $R_0(2\mathsf{Z}+i0)$ of the spectral family $P_{\mathrm{zero}} = \{P_{\mathrm{zero}}(0)-\zeta\}_{\zeta\geq 0}$ of the Schr\"odinger operator $P_{\mathrm{zero}}(0)$ on $X_{1/2}$ by 
	\begin{equation} 
		e^{+i\Phi(-;0)} \tilde{R}_0(0;\mathsf{Z}) e^{-i\Phi(-;0)} f = R_0(2\mathsf{Z} +i0) f =  x_{1/2}^{-1-n/2} R(0;\mathsf{Z}+i0) x_{1/2}^{-1+n/2} f, 
	\end{equation} 	
	which holds for all $f\in \calS(X)$.

	Call the ($\mathsf{Z}$-dependent) domain and codomain of \cref{eq:misc_tpo} 
	\begin{align*}
	\tilde{\calX}_{m,\varsigma_{1/2},\ell_{1/2}} &= \{ u \in H_{\mathrm{scb}}^{m,\varsigma_{1/2},\ell_{1/2}}(X_{1/2}) : x_{1/2}^{-1-n/2} \tilde{P}(0;\mathsf{Z})x_{1/2}^{-1+n/2}  u \in H_{\mathrm{scb}}^{m-2,\varsigma_{1/2}+1, \ell_{1/2}+1}(X_{1/2}) \}  \\ 
	&= \{ u \in H_{\mathrm{scb}}^{m,\varsigma_{1/2},\ell_{1/2}}(X_{1/2}) : \tilde{P}(0;\mathsf{Z})x_{1/2}^{-1+n/2}  u \in H_{\mathrm{scb}}^{m-2,\varsigma_{1/2}+2+n/2, \ell_{1/2}+2+n/2}(X_{1/2}) \}  \\
	\tilde{\calY}_{m,\varsigma_{1/2},\ell_{1/2}} &=  H_{\mathrm{scb}}^{m-2,\varsigma_{1/2}+1, \ell_{1/2}+1}(X_{1/2}). 
	\end{align*} 
	Setting 
	\begin{align}
		\begin{split} 
		\calX_{m,\varsigma,\ell} &=  x^{-1+n/2}_{1/2} \tilde{\calX}_{m,\varsigma_{1/2},\ell_{1/2}} \\ 
		&= \{ u \in H_{\mathrm{scb}}^{m,\varsigma+n/2,\ell+n/2}(X_{1/2}) :  \tilde{P}(0;\mathsf{Z})   u \in H_{\mathrm{scb}}^{m-2,\varsigma+3+n/2, \ell+3+n/2}(X_{1/2}) \} 
		\end{split} \label{eq:misc_x12} \\
	\calY_{m,\varsigma,\ell} &= H_{\mathrm{scb}}^{m-2,\varsigma+3+n/2, \ell+3+n/2}(X_{1/2}) \label{eq:misc_y12}
	\end{align}
	for $\varsigma = \varsigma_{1/2}-1$ and $\ell = \ell_{1/2}-1$, 
	we have a commutative diagram 
	\begin{equation}
		\xymatrix{
			\tilde{\calX}_{m,\varsigma_{1/2},\ell_{1/2}} \ar[rrr]^{x_{1/2}^{-1-n/2} \tilde{P}(0)x_{1/2}^{-1+n/2} } \ar[d]_{\times x_{1/2}^{-1+n/2}} &&& \tilde{\calY}_{m,\varsigma_{1/2},\ell_{1/2}} \ar[d]^{\times x_{1/2}^{1+n/2}} \\
			\calX_{m,\varsigma,\ell} \ar[rrr]^{\tilde{P}(0)} &&& \calY_{m,\varsigma,\ell} 
		}
	\end{equation}
	in the category of Banach spaces. The vertical arrows are manifestly isomorphisms, and as observed the top horizontal arrow is as well. 
	
	Hence, $\tilde{P}(0)=\tilde{P}(0;\mathsf{Z})$ is invertible, and the inverse $\tilde{R}_+(0)=\tilde{R}_+(0;\mathsf{Z})$ has the properties specified in the proposition. 
\end{proof}
And for the b-Sobolev spaces:
\begin{proposition}
	\label{prop:LA_at_0b} 
	If $P(\sigma)$ is the spectral family of an attractive Coulomb-like Schr\"odinger operator on $X$, for any $m,\ell=2l \in \bbR$ satisfying $\ell<-3/2<m+\ell$, 
	\begin{align} 
		\begin{split} 
		\tilde{P}(0;\mathsf{Z}) &:\{ u \in H_{\mathrm{b}}^{m,\ell+n/2}(X_{1/2}) :  \tilde{P}(0;\mathsf{Z})   u \in H_{\mathrm{b}}^{m, \ell+3+n/2}(X_{1/2}) \}    \to H_{\mathrm{b}}^{m, \ell+3+n/2}(X_{1/2}) \\
		&:\{ u \in H_{\mathrm{b}}^{m,l}(X) :  \tilde{P}(0;\mathsf{Z})   u \in H_{\mathrm{b}}^{m, l+3/2}(X) \}    \to H_{\mathrm{b}}^{m, l+3/2}(X)
		\end{split} 
		\label{eq:misc_o12}
	\end{align} 
	is invertible. 
\end{proposition}
\begin{proof}
	The equality 
	\begin{equation}
		\{ u \in H_{\mathrm{b}}^{m,\ell+n/2}(X_{1/2}) :  \tilde{P}(0;\mathsf{Z})   u \in H_{\mathrm{b}}^{m, \ell+3+n/2}(X_{1/2}) \} = \{ u \in H_{\mathrm{b}}^{m,l}(X) :  \tilde{P}(0;\mathsf{Z})   u \in H_{\mathrm{b}}^{m, l+3/2}(X) \} 
	\end{equation}
	follows from \Cref{lem:X12_conv}. 
	
	First of all, setting $\varsigma = m+\ell$, 
	\begin{equation}
		\{ u \in H_{\mathrm{b}}^{m,\ell+n/2}(X_{1/2}) :  \tilde{P}(0;\mathsf{Z})   u \in H_{\mathrm{b}}^{m, \ell+3+n/2}(X_{1/2}) \} \subset \calX_{m,\varsigma,\ell}, 
	\end{equation}
	and so \cref{eq:misc_o12} is injective (by \Cref{prop:LA_at_0}). Conversely, 
	\begin{equation}
		H_{\mathrm{b}}^{m, l+3/2}(X) \subset \tilde{\calY}_{m,\varsigma,\ell}, 
	\end{equation}
	so given any $f\in \smash{H_{\mathrm{b}}^{m, l+3/2}(X)}$ there exists a $u\in \tilde{\calX}_{m,\varsigma,\ell}$ such that $\tilde{P}(0;\mathsf{Z})u=f$. Thus, $u\in \smash{H_{\mathrm{b}}^{m,\ell+n/2}(X_{1/2})}$, and we already know that $f = \tilde{P}(0;\mathsf{Z}) u\in \smash{H_{\mathrm{b}}^{m, \ell+3+n/2}(X_{1/2})}$, so $u$ is actually in the codomain of \cref{eq:misc_o12}. Thus, \cref{eq:misc_o12} is surjective. 
\end{proof}

\section{Symbolic estimates}
\label{sec:symbolic}

We now proceed to establish quantitative control of $u \in \calS'(X)$ in terms of $\tilde{P} u$ microlocally in the symbolic region of the leC-phase space $\smash{{}^{\mathrm{leC}}\overline{T}^* X}$, meaning at $\mathrm{df}\cup \mathrm{sf} \cup \mathrm{ff}$. 
Since we make no attempt to be uniform in the $\sigma\to\infty$ limit, we simply restrict attention to $\sigma \in [0,\Sigma]$ for some arbitrary $\Sigma>0$, and the estimates will all depend on $\Sigma$ in some unexamined way. The main result of this section, duplicated below as \Cref{prop:symbolic_final}, says:
\begin{itemize}
	\item for every $\Sigma>0$, $N\in \bbN$, and $m,s,\varsigma,l, \ell,s_0,\varsigma_0 \in \bbR$ satisfying  $l<-1/2<s_0<s$ and $\ell<-3/2<\varsigma_0<\varsigma\leq \ell+s-l$, there exists a constant $C=C(\tilde{P},\Sigma,N,m,s,\varsigma,l,\ell)>0$ such that
	\begin{equation}
		\lVert u \rVert_{H_{\mathrm{leC}}^{m,s,\varsigma,l,\ell} } \leq C( \lVert \tilde{P} u \rVert_{H_{\mathrm{leC}}^{m-2,s+1,\varsigma+3,l+1,\ell+3}} + \lVert u \rVert_{H_{\mathrm{b,leC}}^{-N,l,\ell}} ) 
		\label{eq:misc_gg1}
	\end{equation}
	holds for all $u\in \calS'(X)$ and $\sigma \in [0,\Sigma]$ such that $\lVert u  \rVert_{H_{\mathrm{leC}}^{-N,s_0,\varsigma_0,-N,-N}(X)(\sigma)}<\infty$.  
\end{itemize}
(As mentioned in the introduction, \cite[\S5]{VasyLA} suffices for the analysis of the $\sigma\to\infty$ regime.) The contents of this section should be compared to the contents of \cite[\S4]{VasyLA}, as our argument below is very similar to some of the symbolic computations there. Using \Cref{lem:combination}, in order to prove \cref{eq:misc_gg1} it suffices to establish quantitative control of $u$ within each member of some finite collection of open subsets of the leC-phase space covering $\mathrm{df}\cup \mathrm{sf} \cup \mathrm{ff}$ --- see \Cref{fig:sketch}.

 By \cref{eq:misc_v9v}, 
\begin{equation}
	\tilde{P} = -(x^2 \partial_x)^2 + x^2 \triangle_{\partial X} +2i x (\sigma^2+\mathsf{Z} x)^{1/2} x \partial_x  \bmod S\operatorname{Diff}_{\mathrm{leC}}^{2,-1,-3,-1,-3}(X) 
\end{equation}
(near $\partial X$). 
Thus, the leC-principal symbol $\sigma_{\mathrm{leC}}^{2,0,-2,-1,-3}(\tilde{P}) \in S_{\mathrm{leC}}^{2,0,-2,-1,-3}(X)/S_{\mathrm{leC}}^{1,-1,-3,-1,-3}(X)$ of $\tilde{P}$ has a representative of the form 
\begin{equation}
	\tilde{p} = \tilde{p}_0 + \tilde{p}_{1,2}
\end{equation}
for $\tilde{p}_{1,2} \in S_{\mathrm{leC}}^{2,-1,-3,-1,-3}(X)$, where $\tilde{p}_0$ is a representative of $\sigma_{\mathrm{leC}}^{2,0,-2,-1,-3}(\tilde{P}_0)$. Thus, 
\begin{equation} 
	\tilde{p}_0 =  x^2 \xi_{\mathrm{b}}^2 + x^2 \eta_{\mathrm{b}}^2 -2 x(\sigma^2+\mathsf{Z} x)^{1/2} \xi_{\mathrm{b}}
\end{equation} 
near $\partial X$. 
(Recall from \S\ref{sec:calculus} that $\xi_{\mathrm{b}}$ is the b-cofiber coordinate dual to $x$ and $\eta=\eta_{\mathrm{b}} \in T^* \partial X$.) 
The ellipticity of $\tilde{P} \in \smash{S\operatorname{Diff}_{\mathrm{leC}}^{2,0,-2,-1,-3}(X)}$ (see \Cref{prop:df_ellipticity}) at and near $\mathrm{df}$ makes establishing control there trivial, so we work on $\smash{{}^{\mathrm{leC}} T^* X = {}^{\mathrm{leC}} \overline{T}^* X \backslash \mathrm{df}}$ and establish control at $\{x=0\}\subset{}^{\mathrm{leC}} T^* X $. Thus, $\tilde{p}_{1,2}$ is irrelevant for the symbolic considerations below (except those in \S\ref{subsection:ellipticity}), all of which are restricted away from $\mathrm{df}$.

In order to investigate the dynamics away from $\mathrm{bf}\cup \mathrm{tf}$, we introduce new coordinates on $\bbR^+_\sigma \times T^* X^\circ$ (over some collar neighborhood of $\partial X$, not including the boundary itself) by 
\begin{align}
	\xi_{\mathrm{sc,leC}} &= \xi_{\mathrm{b}} \varrho_{\mathrm{bf}_{00}}\varrho_{\mathrm{tf}_{00}} = \xi_{\mathrm{b}} \varrho_{\mathrm{sf}}\varrho_{\mathrm{ff}}\varrho_{\mathrm{bf}}\varrho_{\mathrm{tf}} \label{eq:xisc}\\
	\eta_{\mathrm{sc,leC}} &= \eta \varrho_{\mathrm{bf}_{00}}\varrho_{\mathrm{tf}_{00}} = \eta \varrho_{\mathrm{sf}}\varrho_{\mathrm{ff}}\varrho_{\mathrm{bf}}\varrho_{\mathrm{tf}} \in T^* \partial X,
	\label{eq:etasc}
\end{align}
for $\xi_{\mathrm{b}}\in \bbR$ and $\eta = \eta_{\mathrm{b}} \in T^* \partial X$. More explicitly, $\xi_{\mathrm{sc,leC}} = \xi_{\mathrm{b}} x (\sigma^2+\mathsf{Z}x)^{-1/2}$, $\eta_{\mathrm{sc,leC}} = \eta x (\sigma^2+\mathsf{Z}x)^{-1/2}$. These extend to fiber coordinates on the bundle $[0,\bar{x})^{\mathrm{sp}}_{\mathrm{res}} \times \bbR_{\xi_{\mathrm{sc,leC}}} \times (T^* \partial X)_{\eta_{\mathrm{sc,leC}}} \to [0,\bar{x})^{\mathrm{sp}}_{\mathrm{res}}$, and we write
\begin{multline} 
	{}^{\mathrm{sc,leC}} T^*_{\partial X} X = ( \partial [0,\bar{x})^{\mathrm{sp}}_{\mathrm{res}} \backslash \mathrm{zf}^\circ) \times \bbR_{\xi_{\mathrm{sc,leC}}} \times (T^* \partial X)_{\eta_{\mathrm{sc,leC}}} \\ = \partial (([0,\bar{x})^{\mathrm{sp}}_{\mathrm{res}}\backslash \mathrm{zf}^\circ) \times \bbR_{\xi_{\mathrm{sc,leC}}} \times (T^* \partial X)_{\eta_{\mathrm{sc,leC}}}).
	\label{eq:bdy_trv}
\end{multline} 
(We do not endow ${}^{\mathrm{sc,leC}} T^*_{\partial X} X$ with any more structure than that of a set.)
Let $o$ denote the zero section of $T^* \partial X$. 
While ${}^{\mathrm{sc,leC}} T^*_{\partial X} X$ is not a mwc,  $[0,\bar{x})^{\mathrm{sp}}_{\mathrm{res}} \times \bbR_{\xi_{\mathrm{sc,leC}}}\times T^* \partial X$ is.

\begin{figure}
	\begin{center} 
		\begin{tikzpicture}[scale=.8]
		\filldraw[fill=gray!5, dashed] (0,0) circle (3);
		\draw[fill=gray] (0,2) circle (7pt);
		\draw[fill=gray] (0,0) circle (7pt);
		\draw[red] (0,1) circle (1);
		\filldraw[red] (1,.9) -- (1.1,1.1) -- (.9,1.1) -- cycle;
		\filldraw[red] (-1,.9) -- (-1.1,1.1) -- (-.9,1.1) -- cycle;
		\filldraw[red] (0,2) circle (2pt); 
		\filldraw[red] (0,0) circle (2pt); 
		\node[red, above right] (rp) at (0,2) {$\,\calR_+$};
		\node[red, below left] (ro) at (0,0) {$\calR_0$};
		\node[red, right] (char) at (1,1) {$\;\Sigma$};
		\node[black, above] (df) at (0,3) {${}^{\mathrm{sc,leC}} T^*_{p} X $};
		\end{tikzpicture}
	\end{center}
	\caption{Schematic of the proof of \cref{eq:misc_gg1} when $\dim \partial X=1$. The situation over a point $p\in( \partial [0,\bar{x})^{\mathrm{sp}}_{\mathrm{res}} \backslash \mathrm{zf}^\circ)\times \partial X$ in the sc,leC-phase space is illustrated, with $\Sigma = \operatorname{Char}_{\mathrm{sc,leC}}^{2,0,-2}(\tilde{P})$ in red. There are four subsets of interest: the dark gray neighborhood of $\calR_+$, controlled using a high regularity radial point estimate, the dark gray neighborhood of $\calR_0=o_\partial$, controlled using a low regularity radial point estimate (and an elliptic estimate), a neighborhood of the rest of the characteristic set which stays away from $\calR_+,\calR_0$ over which we can propagate regularity, and the rest of the cofiber (light gray background), where elliptic estimates apply. 	
	Cf.\ \cite[Figure 2]{MelroseSC}.
	The direction of the Hamiltonian flow of $\tilde{P}$ is indicated with arrows. (Note that only the ``vertical'' components of the flow are drawn --- the Hamiltonian flow also changes $p$.)}
	\label{fig:sketch}
\end{figure}

The punctured space ${}^{\mathrm{sc,leC}} T^*_{\partial X} X \backslash o_\partial$, $o_\partial =\{x=0,\xi_{\mathrm{sc,leC}} = 0 \text{ and } \eta_{\mathrm{sc,leC}} \in o\}$, can be identified with $\mathrm{sf}\cup \mathrm{ff} \backslash (\mathrm{df}\cup \mathrm{bf}\cup \mathrm{tf})$:
\begin{proposition}
	The identity map $i:\{0\leq x <\bar{x}\}\cap (\bbR_\sigma^+\times T^* X^\circ)\to \{0\leq x <\bar{x}\}\cap (\bbR_\sigma^+\times T^* X^\circ)$ composed with $(D\iota)^*:\{0\leq x <\bar{x}\}\cap (\bbR_\sigma^+\times T^* X^\circ) \to \bbR_\sigma^+\times (0,\bar{x})_x\times \bbR_{\xi}\times T^* \partial X $  (where $\iota$ is the boundary collar) extends (uniquely) to a smooth map
	\begin{equation}
		\bar{i} : \{0\leq x<\bar{x}\}\cap {}^{\mathrm{leC}} T^* X \to [0,\bar{x})^{\mathrm{sp}}_{\mathrm{res}} \times \bbR_{\xi_{\mathrm{sc,leC}}}\times (T^* \partial X)_{\eta_{\mathrm{sc,leC}}}
		\label{eq:misc_iba}
	\end{equation}
	restricting to a diffeomorphism 
	$\{0\leq x<\bar{x}\}\cap {}^{\mathrm{leC}} \overline{T}^* X \backslash (\mathrm{df}\cup \mathrm{bf}\cup \mathrm{tf}) \to ([0,\bar{x})^{\mathrm{sp}}_{\mathrm{res}} \times \bbR_{\xi_{\mathrm{sc,leC}}}\times (T^* \partial X)_{\eta_{\mathrm{sc,leC}}}) \backslash o_\partial$. 
\end{proposition}
\begin{proof}
	Here we are using the boundary collar $\iota$ (really $(D\iota)^*$) to identify $\{0<x<\bar{x}\}\cap (\bbR^+_\sigma\times T^* X^\circ)$ with $\bbR_\sigma^+\times (0,\bar{x})_x\times \bbR_{\xi}\times T^* \partial X$. 
	
	Using $\iota$, we see that $\xi_{\mathrm{sc,leC}}$ is a smooth function on the domain of \cref{eq:misc_iba}, as is $\eta_{\mathrm{sc,leC}} \in T^* \partial X$. Likewise, the (smooth) projection ${}^{\mathrm{leC}}T^* X \to X^{\mathrm{sp}}_{\mathrm{res}}$ fits into a composition 
	\begin{equation}
		 \{0\leq x<\bar{x}\}\cap {}^{\mathrm{leC}}T^* X  \to  \{0\leq x<\bar{x}\}\cap X^{\mathrm{sp}}_{\mathrm{res}} \overset{\iota^{-1}}{\to} \hat{X}^{\mathrm{sp}}_{\mathrm{res}} = [0,\bar{x})^{\mathrm{sp}}_{\mathrm{res}} \times \partial X \to  [0,\bar{x})^{\mathrm{sp}}_{\mathrm{res}},
		 \label{eq:misc_k12}
	\end{equation}
	where $\hat{X} = [0,\bar{x})\times \partial X$, which shows that the $[0,\bar{x})^{\mathrm{sp}}_{\mathrm{res}}$ component of \cref{eq:misc_iba} is smooth.  Thus, $\bar{i}$ is smooth. 
	
	The diffeomorphism clause follows from the inversion formulas 
	\begin{align}
		\varrho_{\mathrm{df}} &= \Big[ 1 + \frac{x}{\sigma^2+\mathsf{Z} x} +  \Big( \frac{x^2}{\sigma^2+\mathsf{Z} x} + \xi_{\mathrm{sc,leC}}^2+\eta_{\mathrm{sc,leC}}^2 \Big)^{1/2}\Big]^{-1}, \\
		\varrho_{\mathrm{sf}} &= \frac{x}{\sigma^2+\mathsf{Z}x}\Big[1+\Big(\frac{x}{\sigma^2+\mathsf{Z}x}+\Big( \frac{x^2}{\sigma^2+\mathsf{Z} x} + \xi_{\mathrm{sc,leC}}^2+\eta_{\mathrm{sc,leC}}^2 \Big)^{1/2} \Big)^{-1} \Big], \\
		\varrho_{\mathrm{ff}} &= (\sigma^2+\mathsf{Z}x)^{1/2} \Big( \frac{x}{\sigma^2+\mathsf{Z}x} \Big( \frac{x^2}{\sigma^2+\mathsf{Z}x} + \xi_{\mathrm{sc,leC}}^2+\eta_{\mathrm{sc,leC}}^2 \Big)^{-1/2}  + 1 \Big), \\
		\varrho_{\mathrm{tf}} &= \Big[ 1 + \frac{x}{\sigma^2+\mathsf{Z}x} \Big( \frac{x^2}{\sigma^2+\mathsf{Z}x} + \xi_{\mathrm{sc,leC}}^2+\eta_{\mathrm{sc,leC}}^2 \Big)^{-1/2} \Big]^{-1},
	\end{align}
	and 
	\begin{multline}
		\varrho_{\mathrm{bf}} = \Big[1+ \Big(1 + \frac{x}{\sigma^2+\mathsf{Z}x}\Big) \Big(\frac{x^2}{\sigma^2+\mathsf{Z}x} + \xi_{\mathrm{sc,leC}}^2+\eta_{\mathrm{sc,leC}}^2 \Big)^{-1/2} \Big]^{-1} \\ \times \Big[ 1 +  \frac{x}{\sigma^2+\mathsf{Z}x} \Big(\frac{x^2}{\sigma^2+\mathsf{Z}x} + \xi_{\mathrm{sc,leC}}^2+\eta_{\mathrm{sc,leC}}^2 \Big)^{-1/2}\Big],
	\end{multline}
	(holding on $\{0<x<\bar{x}\}\cap \bbR^+_\sigma\times T^* X^\circ$), which certainly suffice to define a smooth two-sided inverse  $([0,\bar{x})^{\mathrm{sp}}_{\mathrm{res}} \times \bbR_{\xi_{\mathrm{sc,leC}}}\times (T^* \partial X)_{\eta_{\mathrm{sc,leC}}}) \backslash o_\partial \to \{0\leq x<\bar{x}\}\cap {}^{\mathrm{leC}} \overline{T}^* X \backslash (\mathrm{df}\cup \mathrm{bf}\cup \mathrm{tf})$ of $\bar{i}$. 
\end{proof}

\begin{figure}[t!]
	\begin{center} 
		\begin{tikzpicture}[scale=.8]
		\filldraw[fill=gray!5, dashed] (0,0) circle (3);
		\draw[fill=gray] (0,2) circle (7pt);
		\draw[fill=gray] (20pt,0) circle (7pt);
		\draw[fill=gray] (-20pt,0) circle (7pt);
		\draw[fill=white] (0,0) circle (20pt);
		\draw[red] (0,2) to[out=0,in=0] (20pt,0);
		\draw[red] (0,2) to[out=180,in=180] (-20pt,0);
		\filldraw[red] (1,.9) -- (1.1,1.1) -- (.83,1.05) -- cycle;
		\filldraw[red] (-1,.9) -- (-1.1,1.1) -- (-.83,1.05) -- cycle;
		\filldraw[red] (0,2) circle (2pt);
		\filldraw[red] (-20pt,0) circle (2pt); 
		\filldraw[red] (+20pt,0) circle (2pt); 
		\node[red, above right] (rp) at (0,2) {$\,\calR_+$};
		\node[red, right] (char) at (1,1) {$\;\Sigma$};
		\node[red, below right] (ro) at (20pt,0) {$\calR$};
		\node[black, above] (df) at (0,3) {$[{}^{\mathrm{sc,leC}} T^*_{p} X;(o_\partial)_p ] $};
		\draw[red, dashed] (0,0) ellipse (20pt and 2pt);
		\end{tikzpicture}
	\end{center}
	\caption{The result of blowing up the zero section of ${}^{\mathrm{sc,leC}}T^*_{\partial X} X$ over a fixed point $p\in (\partial [0,\bar{x})_{\mathrm{res}}^{\mathrm{sp}} \backslash (\mathrm{zf}^\circ \cup (\mathrm{bf}\cap \mathrm{tf})))\times \partial X$, with $\Sigma = \operatorname{Char}_{\mathrm{leC}}^{2,0,-2,-1,-3}(\tilde{P})$ (which, according to $\bar{i}$, agrees with $\operatorname{Char}^{2,0,-2}_{\mathrm{sc,leC}}(\tilde{P})$ away from $\mathrm{bf}\cup \mathrm{tf}$ but is deformed by the blow-up) and the vertical Hamiltonian flow on it indicated. Cf.\ \cite[Figure 3]{VasyLA}. The inner black circle depicts either $\mathrm{bf}\cap \mathrm{sf}$ or $\mathrm{tf}\cap \mathrm{ff}$, depending on whether $\sigma>0$ or $\sigma=0$. Over each $p$, it is a copy of $\bbS^{n-1}$. Its equator is the portion of $\calR$ over $p$. The radial point estimate in \S\ref{subsec:rp} is used to control the dark gray neighborhood, while elliptic estimates suffice for the rest of $\mathrm{bf}\cap \mathrm{sf}$ or $\mathrm{tf}\cap \mathrm{ff}$ (which is away from $\calR$).}
	\label{fig:blownup}
\end{figure}

In terms of $\xi_{\mathrm{sc,leC}}$ and $\eta_{\mathrm{sc,leC}}$, $\tilde{p}_0$ can be written as 
\begin{align} 
	\tilde{p}_0 &= (\sigma^2+\mathsf{Z}x)( \xi_{\mathrm{sc,leC}}^2 + g^{-1}_{\partial X}(\eta_{\mathrm{sc,leC}},\eta_{\mathrm{sc,leC}}) - 2 \xi_{\mathrm{sc,leC}}) \\
	\label{eq:ptil}
	&= (\sigma^2+\mathsf{Z}x)( \xi_{\mathrm{sc,leC}}^2 + \eta_{\mathrm{sc,leC}}^2 - 2 \xi_{\mathrm{sc,leC}})
\end{align} 
in $T^* X^\circ$. Weighting, $\tilde{p} \smash{\varrho_{\mathrm{df}}^2 \varrho_{\mathrm{ff}}^{-2} \varrho_{\mathrm{bf}}^{-1} \varrho_{\mathrm{tf}}^{-3}} = \tilde{p}^{2,0,-2,-1,-3}$ induces a well-defined function on $\partial {}^{\mathrm{leC}} T^* X$ and thus on $\smash{{}^{\mathrm{sc,leC}}T^*_{\partial X}X\backslash o_\partial}$. 
Note that this restriction does not depend on $\tilde{p}_{1,2}$. 
The portion of the leC-characteristic set  of $\tilde{P}$ disjoint from $\mathrm{bf}\cup \mathrm{tf}$ can be written as
\begin{equation}
\operatorname{Char}_{\mathrm{leC}}^{2,0,-2,-1,-3}(\tilde{P}) \backslash (\mathrm{bf}\cup \mathrm{tf}) =  (\tilde{p}^{2,0,-2,-1,-3}|_{{}^{\mathrm{sc,leC}} T^*_{\partial X} X \backslash o_\partial})^{-1}(\{0\}).
\label{eq:char}
\end{equation} 
This is equal to  $\{\xi_{\mathrm{sc,leC}}^2 + \eta_{\mathrm{sc,leC}}^2 - 2 \xi_{\mathrm{sc,leC}}=0 \} \backslash o_\partial \subset {}^{\mathrm{sc,leC}} T^*_{\partial X} X$. For convenience, we add in $o_\partial$ to give a new set 
\begin{equation}
	\operatorname{Char}_{\mathrm{sc,leC}}^{2,0,-2}(\tilde{P}) = \{\xi_{\mathrm{sc,leC}}^2 + g^{-1}_{\partial X}(\eta_{\mathrm{sc,leC}},\eta_{\mathrm{sc,leC}}) - 2 \xi_{\mathrm{sc,leC}}=0 \} \subset {}^{\mathrm{sc,leC}}T^*_{\partial X} X, 
	\label{eq:char2}
\end{equation} 
which consists of an off-center sphere over each point of $(\partial [0,\bar{x})_{\mathrm{res}}^{\mathrm{sp}}\backslash \mathrm{zf}^\circ)\times \partial X$ (in terms of the trivialization \cref{eq:bdy_trv}), of radius one and centered at $\xi_{\mathrm{sc,leC}}=1$ and $\eta_{\mathrm{sc,leC}} \in o$. See \Cref{fig:sketch}, which is a translated version of \cite[Figure 2]{MelroseSC}. The portion of 
\begin{equation} 
\operatorname{Char}_{\mathrm{leC}}^{2,0,-2,-1,-3}(\tilde{P}) = (\tilde{p}^{2,0,-2,-1,-3})^{-1}(\{0\}) \cap (\mathrm{df}\cup \mathrm{sf} \cup \mathrm{ff} ) 
\end{equation} 
over a point $p\in (\partial [0,\bar{x})_{\mathrm{res}}^{\mathrm{sp}} \backslash (\mathrm{zf}^\circ \cup (\mathrm{bf}\cap \mathrm{tf}))) \times \partial X$ is depicted in \Cref{fig:blownup}. Note that $\operatorname{Char}_{\mathrm{leC}}^{2,0,-2,-1,-3}(\tilde{P})$ only contains a codimension one subset of $\mathrm{sf}\cap \mathrm{bf}$, $\mathrm{ff}\cap \mathrm{tf}$, that corresponding to $\xi_{\mathrm{sc,leC}} = 0$, while the lift of $\smash{\operatorname{Char}_{\mathrm{sc,leC}}^{2,0,-2}(\tilde{P})}$ contains the whole sets.

Elliptic estimates control $u$ in terms of $\tilde{P}u$ away from this set --- see \S\ref{subsection:ellipticity}. 

As shown in \S\ref{subsec:propagation}, the Hamiltonian flow $H_{\tilde{p}} \in \calV(\bbR^{+}_\sigma\times T^* X^\circ)$ associated to $\tilde{p}$ is given in terms of the sc,leC- coordinates above by 
\begin{multline}
	H_{\tilde{p}_0}  = x (\sigma^2+\mathsf{Z}x)^{1/2} \Big[ 2 (\xi_{\mathrm{sc,leC}}-1) x\partial_x + 2 g^{-1}_{\partial X}(\eta_{\mathrm{sc,leC}},-)  \\ + \frac{2 \sigma^2+\mathsf{Z}x}{\sigma^2+\mathsf{Z}x}  \Big( (\xi_{\mathrm{sc,leC}}-1) \eta_{\mathrm{sc,leC}} \partial_{\eta_{\mathrm{sc,leC}}}  + \xi_{\mathrm{sc,leC}}(\xi_{\mathrm{sc,leC}}-2) \partial_{\xi_{\mathrm{sc,leC}}} \Big) - \frac{2\tilde{p}_0}{\sigma^2+\mathsf{Z} x}   \partial_{\xi_{\mathrm{sc,leC}}} \Big]
	\label{eq:hampres}
\end{multline} 
to leading order at $\partial X$, 
where $\eta_{\mathrm{sc,leC}} \partial_{\eta_{\mathrm{sc,leC}}} \in \calV(T^* \partial X) \subset \calV(\bbR^+_\sigma\times [0,\bar{x})_x \times \bbR_{\xi_{\mathrm{sc,leC}}}\times T^* \partial X)$ denotes the vector field on $T^* \partial X$ given in local coordinates $y_1,\ldots,y_{n-1}$ for $\partial X$ by 
\begin{equation}
	\eta_{\mathrm{sc,leC}} \partial_{\eta_{\mathrm{sc,leC}}} = \sum_{j=1}^{n-1} \eta_{\mathrm{sc,leC},j} \partial_{\eta_{\mathrm{sc,leC},j}},
\end{equation} 
where $\eta_{\mathrm{sc,leC},j}$ is the cofiber component of $\eta_{\mathrm{sc,leC}}$ dual to $y_j$. It will be convenient to work with the weighted Hamiltonian vector fields $H_{\tilde{p}_0}^{2,0,-2} = \varrho_{\mathrm{df}} x^{-1} (\sigma^2+\mathsf{Z} x)^{-1/2} H_{\tilde{p}_0}$, $H_{\tilde{p}}^{2,0,-2} = \varrho_{\mathrm{df}} x^{-1} (\sigma^2+\mathsf{Z} x)^{-1/2} H_{\tilde{p}}$.

The vector field $H_{\tilde{p}}$, as given by \cref{eq:hampres}, 
defines on the sc,leC-characteristic set $\operatorname{Char}_{\mathrm{sc,leC}}^{2,0,-2}(\tilde{P})$ a source-to-sink flow, with  
\begin{itemize}
	\item $\{\xi_{\mathrm{sc,leC}}=\eta_{\mathrm{sc,leC}}=0\} = \calR_0 \subset {}^{\mathrm{sc,leC}}T^*_{\partial X} X$, the ``selected radial set,'' and 
	\item $\{\xi_{\mathrm{sc,leC}}=2,\eta_{\mathrm{sc,leC}}=0\} = \calR_+ \subset \mathrm{sf} \cup \mathrm{ff}$, the ``unselected radial set,''
\end{itemize}
the sink and source (respectively, under our sign conventions) of the flow. 
Hence, leC-analogues of standard propagation and radial point estimates apply away from $o_\partial$ (the dependence on the one parameter $\sigma$ being unimportant), and the proofs are straightforward modifications of the analogous estimates in \cite{MelroseSC}. 
These estimates are proven in \S\ref{subsec:propagation}. Note that the $(\xi_{\mathrm{sc,leC}}-1) x\partial_x$ term in \cref{eq:hampres} does not actually vanish on the portion of $\calR_0,\calR_+$ over the interior of $\mathrm{tf}$ --- rather, it induces a flow from the $\mathrm{bf}$ side to the $\mathrm{zf}$ side. Thus, it is also possible to prove a propagation result in which regularity at $\calR_+ \cap  \mathrm{sf}$ is propagated into $\calR_+ \cap \mathrm{zf}$. We will not need (and hence will not prove) such an estimate, as a much weaker argument in \S\ref{sec:mainproof} suffices to show that for nice $f$ the output $\tilde{R}_+(E)f$ of the conjugated resolvent $\tilde{R}_+(E)$ converges weakly to something satisfying the conjugated version of the Sommerfeld radiation condition as $E\to 0^+$. (But it is still worth pointing out that $\calR_+ \cap \mathrm{zf}$ is not a source of the flow in every direction --- it is a sink in the $\varsigma=\sigma/x^{1/2}$ direction.)

The radial point estimate at the selected radial set is somewhat more nonstandard. Rather than $\calR_0$, we work with  $\calR = \operatorname{Char}^{2,0,-2,-1,-3}_{\mathrm{leC}}(\tilde{P}) \cap (\mathrm{bf}\cup \mathrm{tf})$. See \cite[Figure 3]{VasyLA} for the $\sigma>0$ case.  The face $\mathrm{ff}$ meets $\mathrm{bf}$ at an edge --- see \Cref{fig:LEC_phase_space}, \Cref{fig:later} --- and so $\mathrm{sf}\cup \mathrm{ff} \backslash \mathrm{df}$ is not just the boundary of
\begin{equation} 
[[0,\bar{x})_{\mathrm{res}}^{\mathrm{sp}} \times \bbR_{\xi_{\mathrm{sc,leC}}} \times T^* \partial X ; [0,\bar{x})_{\mathrm{res}}^{\mathrm{sp}} \times \{0\}\times o ],
\end{equation}
which is why \Cref{fig:blownup} depicts the situation only over  $p\in [0,\bar{x})_{\mathrm{sp}}^{\mathrm{res}}\times \partial X$ not on the zero face or corner.  
Away from that edge, our situation looks (even at ff) very much like that in \cite{VasyLA}. An argument similar to that used to prove the radial point estimate there suffices to prove an estimate here that is uniform down to $\sigma=0$, albeit without the desired number of independent orders. 
In order to prove an estimate with the desired number of independent orders, it will be necessary to take into account the aberrant edge. See \S\ref{subsec:rp}. 

The elliptic, propagation, and radial point estimates in \S\ref{subsection:ellipticity}, \S\ref{subsec:propagation}, \S\ref{subsec:rp} are combined in a short epilogue \S\ref{subsec:upshot}, which contains  \Cref{thm:symbolic_main}, as well as the corollary \Cref{prop:symbolic_final} stated above.

\begin{figure}[t!]
	\begin{tikzpicture}[scale=.75]
		\coordinate (ul) at (0,3.5,5) {};
		\coordinate (ulp) at (0,3.75,4.75) {};
		\coordinate (uul) at (0,4,4.5) {};
		\coordinate (uulc) at (3,4,4.5) {};
		\coordinate (ulc) at (3,3.5,5) {};
		\coordinate (ulcp) at (3,3.75,4.75) {};
		\coordinate (urc) at (5,4,3) {};
		\coordinate (urcp) at (5,3.2,3.8) {};
		\coordinate (ur) at (5,4,0) {};
		\coordinate (lc) at (4,3,5) {};
		\coordinate (rc) at (5,3,4) {};
		\coordinate (ld) at (4,0,5) {};
		\coordinate (rd) at (5,0,4) {};
		\draw (ld)--(lc); 
		\draw[red] (lc)--(ulc);
		\draw[red] (ulc)--(ul);
		\draw[red] (urcp)--(ulcp)--(ulp);
		\draw[red] (rc)--(lc);
		\draw (uul)--(uulc)--(urc);
		\draw (rd)--(rc)--(urc)--(ur);
		\draw (ulc)--(uulc);
		\draw[dashed] (uul)--(ul);
		\draw[dashed] (ld)--(rd); 
		\draw[dashed] (ld) --++ (-3,0,0);
		\draw[dashed] (rd) --++ (0,0,-3);
		\draw[dashed] (ul) --++ (0,-3,0);
		\draw[dashed] (uul) --++ (0,0,-3);
		\draw[dashed] (ur) --++ (-3,0,0);
		\draw[dashed] (ur) --++ (0,-3,0);
		\node (df) at (2.5,4,2.5) {$\mathrm{df}$};
		\node (bf) at (2,2,5) {$\mathrm{bf}$};
		\node (zf) at (5,2,2) {$\mathrm{zf}$};
		\node (tf) at (4.5,2,4.5) {$\mathrm{tf}$};
		\node (ff) at (4.4,3.75,4.15) {$\mathrm{ff}$};
		\node (sf) at (-.35,3.75,4.75) {$\mathrm{sf}$};
	\end{tikzpicture}
	\caption{The radial sets $\calR_+,\calR$ (in red) as subsets of the leC-phase space (ignoring the degrees of freedom associated with $T^* \partial X$). The top line represents $\calR_+$, while the bottom line depicts $\calR$. (Note that on $\mathrm{bf} \cap \mathrm{ff}$, $x/(\sigma^2+\mathsf{Z} x)=\sigma^2+\mathsf{Z} x=0$.) }
	\label{fig:later}
\end{figure}

\subsection{Ellipticity}
\label{subsection:ellipticity}

\begin{proposition}
	\label{prop:df_ellipticity}
	$\tilde{P} \in \operatorname{Diff}_{\mathrm{leC}}^{2,0,-2,-1,-3}(X)$ is elliptic in some neighborhood of $\mathrm{df}$. 
\end{proposition}
\begin{proof}
	By assumption, $\tilde{P}$ is elliptic at every point of $\mathrm{df}^\circ$. By \cref{eq:misc_v9m}, $\tilde{P}$ will be elliptic at a point of $\mathrm{sf}\cup \mathrm{ff}$ (including the points of $\partial \mathrm{df} \subset \mathrm{sf}\cup \mathrm{ff}$) if and only if $\tilde{P}_0$ is. $\operatorname{Char}_{\mathrm{leC}}^{2,0,-2,-1,-3}(\tilde{P}_0)$ is contained away from $\mathrm{df}$, so the same then holds for $\operatorname{Char}_{\mathrm{leC}}^{2,0,-2,-1,-3}(\tilde{P})$.
\end{proof}

Via the leC-calculus analogue of the usual argument via parametrix:

\begin{proposition}
	Let $A,A_0,B \in \Psi^{0,0,0,0,0}_{\mathrm{leC}}(X)$, with $\operatorname{WF}_{\mathrm{leC}}^{\prime 0,0}(A),\operatorname{WF}_{\mathrm{leC}}^{\prime}(A_0)\subseteq \operatorname{Ell}_{\mathrm{leC}}^{0,0,0,0,0}(B) \cap \operatorname{Ell}_{\mathrm{leC}}^{2,0,-2,-1,-3}(\tilde{P})$. 
	
	Then, for each $\Sigma>0$, $m,s,\varsigma,l,\ell \in \bbR$, and $N\in \bbN$,  
	there exists a 
	\begin{equation*} 
		C=C(\tilde{P},A,A_0,B,m,s,\varsigma,l,\ell,N,\Sigma)>0
	\end{equation*} 
	such that, 
	for any $u\in  \calS'(X)$, 
	\begin{align}
		\lVert A u \rVert_{H_{\mathrm{leC}}^{m,s,\varsigma,l,\ell}} &\leq C \big[ \lVert B\tilde{P} u \rVert_{H_{\mathrm{leC}}^{m-2,s,\varsigma+2,l+1,\ell+3}}+ \lVert u \rVert_{H_{\mathrm{leC}}^{-N,-N,-N,l,\ell}} \big] \label{eq:misc_aes} \\
		\lVert A_0 u \rVert_{H_{\mathrm{leC}}^{m,s,\varsigma,l,\ell}} &\leq C \big[ \lVert B\tilde{P} u \rVert_{H_{\mathrm{leC}}^{m-2,s,\varsigma+2,l+1,\ell+3}}+ \lVert u \rVert_{H_{\mathrm{leC}}^{-N,-N,-N,-N,-N}} \big] 
		\label{eq:misc_a0e}
	\end{align}
	for all $\sigma \in (0,\Sigma]$ 
	(in the strong sense that if the right-hand side is finite then the left-hand side is as well). 
\end{proposition}
\begin{proof} 
	Let $b \in S_{\mathrm{leC}}^{0,0,0,0,0}(X) = S_{\mathrm{b,leC}}^{0,0,0}(X)$ denote a representative of  $\sigma_{\mathrm{leC}}^{0,0,0,0,0}(B)$. Let $\varphi \in C_{\mathrm{c}}^\infty(\bbR)$ be identically equal to one in some neighborhood of $[0,\Sigma]$ and supported in $(-\infty,2\Sigma)$. 
	
	The set 
	\begin{equation} 
		K=\operatorname{WF}_{\mathrm{leC}}^{\prime 0,0}(A) \cap \{\sigma \leq 2\Sigma\} \subseteq \mathrm{df}\cup \mathrm{sf} \cup \mathrm{ff} \subseteq \partial {}^{\mathrm{leC}} \overline{T}^* X
	\end{equation} 
	is a compact subset of ${}^{\mathrm{leC}} \overline{T}^* X$, so we can  find some $\chi \in S_{\mathrm{leC}}^{0,0,0,0,0}(X)$ such that $\chi=1$ identically in some neighborhood of $K$ and such that $\chi = 0$ identically in some neighborhood of $\operatorname{Char}_{\mathrm{leC}}^{2,0,-2,-1,-3}(B\tilde{P}) = \operatorname{Char}_{\mathrm{leC}}^{0,0,0,0,0}(B) \cup \operatorname{Char}_{\mathrm{leC}}^{2,0,-2,-1,-3}(\tilde{P})$. We can moreover choose $\chi$ such that $\chi=0$ identically in some neighborhood of $b^{-1}(\{0\})\cup \tilde{p}^{-1}(\{0\})$. 
	Consider $f = \chi / b\tilde{p} \in S_{\mathrm{leC}}^{-2,0,2,1,3}(X)$. 
	
	Quantizing, we get 
	some $F = \operatorname{Op}(f)\in \Psi_{\mathrm{leC}}^{-2,0,2,1,3}(X)$ such that -- via the leC-principal symbol short exact sequence -- 
	 \begin{equation} 
		\varphi A (1-F B \tilde{P}) \in \Psi_{\mathrm{leC}}^{-1,-1,-1,0,0}(X). 
		\label{eq:misc_joj}
	\end{equation} 
	Indeed, $ \varphi A (1-F B \tilde{P})  \in \Psi_{\mathrm{leC}}^{0,0,0,0,0}(X)$ by \Cref{prop:algebra}. By \Cref{prop:basic_symbology}, for any representative $a\in S_{\mathrm{leC}}^{0,0,0,0,0}(X)$ of the principal symbol $\sigma_{\mathrm{leC}}^{0,0,0,0,0}(A)$,  
	\begin{align}
		\begin{split} 
		\sigma_{\mathrm{leC}}^{0,0,0,0,0}(\varphi A (1-F B \tilde{P})) &= \varphi a  (1- \chi) \bmod S_{\mathrm{leC}}^{-1,-1,-1,0,0}(X) \\ 
		&= 0 \bmod S_{\mathrm{leC}}^{-1,-1,-1,0,0}(X), 
		\end{split}
	\end{align}
	which implies \cref{eq:misc_joj}. 
	 So, for $\sigma\in [0,\Sigma]$, 
	\begin{align}
		\begin{split} 
			\lVert A u \rVert_{H_{\mathrm{leC}}^{m,s,\varsigma,l,\ell}} &\leq  \lVert \varphi AF B\tilde{P} u \rVert_{H_{\mathrm{leC}}^{m,s,\varsigma,l,\ell}} +  \lVert \varphi A (1-F B\tilde{P}) u \rVert_{H_{\mathrm{leC}}^{m,s,\varsigma,l,\ell}} \\
			&\preceq \lVert B \tilde{P} u \rVert_{H_{\mathrm{leC}}^{m-2,s,\varsigma+2,l+1,\ell+3}}+ \lVert \varphi A_1u \rVert_{H_{\mathrm{leC}}^{m-1,s-1,\varsigma-1,l,\ell}} + \lVert u \rVert_{H_{\mathrm{leC}}^{-N,-N,-N,l,\ell} } \\
			&= \lVert B \tilde{P} u \rVert_{H_{\mathrm{leC}}^{m-2,s,\varsigma+2,l+1,\ell+3}}+ \lVert  A_1u \rVert_{H_{\mathrm{leC}}^{m-1,s-1,\varsigma-1,l,\ell}}+ \lVert u \rVert_{H_{\mathrm{leC}}^{-N,-N,-N,l,\ell} }
			\end{split} 
	\end{align}
	for $A_1$ satisfying the same hypotheses as $A$ but also elliptic on the essential support of $A$.
	Inducting, we conclude the estimate \cref{eq:misc_aes}. 
	
	The second estimate, \cref{eq:misc_a0e}, is proven in a completely analogous manner. 
\end{proof}

\subsection{Propagation}
\label{subsec:propagation} 

The Hamiltonian vector field $H_{\tilde{p}} \in  \calV(\bbR^+_\sigma \times T^* X^\circ)$ associated with the symbol $\tilde{p} \in C^\infty( (0,\infty)_\sigma\times T^* X^\circ )$ is given near $\partial X$ by
\begin{equation}
	H_{\tilde{p}} = (\partial_{\xi_{\mathrm{b}}} \tilde{p}) x \partial_x - (x \partial_x \tilde{p}) \partial_{\xi_{\mathrm{b}}}  + \sum_{j=1}^{n-1} ((\partial_{\eta_j} \tilde{p}) \partial_{y_j} - (\partial_{y_j} \tilde{p}) \partial_{\eta_j}) \in C^\infty((0,\infty)_\sigma;\calV (T^* X^\circ))
	\label{eq:Hp}
\end{equation}
with respect to any set of local coordinates $y=(y_1,\ldots,y_{n-1})$ on $\partial X$. (We will alternatively identify $H_{\tilde{p}}$ as a smooth family of elements of $\calV (T^* X^\circ)$ and as a vector field on $\bbR^+_\sigma \times T^* X^\circ$ without $\partial_\sigma$ component.) 
Together, $x,y ,\xi_{\mathrm{sc,leC}}, \eta_{\mathrm{sc,leC}}$ constitute a coordinate chart for $T^* X^\circ$, so we can rewrite $H_{\tilde{p}}(\sigma) \in C^\infty(T^* X^\circ ; T T^* X^\circ)$ in terms of them, and the result (patching together the various coordinate charts for $\partial X$) can be interpreted as a weighted b-vector field on $ [0,\bar{x})^{\mathrm{sp}}_{\mathrm{res}} \times \bbR_{\xi_{\mathrm{sc,leC}}} \times (T^* \partial X)_{y,\eta_{\mathrm{sc,leC}}}$ (after restricting attention to a small collar neighborhood of $\partial X$). 
To perform this rewrite, we need the following substitutions:
\begin{align}
	\begin{split}
	x\partial_x &\to x\partial_x + \frac{\partial \xi_{\mathrm{sc,leC}}}{\partial x} x\partial_{\xi_{\mathrm{sc,leC}}} + \sum_{j=1}^{n-1} \frac{\partial \eta_{\mathrm{sc,leC},j}}{\partial x} x\partial_{\eta_{\mathrm{sc,leC},j}} \\
	&= x\partial_x + \frac{2\sigma^2+\mathsf{Z}x}{2(\sigma^2+\mathsf{Z}x)} \xi_{\mathrm{sc,leC}} \partial_{\xi_{\mathrm{sc,leC}}} +
	\frac{2\sigma^2+\mathsf{Z}x}{2(\sigma^2+\mathsf{Z}x)} \sum_{j=1}^{n-1} \eta_{\mathrm{sc,leC},j} \partial_{\eta_{\mathrm{sc,leC},j}} 
	\end{split}
\end{align} 
and 
\begin{equation}
	\partial_{\xi_{\mathrm{b}}} \to \frac{\partial \xi_{\mathrm{sc,leC}}}{\partial \xi_{\mathrm{b}}} \partial_{\xi_{\mathrm{sc,leC}}} 
	= \frac{x}{\sqrt{\sigma^2+\mathsf{Z}x}} \partial_{\xi_{\mathrm{sc,leC}}}, \qquad 
	\partial_{\eta_j} \to \frac{\partial \eta_{\mathrm{sc,leC},j}}{\partial \eta_j} \partial_{\eta_{\mathrm{sc,leC},j}} 
	= \frac{x}{\sqrt{\sigma^2+\mathsf{Z}x}} \partial_{\eta_{\mathrm{sc,leC},j}} 
\end{equation}
for $j=1,\ldots,n-1$, 
where the partial derivatives are taken with respect to the coordinate system $x,y,\xi_{\mathrm{b}},\eta$. In other words, letting $(x\partial_x)_{\mathrm{old}}$, $(\partial_{y_j})_{\mathrm{old}}$ denote the local vector fields defined using the coordinate system $x,y,\xi_{\mathrm{b}},\eta$, we have 
\begin{align}
	(\partial_{y_j})_{\mathrm{old}} &= \partial_{y_j},\\ 
	(x \partial_x)_{\mathrm{old}} &=  x\partial_x + \frac{2\sigma^2+\mathsf{Z}x}{2(\sigma^2+\mathsf{Z}x)} \xi_{\mathrm{sc,leC}} \partial_{\xi_{\mathrm{sc,leC}}} +
	\frac{2\sigma^2+\mathsf{Z}x}{2(\sigma^2+\mathsf{Z}x)} \sum_{j=1}^{n-1} \eta_{\mathrm{sc,leC},j} \partial_{\eta_{\mathrm{sc,leC},j}},
\end{align}
where the partial derivatives on the right-hand side are defined using the coordinate system $x,y,\xi_{\mathrm{sc,leC}},\eta_{\mathrm{sc,leC}}$. 
In terms of this new notation, \cref{eq:Hp} says 
\begin{equation}
	H_{\tilde{p}} = (\partial_{\xi_{\mathrm{b}}} \tilde{p}) (x \partial_x)_{\mathrm{old}} - ((x \partial_x)_{\mathrm{old}} \tilde{p}) \partial_{\xi_{\mathrm{b}}}  + \sum_{j=1}^{n-1} \Big((\partial_{\eta_j} \tilde{p}) (\partial_{y_j})_{\mathrm{old}} - ((\partial_{y_j})_{\mathrm{old}} \tilde{p}) \partial_{\eta_j}\Big). 
	\label{eq:Hpnew}
\end{equation}
The same holds for $\tilde{p}_0$ in place of $\tilde{p}$. 
The $\partial_{y_j}$ component of $H_{\tilde{p}_0} = H_{\tilde{p}} - H_{\tilde{p}_{1,2}}$ is given by 
\begin{equation}
	\partial_{\eta_j} \tilde{p}_0= \frac{x}{\sqrt{\sigma^2+\mathsf{Z}x}} \frac{\partial \tilde{p}_0}{ \partial \eta_{\mathrm{sc,leC}}} = 2x \sqrt{\sigma^2+\mathsf{Z}x} \sum_{k=1}^{n-1} g^{kj}\eta_{\mathrm{sc,leC},k}.
\end{equation}
On the other hand, the $\partial_{\eta_{\mathrm{sc,leC},j}}$ component of $H_{\tilde{p}_0}$ is given by
\begin{align} 
	\begin{split} 
	\frac{2\sigma^2+\mathsf{Z}x}{2\sigma^2+2\mathsf{Z}x} \eta_{\mathrm{sc,leC},j} \frac{\partial \tilde{p}_0}{\partial \xi_{\mathrm{b}}} - \frac{x}{(\sigma^2+\mathsf{Z}x)^{1/2}} \frac{\partial \tilde{p}_0}{\partial y_j }   &= \frac{2\sigma^2+\mathsf{Z}x}{2\sigma^2+2\mathsf{Z}x} \frac{x}{(\sigma^2+\mathsf{Z}x)^{1/2}} \eta_{\mathrm{sc,leC},j} \frac{\partial \tilde{p}_0}{\partial \xi_{\mathrm{sc,leC}}} \\
	&= \frac{2\sigma^2+\mathsf{Z}x}{(\sigma^2+\mathsf{Z}x)^{1/2}} x \eta_{\mathrm{sc,leC},j} (\xi_{\mathrm{sc,leC}} -1),
	\end{split}
\end{align}
while the $\partial_{\xi_{\mathrm{sc,leC}}}$ component is given by 
\begin{multline}
	\frac{2\sigma^2+\mathsf{Z}x}{2\sigma^2+2\mathsf{Z}x} \xi_{\mathrm{sc,leC}} \frac{\partial \tilde{p}_0}{\partial \xi_{\mathrm{b}}} - \frac{x^2}{(\sigma^2+\mathsf{Z}x)^{1/2}} \Big(\frac{\partial \tilde{p}_0}{\partial x }\Big)_{\mathrm{old}} \\ = \frac{2\sigma^2+\mathsf{Z}x}{(\sigma^2+\mathsf{Z}x)^{1/2}} x\xi_{\mathrm{sc,leC}} (\xi_{\mathrm{sc,leC}}-1) - \frac{x}{(\sigma^2+\mathsf{Z}x)^{1/2}} \Big[ 2\tilde{p}_0 +  \xi_{\mathrm{sc,leC}} (2\sigma^2+\mathsf{Z} x)\Big] \\
	= \frac{2\sigma^2+\mathsf{Z}x}{(\sigma^2+\mathsf{Z}x)^{1/2}} x\xi_{\mathrm{sc,leC}}(\xi_{\mathrm{sc,leC}}-2) - \frac{2 x \tilde{p}_0}{(\sigma^2+\mathsf{Z}x)^{1/2}},
\end{multline}
and the $x\partial_x$ component is $x (\sigma^2+\mathsf{Z}x)^{-1/2} \partial_{\xi_{\mathrm{sc,leC}}} \tilde{p}_0 = 2x (\sigma^2+\mathsf{Z}x)^{1/2} (\xi_{\mathrm{sc,leC}}-1)$. To summarize:

\begin{propositionp}
	\label{prop:Hocompe}
	In terms of the coordinates $x,y ,\xi_{\mathrm{sc,leC}},\eta_{\mathrm{sc,leC}}$, 
	\begin{multline}
		H_{\tilde{p}_0}  = x (\sigma^2+\mathsf{Z}x)^{1/2} \Big[ 2 (\xi_{\mathrm{sc,leC}}-1) x\partial_x + 2 g^{-1}_{\partial X}(\eta_{\mathrm{sc,leC}},-)  \\ + \frac{2 \sigma^2+\mathsf{Z}x}{\sigma^2+\mathsf{Z}x}  \Big(  \xi_{\mathrm{sc,leC}}(\xi_{\mathrm{sc,leC}}-2) \partial_{\xi_{\mathrm{sc,leC}}} +\sum_{j=1}^{n-1} \eta_{\mathrm{sc,leC},j}(\xi_{\mathrm{sc,leC}}-1)  \partial_{\eta_{\mathrm{sc,leC},j}} \Big) - \frac{2\tilde{p}_0}{\sigma^2+\mathsf{Z}x}   \partial_{\xi_{\mathrm{sc,leC}}} \Big]
	\end{multline}
	(in the relevant neighborhood of ${}^{\mathrm{leC}}T^* X$). 
\end{propositionp}

Letting $H_{\tilde{p}}^{-,0,-2} = x^{-1} (\sigma^2+\mathsf{Z} x)^{-1/2}H_{\tilde{p}} \in \calV(\bbR^+_\sigma \times T^* X^\circ)$, $H_{\tilde{p}}^{-,0,-2}$ defines a b-vector field on
\begin{equation} 
	{}^{\mathrm{sc,leC}}T^* X = [0,\bar{x})_{\mathrm{res}}^{\mathrm{sp}} \times \bbR_{\xi_{\mathrm{sc,leC}}} \times (T^* \partial X)_{y,\eta_{\mathrm{sc,leC}}}.
\end{equation} 
Likewise for $\tilde{p}_0$. Note that $H_{\tilde{p}}^{-,0,-2}$ and $H_{\tilde{p}_0}^{-,0,-2}$ agree at every point of $\mathrm{sf}\cup \mathrm{ff}$. 
Restricting to $\{x=0\}$, $H_{\tilde{p}}^{-,0,-2}$ can be considered as a family  
\begin{equation} 
	H: \partial ([0,\bar{x})_{\mathrm{res}}^{\mathrm{sp}} \backslash  \mathrm{zf}^\circ) \to \calV(\bbR_{\xi_{\mathrm{sc,leC}}}\times (T^* \partial X)_{y,\eta_{\mathrm{sc,leC}}})
\end{equation} 
of vector fields on the fiber $\bbR \times T^* \partial X$.
In order to understand $H$, consider $\smash{H_{\tilde{p}}^{-,0,-2}}$ over a neighborhood of a subset of the interior of the transition face tf of $[0,\bar{x})^{\mathrm{sp}}_{\mathrm{res}}$, which we can parametrize in terms of $x \in [0,\bar{x})$ and $\lambda \in \bbR^+$ by writing $\sigma^2 = \mathsf{Z} \lambda x$. In this neighborhood, $x^{1/2}$ is a bdf for the transition face. Then, $H$ can be thought of as a family of vector fields on $\bbR_{\xi_{\mathrm{sc,leC}}}\times T^* \partial X$ dependent on the parameter $\lambda$, with explicit formula 
\begin{equation} 
H  =  \frac{2\lambda+1}{\lambda+1}  \Big[(\xi_{\mathrm{sc,leC}}-1)\eta_{\mathrm{sc,leC}}\partial_{\eta_{\mathrm{sc,leC}}}  +(\xi_{\mathrm{sc,leC}}-2)\xi_{\mathrm{sc,leC}}\partial_{\xi_{\mathrm{sc,leC}}}  \Big] + 2 g^{-1}_{\partial X}(\eta_{\mathrm{sc,leC}},-) 
\label{eq:misc_102}
\end{equation} 
when $\tilde{p}_0=0$. The $\partial_y$-component of $H$ is independent of $\lambda$ and comes from the projection of the geodesic flow on $T^* \partial X$ down to $\partial X$. The cofiber components of $H$ (bracketed) depend on $\lambda$ only in the form of an overall factor that -- crucially -- is nonzero for all $\lambda \in [0,\infty)$. (In fact, \cref{eq:misc_102} makes sense as a family of vector fields on $\bbR\times T^* \partial X$ for all $\lambda > -1$, and it is non-vanishing for $\lambda>-1/2$, but we do not consider this ``extended transition face'' here. The vanishing at $\lambda=-1/2$ is one sign that the consideration of negative $\lambda$ would require solving a b-problem analogous to the b-problem encountered in the low-energy analysis of Coulomb-free Schr\"odinger operators. Cf.\ \Cref{rem:intuition}.) 
From \cref{eq:misc_102}, we read off the following crucial observation: over the transition face of $[0,\bar{x})^{\mathrm{sp}}_{\mathrm{res}}$, parametrized as above, $H$, when restricted to the characteristic set of $\tilde{p}$ and away from $\mathrm{bf}\cup\mathrm{tf}$, vanishes if and only if $\eta_{\mathrm{sc,leC}} = 0$ and $\xi_{\mathrm{sc,leC}} = 2$, i.e.\ on $\calR_+$.
Between this submanifold and the zero section $\calR_0$, $H$ has (over each point in $\partial [0,\bar{x})_{\mathrm{res}}^{\mathrm{sp}}$) a source-to-sink flow within $\operatorname{Char}_{\mathrm{sc,leC}}^{2,0,-2}(\tilde{P})$. 
In order to see that $\{\xi_{\mathrm{sc,leC}} = \eta_{\mathrm{sc,leC}} = 0\}$ is a sink of the flow, observe that 
\begin{equation} 
	H (\xi_{\mathrm{sc,leC}}^2+\eta_{\mathrm{sc,leC}}^2) = \frac{4\lambda+2}{\lambda+1} \Big[ \eta_{\mathrm{sc}}^2 (\xi_{\mathrm{sc,leC}}-1) + \xi_{\mathrm{sc}}^2 (\xi_{\mathrm{sc,leC}}-2)\Big] - \frac{4\tilde{p}_0\xi_{\mathrm{sc,leC}}}{\sigma^2+\mathsf{Z} x}
\end{equation} 
(note that $\tilde{p}_0/(\sigma^2+\mathsf{Z} x) = \xi_{\mathrm{sc,leC}}^2+\eta_{\mathrm{sc,leC}}^2-2\xi_{\mathrm{sc,leC}}$ is well-defined). 

The same computation yields:
\begin{propositionp}
	\label{prop:radial_comp}
	$H^{-,0,-2}_{\tilde{p}_0}(\xi_{\mathrm{sc,leC}}^2+\eta_{\mathrm{sc,leC}}^2) = \beta_{0,1} (\xi_{\mathrm{sc,leC}}^2+\eta_{\mathrm{sc,leC}}^2) + F_{0,2} + F_{0,3}$ for 
	\begin{equation}
		\beta_{0,1} = \frac{4\sigma^2+2\mathsf{Z}x}{\sigma^2+\mathsf{Z}x} (\xi_{\mathrm{sc,leC}} - 1), \qquad F_{0,2} = - \xi_{\mathrm{sc,leC}}^2 \frac{4\sigma^2+2\mathsf{Z}x}{\sigma^2+\mathsf{Z}x}, \qquad F_{0,3} = - \frac{4 \tilde{p}_0 \xi_{\mathrm{sc,leC}}}{\sigma^2+\mathsf{Z}x}. 
	\end{equation}
	These extend to symbols on the leC-phase space. The first two are nonpositive in the vicinity of $\mathrm{bf}\cup \mathrm{tf}$, while $F_{0,3}$ vanishes cubically there. 
\end{propositionp}

Note a similar statement holds if we replace $\tilde{p}_0$ by $\tilde{p}$, if instead of $F_{0,3} = -4\tilde{p}_0 \xi_{\mathrm{sc,leC}} / (\sigma^2+\mathsf{Z}x)$ we use 
\begin{equation} 
	H_{\tilde{p}_{1,2}}^{-,0,-2}(\xi_{\mathrm{sc,leC}}^2+\eta_{\mathrm{sc,leC}}^2) -4\tilde{p}_0 \xi_{\mathrm{sc,leC}} / (\sigma^2+\mathsf{Z}x),
\end{equation} 
which is the sum of a term suppressed by a factor of $x (\sigma^2+\mathsf{Z} x)^{-1/2}$ and a cubically vanishing term.

Consider, for each pair of  $\Theta_1,\Theta_2 \in (0,\pi)$ with $\Theta_1<\Theta_2$, the set $\calP[\Theta_1,\Theta_2] \subset \mathrm{sf}\cup \mathrm{ff}$ defined by 
$\calP[\Theta_1,\Theta_2] = \operatorname{Char}_{\mathrm{leC}}^{2,0,-2}(\tilde{P})  \cap \{ \arccos( \xi_{\mathrm{sc,leC}}-1) \in [\Theta_1,\Theta_2]\} $. The following proposition is a symbolic version of the statement that the Hamiltonian flow is source-to-sink, $\calR_+$ to $\calR_0$. 
\begin{proposition}
	\label{prop:flow_comp}
	 Let $\Theta \in S_{\mathrm{cl,leC}}^{0,0,0,0,0}(X)$ satisfy $\Theta = \operatorname{arccos}(\xi_{\mathrm{sc,leC}} - 1)$ in some neighborhood of $ \calP[\Theta_1,\Theta_2]$.

	For any pair if $\Theta_1,\Theta_2 \in (0,\pi)$ with $\Theta_1<\Theta_2$, the symbol $\alpha \in S^{0,0,0,0,0}_{\mathrm{leC}}(X)$  defined by $H_{\tilde{p}}^{2,0,-2} \Theta  =  \alpha$
	satisfies $\alpha>0$ on $\calP[\Theta_1,\Theta_2]$. 
\end{proposition}
\begin{proof}
	Given such $\Theta \in S_{\mathrm{cl,leC}}^{0,0,0,0,0}(X)$, $\Theta$ is equal to $\operatorname{arccos}(\xi_{\mathrm{sc,leC}} - 1)$ in some neighborhood $U\subset {}^{\mathrm{leC}} \overline{T}^* X$ of $\calP[\Theta_1,\Theta_2]$. Thus, 
	\begin{equation}
		H_{\tilde{p}_0}^{-,0,-2} \Theta =  \frac{2\sigma^2 +\mathsf{Z} x}{\sigma^2+\mathsf{Z}x} \xi_{\mathrm{sc,leC}}^{1/2} (2-\xi_{\mathrm{sc,leC}})^{1/2} + \frac{2\tilde{p}_0}{\sigma^2+\mathsf{Z}x} \frac{1}{\xi_{\mathrm{sc,leC}}^{1/2} (2-\xi_{\mathrm{sc,leC}}^2)^{1/2} }
		\label{eq:misc_751}
	\end{equation}
in some neighborhood of $\calP[\Theta_1,\Theta_2]$, where we are taking positive square roots. Since $\tilde{p}_0$ vanishes on $\calP[\Theta_1,\Theta_2]$, the expression on the right-hand side is positive on $\calP[\Theta_1,\Theta_2]$. 

Since $\varrho_{\mathrm{df}}$ is positive on $\calP[\Theta_1,\Theta_2]$, the same statement applies to $H^{2,0,-2}_{\tilde{p}_0} \Theta$. And since $H_{\tilde{p}_{1,2}}^{2,-0,-2}\Theta$ vanishes at $\calP[\Theta_1,\Theta_2]$, the same statement applies to $\alpha=H_{\tilde{p}}^{2,0,-2} \Theta$.
\end{proof}

Clearly, there exist $\Theta = \Theta_{\Theta_1,\Theta_2} \in  S_{\mathrm{cl,leC}}^{0,0,0,0,0}(X)$ satisfying the hypotheses of the previous proposition.

\begin{proposition} 
	\label{prop:propagation_estimate}
	Suppose that $G_1,G_2,G_3 \in \Psi_{\mathrm{leC}}^{-\infty,0,0,-\infty,-\infty}(X)$ satisfy the following: there exist 
	\begin{equation} 
	\Theta_1,\Theta_2,\Theta_3,\Theta_4,\Theta_5 \in (0,\pi)
	\end{equation} 
	satisfying $\Theta_1<\Theta_2<\Theta_3<\Theta_4<\Theta_5$ and 
	\begin{itemize}
		\item $\operatorname{WF}'_{\mathrm{leC}}(G_1) \cap \operatorname{Char}_{\mathrm{leC}}^{2,0,-2,-1,-3}(\tilde{P}) \subseteq \calP[\Theta_3,\Theta_4]$, 
		\item $\calP[\Theta_1,\Theta_2] \subseteq \operatorname{Ell}_{\mathrm{leC}}^{0,0,0,0,0}(G_3)$, 
		\item $\calP[\Theta_1,\Theta_5] \subseteq \operatorname{Ell}_{\mathrm{leC}}^{0,0,0,0,0}(G_2)$,
	\end{itemize}
	and $\operatorname{WF}'_{\mathrm{leC}}(G_1) \subseteq \operatorname{Ell}_{\mathrm{leC}}^{0,0,0,0,0}(G_2)$. 
	Then, for every $\Sigma>0$, $N\in \bbN$, $m,s,\varsigma,l,\ell \in \bbR$, there exists a constant 
	\begin{equation} 
	C=C(\tilde{P},G_1,G_2,G_3,\Theta_1,\Theta_2,\Theta_3,\Theta_4,\Theta_5,\Sigma,N,m,s,\varsigma,l,\ell)>0 
	\end{equation} 
	such that 
	\begin{equation}
		\lVert G_1 u \rVert_{H_{\mathrm{leC}}^{m,s,\varsigma,l,\ell}} \leq C\Big[ \lVert G_2 \tilde{P} u \rVert_{H_{\mathrm{leC}}^{-N,s+1,\varsigma+3,-N,-N} } + \lVert G_3 u \rVert_{H_{\mathrm{leC}}^{-N,s,\varsigma,-N,-N}} + \lVert u \rVert_{H_{\mathrm{leC}}^{-N,-N,-N,-N,-N}} \Big]
		\label{eq:propagation_estimate}
	\end{equation}
	holds for all $u \in \calS'(X)$ and $\sigma \in [0,\Sigma]$ (in the strong sense that if the right-hand side is finite, then the left-hand side is as well, and the stated inequality holds). 
\end{proposition} 
\begin{proof}
	Throughout the argument below, we can take $N$ to be sufficiently large such that any of the finitely many functions of $m,s,\varsigma,l,\ell$ that arise can be bounded below by $-N$. 
	
	We may assume without loss of generality that $G_2,G_3$ are essentially supported away from $\mathrm{bf}\cup\mathrm{tf}$ and that $G_2$ is essentially supported away from $\calR_+$:
	\begin{equation}
	\calR_+ \cap \operatorname{WF}'_{\mathrm{leC}}(G_2)=\varnothing.
	\label{eq:misc_cas}
	\end{equation}
	
	Let $\varphi \in C_{\mathrm{c}}^\infty(\bbR)$ satisfy $\operatorname{supp}\varphi \subset [\Theta_1,\Theta_5]$ and
	\begin{equation} 
	\varphi'(\theta) = \varphi_0(\theta)^2 + \varphi_1(\theta) 
	\end{equation} 
	for $\varphi_0 \in C_{\mathrm{c}}^\infty(\bbR)$ nonvanishing on $[\Theta_3,\Theta_4]$ and $\varphi_1 \in C_{\mathrm{c}}^\infty (\bbR)$ supported within $(\Theta_1,\Theta_2)$. The construction of such $\varphi$ is standard. 
	Moreover, for any closed interval $I\subset ((\Theta_1+3\Theta_2)/4,\Theta_5)$ and any $\epsilon>0$, we can construct $\varphi$ such that $\epsilon \varphi_0^2 \geq  \varphi $ in $I$.  This construction is also standard ---  we consider $\varphi_{00} \in C^\infty(\bbR)$ given by 
	\begin{equation}
		\varphi_{00} = 
		 \begin{cases}
			e^{-\digamma / (\Theta_5-\Theta)} & (\Theta \leq \Theta_5) \\ 
			0 & (\Theta>\Theta_5) 
		\end{cases}
	\end{equation}
	for a parameter $\digamma=\digamma(\sigma)>0$ 
	and $\varphi_{01} \in C^\infty(\bbR)$ that is identically equal to one in some neighborhood of $[\Theta_2,\infty)$ and identically zero in some neighborhood of $(-\infty,\Theta_1]$; setting $\varphi =- \varphi_{00} \varphi_{01}^2$ we have 
	\begin{equation}
		\varphi' = -2\varphi_{00} \varphi_{01} \varphi_{01}' - \varphi_{00}' \varphi_{01}^2 =-2\varphi_{00}\varphi_{01} \varphi_{01}' + \digamma (\Theta_5-\Theta)^{-2} \varphi_{00} \varphi_{01}^2. 
	\end{equation}
	Setting $\varphi_0 = \digamma^{1/2} (\Theta_5-\Theta)^{-1} \varphi_{00}^{1/2} \varphi_{01}$ and $\varphi_1 = - 2 \varphi_{00} \varphi_{01} \varphi_{01}'$, we see that $\varphi,\varphi'$ have the desired form.

	Fix $\Theta \in S_{\mathrm{cl,leC}}^{0,0,0,0,0}(X)$ that is equal to $\operatorname{arccos}(\xi_{\mathrm{sc,leC}} - 1)$ in some neighborhood of $\calP[\Theta_1,\Theta_5]$. Pick a neighborhood $U_0$ of $\calP[\Theta_1,\Theta_5]$ on which $\Theta$ is identically $\operatorname{arccos}(\xi_{\mathrm{sc,leC}} - 1)$ and such that $\alpha$ is bounded below on $U_0$ (in compact sets worth of $\sigma$). 
	Let 
	\begin{equation} 
	\varphi(\Theta) \in S_{\mathrm{cl,leC}}^{-\infty,0,0,-\infty,-\infty}(X)
 	\end{equation}
 	denote a symbol equal to $\varphi\circ \Theta$ on some neighborhood $U\Subset U_0$ of $\calP[\Theta_1,\Theta_5]$, and let $\psi \in S_{\mathrm{cl,leC}}^{-\infty,0,0,-\infty,-\infty}(X)$ have $\operatorname{supp} \psi \Subset U$ and be identically equal to $1$ on some neighborhood of $\calP[\Theta_1,\Theta_5]$. We can take both of these to be supported away from $\mathrm{bf}\cup\mathrm{tf}$.
 	We can choose $\psi$ such that
 	\begin{equation}
 		\operatorname{supp} (\varphi(\Theta) H_{\tilde{p}}^{2,0,-2}\psi) \cap (\tilde{p}^{2,0,-2,-1,-3})^{-1}(\{0\})=\varnothing 
 		\label{eq:misc_spp}  
 	\end{equation}
 	and $\operatorname{supp} \psi \cap (\mathrm{df} \cup \calR_+ ) = \varnothing$. 
 	Consider the symbol 
	\begin{equation}
		a_0= \varphi(\Theta) \psi^2 \in S_{\mathrm{cl,leC}}^{-\infty,0,0,-\infty,-\infty}(X).  
	\end{equation}
	We then compute 
	\begin{equation}
		H_{\tilde{p}}^{2,0,-2} a_0 = (\varphi_0(\Theta)^2 + \varphi_1(\Theta))  \psi^2 \alpha 
		+ 2\psi \varphi(\Theta)  H_{\tilde{p}}^{2,0,-2} \psi, 
		\label{eq:misc_jik}
	\end{equation} 
	where $\alpha$ is as in \Cref{prop:flow_comp}. 
	
	Now set, for to-be-decided $K>0$, set $\phi_\varepsilon(x) = (1+\varepsilon x^{-1})^{-K}$, for each $\varepsilon\geq 0$. Now set $a^{(\varepsilon)}_0 = \phi_\varepsilon^2 a_0$. 
	
	In \S\ref{sec:operator}, we checked that $\Im \tilde{P} \in S \operatorname{Diff}_{\mathrm{leC}}^{2,-1,-3,-1-\delta,-3-2\delta}(X) \subset \Psi_{\mathrm{leC}}^{2,-1,-3,-1-\delta,-3-2\delta}(X)$. 
	Let $p_1$ denote a representative of $\sigma_{\mathrm{leC}}^{2,-1,-3,-1-\delta,-3-2\delta}(-2\Im \tilde{P})$. 
	Then, 
	\begin{equation}
		H_{\tilde{p}}^{2,0,-2} a^{(\varepsilon)}_0  =  (\varphi_0(\Theta)^2 + \varphi_1(\Theta)) \phi_\varepsilon^2  \psi^2 \alpha    +2\psi \varphi(\Theta)  H_{\tilde{p}}^{2,0,-2} \psi  + 2 K \psi^2 \phi_\varepsilon^2 (\varepsilon x^{-1})(1+\varepsilon x^{-1})^{-1} \varphi \beta_1 ,
	\end{equation}
	\begin{multline}
		H_{\tilde{p}}^{2,0,-2} a^{(\varepsilon)}_0 + \varrho_{\mathrm{df}}\varrho_{\mathrm{sf}}^{-1}\varrho_{\mathrm{ff}}^{-3} \varrho_{\mathrm{bf}}^{-1}\varrho_{\mathrm{tf}}^{-3}   p_1 a^{(\varepsilon)}_0  =  (\varphi_0(\Theta)^2 + \varphi_1(\Theta)) \phi_\varepsilon^2  \psi^2 \alpha    +2\psi \varphi(\Theta)  H_{\tilde{p}}^{2,0,-2} \psi \\ + 2 K \psi^2 \phi_\varepsilon^2 (\varepsilon x^{-1})(1+\varepsilon x^{-1})^{-1} \varphi \beta_1   + \psi^2 p_1  \varrho_{\mathrm{df}}\varrho_{\mathrm{sf}}^{-1}\varrho_{\mathrm{ff}}^{-3} \varrho_{\mathrm{bf}}^{-1}\varrho_{\mathrm{tf}}^{-3} \phi_\varepsilon^2 \varphi ,
		\label{eq:misc_jkk}
	\end{multline} 
	where  $\beta_1 \in S_{\mathrm{leC}}^{0,0,0,0,0}(X)$ is defined by $H_{\tilde{p}}^{2,0,-2} x = \beta_1 x$.

	We now let $a = \varrho_{\mathrm{sf}}^{-2s-1} \varrho_{\mathrm{ff}}^{-2\varsigma-3} \varrho_{\mathrm{bf}}^{-2l-1} \varrho_{\mathrm{tf}}^{-2\ell-3} a_0$, $a^{(\varepsilon)} =  \smash{\varrho_{\mathrm{sf}}^{-2s-1} \varrho_{\mathrm{ff}}^{-2\varsigma-3} \varrho_{\mathrm{bf}}^{-2l-1} \varrho_{\mathrm{tf}}^{-2\ell-3} a_0^{(\varepsilon)}}$. Then, 
	$a^{(\bullet)} \in L^\infty([0,1]_\varepsilon; S_{\mathrm{leC}}^{-\infty,2s+1,2\varsigma+3,-\infty,-\infty}(X))$ and 
	\begin{equation}
		H_{\tilde{p}}^{2,0,-2} a^{(\varepsilon)} = \varrho_{\mathrm{sf}}^{-2s-1} \varrho_{\mathrm{ff}}^{-2\varsigma-3} \varrho_{\mathrm{bf}}^{-2l-1} \varrho_{\mathrm{tf}}^{-2\ell-3} (H_{\tilde{p}}^{2,0,-2} a_0^{(\varepsilon)}+ a_0^{(\varepsilon)} p_2)
	\end{equation}
	for some $p_2 \in S_{\mathrm{leC}}^{0,0,0,0,0}(X)$. Thus, 
	\begin{multline}
	H_{\tilde{p}}^{2,0,-2} a^{(\varepsilon)} + \varrho_{\mathrm{df}}\varrho_{\mathrm{sf}}^{-1}\varrho_{\mathrm{ff}}^{-3} \varrho_{\mathrm{bf}}^{-1}\varrho_{\mathrm{tf}}^{-3}   p_1 a^{(\varepsilon)}  =   \varrho_{\mathrm{sf}}^{-2s-1} \varrho_{\mathrm{ff}}^{-2\varsigma-3} \varrho_{\mathrm{bf}}^{-2l-1} \varrho_{\mathrm{tf}}^{-2\ell-3}\Big[ (\varphi_0(\Theta)^2 + \varphi_1(\Theta)) \phi_\varepsilon^2  \psi^2 \alpha  \\  +2\psi \varphi(\Theta)\phi_\varepsilon^2   H_{\tilde{p}}^{2,0,-2} \psi  + 2 K \psi^2 \phi_\varepsilon^2 (\varepsilon x^{-1})(1+\varepsilon x^{-1})^{-1} \varphi \beta_1  + \phi^2_{\varepsilon}\varphi(\Theta)\psi^2 p_2 \\ + \psi^2 p_1  \varrho_{\mathrm{df}}\varrho_{\mathrm{sf}}^{-1}\varrho_{\mathrm{ff}}^{-3} \varrho_{\mathrm{bf}}^{-1}\varrho_{\mathrm{tf}}^{-3} \phi_\varepsilon^2 \varphi \Big].
	\end{multline} 
	Dividing by $\varrho_{\mathrm{df}} x^{-1} (\sigma^2+\mathsf{Z}x)^{-1/2}$, 
	\begin{multline}
	H_{\tilde{p}} a^{(\varepsilon)} +  p_1 a^{(\varepsilon)}  =  \varrho_{\mathrm{df}}^{-1} \varrho_{\mathrm{sf}}^{-2s} \varrho_{\mathrm{ff}}^{-2\varsigma} \varrho_{\mathrm{bf}}^{-2l} \varrho_{\mathrm{tf}}^{-2\ell}\Big[ (\varphi_0(\Theta)^2 + \varphi_1(\Theta)) \phi_\varepsilon^2  \psi^2 \alpha  \\  +2\psi \varphi(\Theta) \phi_\varepsilon^2  H_{\tilde{p}}^{2,0,-2} \psi  + 2 K \psi^2 \phi_\varepsilon^2 (\varepsilon x^{-1})(1+\varepsilon x^{-1})^{-1} \varphi \beta_1  + \phi^2_{\varepsilon}\varphi(\Theta)\psi^2 p_2 \\ + \psi^2 p_1  \varrho_{\mathrm{df}}\varrho_{\mathrm{sf}}^{-1}\varrho_{\mathrm{ff}}^{-3} \varrho_{\mathrm{bf}}^{-1}\varrho_{\mathrm{tf}}^{-3} \phi_\varepsilon^2 \varphi \Big].
	\end{multline} 
	
	For each $K,\underline{\delta}>0$, we may choose $\psi=\smash{\psi_{K,\underline{\delta}}}$ (perhaps dilating it if necessary) such that its essential support is a subset of $\operatorname{Ell}_{\mathrm{leC}}^{0,0,0,0,0}(G_2)$ and such that, taking $\digamma=\digamma_{K,\underline{\delta},\psi} \in C^\infty([0,\infty)_\sigma;\bbR^+)$ sufficiently large, 
	\begin{multline}
		b_\varepsilon =  \varrho_{\mathrm{df}}^{-1/2} \varrho_{\mathrm{sf}}^{-s} \varrho_{\mathrm{ff}}^{-\varsigma} \varrho_{\mathrm{bf}}^{-l} \varrho_{\mathrm{tf}}^{-\ell}   \phi_\varepsilon \varphi_{00}^{1/2} \varphi_{01}  \psi \Big[  \digamma (\Theta_5-\Theta)^{-2} \alpha     - 2 K \frac{\varepsilon x^{-1}}{1+\varepsilon x^{-1}}  \beta_1  - p_2 \\  -  p_1  \varrho_{\mathrm{df}}\varrho_{\mathrm{sf}}^{-1}\varrho_{\mathrm{ff}}^{-3}\varrho_{\mathrm{bf}}^{-1} \varrho_{\mathrm{tf}}^{-3}  - 2 \underline{\delta} \phi_\varepsilon^2 \psi^2 \varphi \Big]^{1/2}
	\end{multline}
	is a well-defined uniform family of leC-symbols, specifically $b_\bullet \in L^\infty([0,1]_\varepsilon ; S_{\mathrm{leC}}^{-\infty,s,\varsigma,-\infty,-\infty}(X))$.
	In addition, we set 
	\begin{equation}
		e_\varepsilon  =  \varrho_{\mathrm{df}}^{-1}  \varrho_{\mathrm{sf}}^{-2s} \varrho_{\mathrm{ff}}^{-2\varsigma} \varrho_{\mathrm{bf}}^{-2l} \varrho_{\mathrm{tf}}^{-2\ell} \varphi_1(\Theta) \phi_\varepsilon^2 \psi^2  \alpha, 
	\end{equation}
	\begin{equation}
		f_\varepsilon =  2\varrho_{\mathrm{df}}^{-1} \varrho_{\mathrm{sf}}^{-2s} \varrho_{\mathrm{ff}}^{-2\varsigma} \varrho_{\mathrm{bf}}^{-2l} \varrho_{\mathrm{tf}}^{-2\ell}  \varphi(\Theta) \phi_\varepsilon^2 \tilde{p}^{-1} \psi H_{\tilde{p}}^{2,0,-2}\psi 
		\label{eq:misc_fve}
	\end{equation}
	(the division by $\tilde{p}$ in \cref{eq:misc_fve} being unproblematic by \cref{eq:misc_spp}).
	Thus,
	\begin{align}
	e_\bullet &\in  L^\infty([0,1]_\varepsilon ; S_{\mathrm{leC}}^{-\infty,2s,2\varsigma,-\infty,-\infty}(X)),\\  
	f_\bullet &\in L^\infty([0,1]_\varepsilon ; S_{\mathrm{leC}}^{-\infty,2s,2\varsigma+2,-\infty,-\infty}(X)). 
	\end{align} 
	In terms of these, 
	\begin{equation} 
		H_{\tilde{p}} a^{(\varepsilon)} +p_1 a^{(\varepsilon)} = 2\underline{\delta}   \varrho_{\mathrm{df}}^{-1} \varrho_{\mathrm{sf}}^{2s+2} \varrho_{\mathrm{ff}}^{2\varsigma+6} \varrho_{\mathrm{bf}}^{2l+2} \varrho_{\mathrm{tf}}^{2\ell+6}   a^{(\varepsilon)2} + b_\varepsilon^2 +e_\varepsilon + f_\varepsilon \tilde{p}. 
	\end{equation}	
	Setting $A_\varepsilon = (1/2)(\operatorname{Op}(a^{(\varepsilon)}) + \operatorname{Op}(a^{(\varepsilon)})^*)$, $B_\varepsilon = \operatorname{Op}(b_\varepsilon)$, $E_\varepsilon = \operatorname{Op}(e_\varepsilon)$, $F_\varepsilon=\operatorname{Op}(f_\varepsilon)$, 
	\begin{equation}
		-i [\Re \tilde{P},A_\varepsilon] -  \{\Im \tilde{P},A_\varepsilon\} =  2 \underline{\delta} A_\varepsilon\Lambda_{1/2,-s-1,-\varsigma-3,-l-1,-\ell-3}^2 A_\varepsilon + B_\varepsilon^* B_\varepsilon + E_\varepsilon + F_\varepsilon^* \tilde{P} \\ +R_\varepsilon
		\label{eq:misc_l09}
	\end{equation}
	for some $R_\bullet \in L^\infty([0,1]_\varepsilon;\Psi_{\mathrm{leC}}^{-\infty,2s-1,2\varsigma-1,2l,2\ell}(X))$. We have 
	\begin{align}
		\begin{split} 
		A_\bullet &\in L^\infty([0,1]_\varepsilon ; \Psi_{\mathrm{leC}}^{-\infty,2s+1,2\varsigma+3,-\infty,-\infty}(X) ), \\ 
		F_\bullet &\in  L^\infty([0,1]_\varepsilon;\Psi_{\mathrm{leC}}^{-\infty,2s,2\varsigma+2,-\infty,-\infty}(X)),\\
		B_\bullet &\in L^\infty([0,1]_\varepsilon;\Psi_{\mathrm{leC}}^{-\infty,s,\varsigma,-\infty,-\infty}(X)), \\
		E_\bullet &\in L^\infty([0,1]_\varepsilon;\Psi_{\mathrm{leC}}^{-\infty,2s,2\varsigma,-\infty,-\infty}(X)),
		\end{split}
	\label{eq:misc_efq}
	\end{align}
	and
	\begin{multline} 
		\operatorname{WF}'_{L^\infty,\mathrm{leC}}(A_\bullet),	\operatorname{WF}'_{L^\infty,\mathrm{leC}}(B_\bullet), \operatorname{WF}'_{L^\infty,\mathrm{leC}}(E_\bullet),\\ \operatorname{WF}'_{L^\infty,\mathrm{leC}}(F_\bullet),\operatorname{WF}'_{L^\infty,\mathrm{leC}}(R_\bullet)\subset \operatorname{supp} \psi\varphi(\Theta), 
	\end{multline}  
	where the last of these inclusions (the one for $R_\bullet$) follows from the one for $A_\bullet$ and 
	$\operatorname{WF}'_{L^\infty}([\tilde{P},A_\bullet]) \subset \operatorname{WF}'_{L^\infty}(A_\bullet)$. 
	
	Now, for each $m_0,s_0,\varsigma_0,l_0,\ell_0 \in \bbR$, there exist $K>0$ (dependent on $m_0,s_0,\varsigma_0,l_0,\ell_0$ and $m,l,\ell$ but nothing else) such that, given $\{u(-;\sigma)\}_{\sigma\geq 0} \in L^\infty([0,2\Sigma]; H_{\mathrm{leC}}^{m_0,s_0,\varsigma_0,l_0,\ell_0}(X) )$, 
	it is the case that, for any $\varepsilon>0$ (and for each $\sigma>0$, implicit in the notation below),
	\begin{equation}
		2  \Im \langle \tilde{P} u , A_\varepsilon u \rangle_{L^2} =   -\langle \{\Im \tilde{P}, A_\varepsilon \}u,u\rangle_{L^2} + i \langle [\Re \tilde{P}(\sigma),A_\varepsilon]u,u \rangle_{L^2} ,
		\label{eq:misc_l94}
	\end{equation}
	where the pairings above are well-defined (and where we are using the convention that $\langle-,-\rangle_{L^2}$ is antilinear in the first slot). Indeed, $A_\varepsilon$, $\{\Im \tilde{P}, A_\varepsilon \}$, and $[\Re \tilde{P}(\sigma),A_\varepsilon]$ are all smoothing operators -- i.e.\ lying in $\Psi_{\mathrm{scb}}^{-\infty,\infty,\infty}(X)$ if $\sigma>0$ and $\Psi_{\mathrm{scb}}^{-\infty,\infty,\infty}(X_{1/2})$ if $\sigma=0$ (in a uniform sense made precise by the leC-calculus, but since we simply need to justify some integration by parts $\sigma$-wise the uniformity is not important here) -- and by taking $K$ large they can be made to induce an arbitrarily large amount of decay for each $\varepsilon>0$. Given $N$, we fix  $m_0,s_0,\varsigma_0,l_0,\ell_0 \in \bbR$ such that $m_0,s_0,\varsigma_0,l_0,\ell_0<-N$. Then, we can take $K$ depending on $m,l,\ell,N$ and nothing else such that \cref{eq:misc_l94} holds for all $\{u(-;\sigma)\}_{\sigma\geq 0} \in L^\infty([0,\Sigma]; H_{\mathrm{leC}}^{-N,-N,-N,-N,-N}(X) )$.

	Applying \cref{eq:misc_l94} to $\{u(-;\sigma)\}_{\sigma>0} \in H_{\mathrm{leC}}^{-N,-N,-N,-N,-N}(X)$ and pairing against $u$ (after taking $K$ large enough), we have  
	\begin{multline}
		2 \Im \langle \tilde{P}u,A_\varepsilon u \rangle_{L^2} = \lVert B_\varepsilon u \rVert^2_{L^2} + \langle u, E_\varepsilon u \rangle_{L^2} + \langle \tilde{P} u, F_\varepsilon u \rangle_{L^2} + \langle R_\varepsilon u,u \rangle_{L^2} \\ + 2 \underline{\delta} \lVert \Lambda_{1/2,-s-1,-\varsigma-3,-l-1,-\ell-3}^2 A_\varepsilon u \rVert^2_{L^2}
	\end{multline} 
	for $\varepsilon>0$. 
	So, 
	\begin{multline}
		\lVert B_\varepsilon u \rVert_{L^2}^2 + 2\underline{\delta} \lVert \Lambda_{1/2,-s-1,-\varsigma-3,-l-1,-\ell-3} A_\varepsilon u \rVert^2_{L^2} \leq \\2| \langle \tilde{P}u,A_\varepsilon u \rangle_{L^2}| + | \langle \tilde{P}u,F_\varepsilon u \rangle_{L^2}| + | \langle R_\varepsilon u, u \rangle_{L^2}| + +|\langle u,E_\varepsilon u \rangle_{L^2}| .
		\label{eq:misc_lz1}
	\end{multline}

	Fix self-adjoint $G \in \Psi^{-\infty,0,0,-\infty,-\infty}_{\mathrm{leC}}(X)$ (constructed via $\operatorname{Op}$) such that $\operatorname{WF}_{\mathrm{leC}}^{'0,0}(1-G)$ is disjoint from a neighborhood of the $L^\infty$-esssupp of $a,b,f,e$ and such that $\operatorname{WF}'_{\mathrm{leC}}(G)$ is disjoint from $\calR_+$ and satisfies $\operatorname{WF}'_{\mathrm{leC}}(G)\subset \operatorname{Ell}_{\mathrm{leC}}^{0,0,0,0,0}(G_2)$. (Both $a,f$ are supported away from $\calR_+$ and $\operatorname{Char}_{\mathrm{leC}}^{0,0,0,0,0}(G_2)$, so such a $G$ exists.) Then (for $K$ large enough):
	\begin{itemize}
		\item Writing $\tilde{P} = (1-G)\tilde{P} + G \tilde{P}$  and noting that 
		\begin{equation}
			(1-G)A_\bullet \in L^\infty([0,1]_\varepsilon;\Psi_{\mathrm{leC}}^{-\infty,-\infty,-\infty,-\infty,-\infty}(X)), 
		\end{equation} 
		we have, for each $N\in \bbN$, 
		\begin{align}
			\begin{split} 
			2| \langle \tilde{P}u,A_\varepsilon u \rangle_{L^2}| &\leq 2|\langle G \tilde{P}u ,A_\varepsilon u \rangle_{L^2} | +2|\langle  \tilde{P} u,(1-G)A_\varepsilon u \rangle_{L^2} | \\
			&\preceq 2^{-1}\overline{\delta}^{-1}\lVert G \tilde{P} u \rVert^2_{\calY_N} + 2 \overline{\delta} \lVert \Lambda_{1/2,-s-1,-\varsigma-3,-l-1,-\ell-3} A_\varepsilon u \rVert^2_{L^2} + \lVert u \rVert^2_{\calE_{N}}, 
			\label{eq:misc_lz2}
			\end{split} 
		\end{align}
		for any $\overline{\delta}>0$, where the constant in \cref{eq:misc_lz2} is independent of $\overline{\delta}$. We have abbreviated $\calE_N = H_{\mathrm{leC}}^{-N,-N,-N,-N,-N}(X)$, 
		\begin{equation}
			\calY_N = H_{\mathrm{leC}}^{-N,s+1,\varsigma+3,l+1,\ell+3}(X).
		\end{equation}
		We also set $\calY_{*,N} = H_{\mathrm{leC}}^{-N,-(s+1),-(\varsigma+3),-(l+1),-(\ell+3)}(X)$ (so dual in all orders except that at $\mathrm{df}$).
		\item 
		Similarly, we can chose self-adjoint $\bar{G}_3 \in \Psi_{\mathrm{leC}}^{-\infty,0,0,-\infty,-\infty}(X)$ with 
		\begin{equation} 
			\operatorname{WF}^{'0,0}_{\mathrm{leC}}(1-\bar{G}_3) \cap \operatorname{WF}'_{L^\infty,\mathrm{leC}}(E_\bullet) = \varnothing,
		\end{equation} 
		$\operatorname{WF}'_{\mathrm{leC}}(\bar{G}_3) \cap \operatorname{Char}_{\mathrm{leC}}^{2,0,-2,-1,-3}(\tilde{P}) \subseteq \calP[\Theta_1,\Theta_2]$, and $\operatorname{WF}'_{\mathrm{leC}}(\bar{G}_3) \subseteq \operatorname{Ell}_{\mathrm{leC}}^{0,0,0,0,0}(G_2)$. Then 
		\begin{align}
			\begin{split} 
			| \langle u,E_\varepsilon u \rangle_{L^2}| &\leq |\langle \bar{G}_3 u ,E_\varepsilon u \rangle_{L^2} | +|\langle   u,(1-\bar{G}_3)E_\varepsilon u \rangle_{L^2} | \\
			&\preceq \lVert \bar{G}_3  u \rVert^2_{\calX_N} +   \lVert E_\varepsilon u \rVert^2_{\calX_{N}^*} + \lVert u \rVert^2_{\calE_{N}}, \\
			&\preceq \lVert \bar{G}_3  u \rVert^2_{\calX_N} +   \lVert E_\varepsilon u \rVert^2_{\calX_{*,N}} + \lVert u \rVert^2_{\calE_{N}},
			\label{eq:misc_lz5}
			\end{split} 
		\end{align}
		where $\calX_N = H_{\mathrm{leC}}^{-N,s,\varsigma,l,\ell}(X)$, $\calX_{*,N} = H_{\mathrm{leC}}^{-N,-s,-\varsigma,-l,-\ell}(X)$.  
		
		The bound $\lVert E_\varepsilon u \rVert^2_{\calX_{N}^*}  \preceq  \lVert E_\varepsilon u \rVert^2_{\calX_{*,N}} + \lVert u \rVert^2_{\calE_{N}}$ follows (using \cref{eq:misc_efq}) from the construction via $\operatorname{Op}$ of $H_1 \in \Psi_{\mathrm{leC}}^{-1,0,0,0,0}(X)$ and $H_2 \in \Psi_{\mathrm{leC}}^{0,0,0,0,0}(X)$ such that 
		\begin{equation} 
			\operatorname{WF}'_{L^\infty,\mathrm{leC}}(E_\bullet) \cap \operatorname{WF}^{'0,0}_{\mathrm{leC}}(H_2) = \varnothing
		\end{equation} 
		and $1=H_1+H_2$. Then, we can compute  
		\begin{align}
			\begin{split} 
			\lVert E_\varepsilon u \rVert^2_{\calX_{N}^*}  &\leq \lVert H_1 E_\varepsilon u \rVert^2_{\calX_{N}^*} +\lVert H_2 E_\varepsilon u \rVert^2_{\calX_{N}^*}  \\
			&\preceq \lVert  E_\varepsilon u \rVert^2_{H_{\mathrm{leC}}^{N-1,s,\varsigma,l,\ell} } +\lVert H_2 E_\varepsilon u \rVert^2_{\calX_{N}^*} \\
			&\preceq \lVert  E_\varepsilon u \rVert^2_{H_{\mathrm{leC}}^{N-1,s,\varsigma,l,\ell}} +\lVert  u \rVert^2_{\calE_N}. 
			\end{split} 
		\end{align}
		Proceeding inductively, we deduce $\lVert E_\varepsilon u \rVert^2_{\calX_{N}^*}  \preceq  \lVert E_\varepsilon u \rVert^2_{\calX_{*,N}} + \lVert u \rVert^2_{\calE_{N}}$. 
		This argument will be used below without further comment.
		\item 
		\begin{align}
			\begin{split} 
			| \langle \tilde{P}u,F_\varepsilon u \rangle_{L^2}| &\leq |\langle G \tilde{P}u ,F_\varepsilon u \rangle_{L^2} | +|\langle  \tilde{P} u,(1-G)F_\varepsilon u \rangle_{L^2} | \\
			&\preceq \lVert G \tilde{P} u \rVert^2_{\calY_N} +   \lVert F_\varepsilon u \rVert^2_{\calY_{*,N}} + \lVert u \rVert^2_{\calE_{N}}.
			\end{split} 
			\intertext{We have $\operatorname{WF}'_{L^\infty,\mathrm{leC}}(F_\bullet) \cap \operatorname{Char}^{2,0,-2,-1,-3}(G\tilde{P})=\varnothing$ (since, by \cref{eq:misc_spp}, $f_\bullet$ is supported away from the characteristic set), so we can deduce via the elliptic parametrix construction that} 
			\begin{split} 
			\lVert F_\varepsilon u \rVert_{\calY_{*,N}}\preceq \lVert F_\varepsilon u \rVert_{\calY^*_N} &\preceq \lVert G \tilde{P} u \rVert_{H_{\mathrm{leC}}^{-N,s-1,\varsigma-1,-N,-N}} + \lVert u \rVert_{\calE_N} \\
			&\preceq \lVert G\tilde{P} u \rVert_{\calY_N}^2 + \lVert u \rVert_{\calE_N}^2 
			\end{split} 
			\intertext{for $N$ sufficiently large, so}
			|\langle \tilde{P}u,F_\varepsilon u \rangle_{L^2}| &\preceq \lVert G \tilde{P} u \rVert^2_{\calY_N}  + \lVert u \rVert^2_{\calE_{N}}. \label{eq:misc_lz3}
		\end{align}
		\item 
		Writing $R_\varepsilon = (1-G^2)R_\varepsilon + G^2 R_\varepsilon$, since $1-G^2=(1-G)(1+G)$ implies $\operatorname{WF}^{'0,0}_{L^\infty,\mathrm{leC}}(1-G^2) \subset \operatorname{WF}^{'0,0}_{L^\infty,\mathrm{leC}}(1- G)$, for each $N\in \bbN$ we have 
		\begin{align}
			\begin{split} 
			|\langle R_\varepsilon u,u \rangle_{L^2}  |\preceq ( \lVert G R_\varepsilon u \rVert_{\calZ_{*,N}}\lVert Gu \rVert_{\calZ_N}+ \lVert u \rVert_{\calE_N}^2)&\preceq ( \lVert G R_\varepsilon u \rVert_{\calZ_{*,N}}^2+\lVert Gu \rVert_{\calZ_N}^2+ \lVert u \rVert_{\calE_N}^2)  \\
			&\preceq ( \lVert  \tilde{R}_\varepsilon Gu \rVert_{\calZ_{*,N}}^2+\lVert G u \rVert_{\calZ_N}^2+ \lVert u \rVert_{\calE_N }^2)  \\
			&\preceq ( \lVert G u \rVert_{\calZ_N}^2+ \lVert u \rVert_{\calE_N}^2) \label{eq:misc_lz4}
			\end{split} 
		\end{align}
		for some $\tilde{R}_\bullet \in L^\infty([0,1]_\varepsilon;\Psi_{\mathrm{leC}}^{-\infty,2s-1,2\varsigma-1,2l,2\ell}(X))$, where
		\begin{align} 
			\begin{split}
			\calZ_N& =H_{\mathrm{leC}}^{-N,(2s-1)/2, (2\varsigma-1)/2, l,\ell}(X), \\
			\calZ_{*,N}&=H_{\mathrm{leC}}^{-N,-(2s-1)/2, -(2\varsigma-1)/2, -l,-\ell}(X) .
			\end{split}
		\end{align}  
	\end{itemize}
	Combining \cref{eq:misc_lz1}, \cref{eq:misc_lz2} with $\overline{\delta}$ sufficiently small, \cref{eq:misc_lz3}, \cref{eq:misc_lz4}, we have proven that 
	\begin{align}
		\lVert B_\varepsilon u \rVert_{L^2}^2 &\preceq \lVert G\tilde{P} u \rVert_{\calY_N}^2 + \lVert G u \rVert_{\calZ_N}^2 + \lVert \bar{G}_3 u \rVert_{\calX_N}^2 + \lVert E_\varepsilon u \rVert^2_{\calX_{*,N}} + \lVert u \rVert_{\calE_N}^2, 
		\intertext{i.e.\ }
		\lVert \hat{B}_\varepsilon u \rVert_{\calX_N}^2 &\preceq \lVert G\tilde{P} u \rVert_{\calY_N}^2 + \lVert G u \rVert_{\calZ_N}^2 + \lVert \bar{G}_3 u \rVert_{\calX_N}^2+ \lVert \hat{E}_\varepsilon u \rVert^2_{\calX_N} + \lVert u \rVert_{\calE_N}^2, \\
		\lVert \hat{B}_\varepsilon u \rVert_{\calX_N} &\preceq \lVert G\tilde{P} u \rVert_{\calY_N} + \lVert G u \rVert_{\calZ_N} + \lVert \bar{G}_3 u \rVert_{\calX_N} + \lVert \hat{E}_\varepsilon u \rVert_{\calX_{N}} + \lVert u \rVert_{\calE_N}, 
		\label{eq:misc_9iz}
	\end{align}
 	where $\hat{B}_\varepsilon = \Lambda_{0,-s,-\varsigma,0,0} B_\varepsilon$ and $\hat{E}_\varepsilon = \Lambda_{0,-2s,-2\varsigma,0,0} E_\varepsilon$.
	Hence, $\hat{B}_\bullet,\hat{E}_\bullet \in \Psi_{\mathrm{leC}}^{0,0,0,0,0}(X)$. 
	By the choice of $\psi$, $\operatorname{WF}'_{L^\infty,\mathrm{leC}}(B_\bullet), \operatorname{WF}'_{L^\infty,\mathrm{leC}}(E_\bullet) \subseteq \operatorname{Ell}_{\mathrm{leC}}^{0,0,0,0,0}(G_2)$. 
	
	Let 
	$\bar{\calX}_N = H_{\mathrm{leC}}^{-N,s,\varsigma,-N,-N}(X)$
	Since  $\operatorname{WF}_{L^\infty,\mathrm{leC}}'(E_\bullet) \cap \operatorname{Char}_{\mathrm{leC}}^{2,0,-2,-1,-3}(\tilde{P})\subset \calP[\Theta_1,\Theta_2]$, 
	\begin{equation}
		\operatorname{WF}_{L^\infty,\mathrm{leC}}'(E_\bullet) \subset \operatorname{Ell}_{\mathrm{leC}}^{2,0,-2,-1,-3}(G\tilde{P}) \cup \operatorname{Ell}_{\mathrm{leC}}^{0,0,0,0,0}(G_3).
	\end{equation}
	Consequently, $\lVert \hat{E}_\varepsilon u \rVert_{\calX_N} \preceq \lVert G \tilde{P} u \rVert_{\bar{\calY}_N}+ \lVert G_3 u \rVert_{\bar{\calX}_N} + \lVert u \rVert_{\calE_N}$
	for sufficiently large $N$, where $\bar{\calY}_N = H_{\mathrm{leC}}^{-N,s+1,\varsigma+3,-N,-N}(X)$. Likewise,  
	\begin{equation} 
		\lVert \bar{G}_3 u \rVert_{\calX_N} \preceq \lVert G \tilde{P} u \rVert_{\bar{\calY}_N}+ \lVert G_3 u \rVert_{\bar{\calX}_N} + \lVert u \rVert_{\calE_N}.
	\end{equation}

	Let $\bar{\calZ}_N = H_{\mathrm{leC}}^{-N, s-1/2,\varsigma-1/2,-N,-N}(X)$.
	Because $\operatorname{WF}'_{\mathrm{leC}}(G) \subseteq \operatorname{Ell}_{\mathrm{leC}}^{0,0,0,0,0}(G_2)$, 
	\begin{equation} 
		\lVert G\tilde{P} u \rVert_{\bar{\calY}_N} \preceq \lVert G_2\tilde{P} u \rVert_{\bar{\calY}_N} + \lVert u \rVert_{\calE_N}
	\end{equation} 
	and $\lVert G u \rVert_{\calZ_N} \preceq \lVert G_2 u \rVert_{\bar{\calZ}_N} + \lVert u \rVert_{\calE_N}$ for sufficiently large $N$. We have therefore shown that 
	\begin{equation} 
		\lVert \hat{B}_\varepsilon u \rVert_{\calX_N} \preceq \lVert G_2\tilde{P} u \rVert_{\bar{\calY}_N} + \lVert G_2 u \rVert_{\bar{\calZ}_N} + \lVert G_3 u \rVert_{\bar{\calX}_N} + \lVert u \rVert_{\calE_N}. 
	\end{equation} 
	Thus, for each $\sigma\geq 0$, $\hat{B}_\varepsilon u(-;\sigma) $ is uniformly bounded in $\calX_N(\sigma)$ as $\varepsilon\to 0^+$. 
	
	For each $\sigma\geq 0$, given any sequence $\{\varepsilon_k\}_{k\in \bbN}\subset (0,1]$ with $\varepsilon_k\to 0$ as $k\to\infty$, there exists -- by the Banach--Alaoglu theorem -- a subsequence $\varepsilon_{k_\kappa}$ thereof such that 
	\begin{equation} 
		\{\smash{\hat{B}}_{\varepsilon_{k_\kappa}}(\sigma) u(-;\sigma)\}_{\kappa \in \bbN} \subset \calX_N(\sigma) 
	\end{equation}  
	is weakly convergent in the scb-Sobolev space $\calX_N(\sigma)$ and in fact in any closed ball in $\calX_N(\sigma)$ in which $\smash{\hat{B}_{\varepsilon_k} u(-;\sigma)}$ is eventually contained. Call the weak limit $v = v(N,\sigma,\{\varepsilon_{k_\kappa}\}_{\kappa \in \bbN})\in \calX_N(\sigma)$. The preceding clause means that 
	\begin{equation} 
		\lVert v \rVert_{\calX_N(\sigma)}\leq \limsup_{\kappa\to\infty} \lVert \hat{B}_{\varepsilon_{k_\kappa}}(\sigma) u(-;\sigma) \rVert_{\calX_N(\sigma)}. 
		\label{eq:misc_jhq}
	\end{equation} 
	The family  $\{B_\varepsilon(\sigma)\}_{\varepsilon \in [0,1]}$ was constructed so that it is continuous down to $\varepsilon=0$ with respect to the topology of some space of high order $\Psi$DOs. (This follows from the analogous observation for $b_\bullet(\sigma)$ and the continuity of the quantization map.) Consequently,
	\begin{equation} 
		\hat{B}_{\varepsilon_k}(\sigma) u(-;\sigma) \to \hat{B}_0(\sigma) u(-;\sigma)
	\end{equation} 
	in the topology of $\calS'(X)$ as $k\to\infty$. But $\hat{B}_{\varepsilon_{k_\kappa}}(\sigma)u(-;\sigma)\to v$ in $\calS'(X)$, so $v = \hat{B}_0(\sigma) u(-;\sigma)$. The sequence $\{\varepsilon_k\}_{k\in \bbN}$ was arbitrary, so we can actually conclude from  \cref{eq:misc_jhq} that 
	\begin{equation} 
		\lVert \hat{B}_{0}(\sigma) u(-;\sigma) \rVert_{\calX_N(\sigma)}\leq  \liminf_{\varepsilon\to0^+} \lVert \hat{B}_\varepsilon u(-;\sigma) \rVert_{\calX_N(\sigma)} .
	\end{equation} 
	This applies for each $\sigma\geq 0$, so
	\begin{align} 
		\lVert \hat{B}_0 u \rVert_{\calX_N} &\preceq \lVert G_2\tilde{P} u \rVert_{\bar{\calY}_N} + \lVert G_2 u \rVert_{\bar{\calZ}_N} + \lVert G_3 u \rVert_{\bar{\calX}_N} + \lVert u \rVert_{\calE_N}. 
		\intertext{Since $\varphi$ is nonvanishing on $[\Theta_3,\Theta_4]$, we have $\operatorname{Ell}^{0,0,0,0,0}_{\mathrm{leC}}(\hat{B}_0) \supseteq \calP[\Theta_3,\Theta_4] \supseteq \operatorname{Char}_{\mathrm{leC}}^{2,0,-2,-1,-3}(\tilde{P}) \cap \operatorname{WF}'_{\mathrm{leC}}(G_1)$, so (via elliptic regularity) $\lVert G_1 u \rVert_{\calX}\preceq \lVert \hat{B}_0 u \rVert_{\calX_N}+ \lVert G_2 \tilde{P} u \rVert_{\bar{\calY}_N}+ \lVert u \rVert_{\calE_N}$, where $\calX=H_{\mathrm{leC}}^{m,s,\varsigma,l,\ell}(X)$. This yields} 
		\lVert G_1 u \rVert_{\calX} &\preceq \lVert G_2\tilde{P} u \rVert_{\bar{\calY}_N} + \lVert G_2 u \rVert_{\bar{\calZ}_N} + \lVert G_3 u \rVert_{\bar{\calX}_N} + \lVert u \rVert_{\calE_N}. 
		\label{eq:misc_olo}
	\end{align}
	Observe that the leC-Sobolev space $\calZ_N$ is lower order than $\calX$ at sf and ff. Since the leC- Sobolev spaces $\calX_N,\calY_N$ get bigger as $s,\varsigma$ decrease, an inductive argument (which we can carry out because \cref{eq:misc_cas}) upgrades \cref{eq:misc_olo} to \cref{eq:propagation_estimate}. 
\end{proof}

Since $H_{\tilde{p}}^{2,0,-2}$, viewed as a vector field on ${}^{\mathrm{sc,leC}} T^*X$, vanishes at the two radial sets, in order to carry out a positive commutator argument we must take into account the previously negligible radial component 
\begin{equation}
	(H_{{\tilde{p}_0}}^{-,0,-2} x) \partial_x = 2( \xi_{\mathrm{sc,leC}} - 1 ) x \partial_x
\end{equation}
of the rescaled flow. For $\xi_{\mathrm{sc,leC}} = 0,2$, this is $\pm 2 x \partial_x$, which (projecting down to $X$) is the ur-example of a nondegenerate radial b-vector field on $X$. Moreover, $\calR_+$ is a source in the radial direction (as well as the other directions, as seen earlier), so $\calR_+$ is a source for the Hamiltonian flow in all directions. (Similarly, $\calR_0$ is a sink in all directions.) It is therefore straightforward to prove a radial point estimate at $\calR_+$. 

To begin:
\begin{proposition}
	\label{prop:basic_weight_comp} 
	There exist $\beta_1,\beta_2 \in S_{\mathrm{leC}}^{0,0,0,0,0}(X) $ such that 
	\begin{align}
		H_{\tilde{p}}^{2,0,-2} x &= \beta_1 x \label{eq:misc_815}\\
		H_{\tilde{p}}^{2,0,-2} (\sigma^2+\mathsf{Z}x)^{1/2} &= \beta_2 (\sigma^2+\mathsf{Z}x)^{1/2}, \label{eq:misc_816} 
	\end{align}
	with $\beta_1,\beta_2>0$ on $\calR_+$.
\end{proposition}
\begin{proof}
	It suffices to work near $\partial X$ and to consider $\tilde{p}_0$ in place of $\tilde{p}$. Applying \Cref{prop:Hocompe}, 
	\begin{align}
		H_{\tilde{p}_0}^{-,0,-2} x &= 2 (\xi_{\mathrm{sc,leC}} - 1) x, \\
		H_{\tilde{p}_0}^{-,0,-2}(\sigma^2+\mathsf{Z} x)^{1/2} &=  (\xi_{\mathrm{sc,leC}} - 1) \mathsf{Z} x (\sigma^2+\mathsf{Z} x)^{-1/2}. 
	\end{align} 
	So, defining $\beta_1,\beta_2$ by \cref{eq:misc_815} and \cref{eq:misc_816}, $\beta_1,\beta_2 \in S_{\mathrm{cl,leC}}^{0,0,0,0,0}(X)$ and are given at $\{x=0\}$ by $\beta_1 = 2 \varrho_{\mathrm{df}} (\xi_{\mathrm{sc,leC}} - 1)$ and $\beta_2 = \varrho_{\mathrm{df}} (\xi_{\mathrm{sc,leC}} -1) \mathsf{Z}x (\sigma^2+\mathsf{Z}x)^{-1}$. Thus, $\beta_1,\beta_2>0$ on $\calR_+$. 
\end{proof}

\begin{proposition}
	\label{prop:R+}
	Suppose that $G_1,G_2,G_3 \in \Psi_{\mathrm{leC}}^{-\infty,0,0,-\infty,-\infty}(X)$ satisfy 
	\begin{enumerate}
		\item $\operatorname{WF}'_{\mathrm{leC}}(G_1)\subseteq \operatorname{Ell}_{\mathrm{leC}}^{0,0,0,0,0}(G_2)$, 
		\item $\calR_+ \subset \operatorname{Ell}_{\mathrm{leC}}^{0,0,0,0,0}(G_3), \operatorname{Ell}_{\mathrm{leC}}^{0,0,0,0,0}(G_1)$,
		\label{it:G3_prop}
		\item there exists some $\Theta \in (0,\pi)$ such that 
		\begin{equation} 
		\operatorname{WF}'_{\mathrm{leC}}(G_1) \cap \operatorname{Char}_{\mathrm{leC}}^{2,0,-2,-1,-3}(\tilde{P}) \subseteq \calR_+ \cup \bigcup_{0<\Theta'<\Theta} \calP [\Theta',\Theta] \subseteq  \operatorname{Ell}_{\mathrm{leC}}^{0,0,0,0,0}(G_2). 
		\label{eq:misc_ju1}
		\end{equation} 
	\end{enumerate}
	Then, for every $\Sigma>0$ and $N\in \bbN$, $m,s,\varsigma,l,\ell,s_0,\varsigma_0 \in \bbR$ such that $s>s_0>-1/2$ and $\varsigma> \varsigma_0>-3/2$,  there exists some constant 
	\begin{equation} 
	C = C(\tilde{P},G_1,G_2,\Sigma,m,s,\varsigma,l,\ell,s_0,\varsigma_0,N)>0 
	\label{eq:misc_lk1}
	\end{equation} 
	such that, for all $u\in \calS'(X)$, 
	\begin{equation}
		\lVert G_1 u \rVert_{H_{\mathrm{leC}}^{m,s,\varsigma,l,\ell}} \leq C\Big[ \lVert G_2 \tilde{P} u \rVert_{H_{\mathrm{leC}}^{-N,s+1,\varsigma+3,-N,-N}} + \lVert G_3 u \rVert_{H_{\mathrm{leC}}^{-N,s_0,\varsigma_0,-N,-N}} + \lVert u \rVert_{H_{\mathrm{leC}}^{-N,-N,-N,-N,-N}} \Big]
		\label{eq:misc_bbb}
	\end{equation}
	holds (in the usual strong sense, i.e.\ the left-hand side is finite if the right-hand side is) for all $\sigma \in [0,\Sigma]$. 
\end{proposition}

\begin{proof}
	The symbolic constructions will only be specified for $\sigma \in [0,\Sigma]$, which is evidently unproblematic. 
	Also, it suffices to take $N$ to be sufficiently large such that any of the finitely many functions of $m,s,\varsigma,l,\ell$ that arise below can be bounded below by $-N$. 
	
	We can assume without loss of generality that  $\operatorname{WF}'_{\mathrm{leC}}(G_3)  \subseteq \operatorname{Ell}_{\mathrm{leC}}^{0,0,0,0,0}(G_2)$ and 
	\begin{equation} 
		\operatorname{WF}'_{\mathrm{leC}}(G_3) \cap \operatorname{Char}_{\mathrm{leC}}^{2,0,-2,-1,-3}(\tilde{P}) \subseteq \calR_+\cup  \cup_{0<\Theta'<\Theta} \calP [\Theta',\Theta] \subseteq \operatorname{Ell}_{\mathrm{leC}}^{0,0,0,0,0}(G_2).
		\label{eq:misc_inq}
	\end{equation}

	We  read off \Cref{prop:basic_weight_comp} that $\sup_{\calR_+} \beta_1, \sup_{\calR_+} \beta_2 >0$.

	Now chose nonnegative $\rho \in S_{\mathrm{cl,leC}}^{0,0,0,0,0}(X)$ equal to $(\xi_{\mathrm{sc,leC}}-2)^2 + \eta_{\mathrm{sc,leC}}^2$ in some neighborhood of $\calR_+ = \{\xi_{\mathrm{sc,leC}} = 2, \eta_{\mathrm{sc,leC}} = 0\} \subset \mathrm{sf}\cup \mathrm{ff}$. 
	
	There exist some symbols, which we call $\tilde{\beta}_0, \tilde{F}_2, \tilde{F}_3,\tilde{F}_4 \in S_{\mathrm{leC}}^{0,0,0,0,0}(X)$ 
	such that 
	\begin{align}
		\begin{split} 
		H_{\tilde{p}}^{2,0,-2}\rho &= \tilde{\beta}_0 \rho + \tilde{F}_2 + \tilde{F}_3 + x (\sigma^2+\mathsf{Z} x)^{-1/2} \tilde{F}_4 \\ 
		\inf \tilde{\beta}_0|_{\calR_+}&>0, 
		\label{eq:misc_tb0}
		\end{split} 
	\end{align}
	$\tilde{\beta}_0,\tilde{F}_2\geq 0$ everywhere, and $\tilde{F}_3$ vanishes cubically at $\calR_+$ (uniformly in $\sigma$).  This computation is completely analogous to the one in the proof of \Cref{prop:radial_comp}. 
	We now consider the weight (for to-be-decided $l',\ell'\in \bbR$)
	\begin{equation}
		a_0  = x^{-l'} (\sigma^2 + \mathsf{Z} x)^{-\ell'/2+l'} \in S_{\mathrm{cl,leC}}^{0,l',\ell',l',\ell'}(X).
		\label{eq:a0def}
	\end{equation}
	Then, the symbol $\beta \in S_{\mathrm{leC}}^{0,0,0,0,0}(X)$ defined by
	\begin{equation}
		\beta = -l' \beta_1 + (2l'-\ell') \beta_2 = a_0^{-1} H_{\tilde{p}}^{2,0,-2} a_0 
	\end{equation}
	has a definite sign near $\calR_+$ if $l',-l'+\ell'/2\neq 0$ and have the same sign.  Using the explicit formula for $\beta_1,\beta_2$ in the proof of \Cref{prop:basic_weight_comp}, 
	\begin{equation} 
		\beta  = -\varrho_{\mathrm{df}} \Big( 2l'  + (\ell'-2l') \frac{\mathsf{Z}x}{\sigma^2+\mathsf{Z}x} \Big)(\xi_{\mathrm{sc,leC}}-1)
	\end{equation} 
	at $\{x=0\}$. 
	Negativity on $\calR_+ \cap \mathrm{sf}$ requires that $l'>0$. negativity on $\calR_+\cap \mathrm{ff}$ requires that $\ell'>0$. And this suffices; for $l',\ell'>0$, $\beta<0$ in some neighborhood of $\calR_+$.

	There exists $\chi \in C_{\mathrm{c}}^\infty(\bbR)$ such that $-\operatorname{sgn}(t) \chi'(t)\chi(t) = \chi_0^2(t)$ for some $\chi_0\in C_{\mathrm{c}}^\infty(\bbR)$ and such that $\chi=1$ identically in some neighborhood of the origin. (The construction is standard and uses a translate of $\exp(-1/t)$.) 
	Replacing $\chi$ with $\chi\circ \operatorname{dil}_\lambda = \chi(\lambda\bullet)$ for sufficiently large $\lambda =\lambda(l',\ell',\chi)$ if necessary,  choose $\chi$  such that
	\begin{equation}
		\beta |_{\operatorname{supp}  \chi(\rho)}<0 
		\label{eq:pos_cond}
	\end{equation}
	and such that $\operatorname{supp} \chi(\tilde{p}^{2,0,-2}) \chi(\rho)$ is disjoint from $\mathrm{df}\cup \mathrm{bf}\cup \mathrm{tf}$, where $\tilde{p}^{2,0,-2} = (\sigma^2+\mathsf{Z}x)^{-1} \tilde{p}$. 
	
	Choose $\psi\in S_{\mathrm{cl,leC}}^{0,0,0,0,0}(X)$ such that $\psi$ is identically equal to one in some neighborhood of $\{x=0\}$ and such that the formula $ (H^{2,0,-2}_{\tilde{p}} \rho)^{1/2} \psi  \chi_0(\rho)$ defines a  symbol:
	\begin{equation}
		 \chi_0(\rho) \psi \sqrt{H^{2,0,-2}_{\tilde{p}} \rho} \in S_{\mathrm{leC}}^{0,0,0,0,0}(X).
	\end{equation}
	(The existence of such a $\psi$ follows from \cref{eq:misc_tb0}.)
	Now set 
	\begin{equation}
		a = a_0 \psi^2 \chi(\tilde{p}^{2,0,-2})^2 \chi(\rho)^2  \in S_{\mathrm{cl,leC}}^{-\infty,l',\ell',-\infty,-\infty}(X).
	\end{equation}
	We compute 
	\begin{multline}
		H_{\tilde{p}}^{2,0,-2} a =   \psi^2 ( \chi(\tilde{p}^{2,0,-2})^2 \chi(\rho)^2  H_{\tilde{p}}^{2,0,-2} a_0 - 2 a_0 \chi_0(\rho)^2 \chi(\tilde{p}^{2,0,-2})^2 H^{2,0,-2}_{\tilde{p}}\rho  \\ + 2a_0 \chi(\tilde{p}^{2,0,-2})\chi'(\tilde{p}^{2,0,-2}) \chi(\rho)^2 \tilde{p}^{2,0,-2} \tilde{q})   + 2 a_0  \chi(\tilde{p}^{2,0,-2})^2 \chi(\rho)^2 \psi H_{\tilde{p}}^{2,0,-2} \psi .
		\label{eq:misc_hf2}
	\end{multline}
	Here $\tilde{q} \in S_{\mathrm{cl,leC}}^{0,0,0,0,0}(X)$ is defined by $\tilde{q} = (\sigma^2+\mathsf{Z}x) H_{\tilde{p}}^{2,0,-2} (\sigma^2+\mathsf{Z}x)^{-1}$, so that 
	\begin{equation}
		H_{\tilde{p}}^{2,0,-2}\tilde{p}^{2,0,-2} = \tilde{q} \tilde{p}^{2,0,-2}. 
	\end{equation}
	In terms of $\beta$, \cref{eq:misc_hf2} says 
	\begin{multline}
		H_{\tilde{p}}^{2,0,-2} a =  \psi^2( \chi(\tilde{p}^{2,0,-2})^2 \chi(\rho)^2  \beta  a_0 - 2 a_0 \chi_0(\rho)^2 \chi(\tilde{p}^{2,0,-2})^2 H^{2,0,-2}_{\tilde{p}}\rho  \\ + 2a_0 \chi(\tilde{p}^{2,0,-2})\chi'(\tilde{p}^{2,0,-2}) \chi(\rho)^2 \tilde{p}^{2,0,-2} \tilde{q}  ) + 2 a_0  \chi(\tilde{p}^{2,0,-2})^2 \chi(\rho)^2 \psi H_{\tilde{p}}^{2,0,-2} \psi .
		\label{eq:hform2}
	\end{multline}
	The first two terms have definite sign (the same sign for $l',\ell'>0$), while the third and fourth terms are unproblematic (being controllable by elliptic or propagation estimates).  
	
	Set $\phi_\varepsilon = (1+\varepsilon x^{-1})^{-K_1} (1+\varepsilon (\sigma^2+\mathsf{Z}x)^{-1/2})^{-K_2}$ and $a^{(\varepsilon)} = \phi_\varepsilon^2 a  \in L^\infty([0,1]_\varepsilon;S_{\mathrm{cl,leC}}^{-\infty,l',\ell',-\infty,-\infty}(X))$
	for to-be-decided $K_1,K_2\in \bbR$. 
	The replacement for \cref{eq:hform2} is 
	\begin{multline}
		H_{\tilde{p}}^{2,0,-2} a^{(\varepsilon)}  = \phi_\varepsilon^2 \Big[ 
		-\chi(\tilde{p}^{2,0,-2})^2 \chi(\rho)^2 a_0 \psi^2 \Big( \beta_1 \Big(l'- \frac{K_1 \varepsilon x^{-1}}{1+\varepsilon x^{-1}}\Big)  +  \beta_2 \Big(\ell'-2l'- \frac{K_2\varepsilon (\sigma^2+\mathsf{Z}x)^{-1/2}}{1+\varepsilon (\sigma^2+\mathsf{Z}x)^{-1/2}} \Big) \Big)  \\  - 2 a_0 \chi_0(\rho)^2 \chi(\tilde{p}^{2,0,-2})^2 \psi^2 H_{\tilde{p}}^{2,0,-2} \rho  +2 a_0 \chi'(\tilde{p}^{2,0,-2})\chi(\tilde{p}^{2,0,-2}) \chi(\rho)^2 \psi^2 \tilde{p}^{2,0,-2} \tilde{q} \\ +
		2 a_0 \chi(\tilde{p}^{2,0,-2})^2\chi(\rho)^2 \psi H_{\tilde{p}}^{2,0,-2}\psi 
		\Big]. 
		\label{eq:misc_01l}
	\end{multline}
	Rewriting the first parenthetical,
	\begin{multline}
		\beta_1 \Big(l'- \frac{K_1 \varepsilon x^{-1}}{1+\varepsilon x^{-1}}\Big)  +  \beta_2 \Big(\ell'-2l'- \frac{K_2\varepsilon (\sigma^2+\mathsf{Z}x)^{-1/2}}{1+\varepsilon (\sigma^2+\mathsf{Z}x)^{-1/2}} \Big) \\ = \varrho_{\mathrm{df}} \Big(2\Big(l'- \frac{K_1 \varepsilon x^{-1}}{1+\varepsilon x^{-1}}\Big) + \frac{\mathsf{Z}x}{\sigma^2+\mathsf{Z}x}\Big(\ell'-2l'- \frac{K_2\varepsilon (\sigma^2+\mathsf{Z}x)^{-1/2}}{1+\varepsilon (\sigma^2+\mathsf{Z}x)^{-1/2}} \Big)  \Big) (\xi_{\mathrm{sc,leC}} - 1)
		\label{eq:misc_k31}
	\end{multline}
	at $\calR_+$. 
	So, we will require that 
	\begin{equation}
		K_1 < l', \qquad K_2  < \ell',
		\label{eq:reg_ineq}
	\end{equation}
	and then the quantity in \cref{eq:misc_k31} is positive in some neighborhood of $\calR_+$. 
	Thus, only a limited amount of ``regularization'' can be performed. This is a standard technicality, and we can deal with it via citing the standard arguments used to handle it elsewhere --- see \cite{VasyGrenoble}. We do not even need to worry about uniformity: we can justify the formal integrations-by-parts below $\sigma$-wise, by citing essentially verbatim the arguments in \cite{VasyGrenoble} for the $\sigma>0$ and applying the argument with $X_{1/2}$ in place of $X$ to handle the $\sigma=0$ case. (In fact, since the $\sigma=0$ case of the proposition follows from the estimates in \cite{VasyLA} applied on $X_{1/2}$, to prove the proposition here it suffices to prove estimates that are uniform as $\sigma\to 0^+$, and thus to restrict attention to the $\sigma>0$ case, for which we can take $K_2=0$ and apply  \cite{VasyGrenoble} essentially verbatim.)
	
	Given that the inequalities \cref{eq:reg_ineq} are satisfied, we can (perhaps dilating $\chi$ or shrinking the support of $\psi$ if necessary) find 
	\begin{equation}
	\underline{\delta}  = \underline{\delta}(K_1,K_2,l',\ell',\chi)>0
	\end{equation} 
	sufficiently small such that  	
	there exist uniform families of leC-symbols 
	\begin{align}
		\begin{split} 
		b_\bullet &\in  L^\infty([0,1]_\varepsilon ; S_{\mathrm{leC}}^{-\infty,(l'-1)/2,(\ell'-3)/2,-\infty,-\infty}(X)),\\  
		e_\bullet &\in  L^\infty([0,1]_\varepsilon ; S_{\mathrm{leC}}^{-\infty,(l'-1)/2,(\ell'-3)/2,-\infty,-\infty}(X)),\\  
		f_\bullet &\in L^\infty([0,1]_\varepsilon ; S_{\mathrm{leC}}^{-\infty,l'-1,\ell'-1,-\infty,-\infty}(X)), \\
		r_\bullet &\in L^\infty([0,1]_\varepsilon ; S_{\mathrm{leC}}^{-\infty,-\infty,-\infty,-\infty,-\infty}(X)).
		\end{split} 
	\end{align} 
	given by 
	\begin{align}
		\begin{split} 
			b_{\varepsilon} &= \varrho_{\mathrm{df}}^{-1/2} \varrho_{\mathrm{sf}}^{1/2} \varrho_{\mathrm{ff}}^{3/2} \varrho_{\mathrm{bf}}^{1/2} \varrho_{\mathrm{tf}}^{3/2}  a_0^{1/2} \chi(\tilde{p}^{2,0,-2}) \chi(\rho) \phi_\varepsilon \psi \Big[-\varrho_{\mathrm{df}} \varrho_{\mathrm{sf}}^{-1} \varrho_{\mathrm{ff}}^{-3} \varrho_{\mathrm{bf}}^{-1} \varrho_{\mathrm{tf}}^{-3}p_1+  \beta_1 \Big(l'- \frac{K_1 \varepsilon x^{-1}}{1+\varepsilon x^{-1}}\Big) \label{eq:misc_b8b} \\&\qquad\qquad\qquad +  \beta_2 \Big(\ell'-2l' - \frac{K_2\varepsilon (\sigma^2+\mathsf{Z}x)^{-1/2}}{1+\varepsilon (\sigma^2+\mathsf{Z}x)^{-1/2}}\Big) - 2 \underline{\delta} \phi_\varepsilon^2 \chi(\tilde{p}^{2,0,-2})^2 \chi(\rho)^2  \Big]^{1/2} \\
			e_{\varepsilon} &= \varrho_{\mathrm{df}}^{-1/2} \varrho_{\mathrm{sf}}^{1/2} \varrho_{\mathrm{ff}}^{3/2} \varrho_{\mathrm{bf}}^{1/2} \varrho_{\mathrm{tf}}^{3/2} a_0^{1/2} \phi_\varepsilon \psi \chi_0(\rho) \chi(\tilde{p}^{2,0,-2}) \sqrt{2 H_{\tilde{p}}^{2,0,-2} \rho}\\
			f_{\varepsilon} &= 2 \varrho_{\mathrm{df}} \varrho_{\mathrm{sf}} \varrho_{\mathrm{ff}} \varrho_{\mathrm{bf}} \varrho_{\mathrm{tf}}  \phi_\varepsilon^2 a_0 \chi(\tilde{p}^{2,0,-2})\chi'(\tilde{p}^{2,0,-2})    \chi(\rho)^2\psi^2 \tilde{q}  \\
			r_\varepsilon &=  2 \varrho_{\mathrm{df}}^{-1} \varrho_{\mathrm{sf}} \varrho_{\mathrm{ff}}^3\varrho_{\mathrm{bf}} \varrho_{\mathrm{tf}}^3  \phi_\varepsilon^2  a_0   \chi(\tilde{p}^{2,0,-2})^2\chi(\rho)^2 \psi H_{\tilde{p}}^{2,0,-2} \psi.
		\end{split}
	\end{align}
	Here $p_1 \in S_{\mathrm{leC}}^{2,-1,-3,-1-\delta,-3-2\delta}(X)$ is as in the proof of the propagation estimate, and we are using \Cref{prop:imaginary_comp}, which shows that $\varrho_{\mathrm{df}}^2 \varrho_{\mathrm{sf}}^{-1} \varrho_{\mathrm{ff}}^{-3} \varrho_{\mathrm{bf}}^{-1} \varrho_{\mathrm{tf}}^{-3}p_1$ vanishes to some fractional order at $\calR_+$ and therefore does not spoil the sign of the quantity under the first square root in \cref{eq:misc_b8b} for an appropriate choice of $\chi,\psi$.

	In terms of these new symbols, we can write 
	\begin{align}
		\begin{split} 
		H_{\tilde{p}} a^{(\varepsilon)} + p_1 a^{(\varepsilon)} &= -2\underline{\delta} \varrho_{\mathrm{df}}^{-1} \varrho_{\mathrm{sf}} \varrho_{\mathrm{ff}}^{3} \varrho_{\mathrm{bf}} \varrho_{\mathrm{tf}}^{3} a^{(\varepsilon)} \phi_\varepsilon^2 \chi(\tilde{p}^{2,0,-2})^2 \chi(\rho)^2  - b_\varepsilon^2 - e_\varepsilon^2 + f_\varepsilon \tilde{p} +r_\varepsilon\\
		&= -2\underline{\delta} \varrho_{\mathrm{df}}^{-1} \varrho_{\mathrm{sf}} \varrho_{\mathrm{ff}}^{3} \varrho_{\mathrm{bf}} \varrho_{\mathrm{tf}}^{3}  a_0^{-1} a^{(\varepsilon)2}  - b_\varepsilon^2 - e_\varepsilon^2 + f_\varepsilon \tilde{p} +r_\varepsilon
		\end{split} 
	\end{align}
	We apply the quantization map $\operatorname{Op}$. Setting $A_\varepsilon=(1/2)(\operatorname{Op}(a^{(\varepsilon)}) + \operatorname{Op}(a^{(\varepsilon)})^*)$, $B_\varepsilon=\operatorname{Op}(b_\varepsilon)$, $E_\varepsilon=\operatorname{Op}(e_\varepsilon)$, $F_\varepsilon=\operatorname{Op}(f_\varepsilon)$,  
	\begin{align}
		\begin{split} 
		A_\bullet &\in L^\infty([0,1]_\varepsilon ; \Psi_{\mathrm{leC}}^{-\infty,l',\ell',-\infty,-\infty}(X) ),  \\
		B_\bullet, E_\bullet &\in L^\infty([0,1]_\varepsilon; \Psi_{\mathrm{leC}}^{-\infty,(l'-1)/2,(\ell'-3)/2,-\infty,-\infty} ), \\
		F_\bullet &\in L^\infty([0,1]_\varepsilon ; \Psi_{\mathrm{leC}}^{-\infty,l'-1,\ell'-1,-\infty,-\infty}(X) ) ,
		\end{split} 
	\end{align}
	and 
	\begin{multline}
		-i [\Re \tilde{P},A_\varepsilon] -\{\Im \tilde{P},A_\varepsilon  \} =  -2\underline{\delta} A_\varepsilon  \Lambda_{1/2,-(l'+1)/2,-(\ell'+3)/2,-(l'+1)/2,-(\ell'+3)/2}^2 A_\varepsilon - B_\varepsilon^* B_\varepsilon - E_\varepsilon^* E_\varepsilon + F_\varepsilon^* \tilde{P} \\+R_\varepsilon
		\label{eq:misc_l99}
	\end{multline}
	for some $R_\bullet \in L^\infty([0,1]_\varepsilon;\Psi_{\mathrm{leC}}^{-\infty,l'-2,\ell'-4,-\infty,-\infty}(X))$. 
	We have 
	\begin{equation} 
		\operatorname{WF}'_{L^\infty,\mathrm{leC}}(A_\bullet),	\operatorname{WF}'_{L^\infty,\mathrm{leC}}(B_\bullet),\operatorname{WF}'_{L^\infty,\mathrm{leC}}(F_\bullet),\operatorname{WF}'_{L^\infty}(E_\bullet),\operatorname{WF}'_{L^\infty,\mathrm{leC}}(R_\bullet)\subset \operatorname{supp} \chi(\tilde{p}^{2,0,-2})\chi(\rho) \psi . 
	\end{equation}  
	
	We now set the parameters $l',\ell'$ in the definition \cref{eq:a0def} of $a_{0}$ to $l'=2s+1>0$ and $\ell'=2\varsigma+3>0$. 
	Fix $K_1\in (0,l')$, $K_2 \in (0,\ell')$ such that $-1/2<s-K_1<s_0$ and $-3/2 < \varsigma-K_2< \varsigma_0$. 
	Suppose now that $u\in \calS'(X)$, $\sigma \in [0,\Sigma]$ are such that
	\begin{equation}
		\lVert G_3 u \rVert_{H_{\mathrm{leC}}^{-N,s_0,\varsigma_0,-N,-N}}<\infty. 
	\end{equation}
	The argument in \cite[\S4.7, above Proposition 5.27]{VasyGrenoble} justifies the computation 
	\begin{align}
		2 \Im \langle \tilde{P}u, A_\varepsilon u \rangle_{L^2} &=     -\langle \{\Im \tilde{P}, A_\varepsilon \}u,u\rangle_{L^2} + i \langle [\Re \tilde{P}(\sigma),A_\varepsilon]u,u \rangle_{L^2} \\
		\begin{split} 
			&=   -\lVert B_\varepsilon u \rVert^2_{L^2} - \lVert E_\varepsilon u \rVert_{L^2}^2 + \langle \tilde{P} u , F_\varepsilon u \rangle_{L^2} + \langle R_\varepsilon u ,u \rangle_{L^2} \\
			&\qquad\qquad\qquad\qquad\qquad- 2\underline{\delta} \lVert \Lambda_{1/2,-s-1,-\varsigma-3, -s-1,-\varsigma-3} A_\varepsilon u \rVert^2_{L^2},
		\end{split}
	\end{align}
	where the individual terms above are all well-defined distributional pairings (in the sense of H\"ormander) or (finite) norms. Thus, 
	\begin{multline}
		\lVert B_\varepsilon u \rVert^2_{L^2} + \lVert E_\varepsilon u \rVert_{L^2}^2+2\underline{\delta} \lVert \Lambda_{1/2,-s-1,-\varsigma-3, -s-1,-\varsigma-3} A_\varepsilon u \rVert^2_{L^2} \leq 2|\langle \tilde{P}u, A_\varepsilon u \rangle_{L^2}| + | \langle \tilde{P} u , F_\varepsilon u \rangle_{L^2}| \\ +|\langle R_\varepsilon u ,u \rangle_{L^2}|.
	\end{multline}
	We estimate each of the terms on the right-hand side as in the proof of the propagation estimate: for self-adjoint $G \in \Psi_{\mathrm{leC}}^{-\infty,0,0,-\infty,-\infty}(X)$ with $1-G$ essentially supported away from the $L^\infty$-essential support of $a,f,e,b$ and with $\operatorname{WF}'_{\mathrm{leC}}(G)\subset \operatorname{Ell}_{\mathrm{leC}}^{0,0,0,0,0}(G_2)$, 
	\begin{equation}
		\lVert \hat{B}_\varepsilon u \rVert_{\calX_N}\leq \lVert \hat{B}_\varepsilon u \rVert_{\calX_N} + \lVert \hat{E}_\varepsilon u \rVert_{\calX_N} \preceq \lVert G \tilde{P} u \rVert_{\calY_N} + \lVert G u \rVert_{\calZ_N} + \lVert u \rVert_{\calE_N}, 
		\label{eq:misc_llk}
	\end{equation}
	where $\calE_N= H_{\mathrm{leC}}^{-N,-N,-N,-N}(X)$, $\calZ_N =H_{\mathrm{leC}}^{-N,(2s-1)/2, (2\varsigma-1)/2, -N,-N}(X)$, $\calX_N = H_{\mathrm{leC}}^{-N,s,\varsigma,-N,-N}(X)$, $\calY_N = H_{\mathrm{leC}}^{-N,s+1,\varsigma+3,-N,-N}(X)$,
	and $\hat{B}_\varepsilon$ and $\hat{E}_\varepsilon$ are given by $\hat{B}_\varepsilon = \Lambda_{0,-s,-\varsigma,0,0} B_\varepsilon$ and 
	\begin{equation} 
		\hat{E}_\varepsilon = \Lambda_{0,-s,-\varsigma,0,0} E_\varepsilon.
	\end{equation} 
	Via elliptic regularity, we can estimate 
	\begin{equation}
		\lVert G \tilde{P} u \rVert_{\calY_N} \preceq \lVert G_2 \tilde{P} u \rVert_{\calY_N} + \lVert u \rVert_{\calE_N}. 
	\end{equation}
	By shrinking the support of $\chi,\psi$ if necessary, we can arrange that the $L^\infty$-esssupp of $a,b,f,e$ is a subset $\operatorname{Ell}_{\mathrm{leC}}^{0,0,0,0,0}(G_3)$, and then we can choose $G$ such that $\operatorname{WF}'_{\mathrm{leC}}(G)\subset \operatorname{Ell}^{0,0,0,0,0}_{\mathrm{leC}}(G_3)$, so that the estimate \cref{eq:misc_llk} implies 
	\begin{equation}
		\lVert \hat{B}_\varepsilon u \rVert_{\calX_N}\preceq \lVert G_2 \tilde{P} u \rVert_{\calY_N} + \lVert G_3 u \rVert_{\calZ_N} + \lVert u \rVert_{\calE_N}. 
	\end{equation}
	Using the Banach--Alaoglu theorem, applied as during the proof of the propagation estimate, we can take $\varepsilon\to 0^+$ to conclude 
	\begin{equation}
		\lVert \hat{B}_0 u \rVert_{\calX_N}\preceq \lVert G_2 \tilde{P} u \rVert_{\calY_N} + \lVert G_3 u \rVert_{\calZ_N}+ \lVert u \rVert_{\calE_N}. 
		\label{eq:misc_eqq}
	\end{equation}

	Let $\calX = H_{\mathrm{leC}}^{m,s,\varsigma,l,\ell}(X)$. 
	Since $\operatorname{Ell}_{\mathrm{leC}}^{0,0,0,0,0}(\hat{B}_0)\supset \calR_+$, \cref{eq:misc_eqq} implies 
	\begin{equation} 
		\lVert G_1 u \rVert_{\calX} \preceq \lVert \hat{B}_0 u \rVert_{\calX_N} + \lVert G_2 \tilde{P} u \rVert_{\calY_N} + \lVert u \rVert_{\calE_N} ,
	\end{equation} 
	where we used the propagation estimate to control $G_1 u$ on $\operatorname{Char}_{\mathrm{leC}}^{2,0,-2,-1,-3}(\tilde{P})$ away from $\calR_+$.  So, 
	\begin{align}
		\begin{split} 
		\lVert G_1 u \rVert_\calX \preceq \lVert \hat{B}_0 u \rVert_{\calX_N} + \lVert G_2 \tilde{P} u \rVert_{\calY_N}  + \lVert u \rVert_{\calE_N} &\preceq \lVert G_2 \tilde{P} u \rVert_{\calY_N}+ \lVert G_3 u \rVert_{\calZ_N} + \lVert u \rVert_{\calE_N} \\
		&\preceq \lVert G_2 \tilde{P} u \rVert_{\calY_N}+ \lVert G_3 u \rVert_{H_{\mathrm{leC}}^{-N,s-1/2, \varsigma-1/2, -N,-N}} + \lVert u \rVert_{\calE_N} . 
		\end{split} 
		\label{eq:misc_p01}
	\end{align}
	An inductive argument (using \cref{eq:misc_inq}) estimating $\lVert G_3 u \rVert_{H_{\mathrm{leC}}^{-N,s-1/2, \varsigma-1/2, -N,-N}}$ finishes the proof. 
\end{proof}

\subsection{The Radial ``Point'' $\calR$}
\label{subsec:rp}

In order to analyze matters uniformly near the corners of leC-phase space (in particular the highlighted edge $\mathrm{bf}\cup \mathrm{ff}$ in \Cref{fig:later}), we work with the bdfs $\varrho_{\mathrm{sf}},\varrho_{\mathrm{bf}},\varrho_{\mathrm{ff}},\varrho_{\mathrm{tf}}$ defined in \cref{eq:bdfs} rather than the leC-adapted momentum coordinates $\xi_{\mathrm{sc,leC}},\eta_{\mathrm{sc,leC}}$ used in the previous section. 

In terms of the bdfs,  
\begin{align} 
\tilde{p}_0 &= \varrho_{\mathrm{df}}^{-2}\varrho_{\mathrm{ff}}^{2}\varrho_{\mathrm{bf}}^{2}\varrho_{\mathrm{tf}}^{4} (1-\varrho_{\mathrm{df}}^2\varrho_{\mathrm{sf}}^2\varrho_{\mathrm{ff}}^2) - 2\Big( \frac{\xi_{\mathrm{b}}}{(1+\xi^2_{\mathrm{b}}+g^{-1}_{\partial X}(\eta_{\mathrm{b}},\eta_{\mathrm{b}}))^{1/2}}\Big)\varrho_{\mathrm{df}}^{-1}\varrho_{\mathrm{ff}}^{2}\varrho_{\mathrm{bf}}\varrho_{\mathrm{tf}}^{3} \\
&= \varrho_{\mathrm{df}}^{-2}\varrho_{\mathrm{ff}}^{2}\varrho_{\mathrm{bf}}^{2}\varrho_{\mathrm{tf}}^{4}(1-\varrho_{\mathrm{df}}^2\varrho_{\mathrm{sf}}^2\varrho_{\mathrm{ff}}^2) - 2 \beth \varrho_{\mathrm{df}}^{-1}\varrho_{\mathrm{ff}}^{2} \varrho_{\mathrm{bf}}\varrho_{\mathrm{tf}}^{3}
\label{eq:misc_844}
\end{align}  
near $\{x=0\}\subset {}^{\mathrm{leC}} \overline{T}^* X$, where $\beth \in S_{\mathrm{cl,leC}}^{0,0,0,0,0}(X)$ is given by $\xi_{\mathrm{b}} (1+\xi_{\mathrm{b}}^2+g^{-1}_{\partial X}(\eta_{\mathrm{b}},\eta_{\mathrm{b}}))^{-1/2}$ near $\{x=0\}$. 
Factoring out common powers of bdfs from \cref{eq:misc_844}, we are left with 
\begin{align} 
	\tilde{p}^{2,0,-2,-1,-3}_0 &= \varrho_{\mathrm{bf}} \varrho_{\mathrm{tf}} (1-\varrho_{\mathrm{df}}^2\varrho_{\mathrm{sf}}^2\varrho_{\mathrm{ff}}^2) -2 \beth \varrho_{\mathrm{df}} \\
	&= \varrho_{\mathrm{bf}} \varrho_{\mathrm{tf}}  -2 \beth \varrho_{\mathrm{df}} \bmod S_{\mathrm{leC}}^{-1,-1,-1,-1,-1}(X).
\end{align}   
On $\mathrm{sf}\cup \mathrm{ff}$, $\tilde{p}^{2,0,-2,-1,-3}_0$ and thus $\tilde{p}^{2,0,-2,-1,-3}$ vanishes if and only if $2 \beth =  \varrho_{\mathrm{bf}} \varrho_{\mathrm{tf}}/\varrho_{\mathrm{df}}$.
Therefore 
\begin{equation} 
	\operatorname{Char}_{\mathrm{leC}}^{2,0,-2,-1,-3}(\tilde{P}) = \{2\beth  = \varrho_{\mathrm{bf}}\varrho_{\mathrm{tf}}/\varrho_{\mathrm{df}}\}\subset {}^{\mathrm{leC}}T^* X. 
\end{equation}
In particular, the portion of the characteristic set that is on the boundary of $\mathrm{bf}\cup \mathrm{tf}$ is precisely $\{\beth=0\} \cap (\mathrm{bf} \cup \mathrm{tf}) \cap (\mathrm{sf} \cup \mathrm{ff})$.

\begin{proposition}
	\label{prop:new_weight_comp}
	We have $H_{\tilde{p}_0}^{2,0,-2} \varrho_{\mathrm{df}_{00}} = F_{0,1} \varrho_{\mathrm{df}_{00}}$ for $F_{0,1} \in S_{\mathrm{cl,leC}}^{0,0,0,0,0}(X)$ given by 
	\begin{equation} 
		F_{0,1}= 2\varrho_{\mathrm{bf}}\varrho_{\mathrm{tf}} \beth (1-\varrho_{\mathrm{df}}^2\varrho_{\mathrm{sf}}^2\varrho_{\mathrm{ff}}^2 ) - 2  \varrho_{\mathrm{df}} \beth^2 - \mathsf{Z}    \varrho_{\mathrm{df}} \varrho_{\mathrm{sf}}  \varrho_{\mathrm{bf}} \beth^2
	\end{equation} 
	near $\{x=0\}\subset {}^{\mathrm{leC}}\overline{T}^* X$. 
\end{proposition}
\begin{proof}
	Applying \cref{eq:Hp} to $\varrho_{\mathrm{df}_{00}}$, 
	\begin{equation}
		H_{\tilde{p}_0} \varrho_{\mathrm{df}_{00}} =  - (x \partial_x  \tilde{p}_0) \partial_{\xi_{\mathrm{b}}} \varrho_{\mathrm{df}_{00}}
	\end{equation}
	near $\{x=0\}$.  Writing $\tilde{p}_0 = x^2 \varrho_{\mathrm{df}_{00}}^{-2} (1- \varrho_{\mathrm{df}_{00}}^{2}) - 2 x \xi_{\mathrm{b}}  (\sigma^2+\mathsf{Z}x)^{1/2} $, 
	\begin{align}
		\begin{split} 
		\partial_x \tilde{p}_0 &= 2x \varrho_{\mathrm{df}_{00}}^{-2} (1- \varrho_{\mathrm{df}_{00}}^{2})  - 2 \xi_{\mathrm{b}}  (\sigma^2 + \mathsf{Z} x)^{1/2} - \mathsf{Z}x \xi_{\mathrm{b}}  (\sigma^2+\mathsf{Z} x)^{-1/2} \\
		\partial_{\xi_{\mathrm{b}}} \varrho_{\mathrm{df}_{00}} &=  - \xi_{\mathrm{b}} \varrho_{\mathrm{df}_{00}}^{3} .
		\end{split} 
	\end{align} 
	Thus,  
	\begin{equation}
		H_{ \tilde{p}_0} \varrho_{\mathrm{df}_{00}}  = 2  x^{2} \xi_{\mathrm{b}} \varrho_{\mathrm{df}_{00}}(1-\varrho_{\mathrm{df}_{00}}^2) - 2 x \xi_{\mathrm{b}}^2 \varrho_{\mathrm{df}_{00}}^{3} (\sigma^2+\mathsf{Z}x)^{1/2}  -  \mathsf{Z}  \xi_{\mathrm{b}}^2 \varrho_{\mathrm{df}_{00}}^{3} x^{2} (\sigma^2+\mathsf{Z}x)^{-1/2} . 
	\end{equation}
	In terms of the bdfs of ${}^{\mathrm{leC}}\overline{T}^* X$, $x^2 \xi_{\mathrm{b}} \varrho_{\mathrm{df}_{00}} = \varrho_{\mathrm{sf}}^2\varrho_{\mathrm{ff}}^4\varrho_{\mathrm{bf}}^2\varrho_{\mathrm{tf}}^4 \beth $, and $x\xi_{\mathrm{b}}^2 \varrho_{\mathrm{df}_{00}}^3 (\sigma^2+\mathsf{Z} x)^{1/2} = \beth^2 \varrho_{\mathrm{df}} \varrho_{\mathrm{sf}}^2 \varrho_{\mathrm{ff}}^4 \varrho_{\mathrm{bf}} \varrho_{\mathrm{tf}}^3$, and $\xi_{\mathrm{b}}^2 \varrho_{\mathrm{df}_{00}}^3x^2 (\sigma^2+\mathsf{Z}x)^{-1/2} = \beth^2 \varrho_{\mathrm{df}} \varrho_{\mathrm{sf}}^3 \varrho_{\mathrm{ff}}^4 \varrho_{\mathrm{bf}}^2 \varrho_{\mathrm{tf}}^3$. 	
	Adding everything together, we find 
	\begin{align} 
		H_{\tilde{p}_0} \varrho_{\mathrm{df}_{00}} &= 2\varrho_{\mathrm{sf}}^2\varrho_{\mathrm{ff}}^4\varrho_{\mathrm{bf}}^2\varrho_{\mathrm{tf}}^4 \beth(1-\varrho_{\mathrm{df}}^2\varrho_{\mathrm{sf}}^2\varrho_{\mathrm{ff}}^2 ) - 2   \varrho_{\mathrm{df}} \varrho_{\mathrm{sf}}^2 \varrho_{\mathrm{ff}}^4 \varrho_{\mathrm{bf}} \varrho_{\mathrm{tf}}^3\beth^2 - \mathsf{Z}   \varrho_{\mathrm{df}} \varrho_{\mathrm{sf}}^3 \varrho_{\mathrm{ff}}^4 \varrho_{\mathrm{bf}}^2 \varrho_{\mathrm{tf}}^3 \beth^2 .
		\intertext{	 Dividing by $\varrho_{\mathrm{df}}^{-1} x(\sigma^2+\mathsf{Z}x)^{1/2} = \varrho_{\mathrm{df}}^{-1} \varrho_{\mathrm{sf}} \varrho_{\mathrm{ff}}^3 \varrho_{\mathrm{bf}} \varrho_{\mathrm{tf}}^3$,}
		H_{\tilde{p}_0}^{2,0,-2} \varrho_{\mathrm{df}_{00}} &= (2\varrho_{\mathrm{bf}}\varrho_{\mathrm{tf}} \beth (1-\varrho_{\mathrm{df}}^2\varrho_{\mathrm{sf}}^2\varrho_{\mathrm{ff}}^2 ) - 2  \varrho_{\mathrm{df}} \beth^2 - \mathsf{Z}    \varrho_{\mathrm{df}} \varrho_{\mathrm{sf}}  \varrho_{\mathrm{bf}} \beth^2) \varrho_{\mathrm{df}_{00}},
	\end{align} 
	as claimed. 
\end{proof}

Letting $\beta_1,\beta_2$ be as in \Cref{prop:basic_weight_comp}, $\beta_1,\beta_2<0$ on $\calR_0$ and thus on $\calR$. By \Cref{prop:new_weight_comp}, there exists an $F_1 \in S_{\mathrm{leC}}^{0,0,0,0,0}(X)$ such that 
\begin{equation}
	H_{\tilde{p}}^{2,0,-2} \varrho_{\mathrm{df}_{00}} = F_1 \varrho_{\mathrm{df}_{00}}, 
\end{equation}
with $F_1$ vanishing on $\calR$. 
By \Cref{prop:radial_comp}, it is the case that for any fixed $\rho \in S^{0,0,0,0,0}_{\mathrm{cl,leC}}(X)$ equal to $\xi_{\mathrm{sc,leC}}^2 + \eta_{\mathrm{sc,leC}}^2$ in some neighborhood of $\mathrm{bf}\cup \mathrm{tf}$, there exist some symbols  
\begin{equation}
	\beta_0, F_2, F_3 ,F_4 \in S_{\mathrm{leC}}^{0,0,0,0,0}(X)
\end{equation}
such that 
\begin{align}
	H_{\tilde{p}}^{2,0,-2}\rho &= \beta_0 \rho + F_2 + F_3 + x(\sigma^2+\mathsf{Z} x)^{-1/2} F_4\\ 
	 \beta_0|_{\calR_0}&<0, 
\end{align}
$\beta_0,F_2\leq 0$ everywhere, and $F_3$ vanishes cubically at $\calR$ uniformly in $[0,\Sigma]$.  
We may choose $\rho$ such that it is nonnegative everywhere. 

It is necessary to have another weight with semidefinite sign under the Hamiltonian flow.
We may use 
\begin{equation}
	\varrho_{\mathrm{ff}} = \varrho_{\mathrm{df}_{00}}+\varrho_{\mathrm{tf}_{00}} \in S^{0,0,0,0,0}_{\mathrm{cl,leC}}(X). 
\end{equation}
We have $H_{\tilde{p}}^{2,0,-2} \varrho_{\mathrm{ff}} = \beta_3 \varrho_{\mathrm{ff}}$ for $\beta_3 =  (F_1 \varrho_{\mathrm{df}_{00}} +\beta_2 \varrho_{\mathrm{tf}_{00}})(\varrho_{\mathrm{df}_{00}}+\varrho_{\mathrm{tf}_{00}})^{-1} \in S_{\mathrm{leC}}^{0,0,0,0,0}(X)$. We can write this as $\beta_3 = F_1 \varrho_{\mathrm{df}} \varrho_{\mathrm{sf}} + \beta_2 \varrho_{\mathrm{tf}}$. Thus, $\beta_3|_\calR\leq 0$. Note that $\beta_3$ vanishes at $\calR\cap \mathrm{tf}$. Normally, this would be problematic as far as the radial point estimate is concerned, but this fourth weight is used only to give us an extra independent order and not to manufacture positivity, so a semidefinite sign is actually acceptable.

We now have enough basic weights to construct our commutants: the basic weights are $x= \varrho_{\mathrm{bf}} \varrho_{\mathrm{sf}} \varrho_{\mathrm{tf}}^2 \varrho_{\mathrm{ff}}^2$, $(\sigma^2+\mathsf{Z}x)^{1/2} = \varrho_{\mathrm{tf}} \varrho_{\mathrm{ff}}$, $\varrho_{\mathrm{df}_{00}} = \varrho_{\mathrm{df}} \varrho_{\mathrm{sf}} \varrho_{\mathrm{ff}}$, and $\varrho_{\mathrm{ff}}$. For any $s,\varsigma,l,\ell \in \bbR$, 
\begin{equation}
	x^l (\sigma^2+\mathsf{Z}x)^{\ell/2-l} \varrho_{\mathrm{df}_{00}}^{s-l} \varrho_{\mathrm{ff}}^{\varsigma -\ell  - s+l} = \varrho_{\mathrm{df}}^{s-l} \varrho_{\mathrm{sf}}^s \varrho_{\mathrm{ff}}^{\varsigma} \varrho_{\mathrm{bf}}^l \varrho_{\mathrm{tf}}^\ell. 
	\label{eq:misc_alo}
\end{equation}
We do not care about the order at df (because $\calR$ is disjoint from df), so the weights of the form \cref{eq:misc_alo} (which give four independent orders) suffice. 
We now consider the weight (dependent on parameters $s,\varsigma,l,\ell \in \bbR$)  
\begin{equation} 
	a_0 = 	x^{-l} (\sigma^2+\mathsf{Z}x)^{-\ell/2+l} \varrho_{\mathrm{df}_{00}}^{-s+l} \varrho_{\mathrm{ff}}^{-\varsigma +\ell  + s-l} \in S_{\mathrm{cl,leC}}^{s-l,s,\varsigma,l,\ell}(X).
\end{equation} 
Per the above, 
\begin{equation} 
	\beta =-((s-l) F_1 + (\varsigma-\ell-s+l) \beta_3  + l \beta_1  + (\ell-2l) \beta_2) = a_0^{-1}\smash{H_{\tilde{p}}^{2,0,-2} a_0} \in \smash{S^{0,0,0,0,0}_{\mathrm{leC}}(X)}
\end{equation} 
will have a definite sign near  $\calR$ if $l,\ell$ have the same definite sign and $\varsigma-\ell-s+l$ has the same sign semidefinitely.  We will choose these parameters such that $\beta<0$ near $\calR$.

Hence, multiplying $a_0$ by an appropriate microlocal cutoff, we can arrange for $a$ to be everywhere monotonic under the Hamiltonian flow, strictly so near $\calR$. 
Specifically, fix $\chi \in C_{\mathrm{c}}^\infty(\bbR;[0,1])$ as in the proof of \Cref{prop:R+}. Dilating $\chi$ if necessary, we can find $\psi \in S_{\mathrm{cl,leC}}^{0,0,0,0,0}(X)$ that is identically equal to one in some neighborhood of $\mathrm{sf}\cup \mathrm{ff}$ such that 
\begin{equation}
	\beta<0 \label{eq:pos_cond_ii}
\end{equation}
on $\operatorname{supp} \chi(\tilde{p}^{2,0,-2,-1,-3}) \chi(\rho) \psi$ and such that 
\begin{equation}
		\chi(\tilde{p}^{2,0,-2,-1,-3})\chi_0(\rho) \psi \sqrt{-H^{2,0,-2}_{\tilde{p}} \rho} \in S_{\mathrm{leC}}^{0,0,0,0,0}(X)
	\label{eq:pos_cond_iii}
\end{equation}
and $\operatorname{supp} \chi(\tilde{p}^{2,0,-2,-1,-3}) \chi(\rho)  \psi \cap (\calR_+\cup \mathrm{df}) = \varnothing$.
Now set 
\begin{equation}
	a = a_0 \chi(\tilde{p}^{2,0,-2,-1,-3})^2 \chi(\rho)^2 \psi^2 \in S_{\mathrm{cl,leC}}^{s-l,s,\varsigma,l,\ell}(X).
\end{equation}
The three factors $\chi(\tilde{p}^{2,0,-2,-1,-3}), \chi(\rho),\psi  \in S_{\mathrm{cl,leC}}^{0,0,0,0,0}(X)$ together microlocalize near $\calR$.

We can write $H^{2,0,-2}_{\tilde{p}} \tilde{p}^{2,0,-2,-1,-3} = \tilde{q} \tilde{p}^{2,0,-2,-1,-3}$ for  $\tilde{q} \in S_{\mathrm{leC}}^{0,0,0,0,0}(X)$ defined by 
\begin{equation} 
	\tilde{q} = \smash{\varrho_{\mathrm{df}}^{-2}\varrho_{\mathrm{ff}}^{2} \varrho_{\mathrm{bf}}^{1}\varrho_{\mathrm{tf}}^{3}H^{2,0,-2}_{\tilde{p}}(\varrho_{\mathrm{df}}^2\varrho_{\mathrm{ff}}^{-2} \varrho_{\mathrm{bf}}^{-1}\varrho_{\mathrm{tf}}^{-3})}.
\end{equation} 
Then, 
\begin{multline}
	H_{\tilde{p}}^{2,0,-2} a =   \psi^2 \chi(\tilde{p}^{2,0,-2,-1,-3})^2 \chi(\rho)^2  H_{\tilde{p}}^{2,0,-2} a_0  - 2  a_0 \chi_0(\rho)^2 \chi(\tilde{p}^{2,0,-2,-1,-3})^2 \psi^2H^{2,0,-2}_{\tilde{p}} \rho   \\ + 2a_0 \chi(\tilde{p}^{2,0,-2,-1,-3})\chi'(\tilde{p}^{2,0,-2,-1,-3}) \chi(\rho)^2 \psi^2 \tilde{p}^{2,0,-2,-1,-3} \tilde{q}  \\ + 2a_0 \chi(\tilde{p}^{2,0,-2,-1,-3})^2 \chi(\rho)^2 \psi H_{\tilde{p}}^{2,0,-2}\psi   .
	\label{eq:hform}
\end{multline}
Observe: 
\begin{enumerate}
	\item by \cref{eq:pos_cond_ii}, the first term, 
	\begin{equation} 
		\psi^2 \chi(\tilde{p}^{2,0,-2,-1,-3})^2 \chi(\rho)^2  H_{\tilde{p}}^{2,0,-2} a_0 =  \psi^2\chi(\tilde{p}^{2,0,-2,-1,-3})^2 \chi(\rho)^2 \beta a_0 
	\end{equation} 
	will have a definite sign (the same as $\beta$, negative if $l,\ell<0$) for appropriate $\chi,\psi$, 
	\item the second term, 
	\begin{multline} 
		-a_0 \chi_0(\rho)^2 \chi(\tilde{p}^{2,0,-2,-1,-3})^2\psi^2  H^{2,0,-2}_{\tilde{p}} \rho \\ = -a_0 \chi_0(\rho)^2 \chi(\tilde{p}^{2,0,-2,-1,-3})^2 \psi^2 (\beta_0\rho+F_2+F_3 + x(\sigma^2+\mathsf{Z} x)^{-1/2} F_4 ),
	\end{multline} 
	also has a definite sign (positive, since $\beta_0,F_2,F_3|_\calR\leq 0$) for appropriate $\chi,\psi$  and is supported in an annulus around $\calR_+$, which should intersect $\operatorname{Char}_{\mathrm{leC}}^{2,0,-2,-1,-3}(\tilde{P})\backslash (\calR\cup \calR_+)$,  
	\item the third sand fourth terms are supported away from $\operatorname{Char}^{2,0,-2,-1,-3}(\tilde{P}) $ and are therefore unproblematic (as that region of phase space is controlled via elliptic estimates or, in the case of $\mathrm{bf}^\circ \cup \mathrm{tf}^\circ$ where the fourth term might have some support, cannot be controlled by symbolic considerations anyways). 
\end{enumerate}
We will prove a low order radial point estimate. This means that 
\begin{equation}
	l< 0, \qquad \ell<0 \qquad \varsigma\leq \ell+s-l, 
\end{equation}
so that $\beta<0$ near $\calR$. 
Thus, the first and second terms in \cref{eq:hform} have the opposite sign near $\calR$. The second term thus contributes to the right-hand side of the radial point estimate, but this term can itself be controlled using a radial point estimate at $\calR_+$ in conjunction with a propagation estimate and will therefore be unproblematic as well.

In order to ``regularize,''  set $\phi_\varepsilon = (1+\varepsilon x^{-1})^{-K_1}(1+\varepsilon \varrho_{\mathrm{df}_{00}}^{-1})^{-K_3}$, 
\begin{equation}a^{(\varepsilon)} = \phi_\varepsilon^2 a  \in L^\infty([0,1]_\varepsilon;S_{\mathrm{cl,leC}}^{-\infty,s,\varsigma,l,\ell}(X)), 
\end{equation}
for to-be-decided $K_1,K_3\in \bbR$. We then compute that 
\begin{multline}
	H_{\tilde{p}}^{2,0,-2} a^{(\varepsilon)}  = \phi_\varepsilon^2 \Big[ -\chi(\tilde{p}^{2,0,-2,-1,-3})^2 \chi(\rho)^2 \psi^2 a_0 \Big(  (\varsigma-\ell-s+l)\beta_3\\ + F_1\Big(s-l- \frac{K_3 \varepsilon \varrho_{\mathrm{df}_{00}}^{-1}}{1+\varepsilon \varrho_{\mathrm{df}_{00}}^{-1}}\Big)  + \beta_1 \Big(l- \frac{K_1 \varepsilon x^{-1}}{1+\varepsilon x^{-1}}\Big)  +  \beta_2 (\ell-2l) \Big)   - 2 a_0 \chi_0(\rho)^2 \chi(\tilde{p}^{2,0,-2,-1,-3})^2 \psi^2 H_{\tilde{p}}^{2,0,-2} \rho  \\  +2 a_0 \chi(\tilde{p}^{2,0,-2,-1,-3})\chi'(\tilde{p}^{2,0,-2,-1,-3}) \chi(\rho)^2  \psi^2 \tilde{p}^{2,0,-2,-1,-3} \tilde{q} \\ + 2 a_0 \chi(\tilde{p}^{2,0,-2,-1,-3})^2 \chi(\rho)^2 \psi H_{\tilde{p}}^{2,0,-2}\psi \Big]. 
	\label{eq:misc_01h}
\end{multline}
In contrast to the previous radial point estimate, we can make $K_1,K_3$ arbitrarily large without affecting the sign of the parenthetical term, although we might need to choose $\psi$ with smaller support and replace $\chi$ with $\chi\circ \operatorname{dil}_\lambda$  for 
\begin{equation} 
	\lambda = \lambda(l,\ell,s,\varsigma,K_1,K_3)>0
\end{equation}
sufficiently large 
to ensure that the term proportional to $F_1$ in \cref{eq:misc_01h} does not spoil that sign. 
So, for some choice of $\psi,\chi$, we can choose $\underline{\delta}  = \underline{\delta}(K_1,K_3,s,\varsigma,l,\ell,\chi,\psi)>0$ sufficiently small such that there exist 
well-defined uniform families of leC-symbols 
\begin{align}
	\begin{split}
	b_\bullet &\in  L^\infty([0,1]_\varepsilon ; S_{\mathrm{leC}}^{-\infty,(s-1)/2,(\varsigma-3)/2,(l-1)/2,(\ell-3)/2}(X)),\\  
	e_\bullet &\in  L^\infty([0,1]_\varepsilon ; S_{\mathrm{leC}}^{-\infty,(s-1)/2,(\varsigma-3)/2,-\infty,-\infty}(X)),\\  
	f_\bullet &\in L^\infty([0,1]_\varepsilon ; S_{\mathrm{leC}}^{-\infty,s-1,\varsigma-1,l,\ell}(X)), \\
	r_\bullet &\in L^\infty([0,1]_\varepsilon ; S_{\mathrm{leC}}^{-\infty,-\infty,-\infty,l-1,\ell-3}(X)),
	\end{split} 
\end{align} 
such that, in some neighborhood of $\{x=0\}\subset {}^{\mathrm{leC}}\overline{T}^* X$, 
\begin{align}
	\begin{split} 
	b_\varepsilon &= \varrho_{\mathrm{df}}^{-1/2} \varrho_{\mathrm{sf}}^{1/2} \varrho_{\mathrm{ff}}^{3/2} \varrho_{\mathrm{bf}}^{1/2} \varrho_{\mathrm{tf}}^{3/2}  a_0^{1/2} \chi(\tilde{p}^{2,0,-2,-1,-3}) \chi(\rho) \phi_\varepsilon \psi \Big[ F_1\Big(m- \frac{K_3 \varepsilon \varrho_{\mathrm{df}_{00}}^{-1}}{1+\varepsilon \varrho_{\mathrm{df}_{00}^{-1}}}\Big)    + \beta_1 \Big(l- \frac{K_1 \varepsilon x^{-1}}{1+\varepsilon x^{-1}}\Big)  \\&\qquad+  \beta_2 (\ell-2l)  +\beta_3  (\varsigma-\ell-s+l)- 2 \underline{\delta}\phi_\varepsilon^2 \chi(\tilde{p}^{2,0,-2,-1,-3})^2 \chi(\rho)^2 \psi^2- \varrho_{\mathrm{df}} \varrho_{\mathrm{sf}}^{-1} \varrho_{\mathrm{ff}}^{-3} \varrho_{\mathrm{bf}}^{-1} \varrho_{\mathrm{tf}}^{-3} p_1 \Big]^{1/2} \\
	e_{\varepsilon} &= \varrho_{\mathrm{df}}^{-1/2} \varrho_{\mathrm{sf}}^{1/2} \varrho_{\mathrm{ff}}^{3/2} \varrho_{\mathrm{bf}}^{1/2} \varrho_{\mathrm{tf}}^{3/2} \phi_\varepsilon a_0^{1/2} \chi_0(\rho) \chi(\tilde{p}^{2,0,-2,-1,-3}) \psi \sqrt{-2 H_{\tilde{p}}^{2,0,-2} \rho}\\
	f_{\varepsilon} &= 2 \varrho_{\mathrm{df}} \varrho_{\mathrm{sf}} \varrho_{\mathrm{ff}}  a_0 \phi_\varepsilon^2 \chi(\tilde{p}^{2,0,-2,-1,-3})\chi'(\tilde{p}^{2,0,-2,-1,-3})  \chi(\rho)^2 \psi^2\tilde{q} \\
	r_\varepsilon &=   2\varrho_{\mathrm{df}}^{-1}\varrho_{\mathrm{sf}}\varrho_{\mathrm{ff}}^3\varrho_{\mathrm{bf}}\varrho_{\mathrm{tf}}^3  a_0 \phi_\varepsilon^2\chi(\tilde{p}^{2,0,-2,-1,-3})^2 \chi(\rho)^2 \psi H_{\tilde{p}}^{2,0,-2}\psi
	. 
	\end{split} 
\end{align}
In terms of these new symbols, we can write 
\begin{equation}
	H_{\tilde{p}} a^{(\varepsilon)} + p_1 a^{(\varepsilon)}
	= -2\underline{\delta} \varrho_{\mathrm{df}}^{-1} \varrho_{\mathrm{sf}} \varrho_{\mathrm{ff}}^3 \varrho_{\mathrm{bf}} \varrho_{\mathrm{tf}}^3 a_0^{-1} a^{(\varepsilon)2} - b_\varepsilon^2 + e_\varepsilon^2 + f_\varepsilon \tilde{p} + r_\varepsilon.
\end{equation}
We now apply $\operatorname{Op}$. Setting $A_{\varepsilon}=(1/2)(\operatorname{Op}(a^{(\varepsilon)}) + \operatorname{Op}(a^{(\varepsilon)})^*)$, $B_\varepsilon=\operatorname{Op}(b_\varepsilon)$, $E_\varepsilon=\operatorname{Op}(e_\varepsilon)$, $F_\varepsilon=\operatorname{Op}(f_\varepsilon)$, 
we have 
\begin{align}
	\begin{split} 
	A_\bullet &\in L^\infty([0,1]_\varepsilon ; \Psi_{\mathrm{leC}}^{-\infty,s,\varsigma,l,\ell}(X)),  \\
	B_\bullet &\in  L^\infty([0,1]_\varepsilon ; \Psi_{\mathrm{leC}}^{-\infty,(s-1)/2,(\varsigma-3)/2,(l-1)/2,(\ell-3)/2}(X)),\\  
	E_\bullet &\in  L^\infty([0,1]_\varepsilon ; \Psi_{\mathrm{leC}}^{-\infty,(s-1)/2,(\varsigma-3)/2,-\infty,-\infty}(X)),\\  
	F_\bullet &\in L^\infty([0,1]_\varepsilon ; \Psi_{\mathrm{leC}}^{-\infty,s-1,\varsigma-1,l,\ell}(X)), 
	\end{split} 
\end{align} 
and
\begin{multline}
	-i [\Re \tilde{P},A_\varepsilon]- \{\Im \tilde{P},A_\varepsilon\} =  -2 \underline{\delta}A_{\varepsilon}  \Lambda_{1/2,-(s+1)/2,-(\varsigma+3)/2,-(l+1)/2,-(\ell+3)/2}^2 A_{\varepsilon} - B_\varepsilon^* B_\varepsilon  + E_\varepsilon^* E_\varepsilon  \\  + F_\varepsilon^* \tilde{P} +R_\varepsilon
	\label{eq:misc_l29}
\end{multline}
for some $R_\bullet \in L^\infty([0,1]_\varepsilon;\Psi_{\mathrm{leC}}^{-\infty,s-2,\varsigma-4,l-1,\ell-3}(X))$. 
Moreover, we necessarily have 
\begin{multline} 
	\operatorname{WF}'_{L^\infty,\mathrm{leC}}(A_\bullet),\operatorname{WF}'_{L^\infty,\mathrm{leC}}(B_\bullet),\operatorname{WF}'_{L^\infty,\mathrm{leC}}(F_\bullet),\\ \operatorname{WF}'_{L^\infty,\mathrm{leC}}(E_\bullet),\operatorname{WF}'_{L^\infty,\mathrm{leC}}(R_\bullet)\subset \operatorname{supp} \chi(\tilde{p}^{2,0,-2,-1,-3})\chi(\rho) \psi, 
\end{multline}  
where the last of these inclusions (the one for $R_\bullet$) follows from the one for $A_\bullet$ and 
$\operatorname{WF}'_{L^\infty,\mathrm{leC}}([\tilde{P},A_\bullet]) \subset \operatorname{WF}'_{L^\infty,\mathrm{leC}}(A_\bullet)$.

For each $m_0,s_0,\varsigma_0,l_0,\ell_0 \in \bbR$, there exist some $K_{1,0},K_{3,0}>0$ (dependent on $m_0$, $s_0$, $\varsigma_0$, $l_0$, $\ell_0$ and $m,s,\varsigma,l,\ell$) such that, given $\{u(-;\sigma)\}_{\sigma>0}  \subset \calS'(X)$
with $u(-;0) \in H_{\mathrm{scb}}^{m,\varsigma,\ell}(X_{1/2})$ and $u(-;\sigma) \in H_{\mathrm{scb}}^{m,s,l}(X)$ for all $\sigma>0$, 
if we take $K_1>K_{1,0},K_3>K_{3,0}$ in the construction above then it is the case that (for any $\varepsilon>0$, and for each $\sigma>0$, implicit in the notation),
\begin{equation}
		2  \Im \langle \tilde{P} u , A_\varepsilon u \rangle_{L^2} =   -\langle \{\Im \tilde{P}, A_\varepsilon \}u,u\rangle_{L^2} + i \langle [\Re \tilde{P}(\sigma),A_\varepsilon]u,u \rangle_{L^2} ,
		\label{eq:misc_12j}
\end{equation}
where the pairings above are well-defined distributional pairings (with the left argument of each inner product in the dual Sobolev space to a Sobolev space in which the right argument lies). 
Applying \cref{eq:misc_l29}  to $\{u(-;\sigma)\}_{\sigma>0}$ as above and pairing against $u$ (and taking $K_1,K_3$ large enough), we have  
\begin{multline}
	2 \Im \langle \tilde{P}u,A_\varepsilon u \rangle_{L^2} = -\lVert B_\varepsilon u \rVert^2_{L^2} + \lVert E_\varepsilon u \rVert_{L^2}^2 + \langle \tilde{P} u, F_\varepsilon u \rangle_{L^2} + \langle R_\varepsilon u,u \rangle_{L^2} \\ - 2\underline{\delta} \lVert x\Lambda_{1/2,-(s+1)/2,-(\varsigma+3)/2,-(l+1)/2,-(\ell+3)/2} A_\varepsilon u \rVert^2_{L^2}.  \label{eq:misc_0kk}
\end{multline} 
\Cref{eq:misc_0kk} implies
\begin{multline}
	\lVert B_\varepsilon u \rVert_{L^2}^2 + 2\underline{\delta}  \lVert \Lambda_{1/2,-(s+1)/2,-(\varsigma+3)/2,-(l+1)/2,-(\ell+3)/2} A_{\varepsilon} u \rVert^2_{L^2} \\ \leq 2| \langle \tilde{P}u,A_\varepsilon u \rangle_{L^2}| + | \langle \tilde{P}u,F_\varepsilon u \rangle_{L^2}| + | \langle R_\varepsilon u, u \rangle_{L^2}| + \lVert E_\varepsilon u  \rVert_{L^2}^2 .
	\label{eq:misc_k64}
\end{multline}

Suppose $G \in \Psi^{-\infty,0,0,0,0}_{\mathrm{leC}}(X)$ is self-adjoint and such that \begin{equation} 
	\operatorname{WF}_{\mathrm{leC}}^{\prime0,0}(1-G) \cap \operatorname{supp} \chi(\tilde{p}^{2,0,-2,-1,-3})\chi(\rho) \psi = \varnothing.
\end{equation}  
We now estimate the terms in \cref{eq:misc_k64} as in the propagation estimate, except we now must keep track of orders at $\mathrm{bf}$, $\mathrm{tf}$: 
\begin{itemize}
	\item Writing $\tilde{P} = (1-G)\tilde{P}+G\tilde{P}$ 
	we have, for each $N\in \bbN$ (and $K_1,K_3$ large enough), 
	\begin{equation}
		 |\langle \tilde{P} u,A_{\varepsilon} u \rangle_{L^2} | \leq  | \langle G\tilde{P} u,A_{\varepsilon} u \rangle_{L^2} |+ |\langle \tilde{P} u,(1-G)A_{\varepsilon} u \rangle_{L^2} |.
	\end{equation} 
	By \Cref{lem:duality}, for any $\overline{\delta}>0$,  
	\begin{align}
		\begin{split} 
		| \langle G\tilde{P} u,A_{\varepsilon} u \rangle_{L^2} | 
		 &\preceq \overline{\delta}^{-1}\lVert G \tilde{P} u \rVert_{\calY_N}^2 + \overline{\delta}\lVert A_\varepsilon u \rVert_{\calY_{N}^*}^2  + \overline{\delta} \lVert u \rVert_{\bar{\calE}_N}^2  \\
		&\preceq \overline{\delta}^{-1}\lVert G \tilde{P} u \rVert_{\calY_N}^2 + \overline{\delta}\lVert A_\varepsilon u \rVert_{\calY_{*,N}}^2   \\
		&\preceq \overline{\delta}^{-1} \lVert G \tilde{P} u \rVert_{\calY_N}^2 + \overline{\delta}\lVert \Lambda_{1/2,-(s+1)/2,-(\varsigma+3)/2,-(l+1)/2,-(\ell+3)/2} A_\varepsilon u \rVert_{L^2}^2 + \overline{\delta} \lVert u \rVert_{\bar{\calE}_N}^2 
		\end{split} 
	\end{align}
	(where the constant does not depend on $\overline{\delta}$) 
	for	
	\begin{align}
		\begin{split}
			\calY_N&=H_{\mathrm{leC}}^{-N,(s+1)/2,(\varsigma+3)/2,(l+1)/2,(\ell+3)/2}(X),\\ \calY_{*,N}&=H_{\mathrm{leC}}^{-N,-(s+1)/2,-(\varsigma+3)/2,-(l+1)/2,-(\ell+3)/2}(X), \\
			\bar{\calE}_N &= H_{\mathrm{leC}}^{-N,-N,-N,-N,-N}(X). 
		\end{split}
	\end{align}
	On the other hand, noting that $ (1-G)A_{\bullet} \in L^\infty ([0,1]_\varepsilon ; \Psi_{\mathrm{leC}}^{-\infty, -\infty,-\infty,l/2,\ell/2}(X))$,
	\begin{align}
		\begin{split} 
		|\langle \tilde{P} u,(1-G)A_{\varepsilon} u \rangle_{L^2} | &\preceq \lVert \tilde{P} u \rVert_{H_{\mathrm{leC}}^{-N_0,-N_0,-N_0,\frac{l+1}{2},\frac{\ell+3}{2}} }^2 + \lVert (1-G)A_\varepsilon u \rVert_{H_{\mathrm{leC}}^{N_0,N_0,N_0,-\frac{l+1}{2},-\frac{\ell+3}{2}} }^2 \\
		&\preceq \lVert \tilde{P} u \rVert_{H_{\mathrm{leC}}^{-N_0,-N_0,-N_0,\frac{l+1}{2},\frac{\ell+3}{2}} }^2 + \lVert (1-G)A_\varepsilon u \rVert_{H_{\mathrm{leC}}^{-N_0,-N_0,-N_0,-\frac{l+1}{2},-\frac{\ell+3}{2}} }^2 + \lVert u \rVert_{\bar{\calE}_{N_0}}^2 \\
		&\preceq \lVert u \rVert_{H_{\mathrm{leC}}^{-N,-N,-N,\frac{l-1}{2},\frac{\ell-3}{2}} } ^2
		\end{split} 
	\end{align}
	for $N_0$ sufficiently large (relative to $N$). So, $|\langle \tilde{P} u,(1-G)A_{\varepsilon} u \rangle_{L^2} | \preceq \lVert u \rVert^2_{\calE_N} $ for $\calE_N = H_{\mathrm{leC}}^{-N,-N,-N,\frac{l-1}{2},\frac{\ell-3}{2}}(X)$. 
	
	Thus, taking $\overline{\delta}$ sufficiently small (relative to $\underline{\delta})$, 
	\begin{equation}
		2|\langle \tilde{P} u,A_{\varepsilon} u \rangle_{L^2} | - 2\underline{\delta}  \lVert \Lambda_{1/2,-(s+1)/2,-(\varsigma+3)/2,-(l+1)/2,-(\ell+3)/2} A_{\varepsilon} u \rVert^2_{L^2} \\ \preceq \lVert G \tilde{P} u \rVert_{\calY_N}^2 + \lVert u \rVert_{\calE_N}^2 .
	\end{equation}
	\item Similarly, $|\langle \tilde{P} u , F_\varepsilon u \rangle_{L^2}| \preceq \lVert G \tilde{P} u \rVert_{\calY_N}^2 + \lVert F_\varepsilon u\rVert_{\calY_{*,N_0}}^2 + \lVert u \rVert_{\calE_N}^2$, for any $N_0\in \bbN$. 
	
	We have $\operatorname{WF}_{L^\infty,\mathrm{leC}}(F_\bullet) \cap \operatorname{Char}_{\mathrm{leC}}^{2,0,-2,-1,-3}(\tilde{P})=\varnothing$, so we deduce from elliptic regularity that 
	\begin{equation}
		\lVert F_\varepsilon u \rVert^2_{\calY_{*,N_0}} \preceq \lVert G \tilde{P} u\rVert^2_{\calY_N} + \lVert u \rVert_{\calE_N}^2  
	\end{equation}
	for $N_0$ sufficiently large (relative to $N$). Thus, 
	\begin{equation}
		|\langle \tilde{P} u , F_\varepsilon u \rangle_{L^2}|  \preceq  \lVert G \tilde{P} u\rVert^2_{\calY_N} + \lVert u \rVert_{\calE_N}^2 . 
	\end{equation}
	\item $|\langle R_\varepsilon u ,u \rangle_{L^2} | \preceq \lVert G u \rVert_{\calZ_N}^2 + \lVert u \rVert_{\calE_N}^2$ for 
	\begin{equation}
		\calZ_N = H_{\mathrm{leC}}^{-N,(s-2)/2,(\varsigma-4)/2,(l-1)/2,(\ell-3)/2}(X). 
	\end{equation}
\end{itemize}
Thus,
\begin{align}
	\lVert B_\varepsilon u \rVert_{L^2}^2 &\preceq \lVert G\tilde{P} u \rVert_{\calY_N}^2 + \lVert G u \rVert_{\calZ_N}^2 + \lVert E_\varepsilon u \rVert^2_{L^2} + \lVert u \rVert_{\calE_N}^2, 
	\intertext{i.e., letting $\calX_N = H_{\mathrm{leC}}^{-N,(s-1)/2,(\varsigma-3)/2,(l-1)/2,(\ell-3)/2}(X)$ and $\bar{\calX}_N = H_{\mathrm{leC}}^{-N,(s-1)/2,(\varsigma-3)/2,-N,-N}(X)$, } 
	\lVert \hat{B}_\varepsilon u \rVert_{\calX_N}^2 &\preceq \lVert G\tilde{P} u \rVert_{\calY_N}^2 + \lVert G u \rVert_{\calZ_N}^2 + \lVert \hat{E}_\varepsilon u \rVert^2_{\bar{\calX}_N} + \lVert u \rVert_{\calE_N}^2, \\
	\lVert \hat{B}_\varepsilon u \rVert_{\calX_N} &\preceq \lVert G\tilde{P} u \rVert_{\calY_N} + \lVert G u \rVert_{\calZ_N} + \lVert \hat{E}_\varepsilon u \rVert_{\bar{\calX}_N} + \lVert u \rVert_{\calE_N}, 
	\label{eq:misc_9i9}
\end{align}
where
\begin{align}
	\hat{B}_\varepsilon &= \Lambda_{0,-(s-1)/2,-(\varsigma-3)/2,-(l-1)/2,-(\ell-3)/2} B_\varepsilon \\
	\hat{E}_\varepsilon &= \Lambda_{0,-(s-1)/2,-(\varsigma-3)/2,0,0}E_\varepsilon, 
\end{align}
so that $\hat{B}_\bullet,\hat{E}_\bullet \in \Psi_{\mathrm{leC}}^{0,0,0,0,0}(X)$. 

\begin{proposition}
	Suppose that $G_0,G_1 \in \Psi_{\mathrm{leC}}^{-\infty,0,0,0,0}(X)$, $G_2 \in \Psi_{\mathrm{leC}}^{-\infty,0,0,-\infty,-\infty}$ (with essential support away from $\mathrm{bf}\cup\mathrm{tf}$) satisfy
	\begin{enumerate}
		\item $\operatorname{WF}_{\mathrm{leC}}^{\prime}(G_2) \cap \calR_+ =\varnothing$ and 
		\begin{equation} 
			\operatorname{WF}_{\mathrm{leC}}^{\prime0,0}(G_1) \cap \calR_+ =\varnothing,
		\end{equation}
		\item $\calR, \operatorname{WF}^{\prime 0,0}_{\mathrm{leC}}(G_1) \subset \operatorname{Ell}^{0,0,0,0,0}_{\mathrm{leC}}(G_0)$, 
		\item there exist $\Theta_1,\Theta_2 \in (0,\pi)$ with $\Theta_1<\Theta_2$ such that $\operatorname{Ell}^{0,0,0,0,0}_{\mathrm{leC}}(G_2) \supseteq \calP[\Theta_1,\Theta_2]$ and 
		\begin{equation}
			\operatorname{Ell}_{\mathrm{leC}}^{0,0,0,0,0}(G_0) \supset  \calR\cup \bigcup_{\Theta_3\in (\Theta_2,\pi)}\calP[\Theta_1,\Theta_3] 
		\end{equation}
	\end{enumerate} 
	Then, for any $\Sigma>0$, $N\in \bbN$, and $m,s,\varsigma,l,\ell \in \bbR$ with $l< -1/2$, $\ell< -3/2$, and $\varsigma\leq\ell +s-l$, there exists a constant 
	\begin{equation} 
		C=C(\tilde{P},G_0,G_1,G_2,G_3, \Sigma,N,m,s,\varsigma,l,\ell)>0
	\end{equation} such that  
	\begin{equation}
		\lVert G_1 u \rVert_{H_{\mathrm{leC}}^{m,s,\varsigma,l,\ell}} \leq C\Big[ \lVert G_0 \tilde{P} u \rVert_{H_{\mathrm{leC}}^{-N,s+1,\varsigma+3,l+1,\ell+3} } + \lVert G_2 u \rVert_{H_{\mathrm{leC}}^{-N,s,\varsigma,-N,-N}} + \lVert  u \rVert_{H_{\mathrm{leC}}^{-N,-N,-N,l,\ell}}\Big]
		\label{eq:misc_fin}
	\end{equation}
	holds for all $u \in \calS'(X)$ and $\sigma\in [0,\Sigma]$. 
\end{proposition}
\begin{proof} 
	We can assume without loss of generality that 
	\begin{equation}
		\calR\subseteq \operatorname{Ell}^{0,0,0,0,0}_{\mathrm{leC}}(G_1)\label{eq:misc_rek}
	\end{equation} and 
	\begin{equation}
		\operatorname{Ell}^{0,0,0,0,0}_{\mathrm{leC}}(G_0) \supseteq \operatorname{WF}'_{\mathrm{leC}}(G_2).
	\end{equation} 
	
	Let $s_0 = 2s+1$, $l_0 = 2l+1$, and $\ell_0 = 2\ell+3$, and $\varsigma_0 = 2 \varsigma+3$. The condition $l<-1/2$ is equivalent to $l_0<0$, $\ell < -3/2$ is equivalent to $\ell_0<0$, and $\varsigma\leq \ell+s-l$ is equivalent to $\varsigma_0\leq \ell_0+s_0-l_0$. We can therefore apply \cref{eq:misc_9i9} with $s_0,\varsigma_0,l_0,\ell_0$ in place of what we called ``$s,\varsigma,l,\ell$'' there. Thus, for each $\varepsilon>0$, 
	\begin{equation}
		\lVert \hat{B}_\varepsilon u \rVert_{\calX_N} \preceq \lVert G\tilde{P} u \rVert_{\calY_N} + \lVert G u \rVert_{\calZ_N} + \lVert \hat{E}_\varepsilon u \rVert_{\bar{\calX}_N} + \lVert u \rVert_{\calE_N},
		\label{eq:misc_alk}
	\end{equation}
	where $G,\hat{B}_\bullet,\hat{E}_\bullet$ are as above and now 
	\begin{equation} 
		\calX_N = H_{\mathrm{leC}}^{-N,s,\varsigma,l,\ell}(X),
	\end{equation} 
	$\calZ_N = H_{\mathrm{leC}}^{-N,s-1/2,\varsigma-1/2,l,\ell}(X)$, $\bar{\calX}_N = H_{\mathrm{leC}}^{-N,s,\varsigma,-N,-N}(X)$, $\calY_N= H_{\mathrm{leC}}^{-N,s+1,\varsigma+3,l+1,\ell+3}(X)$,
	and 
	\begin{equation} 
		\calE_N = H_{\mathrm{leC}}^{-N,-N,-N,l,\ell}(X).
	\end{equation} 
	If necessary, we can retroactively replace $\chi$, $\psi$ such that 
	\begin{equation} \operatorname{WF}^{\prime}_{L^\infty,\mathrm{leC}}(E_\bullet),\operatorname{WF}^{\prime0,0}_{L^\infty,\mathrm{leC}}(B_\bullet) \subset \operatorname{Ell}^{0,0,0,0,0}_{\mathrm{leC}}(G_0).
	\end{equation}  
	
	We can apply the propagation estimate \Cref{prop:propagation_estimate} to bound 
	\begin{equation}
		\lVert \hat{E}_\varepsilon u \rVert_{\bar{\calX}_N} \preceq \lVert G_0 \tilde{P} u  \rVert_{\calY_N}+ \lVert G_2 u \rVert_{\bar{\calX}_N} + \lVert u \rVert_{\calE_N} .
	\end{equation}
	By \cref{eq:misc_rek}, 
	we can now retroactively choose $G$ (for appropriate $\chi,\psi$) such that $\operatorname{WF}_{\mathrm{leC}}^{\prime0,0}(G)\subset \operatorname{Ell}^{0,0,0,0,0}_{\mathrm{leC}}(G_0) \cap \operatorname{Ell}_{\mathrm{leC}}^{0,0,0,0,0}(G_1)$, so that 
	\begin{equation}
		\lVert G \tilde{P} u \rVert_{\calY_N} \preceq \lVert G_0 \tilde{P} u \rVert_{\calY_N} + \lVert u \rVert_{\calE_N} 
	\end{equation}
	and $\lVert G u \rVert_{\calZ_N} \preceq \lVert G_1  u \rVert_{\calZ_N} + \lVert u \rVert_{\calE_N}$.

	The estimate \cref{eq:misc_alk} therefore implies $\lVert \hat{B}_\varepsilon u \rVert_{\calX_N}\preceq \lVert G_0 \tilde{P} u \rVert_{\calY_N} + \lVert G_1 u \rVert_{\calZ_N}+ \lVert G_2 u \rVert_{\bar{\calX}_N} + \lVert u \rVert_{\calE_N}$. Via the compactness argument utilized in the proof of the propagation estimate and previous radial point estimate, $\lVert \hat{B}_{0} u(-;\sigma) \rVert_{\calX_N}\leq  \liminf_{\varepsilon\to0^+} \lVert \hat{B}_\varepsilon u(-;\sigma) \rVert_{\calX_N}$ for each $\sigma \geq 0$, so 
	\begin{align}
		\lVert \hat{B}_0 u \rVert_\calX&\preceq \lVert G_0 \tilde{P} u \rVert_{\calY_N} + \lVert G_1 u \rVert_{\calZ_N}+ \lVert G_2 u \rVert_{\bar{\calX}_N} + \lVert u \rVert_{\calE_N}. 
		\label{eq:misc_g31}
		\intertext{Unlike $\hat{B}_\varepsilon$ for $\varepsilon>0$, $\smash{\operatorname{Ell}^{0,0,0,0,0}_{\mathrm{leC}}}(\hat{B}_0)\supset \calR$. Thus, $\lVert G_1 u \rVert_{\calX_N}\preceq \lVert \hat{B}_0 u \rVert_{\calX_N}  + \lVert G_0 \tilde{P} u \rVert_{\calY_N}  + \lVert G_2 u \rVert_{\bar{\calX}_N} + \lVert u \rVert_{\calE_N}$, using propagation/elliptic estimates to control $G_1 u$ off $\calR$ in terms of $G_0 \tilde{P}u$ and $G_2u$ (and a global $\calE_N$ error). Substituting \cref{eq:misc_g31} into the right-hand side, we get}
		\lVert G_1 u \rVert_{\calX_N} &\preceq \lVert G_0 \tilde{P} u \rVert_{\calY_N}  + \lVert G_1 u \rVert_{\calZ_N}+ \lVert G_2 u \rVert_{\bar{\calX}_N} + \lVert u \rVert_{\calE_N}. 
		\label{eq:misc_oi1}
	\end{align}
	Since the leC- Sobolev spaces $\calX_N,\calY,\calZ_N$ in \cref{eq:misc_oi1} get weaker as $s,\varsigma,l,\ell$ decrease, we can inductively use the family of estimates \cref{eq:misc_oi1} to bound the $\lVert G_1 u \rVert_{\calZ_N}$ term on the right-hand side: for all $N\in \bbN$, 
	\begin{align}
		\lVert G_1 u \rVert_{\calX_N} &\preceq \lVert G_0 \tilde{P} u \rVert_{\calY_N}  + \lVert G_1 u \rVert_{ \calE_N}+ \lVert G_2 u \rVert_{\bar{\calX}_N} + \lVert u \rVert_{\calE_N} \\
		\lVert G_1 u \rVert_{\calX_N} &\preceq \lVert G_0 \tilde{P} u \rVert_{\calY_N}  + \lVert G_2 u \rVert_{\bar{\calX}_N} + \lVert u \rVert_{\calE_N}.
	\end{align}
	This implies \cref{eq:misc_fin}. 
\end{proof}
\subsection{Upshot}
\label{subsec:upshot}

Combining the estimates above (e.g. using \Cref{lem:combination}), we get:
\begin{theoremp}
	\label{thm:symbolic_main}
	Suppose that $G_1,G_2,G_3 \in \Psi^{0,0,0,0,0}_{\mathrm{leC}}(X)$ satisfy 
	\begin{itemize}
		\item $\operatorname{WF}_{\mathrm{leC}}^{\prime 0,0}(G_1), \operatorname{Char}_{\mathrm{leC}}^{2,0,-2,-1,-3}(\tilde{P}) \subseteq \operatorname{Ell}^{0,0,0,0,0}_{\mathrm{leC}}(G_2)$, 
		\item $\calR_+ \subseteq \operatorname{Ell}_{\mathrm{leC}}^{0,0,0,0,0}(G_3)$. 
	\end{itemize}
	For every $m,s,\varsigma,l, \ell,s_0, \varsigma_0 \in \bbR$ satisfying 
	\begin{equation} 
		l<-1/2, \qquad \ell <-3/2 , \qquad -1/2<s_0<s, \qquad -3/2 < \varsigma_0 < \varsigma \leq \ell+s-l, 
	\end{equation}
	there exists, for each $\Sigma>0$ and $N\in \bbN$, a constant $C=C(\tilde{P},G_1,G_2,G_3, \Sigma,N,m,s,\varsigma,l,\ell,s_0,\varsigma_0)>0$
	such that 
	\begin{equation}
		\lVert G_1 u \rVert_{H_{\mathrm{leC}}^{m,s,\varsigma,l,\ell}} \leq C \Big[\lVert G_2 \tilde{P} u \rVert_{H_{\mathrm{leC}}^{m-2,s+1,\varsigma+3,l+1,\ell+3}}  + \lVert G_3 u \rVert_{H_{\mathrm{leC}}^{-N,s_0,\varsigma_0,-N,-N} } + \lVert u \rVert_{H_{\mathrm{b,leC}}^{-N,l,\ell}} \Big]
	\end{equation}
	for all $u\in \calS'(X)$ and $\sigma \in [0,\Sigma]$ (in the usual strong sense that if the right-hand side is finite, then the left-hand side as well, and the stated inequality holds).
\end{theoremp}

It will also be useful to have the following refinement of \Cref{thm:symbolic_main}:
\begin{proposition}
	\label{prop:refined_combined}
	Given $G_1,G_2,m,s,\varsigma,l,\ell, s_0,\varsigma_0$ as in the setup of \Cref{thm:symbolic_main}, if in addition $\operatorname{Ell}_{\mathrm{leC}}^{0,0,0,0,0}(G_1)\supseteq \calR_+$, then there exists a constant 
	\begin{equation} 
		c=c(\tilde{P},G_1,G_2, \Sigma,N,m,s,\varsigma,l,\ell,s_0)>0
	\end{equation} 
	such that, for any $u\in \calS'(X)$,
	\begin{equation}
		\lVert G_1 u \rVert_{H_{\mathrm{leC}}^{m,s,\varsigma,l,\ell}} \leq c  \Big[\lVert G_2 \tilde{P} u \rVert_{H_{\mathrm{leC}}^{m-2,s+1,\varsigma+3,l+1,\ell+3}}   + \lVert u \rVert_{H_{\mathrm{b,leC}}^{-N,l,\ell}} \Big]
		\label{eq:misc_ufu}
	\end{equation}
	holds for all $\sigma \in [0,\Sigma]$ such that $\lVert G_1(\sigma) u  \rVert_{H_{\mathrm{leC}}^{-N,s_0,\varsigma_0,-N,-N}(X)(\sigma) } < \infty$. 
\end{proposition}
\begin{proof}
	It suffices to consider the case $-N<m,l,\ell$. 
	
	Applying \Cref{thm:symbolic_main} with $G_3=G_1$, 
	\begin{equation}
		\lVert G_1 u \rVert_{H_{\mathrm{leC}}^{m,s,\varsigma,l,\ell}} \leq C'( \lVert G_2 \tilde{P} u \rVert_{H_{\mathrm{leC}}^{m-2,s+1,\varsigma+3,l+1,\ell+3}}  + \lVert G_1 u \rVert_{H_{\mathrm{leC}}^{-N,s_0,\varsigma_0,-N,-N} } + \lVert u \rVert_{H_{\mathrm{b,leC}}^{-N,l,\ell}}) 
		\label{eq:misc_utt}
	\end{equation}
	for some constant $C'>0$. 
	We now use the interpolation inequality \Cref{prop:interpolation}, which, for each $\epsilon>0$, allows us to bound 
	\begin{equation}
		\lVert G_1 u \rVert_{H_{\mathrm{leC}}^{-N,s_0,\varsigma_0,-N,-N} } \leq \epsilon \lVert G_1 u \rVert_{H_{\mathrm{leC}}^{m,s,\varsigma,l,\ell}} + C''(\epsilon) \lVert u \rVert_{H_{\mathrm{b,leC}}^{-N,l,\ell}}
		\label{eq:misc_uyy}
	\end{equation}
	for some $C''(\epsilon) =  C(\tilde{P},G_1,G_2, \Sigma,N,m,s,\varsigma,l,\ell,s_0,\varsigma_0,\epsilon)>0$. Taking $\epsilon < 1/2C'$, plugging \cref{eq:misc_uyy} into \cref{eq:misc_utt} yields 
	\begin{equation}
		\lVert G_1 u \rVert_{H_{\mathrm{leC}}^{m,s,\varsigma,l,\ell}} \leq C'( \lVert G_2 \tilde{P} u \rVert_{H_{\mathrm{leC}}^{m-2,s+1,\varsigma+3,l+1,\ell+3}}  + (1+C'')\lVert u \rVert_{H_{\mathrm{b,leC}}^{-N,l,\ell}}) + (1/2) \lVert G_1 u \rVert_{H_{\mathrm{leC}}^{m,s,\varsigma,l,\ell}}.
	\end{equation}
	If $\sigma \in [0,\Sigma]$ is such that
	\begin{equation} 
		\lVert G_1(\sigma) u(-;\sigma) \rVert_{H_{\mathrm{leC}}^{m,s,\varsigma,l,\ell}(\sigma)}<\infty,
	\end{equation} 	
	then we can subtract the last term on the right-hand side from both sides, getting \cref{eq:misc_ufu},
	\begin{equation}
		\lVert G_1 u \rVert_{H_{\mathrm{leC}}^{m,s,\varsigma,l,\ell}} \leq 2C'(1+C'')( \lVert G_2 \tilde{P} u \rVert_{H_{\mathrm{leC}}^{m-2,s+1,\varsigma+3,l+1,\ell+3}}  +\lVert u \rVert_{H_{\mathrm{b,leC}}^{-N,l,\ell}}). 
		\label{eq:misc_uuy}
	\end{equation}
	If $\lVert G_1(\sigma) u(-;\sigma) \rVert_{H_{\mathrm{leC}}^{-N,s_0,\varsigma_0,-N,-N}(X)(\sigma) } < \infty$, then we break into two cases:
	\begin{itemize}
		\item if one of $\lVert G_2 \tilde{P} u \rVert_{H_{\mathrm{leC}}^{m-2,s+1,\varsigma+3,l+1,\ell+3}}$, $\lVert u \rVert_{H_{\mathrm{b,leC}}^{-N,l,\ell}}$ is infinite, then \cref{eq:misc_ufu} holds trivially;
		\item if both are finite, then \Cref{thm:symbolic_main} implies that $\lVert G_1 u \rVert_{H_{\mathrm{leC}}^{m,s,\varsigma,l,\ell}}<\infty$, and therefore \cref{eq:misc_uuy} holds.
	\end{itemize} 
	We conclude that \cref{eq:misc_ufu} holds if we take $C=2C'(1+C'')$. 
\end{proof}

Taking $G_1=G_2=1$ in \Cref{prop:refined_combined}, we get the main claim \cref{eq:misc_gg1} at the beginning of this section:
\begin{propositionp}
	\label{prop:symbolic_final}
	For every $\Sigma>0$, $N\in \bbN$, and $m,s,\varsigma,l, \ell,s_0,\varsigma_0 \in \bbR$ satisfying  $l<-1/2<s_0<s$ and $\ell<-3/2<\varsigma_0<\varsigma\leq \ell+s-l$, there exists a constant $C=C(\tilde{P},\Sigma,N,m,s,\varsigma,l,\ell)>0$ such that 
	\begin{equation}
		\lVert u \rVert_{H_{\mathrm{leC}}^{m,s,\varsigma,l,\ell} } \leq C\Big[ \lVert \tilde{P} u \rVert_{H_{\mathrm{leC}}^{m-2,s+1,\varsigma+3,l+1,\ell+3}} + \lVert u \rVert_{H_{\mathrm{b,leC}}^{-N,l,\ell}} \Big]
	\end{equation}
	holds for all $u\in \calS'(X)$ and $\sigma \in [0,\Sigma]$ such that $\lVert u  \rVert_{H_{\mathrm{leC}}^{-N,s_0,\varsigma_0,-N,-N}(X)(\sigma)}<\infty$.  
\end{propositionp}

\section{Proof of main theorem}
\label{sec:mainproof}

Most of the results in this section are split among three subsections: 
\begin{enumerate}
	\item \S\ref{subsec:estimates}, where estimates regarding the ``leC-normal operator'' $\operatorname{N}(\tilde{P})$ (see \S\ref{sec:operator}) are proven,
	\item \S\ref{subsec:smoothness1}, where -- using the estimates from the previous subsection -- the conormality of the output of the conjugated resolvent family on $X^{\mathrm{sp}}_{\mathrm{res}}$ is established, along with smoothness (in terms of $E=\sigma^2$) at $\mathrm{zf}^\circ$ (with the terms in the Taylor series being conormal distributions on $\mathrm{zf}$) --- see \Cref{prop:grace} --- and 
	\item \S\ref{subsec:smoothness2}, where the proposition needed to upgrade the conormality established in the previous subsection to the existence of asymptotic expansions is proven.
\end{enumerate}

Central to this section is the analysis of the model problem 
\begin{equation}
	2 i \Big[ \hat{x} \partial_{\hat{x}} + \frac{  \hat{x} }{1+\hat{x}}\Big( k+\frac{1}{4} \Big)  + l + \frac{1}{2} \Big] u = f,  \label{eq:model}
\end{equation}
$u,f \in \calD'(\bbR^+_{\hat{x}})$, where $l,k \in \bbR$ will be required to satisfy some threshold inequalities. We record the exact solution to \cref{eq:model}: for any $l,k\in \bbR$, if $f \in \smash{C^\infty(\bbR^+)}$, then $u\in \calD'(\bbR^+)$ satisfies \cref{eq:model} if and only if 
\begin{equation}
	u(\hat{x}) = \hat{x}^{-l-1/2 } (1+\hat{x})^{-k-1/4} \Big[ c+ \frac{i}{2} \int_{\hat{x}}^1  \hat{x}_0^{l-1/2 } (1+\hat{x}_0)^{k+1/4} f(\hat{x}_0) \dd \hat{x}_0   \Big] 
	\label{eq:model_soln}
\end{equation}
for some $c \in \bbC$. 
The model problem above arises rewriting  
\begin{equation} 
	\tilde{\operatorname{N}}(\tilde{P})=x^{-l-n/2}(\sigma^2+\mathsf{Z}x)^{-k}\operatorname{N}(\tilde{P})x^{l+n/2}(\sigma^2+\mathsf{Z}x)^{k}
\end{equation} 
in terms of $\hat{x} = \mathsf{Z}x/\sigma^2$. Indeed, 
\begin{align}
	\begin{split} 
	x^{-1} (\sigma^2+\mathsf{Z} x)^{-1/2}\tilde{\operatorname{N}}(\tilde{P}) &= 2i \Big( x \partial_x +l + \frac{1}{2} \Big) + \frac{2  i \mathsf{Z}x}{\sigma^2+\mathsf{Z}x} \Big(k + \frac{1}{4} \Big)  \\
	&= 2i \Big[\hat{x} \partial_{\hat{x}} + \frac{  \hat{x} }{1+\hat{x}} \Big( k+\frac{1}{4} \Big)  +l + \frac{1}{2}\Big].
	\end{split} 
	\label{eq:misc_i12}
\end{align}
If $\operatorname{N}(\tilde{P})u = f$ then $\tilde{\operatorname{N}}(\tilde{P})u_0 = f_0$ for $u_0 = x^{-l-n/2} (\sigma^2+\mathsf{Z}x)^{-k} u$ and $f_0 = x^{-l-n/2} (\sigma^2+\mathsf{Z}x)^{-k} f$. 

We now spell out the deduction of \Cref{thm:main} (in the form of the more general \Cref{prop:main}) from the results in \S\ref{subsec:smoothness1}, \S\ref{subsec:smoothness2}.

\begin{proposition}
	\label{prop:form0}
	Suppose that $g=g_0 + x g_1 + x^{\alpha_1} g_2$ for $g_1 \in  C^\infty(X; {}^{\mathrm{sc}}\mathrm{Sym}^2 T^* X)$ and $g_2 \in S^{0} (X; {}^{\mathrm{sc}}\mathrm{Sym}^2 T^* X)$ for $\alpha_1>1$. Then the Laplace-Beltrami operator $\triangle_g$ has the form 
	\begin{equation}
		\triangle_g = \triangle_{g_0} +  x\operatorname{Diff}^{2,0,-2}_{\mathrm{scb}}(X) + x^{\alpha_1} S \operatorname{Diff}^{2,0,-2}_{\mathrm{scb}}(X),
		\label{eq:misc_trg}
	\end{equation}
	where $g_0$ is an exactly conic metric.
\end{proposition}
\begin{proof}
	Let $h\mapsto [h]$ denote the natural bundle monomorphism ${}^{\mathrm{sc}}\mathrm{Sym}^2 T^* X \hookrightarrow \operatorname{End}({}^{\mathrm{sc}}TX, {}^{\mathrm{sc}}T^*X)$.
	
	The matrix identity $[g]^{-1} = [g_0]^{-1} - [g]^{-1} (x[g_1] + x^{\alpha_1}[g_2]) [g_0]^{-1}$, applied inductively, yields
	\begin{equation}
		[g]^{-1} -[g_0]^{-1} = [g_0]^{-1} \sum_{k=1}^{K} (-1)^k(  (x[g_1] + x^{\alpha_1} [g_2])[g_0]^{-1})^k + (-1)^{K+1} [g]^{-1} (  (x[g_1] + x^{\alpha_1} [g_2])[g_0]^{-1})^{K+1} 
	\end{equation}
	for each $K\in \bbN$. Noting that $g^{-1} \in S^0(X;{}^{\mathrm{sc}}\!\operatorname{Sym}^2 T X)$, taking $K>\alpha_1$
	leads to $g^{-1} -g_0^{-1} \in x C^\infty (X ; {}^{\mathrm{sc}}\!\operatorname{Sym}^2 T X ) + x^{\alpha_1} S^0(X; {}^{\mathrm{sc}}\!\operatorname{Sym}^2 T X )$. 
	Also, from $[g] = [g_0](1+x[g_0]^{-1} [g_1] + x^{\alpha_1} [g_0]^{-1} [g_2])$, $\det g \in S^0(X;{}^{\mathrm{sc}}\Omega X)$ satisfies 
	\begin{align}
		\det g = (\det g_0)(1+x \operatorname{tr}([g_0]^{-1}[g_1]) + x^2 C^\infty(X)+x^{\alpha_1} S^0(X) ). 
	\end{align}

	Writing 
	\begin{equation} 
		\triangle_g = -\sum_{i,j=1}^n (g^{ij} \partial_i \partial_j  + \partial_i g^{ij}\partial_j + (1/2)(\det g)^{-1}(\partial_i \det g) g^{ij} \partial_j  )
		\label{eq:laplacian_form}
	\end{equation} 
	in local coordinates, we conclude that \cref{eq:misc_trg} holds locally, which suffices to show that the decomposition \cref{eq:misc_trg} can be done globally. 
\end{proof}
Thus:
\begin{propositionp} 
	\label{prop:form}
	Suppose that $g=g_0 + x g_1 + x^{\alpha_1} g_2$ for $g_1 \in  C^\infty(X; {}^{\mathrm{sc}}\mathrm{Sym}^2 T^* X)$ and $g_2 \in S^{0} (X; {}^{\mathrm{sc}}\mathrm{Sym}^2 T^* X)$ for $\alpha_1>1$.	
	If $P(\sigma) = \triangle_g - \sigma^2 - \mathsf{Z} x + V$ is the spectral family of an attractive Coulomb-like Schr\"odinger operator with $V\in x^2 C^\infty(X) + x^{\alpha_2} S^0(X)$ for $\alpha_2>3/2$, then we can decompose 
	\begin{equation}
		P(\sigma) = P_0(\sigma) + P_1 + P_2
		\label{eq:misc_ppp}
	\end{equation}
	for $P_0,P_1,P_2$ having the form specified in the introduction (with $a=0$), except that 
	\begin{equation} 
		a_{00} =  -x^4(g_1)_{00}|_{\partial X} = - \lim_{x\to 0^+} x^4 g_1(\partial_x,\partial_x) 
		\label{eq:a00_def}
	\end{equation} 
	is not necessarily constant (and the attractivity condition \cref{eq:ac} might only be satisfied for small $\sigma$).
	We can arrange that $P_1$ is fully classical and that $P_2$ is classical to order $(\beta_2,\beta_3)$  with $\beta_2 = \alpha_1-1$ and $\beta_3=\alpha_2-3/2$. 
\end{propositionp}
\begin{proof}
	Let $V_0 \in C^\infty(X), V_1 \in S^0(X)$ satisfy $V=x^2 V_0 + x^{\alpha_1} V_1$. 
	It suffices to restrict attention to the boundary-collar. There, we define $P_0$ by 
	\begin{equation}
		P_0(\sigma) = \triangle_{g_0} - x a_{00} (x^2 \partial_x)^2 - \sigma^2 -\mathsf{Z} x.
	\end{equation}
	This has the form specified in the introduction, in the sense that \cref{eq:P0} holds with $a=0$. 
	
	By the proof of \Cref{prop:form0}, there exists some $P_{2,0} \in  \operatorname{Diff}_{\mathrm{scb}}^{2,-2,-3}(X)+x^{\alpha_1} S \operatorname{Diff}_{\mathrm{scb}}^{2,0,-2}(X)$ such that, in any local coordinate patch, 
	\begin{equation}
		\triangle_g  = \triangle_{g_0}  + \sum_{i,j,k} x  g_0^{ik}(g_1)_{k\ell} g_0^{\ell j} \partial_i \partial_j + P_{2,0}. 
	\end{equation}
	Now set $P_1  = \sum_{i,j,k} x  g_0^{ik}(g_1)_{k\ell} g_0^{\ell j} \partial_i \partial_j - x g_0^{0k}(g_1)_{k\ell} g^{\ell0}_0\partial_x^2 =  \sum_{i,j,k} x  g_0^{ik}(g_1)_{k\ell} g_0^{\ell j} \partial_i \partial_j - x^9 (g_1)_{00}\partial_x^2$.
	Then, $P_1$ has the form specified in \cref{eq:misc_000} and is even fully classical. 
	
	Defining $a_{00}$ by \cref{eq:a00_def}, $(g_1)_{00} +x^{-4} a_{00} \in x^{-3} C^\infty(X)$,
	so the operator $P_{2,1} = x^9 (g_1)_{00}\partial_x^2 + x a_{00}(x^2 \partial_x)^2$ is in  $\operatorname{Diff}_{\mathrm{scb}}^{2,-2,-3}(X)$. 
	We therefore set
	\begin{equation} 
		P_2 = P_{2,0} +P_{2,1}+ x^2 V_0+x^{\alpha_2} V_1 \in  \operatorname{Diff}_{\mathrm{sc}}^{2,-2}(X)+ S\operatorname{Diff}_{\mathrm{scb}}^{2,-\alpha_1,-2-\alpha_1}(X) + x^{\alpha_2} S^0(X).
	\end{equation}
	Then \cref{eq:misc_p1i} applies, for $\beta_2=\alpha_1-1$ and $\beta_3=\alpha_2-3/2$. 
\end{proof}

\begin{proposition}
	\label{prop:main}
	Given an asymptotically conic manifold $(X,\iota,g)$ of dimension $\dim X=n\geq 2$ such that $g$ satisfies \cref{eq:g_ass3}, with $a_{00}\in \bbR$, and given $\mathsf{Z}>0$ and $V\in x^2 C^\infty(X)+x^{\alpha_2} S^{0}(X)$ for some $\alpha_2>3/2$, consider the Schr\"odinger operator 
	\begin{equation}
		P = \triangle_g - \mathsf{Z} x + V : \calS'(X)\to \calS'(X). 
	\end{equation}
	Set  
	\begin{equation}
		\Phi(x;\sigma) =    \frac{1}{x} \sqrt{\sigma^2+\mathsf{Z} x-\sigma^2 a_{00} x} + \frac{1}{\sigma}(\mathsf{Z} -\sigma^2 a_{00}) \operatorname{arcsinh} \Big( \frac{\sigma}{x^{1/2}} \frac{1}{(\mathsf{Z}  - \sigma^2 a_{00})^{1/2}} \Big)
	\end{equation}
	for all $\sigma>0$ such that $\mathsf{Z}>\sigma^2 a_{00}$. Suppose that $g$ is classical to $\alpha_1$th order, $\alpha_1>1$. Set $\delta_1=\min\{\alpha_1-1,\alpha_2-1\}$ and $\delta_0 = \min\{\alpha_1-1,\alpha_2-3/2\}$.

	Then, for any $f\in \calS(X)$: 
	\begin{enumerate}[label=(\Roman*)]
		\item  there exists some 
		 \begin{equation} 
		 	u_{0,\pm} \in \calA^{(0,0),\calE,(0,0)}_{\mathrm{loc}}(X^{\mathrm{sp}}_{\mathrm{res}} \cap \{\mathsf{Z}>E a_{00}\}) + \calA_{\mathrm{loc}}^{((0,0),\delta_1),2\delta_0-,(0,0)}(X^{\mathrm{sp}}_{\mathrm{res}} \cap \{\mathsf{Z}>E a_{00}\})
		\end{equation} 
		such that, for $E>0$ satisfying $\mathsf{Z}>E a_{00}$, 
		$u_\pm(-;E^{1/2}) = R(E\pm i 0) f \in \calS'(X)$ can be written as 
		\begin{equation}
			u_\pm  = e^{\pm i \Phi(x;E^{1/2})} x^{(n-1)/2} (E+\mathsf{Z} x)^{-1/4} u_{0,\pm},  
		\end{equation}
		\item  $u_\pm(-;0)  = R(E=0;\mathsf{Z} \pm i0)f$ as 
		$u_\pm(-;0) =e^{\pm i \Phi(x;0)} x^{(n-1)/2} (\mathsf{Z} x)^{-1/4} u_{0,\pm}(-;0)$,
		where $u_{0,\pm}(-;0) \in C^\infty(X_{1/2}) + \calA_{\mathrm{loc}}^{2\delta_0-}(X)$ is the restriction of $u_{0,\pm}$ to $\mathrm{zf} = \operatorname{cl}\{\sigma=0,x>0\} \subset X^{\mathrm{sp}}_{\mathrm{res}}$.
	\end{enumerate} 
	Moreover, the map 
	\begin{equation} 
		\calS(X)\ni f \mapsto u_{0,\pm} \in \calA^{(0,0),\calE,(0,0)}_{\mathrm{loc}}(X^{\mathrm{sp}}_{\mathrm{res}} \cap \{\mathsf{Z}>E a_{00}\}) + \calA_{\mathrm{loc}}^{((0,0),\delta_1),2\delta_0-,(0,0)}(X^{\mathrm{sp}}_{\mathrm{res}} \cap \{\mathsf{Z}>E a_{00}\})
	\end{equation}	
	is continuous.
\end{proposition}
\begin{proof}
	It suffices to prove only the `$+$' case of the proposition, since the `$-$' case is similar (and follows via complex conjugation). 
	Moreover, it suffices to construct $u_{0,\pm}=u_{0,\pm,E_0}$ over every interval of the form $[0,E_0]$ for $E_0>0$ satisfying $\mathsf{Z}>E_0 a_{00}$, since then for $E \leq E_0$ the function $u_{0,\pm,E_0}(-;E^{1/2})$ does not depend on $E_0$  in the sense that for any $E_0,E_0'$ satisfying $\mathsf{Z}>E_0 a_{00},E_0' a_{00}$ it is the case that $u_{0,\pm,E_0}$ and $u_{0,\pm,E_0'}$ agree on $[0,\min\{E_0,E_0'\}]_E$. 
	
	Given the setup of the proposition, \Cref{prop:form} applies, so we can define a family $P=\{P(\sigma)\}_{\sigma\geq 0}$ satisfying the assumptions listed in \S\ref{sec:introduction} such that $P(\sigma) =P(0)-\sigma^2$ for $\sigma^2\leq E_0$, with $P_1$ fully classical and $P_2$ classical to order $(\beta_2,\beta_3)$ for $\beta_2=\alpha_1-1$ and $\beta_2=\alpha_2-3/2$.

	Defining $\tilde{P} = \exp(-i\Phi) P \exp(+i \Phi)$, since $P u_+ = f$, $\tilde{P} u_{00,+} = \tilde{f}$ for $u_{00,+} = \exp(-i\Phi) u_+ \in \calS'(X)$ and $\tilde{f} = \exp(-i\Phi) f \in \calA^{\infty,\infty,(0,0)}(X^{\mathrm{sp}}_{\mathrm{res}})$. Referring to \S\ref{subsec:smoothness1} for the definition of $\tilde{R}_+(\sigma)$, by \cite[Theorem 1.1]{VasyLA} it is the case that 
	\begin{equation}
		u_{00,+}(-;\sigma) = \tilde{R}_+(\sigma) \tilde{f}(-;\sigma)
	\end{equation}
	for each $\sigma>0$ 
	(since, for each $\sigma>0$, $\Phi$ differs from Vasy's phase by a remainder the exponential of which acts as a bounded multiplication operator on b-Sobolev spaces --- cf.\ \cref{eq:osc_+}, with $a=0$).

	By \Cref{prop:grace}, $u_{00,+} = x^{(n-1)/2} (\sigma^2+\mathsf{Z} x)^{-1/4} u_{0,+}$ for some $u_{0,+} \in \calA_{\mathrm{loc}}^{0-,0-,(0,0)}(X_{\mathrm{res}}^{\mathrm{sp}})$, depending continuously on $f$. Then, by \Cref{prop:inductive_smoothness}, we conclude that 
	\begin{equation}
		u_{0,+} \in  \calA^{(0,0),\calE,(0,0)}_{\mathrm{loc}}(X^{\mathrm{sp}}_{\mathrm{res}} \cap \{\mathsf{Z}>E a_{00}\}) + \calA_{\mathrm{loc}}^{((0,0),\delta_1),2\delta_0-,(0,0)}(X^{\mathrm{sp}}_{\mathrm{res}} \cap \{\mathsf{Z}>E a_{00}\}),
	\end{equation}
	depending continuously on $f\in \calS(X)$, with $\delta_1=\min\{\beta_2,1/2+\beta_3\} = \{\alpha_1-1,\alpha_2-1\}$ and $\delta_0=\min\{\beta_2,\beta_3\} = \min\{\alpha_1-1,\alpha_2-3/2\}$. This yields the first half of the proposition (as well as the continuity clause).

	The second clause of this proposition then follows from the second clause of \Cref{prop:grace} and \Cref{prop:LA_at_0}. Indeed, by the latter,  
	\begin{equation}
		e^{-i\Phi(-;0)} u_{+} = e^{-i\Phi(-;0)} R(0;\mathsf{Z}+i0) f = \tilde{R}_+(0) \tilde{f}(-;0).
	\end{equation}
	By the second half of \Cref{prop:grace}, the right-hand side is the restriction to $\mathrm{zf}$ of $u_{+,00}$. 
\end{proof}
\subsection{Estimates}
\label{subsec:estimates} 
\begin{lemma}
	\label{lemma:uniform_normal}
	For any $l,k\in \bbR$ satisfying $l<-1/2$ and $k+l\leq -3/4$, and for any $N\in \bbN$, there exists a $C=C(\bar{x},l,k,N,\mathsf{Z})>0$ such that the linear map 
	\begin{equation} 
	\calN_{l,N} = \calN_{l,N}(k,\sigma):H_{\mathrm{b}}^{-N,l}[0,\bar{x})\to  H_{\mathrm{b}}^{-N-1,l}[0,\bar{x})
	\end{equation} 
	given by 
	\begin{equation} 
		\calN_{l,N} = 2i \Big[ x \partial_x  + \frac{1}{2} \Big] + \frac{2 i \mathsf{Z}x}{\sigma^2+\mathsf{Z}x} \Big( k + \frac{1}{4} \Big) 
	\end{equation} 
		satisfies $\lVert u \rVert_{H_{\mathrm{b}}^{-N,l}[0,\bar{x})} \leq C \lVert \calN_{l,N}(k,\sigma) u \rVert_{H_{\mathrm{b}}^{-N-1,l}[0,\bar{x})}$
		for all $\sigma \in [0,\infty)$ and $u\in H_{\mathrm{b}}^{-N,l}[0,\bar{x})$. 
		
		The same estimate (possibly with a different constant) holds with $\smash{\hat{X}=[0,\bar{x})_x\times \partial X}$ in place of $[0,\bar{x})$ if we replace $H_{\mathrm{b}}^{-N,l}[0,\bar{x})$ with $\smash{\tilde{H}_{\mathrm{b}}^{-N,l}(\hat{X})}$, where the latter is the usual b-Sobolev space (such that $\tilde{H}_{\mathrm{b}}^{0,0}(\hat{X})=L^2_{\mathrm{b}}(\hat{X})$). 
\end{lemma}
\begin{proof} 
	We write the proof for $[0,\bar{x})$, and the proof for $\hat{X}$ is similar (since angular derivatives commute with $\mathcal{N}_{l,N}$) once we interpret $u$ as a $H^m(\partial X)$-valued element of a b-Sobolev space on $[0,\bar{x})$.

	Via the Mellin transform, $\calN_{l,N}(k,0) = 2i(x \partial_x  +k + 3/4):H_{\mathrm{b}}^{-N,l}[0,\bar{x}) \to H_{\mathrm{b}}^{-N-1,l}[0,\bar{x})$ is invertible for $l,k$ as in the lemma statement, so it suffices to restrict attention to the case $\sigma>0$, i.e.\ to prove that 
	\begin{equation} 
		\lVert  u \rVert_{H_{\mathrm{b}}^{-N,0}[0,\bar{x})} \preceq \lVert x^{-l}\calN_{l,N}(k,\sigma) x^l u \rVert_{H_{\mathrm{b}}^{-N-1,0}[0,\bar{x})}
		\label{eq:misc_kj1}
	\end{equation} 
	for all $u\in H_{\mathrm{b}}^{-N,0}[0,\bar{x})$ and $\sigma>0$ (the estimate required to be uniform in $\sigma$). Let $\hat{x} = \mathsf{Z}x / \sigma^2 \bar{x}$. 
	
	Via the dilation invariance of the b-Sobolev spaces, the estimate \cref{eq:misc_kj1} is equivalent to the following:
	\begin{equation}
			\lVert u \rVert_{H_{\mathrm{b}}^{-N,0}[0,\mathsf{Z}/\sigma^2)} \preceq \lVert \widehat{\calN}_{l,k} u \rVert_{H_{\mathrm{b}}^{-N-1,0}[0,\mathsf{Z}/\sigma^2)}
			\label{eq:misc_kj2}
	\end{equation}
	for all $u\in H_{\mathrm{b}}^{-N,0}[0,\mathsf{Z}/\sigma^2)$ and $\sigma>0$, where 
	\begin{equation}
		\widehat{\calN}_{l,k}  = 2 i \Big[ \hat{x} \partial_{\hat{x}} + l + \frac{1}{2} \Big] + \frac{2i \hat{x}}{1+\hat{x}} \Big(k + \frac{1}{4} \Big). 
	\end{equation} 
	We now ``radially'' compactify the nonnegative real axis $[0,\infty)_{\hat{x}}$ (so that $1/(1+\hat{x})$ becomes a bdf for the new boundary face), and we call the result $[0,\infty]$. Observe that $\smash{\widehat{\calN}_{l,k}}$ is a b-differential operator on $[0,\infty]$, with b-decay rate zero at both ends. In order to prove the estimate \cref{eq:misc_kj2}, it suffices to prove that 
	\begin{equation} 
		\lVert u \rVert_{H_{\mathrm{b}}^{-N,0,0}[0,\infty]} \preceq \lVert \widehat{\calN}_{l,k} u \rVert_{H_{\mathrm{b}}^{-N-1,0,0}[0,\infty]}
		\label{eq:misc_juk}
	\end{equation} 
	for all $u\in H_{\mathrm{b}}^{-N,0,0}[0,\infty]$, where the third index is the b-decay order at $\hat{x}=\infty$. 
	
	For $l\neq -1/2$ and $l+k \neq -3/4$, $\widehat{\calN}_{l,k}$ is  (via b-ellipticity and the Taylor series expansion of $\hat{x} / (1+\hat{x})$ around $\hat{x}=[0,\infty)$) Fredholm as an operator $H_{\mathrm{b}}^{-N,0,0}[0,\infty] \to H_{\mathrm{b}}^{-N-1,0,0}[0,\infty]$, and for any $N_0\in \bbN$ we have the estimate 
	\begin{equation} 
		\lVert u \rVert_{H_{\mathrm{b}}^{-N,0,0}[0,\infty]} \preceq \lVert \widehat{\calN}_{l,k} u \rVert_{H_{\mathrm{b}}^{-N-1,0,0}[0,\infty]} + \lVert u \rVert_{H_{\mathrm{b}}^{-N_0,-1,-1}[0,\infty]}
		\label{eq:misc_ol1}
	\end{equation} 
	for $u\in H_{\mathrm{b}}^{-N,0,0}[0,\infty]$.  Once $\smash{\widehat{\calN}}_{l,k}$ is known to be injective, a standard argument allows us to remove the last term of \cref{eq:misc_ol1}, yielding the desired estimate \cref{eq:misc_juk}. It suffices to consider the case $N_0>N$. The standard argument is as follows:
	\begin{itemize}
		\item If we could not remove the last term of \cref{eq:misc_ol1}, then we would be able to find a sequence $\{u_j\}_{j\in\bbN} \subset  H_{\mathrm{b}}^{-N,0,0}[0,\infty]$ with $\lVert u_j \rVert_{ H_{\mathrm{b}}^{-N,0,0}[0,\infty]} = 1$ for all $j$ and 
		\begin{equation} 
			\lVert \widehat{\calN}_{l,k} u_j \rVert_{H_{\mathrm{b}}^{-N-1,0,0}[0,\infty]}\to 0 
			\label{eq:misc_8hh}
		\end{equation} 
		as $j\to\infty$. By the Banach--Alaoglu theorem, we may assume without loss of generality (by passing from $\{u_j\}_{j\in \bbN}$ to a subsequence if necessary) that there exists some $u_\infty \in H_{\mathrm{b}}^{-N-1,0,0}[0,\infty]$ such that $u_j\to u_\infty$  weakly as $j\to\infty$. 
		
		Via the compactness of the inclusion $H_{\mathrm{b}}^{-N,0,0}[0,\infty]\hookrightarrow H_{\mathrm{b}}^{-N_0,-1,-1}[0,\infty]$ for $N_0>N$, $u_j\to u_\infty$ strongly in the latter space. From \cref{eq:misc_ol1}, we deduce that 
		\begin{equation} 
			\lVert u_\infty \rVert_{H_{\mathrm{b}}^{-N_0,-1,-1}[0,\infty]} = 
		\lim_{j\to\infty} \lVert u_j \rVert_{H_{\mathrm{b}}^{-N_0,-1,-1}[0,\infty]}  \succeq 1.
		\end{equation} 
		In particular, (I) $u_\infty$ is nonzero.

		Also, from the strong convergence of $u_j\to u_\infty$ in $H_{\mathrm{b}}^{-N_0,-1,-1}[0,\infty]$, 
		\begin{equation} 
		\widehat{\calN}_{l,k} u_j \to \widehat{\calN}_{l,k} u_\infty
		\end{equation} 
		distributionally. But $\widehat{\calN}_{l,k} u_j \to 0$ strongly in $H_{\mathrm{b}}^{-N-1,0,0}[0,\infty]$, by \cref{eq:misc_8hh}. Thus, (II) $\widehat{\calN}_{l,k} u_\infty = 0$. 
		
		From (I) and (II), we conclude that $\widehat{\calN}_{l,k}$ is not injective. 
	\end{itemize}

	In order to show that 
	\begin{equation}
	\ker_{H_{\mathrm{b}}^{-N,0,0}[0,\infty]} \widehat{\calN}_{l,k} = \{u\in H_{\mathrm{b}}^{-N,0,0}[0,\infty] : \widehat{\calN}_{l,k} u = 0\}
	\label{eq:ker}
	\end{equation} 
	is trivial, we simply appeal to the solution \cref{eq:model_soln} of the ODE (although it is slightly simpler to integrate in the other direction). Indeed, any element $u$ of the kernel \cref{eq:ker} must be given by
	$u(\hat{x}) = c \hat{x}^{-(l+1/2)} (1+\hat{x})^{-k-1/4}$
	for some $c\in \bbC$. 
	If this is nonzero, then it is $\Omega(\hat{x}^{- (l+k+3/4)})$ as $\hat{x}\to\infty$. If $l+k+3/4 \leq 0$, then $u$ fails to lie in $L^2_{\mathrm{b}}[0,\infty]$. 
	By \cref{eq:misc_ol1} (applied with $0$ in place of $N$ and $N$ in place of $N_0$), 
	\begin{equation}
		\lVert u \rVert_{L^2_{\mathrm{b}}[0,\infty] } \preceq \lVert \widehat{\calN}_{l,k} u \rVert_{H_{\mathrm{b}}^{-1,0,0}[0,\infty]} + \lVert u \rVert_{H_{\mathrm{b}}^{-N,-1,-1}[0,\infty]} \preceq \lVert \widehat{\calN}_{l,k} u \rVert_{H_{\mathrm{b}}^{-1,0,0}[0,\infty]} + \lVert u \rVert_{H_{\mathrm{b}}^{-N,0,0}[0,\infty]}.
		\label{eq:misc_k86} 
	\end{equation}
	Since $\lVert u \rVert_{L^2_{\mathrm{b}}[0,\infty] }=\infty$ and $u\in \ker_{H_{\mathrm{b}}^{-N,0,0}[0,\infty]} \widehat{\calN}_{l,k} \Rightarrow \lVert \widehat{\calN}_{l,k} u \rVert_{H_{\mathrm{b}}^{-1,0,0}[0,\infty]}=0$, \cref{eq:misc_k86} implies
	\begin{equation}
		u\notin H_{\mathrm{b}}^{-N,0,0}[0,\infty],
	\end{equation} 
	which contradicts $u\in \ker_{H_{\mathrm{b}}^{-N,0,0}[0,\infty]} \widehat{\calN}_{l,k}$. Thus, $u\in \ker_{H_{\mathrm{b}}^{-N,0,0}[0,\infty]} \widehat{\calN}_{l,k} \Rightarrow u=0$. 
	
	This completes the proof of the lemma. 
\end{proof}

\begin{proposition}
	\label{prop:norm_est}
	For any $l<-1/2$, $k+l\leq -3/4$, and $N\in \bbN$, there exists a constant $C=C(\tilde{P},N,l,\ell ,k)>0$ such that, for all $\sigma\geq 0$,  
	\begin{align}
		\lVert v \rVert_{H_{\mathrm{b}}^{-N,l}(X)} &\leq C\cdot (\lVert (\sigma^2+\mathsf{Z} x)^{-k}  \operatorname{N}(\tilde{P}(\sigma)) ((\sigma^2+\mathsf{Z} x)^{k} v) \rVert_{H_{\mathrm{b,leC}}^{-N,l+1,2l+3}(X)}) 
		\label{eq:misc_983} \\
		\lVert v \rVert_{H_{\mathrm{b,leC}}^{-N,l,2l+2k}(X)} &\leq C\cdot  (\lVert \operatorname{N}(\tilde{P}(\sigma)) v \rVert_{H_{\mathrm{b,leC}}^{-N,l+1,2l+2k+3}(X)})
		\label{eq:misc_984}
	\end{align}
	for all $v\in \calS'(X)$ supported in $\{x<\bar{x}\}\subset X$. 
\end{proposition}
\begin{proof}
	Letting $\tilde{\operatorname{N}}(\tilde{P}(\sigma)) = (\sigma^2+\mathsf{Z} x)^{-k}  \operatorname{N}(\tilde{P}) (\sigma^2+\mathsf{Z} x)^{k} $, 
	$\tilde{\operatorname{N}}(\tilde{P}(\sigma)) = \operatorname{N}(\tilde{P}(\sigma))+  2\mathsf{Z}  kix^2 /(\sigma^2+\mathsf{Z} x)^{1/2}$. Consequently, we can choose the b,leC-Sobolev and b-Sobolev norms such that 
	\begin{equation}
	\lVert \tilde{\operatorname{N}}(\tilde{P}(\sigma)) v \rVert_{H_{\mathrm{b,leC}}^{-N,l+1,2l+3}} = \lVert (2ix(x \partial_x-(n-1)/2) + 2\mathsf{Z}ki x^2 / (\sigma^2+\mathsf{Z} x) )  v \rVert_{H_{\mathrm{b}}^{-N,l+1}(X)}.
	\label{eq:misc_982}
	\end{equation}
 	We can work on $\hat{X} = [0,1)_x \times \partial X$, as 
	\begin{equation}
		\lVert x^{-n/2} w \rVert_{\tilde{H}_{\mathrm{b}}^{-N,l}(\hat{X})}\preceq \lVert w \rVert_{H_{\mathrm{b}}^{-N,l}(X)} \preceq  \lVert x^{-n/2} w \rVert_{\tilde{H}_{\mathrm{b}}^{-N,l}(\hat{X})}
	\end{equation} 
	for all $N,l\in \bbR$ and $w\in \calS'(X)$ supported in $\{x<\bar{x}\}$.
	\Cref{eq:misc_983} is therefore equivalent to 
	\begin{align}
		\lVert x^{-n/2} v \rVert_{\tilde{H}_{\mathrm{b}}^{-N,l}(\hat{X})} &\preceq \lVert x^{-n/2} (2ix(x \partial_x-(n-1)/2) + 2\mathsf{Z}ki x^2 / (\sigma^2+\mathsf{Z} x) )  v \rVert_{\tilde{H}_{\mathrm{b}}^{-N,l+1}(\hat{X})} 
		\intertext{	for $v \in H_{\mathrm{b}}^{-N,l}(X)$ supported in $\{x<\bar{x}\}$, which follows if } 
		\lVert w \rVert_{\tilde{H}_{\mathrm{b}}^{-N,l}(\hat{X})} &\preceq \lVert  (2i(x \partial_x+1/2) + 2\mathsf{Z}ki x / (\sigma^2+\mathsf{Z} x) )   w \rVert_{\tilde{H}_{\mathrm{b}}^{-N,l}(\hat{X})}
	\end{align}
	holds for $w \in \tilde{H}_{\mathrm{b}}^{-N,l}(\hat{X})$, which was the conclusion of \Cref{lemma:uniform_normal}. 
\end{proof}

\begin{proposition}
	\label{prop:central}
	For each $\Sigma>0$, $N\in \bbN$, $m,s,\varsigma,l,\ell\in \bbR$ satisfying $l<-1/2$, $\ell< -3/2$, $s>s_0>-1/2$, $-3/2 < \varsigma_0 < \varsigma \leq \ell+s-l$, there exists a constant 
	$
	C = C(\tilde{P}, \Sigma,N,m,s,\varsigma,l,\ell  )>0 
	$ 
	such that, for any $u\in \calS'(X)$  
	\begin{equation}
		\lVert u \rVert_{H_{\mathrm{leC}}^{m,s,\varsigma,l,\ell}} \leq C\cdot ( \lVert \tilde{P} u \rVert_{H_{\mathrm{leC}}^{m-2,s+1,\varsigma+3,l+1,\ell+3}}  + \lVert u \rVert_{H_{\mathrm{b,leC}}^{-N,l-\delta,\ell-2\delta}}  )
		\label{eq:misc_8t6}
	\end{equation}
	holds for any $\sigma \in [0,\Sigma]$ such that  $u(-;\sigma) \in H_{\mathrm{scb}}^{-N,s_0,-N}(X)$. 
\end{proposition}

\begin{proof}
	Consider $u\in \calS'(X)$ and $\sigma \in [0,\Sigma]$ as in the proposition statement.  By \Cref{prop:refined_combined}, we have 
	\begin{equation}
		\lVert  u \rVert_{H_{\mathrm{leC}}^{m,s,\varsigma,l,\ell}} \preceq \lVert \tilde{P} u \rVert_{H_{\mathrm{leC}}^{m-2,s+1,\varsigma+3,l+1,\ell+3}}  + \lVert u \rVert_{H_{\mathrm{b,leC}}^{-N_0-2,l,\ell}}, 
	\end{equation}
	where $N_0\in \bbN$ is arbitrary.
	We now apply \Cref{prop:norm_est} to estimate the remainder term. 
	Let $\chi \in C_{\mathrm{c}}^\infty(X)$ be supported in $x\leq \bar{x}$ and identically equal to one in some neighborhood of $x=0$. 
	First of all, 
	\begin{equation} 
		\lVert u \rVert_{H_{\mathrm{b,leC}}^{-N_0-2,l,\ell}} \preceq  \lVert \chi u \rVert_{H_{\mathrm{b,leC}}^{-N_0-2,l,\ell}} + \lVert u \rVert_{H_{\mathrm{b,leC}}^{-N_0,-N_0,-N_0}}.
	\end{equation}  
	Set $k=(\ell-2l)/2$. Then, $k+l\leq -3/4$. 
	We now apply the previous proposition with $v= \chi u$, the result being 
	\begin{align}
		\begin{split} 
		\lVert \chi u \rVert_{H_{\mathrm{b,leC}}^{-N_0-2,l,\ell}} &\preceq \lVert \operatorname{N}(\tilde{P}) \chi u \rVert_{H_{\mathrm{b,leC}}^{-N_0-2,l+1,\ell+3}}  \\
		&\leq \lVert \tilde{P} \chi u \rVert_{H_{\mathrm{b,leC}}^{-N_0-2,l+1,\ell+3}} + \lVert E \chi u \rVert_{H_{\mathrm{b,leC}}^{-N_0-2,l+1,\ell+3}}  
		\end{split} 
	\end{align}
	for $E= \operatorname{N}(\tilde{P}) - \tilde{P}$. 
	By \Cref{prop:normal_error}, $E \in \Psi_{\mathrm{b,leC}}^{2,-1-\delta,-3-2\delta}(X)$ for some $\delta \in (0,1/2)$, so 
	\begin{equation}
		\lVert E \chi u \rVert_{H_{\mathrm{b,leC}}^{-N_0-2,l+1,\ell+3}} \preceq \lVert u \rVert_{H_{\mathrm{b,leC}}^{-N_0,l-\delta,\ell-2\delta}}.
	\end{equation}
	On the other hand, since $\chi$ is identically one in some neighborhood of $\partial X$, 
	\begin{align}
		\begin{split} 
		 \lVert \tilde{P} \chi u \rVert_{H_{\mathrm{b,leC}}^{-N_0-2,l+1,\ell+3}} &\preceq  \lVert \tilde{P}  u \rVert_{H_{\mathrm{b,leC}}^{-N_0-2,l+1,\ell+3}} + \lVert u \rVert_{H_{\mathrm{b,leC}}^{-N_0,-N_0,-N_0}} \\
		 &\preceq  \lVert \tilde{P}  u \rVert_{H_{\mathrm{b,leC}}^{m-2,s+1,\varsigma+3,l+1,\ell+3}} + \lVert u \rVert_{H_{\mathrm{b,leC}}^{-N_0,-N_0,-N_0}} 
		 \end{split} 
	\end{align}
	for sufficiently large $N_0$. Combining the estimates above (for sufficiently large $N_0$), we get \cref{eq:misc_8t6}.  
\end{proof}

For each $m,s,\varsigma,l,\ell \in \bbR$, we consider the families $\calX=\calX_{m,s,\varsigma,l,\ell}=\{\calX_{m,s,\varsigma,l,\ell}(\sigma)\}_{\sigma\geq 0}$ and $\calY=\calY_{m,s,\varsigma,l,\ell} = \{\calY_{m,s,\varsigma,l,\ell}(\sigma)\}_{\sigma\geq 0}$ given by 
	\begin{align}
	\calX_{m,s,\varsigma,l,\ell}(\sigma) &=  \{u \in H_{\mathrm{leC}}^{m,s,\varsigma,l,\ell}(X) : \tilde{P} u \in H_{\mathrm{leC}}^{m-2,s+1,\varsigma+3,l+1,\ell+3}(X)    \} \label{eq:misc_xxx}\\
	\calY_{m,s,\varsigma,l,\ell}(\sigma) &= H_{\mathrm{leC}}^{m-2,s+1,\varsigma+3,l+1,\ell+3}(X), 
\end{align}
considered as families of Banach spaces in the usual way, 
\begin{equation} 
	\lVert u \rVert_\calX = \lVert u \rVert_{H_{\mathrm{leC}}^{m,s,\varsigma,l,\ell}} + \lVert \tilde{P} u \rVert_{H_{\mathrm{leC}}^{m-2,s+1,\varsigma+3,l+1,\ell+3}}.
\end{equation}  Note that 
$\calX_{m,s,\varsigma,l,\ell}(0) = \calX_{m,\varsigma,\ell}$ and $\calY_{m,s,\varsigma,l,\ell}(0) = \calY_{m,\varsigma,\ell}$, where the right-hand sides are defined \cref{eq:misc_x12} and \cref{eq:misc_y12}. 
Tautologically, $\tilde{P} : \calX\to \calY$ is bounded, uniformly in $\sigma\geq 0$, as $\lVert \tilde{P} u \rVert_{\calY}\leq \lVert u \rVert_{\calX}$.

\begin{proposition}
	\label{prop:alternative} 
	Given $m,s,\varsigma,l,\ell$ satisfying the inequalities $l<-1/2$, $\ell<-3/2$, $-1/2<s$, $-3/2< \varsigma \leq \ell+s-l$, one of the following two alternatives holds: 
	\begin{itemize}
		\item there exists some $\sigma\geq 0$ and nonzero $u\in \calX_{m,s,\varsigma,l,\ell}(\sigma)$ such that $\tilde{P}(\sigma) u = 0$, 
		\item there exists, for each $\Sigma>0$, a constant $C_0=C_0(\tilde{P},m,s,\varsigma,l,\ell,\Sigma)>0$ such that the estimate 
		\begin{equation}
			\lVert u \rVert_{H_{\mathrm{leC}}^{m,s,\varsigma,l,\ell}} \leq C_0 \lVert   \tilde{P} u \rVert_{H_{\mathrm{leC}}^{m-2,s+1,\varsigma+3,l+1,\ell+3}}
			\label{eq:misc_zl1}
		\end{equation}
		holds for all $\sigma \in [0,\Sigma]$ and all $u\in \calX_{m,s,\varsigma,l,\ell}(\sigma)$.  
	\end{itemize}
\end{proposition}
\begin{proof}
	The following is a variant of the proof of \cite[Theorem 26.1.7]{Hormander}, also used in the proof of the main theorem in \cite{VasyLA}.
	
	Suppose that the second of the two alternatives does not hold, so that there exist $\Sigma>0$ and sequences $\{\sigma_k\}_{k=0}^\infty \subset [0,\Sigma]$ and $\{u_k\}_{k=0}^\infty \subset \calS'(X)$ with $\lVert u_k \rVert_{H_\mathrm{leC}^{m,s,\varsigma,l,\ell}(X)(\sigma_k) }=1$ and 
	\begin{equation} 
		\lVert \tilde{P}(\sigma_k) u_k \rVert_{H_{\mathrm{leC}}^{m-2,s+1,\varsigma+3,l+1,\ell+3}(X)(\sigma_k)} < 1/k 
		\label{eq:misc_kka}
	\end{equation} 
	for all $k\in \bbN$.  By passing to a subsequence if necessary (and noting that \cref{eq:misc_kka} continues to hold upon doing so), we can arrange that $\sigma_k\to \sigma_\infty$ for some $\sigma_\infty \in [0,\Sigma]$.
	
	Even though \cref{eq:misc_zl1} might not hold, by \Cref{prop:central} we at least have the bound 
	\begin{equation}
		1=\lVert u_k \rVert_{H_{\mathrm{leC}}^{m,s,\varsigma,l,\ell}(X)(\sigma_k)} \leq C \cdot (\lVert   \tilde{P} u_k \rVert_{H_{\mathrm{leC}}^{m-2,s+1,\varsigma+3,l+1,\ell+3}(X)(\sigma_k)} + \lVert u_k \rVert_{H_{\mathrm{b,leC}}^{-N,l-\delta,\ell-2\delta}} ),
		\label{eq:misc_oop}
	\end{equation}
	for any $N\in \bbN$, 
	where $C=C(\tilde{P},\Sigma,N,m,s,\varsigma,l,\ell)$ is some constant.
	On the other hand, for sufficiently large $N_0>N$, we can bound 
	\begin{equation}
		\lVert u \rVert_{H_\mathrm{leC}^{m,s,\varsigma,l,\ell}(X)(\sigma_k) } \succeq \lVert \Lambda_{m,s,\varsigma,l,\ell} (\sigma_k) u \rVert_{L^2(X)} + \lVert x^{-l} (\sigma^2+\mathsf{Z}x)^{l-\ell/2} u \rVert_{H_{\mathrm{b}}^{-N_0,0}(X)},
	\end{equation}
	this holding for all $u \in \calS'(X)$ and $\sigma \in [0,\Sigma]$. 
	Consequently, $\{\Lambda_{m,s,\varsigma,l,\ell}(\sigma_k) u_k \}_{k\in \bbN}$ is bounded in $L^2(X)$ and $\{x^{-l} (\sigma_k^2+\mathsf{Z}x)^{l-\ell/2} u_k\}_{k\in \bbN}$ is bounded in $H_{\mathrm{b}}^{-N_0,0}(X)$. 
	By Banach--Alaoglu -- passing to a subsequence if necessary (and once again noting that \cref{eq:misc_kka} continues to hold upon doing so) -- we can assume that there exist some $v\in L^2(X)$ and $w \in H_{\mathrm{b}}^{-N_0,0}(X)$ such that 
	\begin{equation} 
		\Lambda_{m,s,\varsigma,l,\ell}(\sigma_k) u_k \to v
	\end{equation}  
	as $k\to\infty$ weakly in $L^2(X)$ and $x^{-l} (\sigma_k^2+\mathsf{Z}x)^{l-\ell/2} u_k \to w$ as $k\to\infty$ weakly in $H_{\mathrm{b}}^{-N_0,0}(X)$. 
	
	It follows from the latter that $u_k\to x^{l} (\sigma_\infty^2+\mathsf{Z}x)^{\ell/2-l} w$ strongly in some b-Sobolev space. 
	This has two consequences: 
	\begin{itemize}
		\item First, 
		\begin{equation}
			\tilde{P}(\sigma_k)  u_k \to \tilde{P}(\sigma_\infty)( x^{l} (\sigma_\infty^2+\mathsf{Z}x)^{\ell/2-l} w) 
		\end{equation}
		in $\calS'(X)$ as $k\to\infty$ (e.g. using \Cref{prop:operator_cont} and \Cref{prop:tilde_P_inc}).  But, the assumption  $\lVert \tilde{P}(\sigma_k) u_k \rVert_{H_{\mathrm{leC}}^{m-2,s+1,\varsigma+3,l+1,\ell+3}(X)(\sigma_k)} < 1/k$ implies that $\tilde{P}(\sigma_k)u_k \to 0$ in $\calS'(X)$. Therefore $ \tilde{P}(\sigma_\infty)( x^{l} (\sigma_\infty^2+\mathsf{Z}x)^{\ell/2-l} w)  = 0$. 
		\item Second (using \Cref{prop:operator_cont}),  $ \Lambda_{m,s,\varsigma,l,\ell}(\sigma_k) u_k \to \Lambda_{m,s,\varsigma,l,\ell}(\sigma_\infty) ( x^{l} (\sigma_\infty^2+\mathsf{Z}x)^{\ell/2-l} w)$ in $\calS'(X)$. Since $\calS'(X)$ is Hausdorff, this implies that 
		\begin{equation} \Lambda_{m,s,\varsigma,l,\ell}(\sigma_\infty) ( x^{l} (\sigma_\infty^2+\mathsf{Z}x)^{\ell/2-l} w) = v \in L^2(X),
		\end{equation} 
		which in turn implies that $x^{l} (\sigma_\infty^2+\mathsf{Z}x)^{\ell/2-l} w \in H_{\mathrm{leC}}^{m,s,\varsigma,l,\ell}(X)(\sigma_\infty)$ by elliptic regularity (and the fact that $w \in H_{\mathrm{b}}^{-N_0,0}$).
	\end{itemize} 
	Let $u = x^{l} (\sigma_\infty^2+\mathsf{Z}x)^{\ell/2-l} w$. What we proved above is $\tilde{P}(\sigma_\infty) u =0$ and $u\in H_{\mathrm{leC}}^{m,s,\varsigma,l,\ell}(X)(\sigma_\infty)$.
	Thus, 
	\begin{equation} 
		u \in \calX_{m,s,\varsigma,l,\ell}(\sigma_\infty).
	\end{equation}  
	Since $N>N_0$, via the compactness of the inclusion $ H_{\mathrm{b}}^{-N_0,0}(X)\hookrightarrow H_{\mathrm{b}}^{-N,-\delta}(X)$ it is the case that $x^{-l} (\sigma_k^2+\mathsf{Z}x)^{l-\ell/2} u_k \to w$ strongly in $H_{\mathrm{b}}^{-N,-\delta}(X)$, so 
	\begin{equation}
		\lVert w \rVert_{H_{\mathrm{b}}^{-N,-\delta}(X)} = \lim_{k\to\infty} 	\lVert x^{-l} (\sigma_k^2+\mathsf{Z}x)^{l-\ell/2} u_k \rVert_{H_{\mathrm{b}}^{-N,-\delta}(X)}.
		\label{eq:misc_ooi}
	\end{equation}
	On the other hand, we can bound $
	\lVert u_k \rVert_{H_{\mathrm{b,leC}}^{-N,l-\delta,\ell-2\delta}(X)(\sigma_k)} \preceq \lVert x^{-l} (\sigma_k^2+\mathsf{Z}x)^{l-\ell/2}u_k \rVert_{H_{\mathrm{b}}^{-N,-\delta}(X)}$. So, \cref{eq:misc_oop} yields 
	\begin{align}
		\begin{split} 
		1 &\preceq \lVert w \rVert_{H_{\mathrm{b}}^{-N,-\delta}(X)}+ \limsup_{k\to\infty} \lVert   \tilde{P}(\sigma_k) u_k \rVert_{H_{\mathrm{leC}}^{m-2,s+1,\varsigma+3,l+1,\ell+3}(X)(\sigma_k)}  \\
		&= \lVert w \rVert_{H_{\mathrm{b}}^{-N,-\delta}(X)}. 
		\end{split} 
	\end{align}
	Therefore $w\neq 0$. It follows that $u\neq 0$.  We have therefore succeeded in showing that the first of the two alternatives listed in the proposition holds. 
\end{proof}

\begin{proposition}
	\label{prop:invert}
	Suppose that $P$ is the spectral family of an attractive Coulomb-like Schr\"odinger operator for $\sigma\leq \Sigma$. Then, given $m,s,\varsigma,l,\ell\in \bbR$ satisfying  $l<-1/2$, $\ell<-3/2$, $-1/2<s$, $-3/2< \varsigma$, it is the case that,  
	for each $\sigma\geq \Sigma$, 
	\begin{equation} 
		\tilde{P}(\sigma) : \calX_{m,s,\varsigma,l,\ell}(\sigma) \to \calY_{m,s,\varsigma,l,\ell}(\sigma)
		\label{eq:misc_tps}
	\end{equation} 
	is invertible.
\end{proposition}
\begin{proof}
	We already observed the $\sigma=0$ case in \S\ref{sec:0_operator}. The $\sigma>0$ case is essentially proven in \cite[\S4]{VasyLA}. In order to see this, note that, for each $\sigma>0$,  
	\begin{align}
		\calX_{m,s,\varsigma,l,\ell}(\sigma) &= \{u \in H_{\mathrm{scb}}^{m,s,l}(X) : \tilde{\tilde{P}}(\sigma) u \in H_{\mathrm{scb}}^{m-2,s+1,l+1}(X) \}  \\
		\calY_{m,s,\varsigma,l,\ell}(\sigma) &= H_{\mathrm{scb}}^{m-2,s+1,l+1}(X) 
	\end{align}	
	at the level of sets, where 
	\begin{equation} 
	\tilde{\tilde{P}}(\sigma) = e^{ - i \Phi_0+ i\Phi} \tilde{P} e^{+i \Phi_0-i\Phi}
	\label{eq:misc_12v}
	\end{equation}
	is Vasy's conjugated operator, $\Phi_0$ simply being defined by $\Phi_0 = \sigma x^{-1}$.
	Since the leC-Sobolev spaces are just scb-Sobolev spaces for $\sigma>0$, the crux of the previous claim is that the $\tilde{P}$ on the right-hand side of \cref{eq:misc_xxx} can be replaced by Vasy's \cref{eq:misc_12v}.  Indeed, $\Phi_0(x;\sigma)- \Phi(x;\sigma) \in \log x C^\infty([0,\bar{x})_x)+  C^\infty([0,\bar{x})_x)$ for each $\sigma>0$, so 
	\begin{equation} 
	\tilde{\tilde{P}}(\sigma) = \tilde{P}(\sigma) + T(\sigma)
	\end{equation}
	for $T(\sigma)=e^{ - i \Phi_0(-;\sigma)+ i\Phi(-;\sigma)} [\tilde{P}(\sigma), e^{+i \Phi_0(-;\sigma)-i\Phi(-;\sigma)}] \in \operatorname{Diff}_{\mathrm{scb}}^{1,-1,-1}(X)$. Thus,
		\begin{equation}
			T(\sigma) :H_{\mathrm{scb}}^{m,s,l}(X) \to H_{\mathrm{scb}}^{m-2,s+1,l+1}(X).
		\end{equation}
		So, for $u \in H_{\mathrm{scb}}^{m,s,l}(X)$,  $\tilde{\tilde{P}}(\sigma) u \in H_{\mathrm{scb}}^{m-2,s+1,l+1}(X)$ if and only if  $\tilde{P}(\sigma) u \in H_{\mathrm{scb}}^{m-2,s+1,l+1}(X)$.
	The operators we consider have slightly more general $\sigma$-dependence than the ones in Vasy (as some additional assumptions are needed for the $\sigma\notin \bbR$ case of \cite[Theorem 1.1]{VasyLA}), but since we are only considering real $\sigma$ his proof of the real case of \cite[Theorem 1.1]{VasyLA} goes through in this slightly greater generality \emph{mutatis mutandis}. 
	
	Alternatively, the $\sigma>0$ case of \Cref{prop:central} suffices as a replacement for \cite[Prop. 4.16]{VasyLA} in his proof of \cite[Theorem 1.1]{VasyLA}, the rest of which is identical. (For the purpose of the proof above, we do not need to know that the estimate in \Cref{prop:central} is uniform down to $\sigma=0$, so the $\varsigma \leq \ell+s-l$ hypothesis there is not relevant here.)
\end{proof}

\subsection{Smoothness at $\mathrm{zf}$, Conormality elsewhere}
\label{subsec:smoothness1}

For this subsection, we suppose that $P(\sigma)$ is the spectral family of an attractive Coulomb-like Schr\"odinger operator for $\sigma$ in some neighborhood of $[0,\Sigma]$, $\Sigma>0$. 
By analogy with the terminology in \cite{MelroseSC}, we might say that $u\in \calS'(X)$ ``satisfies the conjugated Sommerfeld radiation condition'' for some given $\sigma\geq 0$ if $u\in \calX_{m,s,\varsigma,l,\ell}(X)(\sigma)$ for some $m,s,\varsigma,l,\ell\in \bbR$ satisfying  $l<-1/2$, $\ell<-3/2$, $s>-1/2$, $\ell+s-l\geq \varsigma>-3/2$.
One of the main tasks of this subsection is to show that the limiting resolvent output converges as $\sigma\to 0^+$ to something satisfying the zero energy version of the Sommerfeld radiation condition. 

For each $\sigma \in [0,\Sigma]$, let $\tilde{R}_+(\sigma) : \calY_{m,s,\varsigma,l,\ell}(\sigma) \to  \calX_{m,s,\varsigma,l,\ell}(\sigma)$ denote the set-theoretic inverse to \cref{eq:misc_tps} (which, of course, must actually be an isomorphism of Banach spaces e.g. by the closed graph theorem). (The `$+$' subscript of $\tilde{R}_+ = \{\tilde{R}_+(\sigma)\}_{\sigma\geq 0}$ refers to the choice of sign in defining the conjugation.) This extends the definition of the operator $\tilde{R}_+(0)$ introduced in \S\ref{sec:0_operator} to the $\sigma>0$ case. For each $\sigma \in (0,\Sigma]$, 
\begin{equation} 
	\tilde{R}_+(\sigma) : H_{\mathrm{scb}}^{m-2,s+1,l+1}(X) \to H_{\mathrm{scb}}^{m,s,l}(X)
	\label{eq:misc_rps}
\end{equation} 
is bounded (but not uniformly in $\sigma$).  Considering the case $s=m+l$, $\tilde{R}_+(\sigma):H_{\mathrm{b}}^{m,l+1}(X) \to H_{\mathrm{b}}^{m,l}(X)$ if $l<-1/2<m+l$. 
Note that the mapping properties of the resolvent with respect to the b-Sobolev spaces are slightly lossy, in the sense that we can no longer keep track of the fact that the map \cref{eq:misc_rps} smooths by two orders.
As the notation in \cref{eq:misc_rps} indicates,  the operator $\smash{\tilde{R}_+(\sigma)}$ makes sense as a map 
\begin{equation} 
	\bigcup_{\substack{m,s,l\in \bbR \\ s>-1/2> l}} H_{\mathrm{scb}}^{m-2,s+1,l+1}(X) \to 	\bigcup_{\substack{m,s,l\in \bbR \\ s>-1/2> l}} H_{\mathrm{scb}}^{m,s,l}(X),
	\label{eq:misc_674}
\end{equation} 
hence we can just write ``$\tilde{R}_+(\sigma)$'' without specifying $m,s,l$.
A similar statement holds for $\sigma=0$.

\begin{proposition} 
	\label{prop:misc_estimate}
	Given $m,s,\varsigma,l,\ell\in \bbR$ satisfying  $l<-1/2$, $\ell<-3/2$, $s>-1/2$, $\ell+s-l\geq \varsigma>-3/2$,  there exists some constant $C=C(m,s,\varsigma,l,\ell,\Sigma) >0$ such that 
	\begin{equation}
		\lVert \tilde{R}_+(\sigma) \tilde{f} \rVert_{H_{\mathrm{leC}}^{m,s,\varsigma,l,\ell}(X) (\sigma)} \leq C \lVert \tilde{f} \rVert_{H_{\mathrm{leC}}^{m-2,s+1,\varsigma+3,l+1,\ell+3}(X)(\sigma)} 
		\label{eq:misc_ues}
	\end{equation}
	for all $\sigma \in [0,\Sigma]$ and $\tilde{f} \in H_{\mathrm{leC}}^{m-2,s+1,\varsigma+3,l+1,\ell+3}(X)(\sigma)$.  	Moreover, for $m,l\in \bbR$ with $l<-1/2$ and $-1<m+2l$: 
	\begin{enumerate}[label=(\Roman*)] 
		\item 
		 for any $\tilde{f} \in H_{\mathrm{b}}^{m,l+5/4}(X)$, we have $\lVert (\sigma^2 + \mathsf{Z}x)^{1/4} \tilde{R}_+(\sigma) \tilde{f} \rVert_{H_{\mathrm{b}}^{m,l}} \leq  C_0 \lVert \tilde{f} \rVert_{H_{\mathrm{b}}^{m,l+5/4}}$
		for some constant $C_0=C_0(\tilde{P},m,l,\Sigma)>0$  
		\item  for any $\tilde{f} \in H_{\mathrm{b}}^{m,l+1}(X)$ we have $
			\lVert (\sigma^2 + \mathsf{Z}x)^{1/4} \tilde{R}_+(\sigma)  (\sigma^2 + \mathsf{Z}x)^{1/4}  \tilde{f} \rVert_{H_{\mathrm{b}}^{m,l}} \leq  C_1 \lVert \tilde{f} \rVert_{H_{\mathrm{b}}^{m,l+1}}$
		for some constant $C_1=C_1(\tilde{P},m,l,\Sigma)>0$
	\end{enumerate}
	for all $\sigma \in [0,\Sigma]$.
\end{proposition}
\begin{proof} 
	As a corollary of \Cref{prop:alternative} and \Cref{prop:invert}, we get that for $m,s,\varsigma,l,\ell$ as above and $P$ the spectral family of an attractive Coulomb-like Schr\"odinger operator, 
	\begin{equation}
		\lVert u \rVert_{H_{\mathrm{leC}}^{m,s,\varsigma,l,\ell}(X) (\sigma)} \leq C(m,s,\varsigma,l,\ell,\Sigma) \lVert \tilde{P} u \rVert_{H_{\mathrm{leC}}^{m-2,s+1,\varsigma+3,l+1,\ell+3}}
		\label{eq:misc_931}
	\end{equation}
	holds for all $\sigma \in [0,\Sigma]$ and all $u\in \calX_{m,s,\varsigma,l,\ell}(\sigma)$. We also have 
	\begin{equation}
		\tilde{R}_+(\sigma) : H_{\mathrm{leC}}^{m-2,s+1,\varsigma+3,l+1,\ell+3}(X)(\sigma) \to \calX_{m,s,\varsigma,l,\ell}(\sigma)\subseteq  H_{\mathrm{leC}}^{m,s,\varsigma,l,\ell}(X) (\sigma).  
	\end{equation}
	Taking $u= \tilde{R}_+(\sigma) \tilde{f}$ in \cref{eq:misc_931} yields \Cref{eq:misc_ues}. 
	
	Suppose now that $m,l\in \bbR$ satisfy $l<-1/2$ and $-1<m+2l$ (in which case $-1/2<m+l$ holds as well). First suppose that $\smash{\tilde{f} \in H_{\mathrm{b}}^{m,l+5/4}(X)}$. Applying \cref{eq:misc_ues} (observing that the required inequalities $l<-1/2<m+l$ and $-3/2<m+2l-1/2$ hold), 
	\begin{align}
		\begin{split} 
			\lVert (\sigma^2+\mathsf{Z}x)^{1/4}\tilde{R}_+(\sigma)\tilde{f} \rVert_{H_{\mathrm{b}}^{m,l}} &\preceq 
			\lVert (\sigma^2+\mathsf{Z}x)^{1/4}\tilde{R}_+(\sigma)\tilde{f} \rVert_{H_{\mathrm{leC}}^{m,m+l,m+2l,l,2l}(\sigma)} \\ & \preceq \lVert \tilde{R}_+(\sigma)\tilde{f} \rVert_{H_{\mathrm{leC}}^{m,m+l,m+2l-1/2,l,2l-1/2}(\sigma)} \\
			&\preceq \lVert \tilde{f} \rVert_{H_{\mathrm{leC}}^{m-2, m+l+1,m+2l+5/2,l+1,2l+5/2}(\sigma)} \\
			&\preceq \lVert \tilde{f} \rVert_{H_{\mathrm{b}}^{m,l+5/4}}.
		\end{split}  
		\intertext{Now supposing that $\tilde{f} \in H_{\mathrm{b}}^{m,l+1}(X)$, } 
		\begin{split} 
		\lVert (\sigma^2+\mathsf{Z}x)^{1/4}\tilde{R}_+(\sigma) (\sigma^2+\mathsf{Z}x)^{1/4}\tilde{f} \rVert_{H_{\mathrm{b}}^{m,l}} &\preceq 
		\lVert (\sigma^2+\mathsf{Z}x)^{1/4}\tilde{R}_+(\sigma)(\sigma^2+\mathsf{Z}x)^{1/4}\tilde{f} \rVert_{H_{\mathrm{leC}}^{m,m+l,m+2l,l,2l}(\sigma)} \\ & \preceq \lVert \tilde{R}_+(\sigma) (\sigma^2+\mathsf{Z} x)^{1/4}\tilde{f} \rVert_{H_{\mathrm{leC}}^{m,m+l,m+2l-1/2,l,2l-1/2}(\sigma)} \\
		&\preceq \lVert (\sigma^2+\mathsf{Z} x)^{1/4} \tilde{f} \rVert_{H_{\mathrm{leC}}^{m-2, m+l+1,m+2l+5/2,l+1,2l+5/2}(\sigma)} \\
		&\preceq \lVert  \tilde{f} \rVert_{H_{\mathrm{leC}}^{m-2, m+l+1,m+2l+2,l+1,2l+2}(\sigma)} \\
		&\preceq \lVert \tilde{f} \rVert_{H_{\mathrm{b}}^{m,l+1}}.
	\end{split} 
	\end{align}
\end{proof}

\begin{proposition}
	\label{prop:prelim_convergence}
	For any $m,l\in \bbR$ with $l<-1/2$ and $-1<m+2l$, for any  $f\in H_{\mathrm{b}}^{m,l+5/4}(X)$, 
	\begin{equation} 
		(\sigma^2+\mathsf{Z}x)^{1/4}\tilde{R}_+(\sigma)f \to  \mathsf{Z}^{1/4} x^{1/4}\tilde{R}_+(0)f
		\label{eq:misc_cvg}
	\end{equation}  
	weakly in $H_{\mathrm{b}}^{m,l}(X)$ as $\sigma\to 0^+$. 
	In fact, the map $(\sigma,f)\mapsto (\sigma^2+\mathsf{Z}x)^{1/4}\tilde{R}_+(\sigma)f $ defines a jointly continuous map 
	\begin{equation} 
	[0,\Sigma)\times H_{\mathrm{b}}^{m,l+5/4}(X) \to H_{\mathrm{b}}^{m-\epsilon,l-\epsilon}(X),
	\end{equation} 
	for any $\epsilon>0$, where we are using the strong topologies on $H_{\mathrm{b}}^{m,l+5/4}(X)$ and $H_{\mathrm{b}}^{m-\epsilon,l-\epsilon}(X)$.

	This (applied for slightly smaller $m,l$) implies that 
	\begin{equation} 
		\{(\sigma^2+\mathsf{Z} x)^{1/4}\tilde{R}_+(\sigma)\}_{\sigma\geq 0} \subset \calL(H_{\mathrm{b}}^{m,l+5/4}(X),H_{\mathrm{b}}^{m-\epsilon,l-\epsilon}(X))
	\end{equation} 
	is continuous with respect to the uniform operator topology. 
\end{proposition}
\begin{proof}
	First consider the claim of joint continuity. By the metrizability of $[0,\Sigma)\times \smash{H_{\mathrm{b}}^{m,l+5/4}(X)}$ and $\smash{H_{\mathrm{b}}^{m-\epsilon,l-\epsilon}(X)}$, joint continuity follows from the claim that 
	\begin{itemize}
		\item given $f \in \smash{H_{\mathrm{b}}^{m,l+5/4}(X)}$ and $\sigma_\infty \in [0,\Sigma)$, for any and sequences $\{\sigma_k\}_{k\in \bbN}\subset [0,\Sigma)$ and $\{f_k\}_{k\in \bbN} \subset H_{\mathrm{b}}^{m,l+5/4}(X)$ with $\sigma_k\to \sigma_\infty$ and $f_k \to f$ as $k\to\infty$, 
		\begin{equation}
			(\sigma^2_k+\mathsf{Z}x)^{1/4}\tilde{R}_+(\sigma_k)f_k \to  (\sigma_\infty^2+\mathsf{Z}x)^{1/4}\tilde{R}_+(\sigma_\infty)f
		\end{equation}
		strongly in $ H_{\mathrm{b}}^{m-\epsilon,l-\epsilon}(X)$.
	\end{itemize}
	Since a sequence of elements of a metric space converges to some element  if and only if every subsequence thereof contains a further subsequence converging to that same element, it suffices to show that
	\begin{itemize}
		\item  given any $\{\sigma_k\}_{k\in \bbN} \subset [0,\Sigma)$, $\{f_k\}_{k\in \bbN} \subset \smash{H_{\mathrm{b}}^{m,l+5/4}(X)}$ with $\sigma_k\to \sigma_\infty$ and $f_k \to f$ as $k\to\infty$, there exists a subsequence $\{k_\kappa\}_{\kappa\in \bbN}\subset \{k\}_{k\in \bbN}$  such that 
		\begin{equation} 
			(\sigma^2_{k_\kappa}+\mathsf{Z}x)^{1/4}\tilde{R}_+(\sigma_{k_\kappa})f_{k_\kappa} \to (\sigma^2_\infty+\mathsf{Z}x)^{1/4}\tilde{R}_+(\sigma_\infty)f
		\end{equation} 
		strongly in $H_{\mathrm{b}}^{m-\epsilon,l-\epsilon}(X)$.
	\end{itemize}
	 We now handle the case of $\sigma_\infty = 0$. The case $\sigma_\infty>0$ follows by a similar but even easier argument.

	By \Cref{prop:misc_estimate} and Banach--Alaoglu, we can find a subsequence $\{k_\kappa\}_{\kappa\in \bbN}\subset \{k\}_{k\in \bbN}$ such that  $\smash{(\sigma^2_{k_\kappa}+\mathsf{Z}x)^{1/4}\tilde{R}_+(\sigma_{k_\kappa})f_{k_\kappa}}$ converges weakly in  $H_{\mathrm{b}}^{m,l}(X)$. Let $w\in H_{\mathrm{b}}^{m,l}(X)$ denote the weak limit. We first want to show that  $w =  \mathsf{Z}^{1/4} x^{1/4}\tilde{R}_+(0)f$. 
	\begin{itemize}
		\item We first check that $w$ solves the PDE $\tilde{P}(0)(\mathsf{Z}^{-1/4} x^{-1/4}w) = f$. Indeed, for any $m_0,l_0$, 
		\begin{equation}
			[0,\infty)_\sigma \times H_{\mathrm{b}}^{m_0,l_0}(X) \ni (\sigma, u)\mapsto \tilde{P}(\sigma) u \in \calS'(X) 
		\end{equation}
		is jointly continuous with respect to the strong topology on $H_{\mathrm{b}}^{m_0,l_0}(X)$. (Besides being clear from the explicit formulas for $\tilde{P}$ in \S\ref{sec:operator}, this follows from \Cref{prop:operator_cont} and \Cref{prop:tilde_P_inc}.) 
		Since 
		\begin{equation} 
			\mathsf{Z}^{-1/4} x^{-1/4}(\sigma^2_{k_\kappa}+\mathsf{Z} x)^{1/4}\tilde{R}_+(\sigma_{k_\kappa})f_{k_\kappa} \to \mathsf{Z}^{-1/4} x^{-1/4} w
		\end{equation} 
		weakly in $H_{\mathrm{b}}^{m,l-1/4}(X)$, this convergence occurs strongly in $H_{\mathrm{b}}^{m-\varepsilon,l-1/4-\varepsilon}(X)$ for any $\varepsilon>0$, so (by the aforementioned joint continuity)
		\begin{equation}
			\tilde{P}(\sigma_{k_\kappa})( \mathsf{Z}^{-1/4} x^{-1/4}(\sigma^2_{k_\kappa}+\mathsf{Z} x)^{1/4}\tilde{R}_+(\sigma_{k_\kappa})f_{k_\kappa}) \to \tilde{P}(0)(\mathsf{Z}^{-1/4} x^{-1/4}w) 
		\end{equation}
		in $\calS'(X)$. 
		
		Moreover, it is not difficult to see that $\tilde{P}(\sigma_{k_\kappa})( \mathsf{Z}^{-1/4} x^{-1/4}(\sigma^2_{k_\kappa}+\mathsf{Z} x)^{1/4}\tilde{R}_+(\sigma_{k_\kappa})f_{k_\kappa}) \to f$ in $\calS'(X)$: indeed
			\begin{multline}
				\tilde{P}(\sigma_{k_\kappa})( \mathsf{Z}^{-1/4} x^{-1/4}(\sigma^2_{k_\kappa}+\mathsf{Z} x)^{1/4}\tilde{R}_+(\sigma_{k_\kappa})f_{k_\kappa}) \\ = [\tilde{P}(\sigma_{k_\kappa}), \mathsf{Z}^{-1/4} x^{-1/4}(\sigma^2_{k_\kappa}+\mathsf{Z} x)^{1/4}]\tilde{R}_+(\sigma_{k_\kappa})f_{k_\kappa} + f_{k_\kappa},  
			\end{multline}
			and $[\tilde{P}(\sigma), \mathsf{Z}^{-1/4} x^{-1/4}(\sigma^2+\mathsf{Z} x)^{1/4}] \in S\operatorname{Diff}_{\mathrm{leC}}^{1,-3/4,-3,-3/4,-3}(X)$ satisfies 
			\begin{equation} 
				[\tilde{P}(\sigma), \mathsf{Z}^{-1/4} x^{-1/4}(\sigma^2+\mathsf{Z} x)^{1/4}]|_{\sigma = 0} = 	[\tilde{P}(0),1] = 0,
			\end{equation} 
			so the boundedness of $\smash{\tilde{R}_+(\sigma_{k_\kappa})f_{k_\kappa}}$  in  $H_{\mathrm{b}}^{m,l}(X)$ in some b-Sobolev space (as given by \Cref{prop:misc_estimate}) and \Cref{prop:operator_cont}
			 show that $[\tilde{P}(\sigma_{k_\kappa}), \mathsf{Z}^{-1/4} x^{-1/4}(\sigma^2_{k_\kappa}+\mathsf{Z} x)^{1/4}]\tilde{R}_+(\sigma_{k_\kappa})f_{k_\kappa} \to 0$ in $\calS'(X)$.

		Since $\calS'(X)$ is Hausdorff, it follows that $ \tilde{P}(0)(\mathsf{Z}^{-1/4} x^{-1/4} w)  = f$. 
		\item 
		Thus, we have 
		\begin{align}
				\tilde{P}(0)(\mathsf{Z}^{-1/4} x^{-1/4} w)  &= f. 
				\intertext{But, also,} 
				\tilde{P}(0)(\tilde{R}_+(0) f) &= f. 
			\label{eq:misc_pde}
		\end{align}
		Set $l_0 = l-1/4$, so that $f\in H_{\mathrm{b}}^{m,l+5/4}(X) = H_{\mathrm{b}}^{m,l_0+3/2}(X)$. In terms of $\ell_0=2l_0$, the inequalities $l<-1/2$ and $-1<m+2l$ become $\ell_0<-3/2$ and $-3/2<m+\ell_0$, so \Cref{prop:LA_at_0b} applies (with $l_0$ in place of $l$ and $\ell_0$ in place of $\ell$):
		\begin{equation}
			\tilde{P}(0) : \{u \in H_{\mathrm{b}}^{m,l-1/4}(X) : \tilde{P}(0) u \in H_{\mathrm{b}}^{m,l+5/4}(X) \} \to H_{\mathrm{b}}^{m,l+5/4}(X) 
			\label{eq:misc_cdd}
		\end{equation}
		is invertible, and the inverse is $\tilde{R}_+(0)$. Thus, $u=\tilde{R}_+(0)f$ is the unique solution to $\tilde{P}(0) u = f$ in the domain of \cref{eq:misc_cdd}. But  $\mathsf{Z}^{-1/4} x^{-1/4} w$ is in the domain, and as we saw in \cref{eq:misc_pde} solves this PDE. We conclude that
		\begin{equation} 
			\mathsf{Z}^{-1/4} x^{-1/4} w = \tilde{R}_+(0)f.
		\end{equation}  
	\end{itemize}
	Via the compactness of the inclusion $H_{\mathrm{b}}^{m,l}\hookrightarrow H_{\mathrm{b}}^{m-\epsilon,l-\epsilon}$, we conclude that 
	\begin{equation}
		(\sigma^2_{k_\kappa}+\mathsf{Z}x)^{1/4}\tilde{R}_+(\sigma_{k_\kappa})f_{k_\kappa} \to \mathsf{Z}^{1/4}x^{1/4} \tilde{R}_+(0)f
	\end{equation} 
	strongly in $H_{\mathrm{b}}^{m-\epsilon,l-\epsilon}$. This completes the proof of joint continuity. 
	
	By \Cref{prop:misc_estimate}, for each $\Sigma>0$ the set  $\{(\sigma^2+\mathsf{Z}x)^{1/4}\tilde{R}_+(\sigma)f\}_{\sigma \in [0,\Sigma]}$ is bounded in $H_{\mathrm{b}}^{m,l}(X)$, and the result above shows that \cref{eq:misc_cvg} holds in the topology generated by Schwartz test functions (and, in fact, even the strong topologies $\smash{H_{\mathrm{b}}^{m-\epsilon,l-\epsilon}}$). It follows from the conjunction of these observations  that \cref{eq:misc_cvg} holds with respect to the weak topology of $\smash{H_{\mathrm{b}}^{m,l}(X)}$ (since $\calS(X)$ is dense in $\smash{H_{\mathrm{b}}^{m,l}(X)}$). 
	
 	We now deduce the uniform continuity statement from the joint continuity statement. Suppose, to the contrary, that we have some $m,l\in \bbR$ with $l<-1/2$ and $-1<m+2l$ and some $\epsilon>0$ such that $\{(\sigma^2+\mathsf{Z} x)^{1/4}\tilde{R}_+(\sigma)\}_{\sigma\geq 0}$ is \emph{not} continuous with respect to the uniform operator topology, generated by the norm 
 	\begin{equation} 
 		\lVert- \rVert_{\calL(H_{\mathrm{b}}^{m,l+5/4}(X),H_{\mathrm{b}}^{m-\epsilon,l-\epsilon}(X))}. 
 	\end{equation} 
 	We handle the case of a discontinuity at $\sigma=0$, and the case of $\sigma>0$ follows via a similar, easier argument. The discontinuity statement means that there exists some $\varepsilon>0$ such that there exist sequences $\{\sigma_k\}_{k\in \bbN}$ with $\sigma_k\to 0^+$ and $\{f_k\}_{k\in \bbN} \subset \smash{H_{\mathrm{b}}^{m,l+5/4}}$ with $\lVert f_k \rVert_{H_{\mathrm{b}}^{m,l+5/4}}\leq 1$ such that 
 	\begin{equation}
 		\lVert (\sigma^2_k+\mathsf{Z} x)^{1/4}\tilde{R}_+(\sigma_k) f_k - (\mathsf{Z} x)^{1/4}\tilde{R}_+(0) f_k \rVert_{H_{\mathrm{b}}^{m-\epsilon,l-\epsilon}} \geq \varepsilon 
 		\label{eq:misc_ls1}
 	\end{equation} 
 	for all $k$. By Banach--Alaoglu, we can choose these sequences such that there exists some 
 	\begin{equation} 
 		f_\infty \in H_{\mathrm{b}}^{m,l+5/4}
 	\end{equation} 
 	such that $f_k\to f_\infty$ weakly, which implies strong convergence in $\smash{H_{\mathrm{b}}^{m-\epsilon',l+5/4-\epsilon'}}$ for any $\epsilon'>0$. But then, by the joint continuity statement already proven, as long as $\epsilon'$ is sufficiently small such that $-1<m+2l-3\epsilon'$, we have
 	\begin{equation} 
 		\lVert (\sigma^2_k+\mathsf{Z} x)^{1/4}\tilde{R}_+(\sigma_k) f_k - (\mathsf{Z} x)^{1/4}\tilde{R}_+(0) f_k \rVert_{H_{\mathrm{b}}^{m-\epsilon'-\epsilon'',l-\epsilon'-\epsilon''}} \to 0
 		\label{eq:misc_ls2}
 	\end{equation} 
 	as $k\to\infty$, for any $\epsilon''>0$. But we can take $\epsilon',\epsilon''$ sufficiently small such that $\epsilon'+\epsilon''<\epsilon$, in which case \cref{eq:misc_ls2} contradicts \cref{eq:misc_ls1}.
\end{proof}

Via Sobolev embedding, the previous proposition already yields the following corollary on the continuity of the resolvent output at zero energy:
\begin{itemize}
	\item 
	for any $f\in H^{\infty, l+5/4}_{\mathrm{b}}(X)$, $l<-1/2$, and for any $\chi \in C_{\mathrm{c}}^\infty(X^\circ)$,  $ \chi \tilde{R}_+(\sigma)f \to \chi \tilde{R}_+(0)f$ in $C_{\mathrm{c}}^\infty(X^\circ)$ as $\sigma\to 0^+$. 
\end{itemize}
We will need to strengthen this result to apply to $\partial_\sigma$ derivatives of the resolvent output. In order to handle the compositions that arise, we will use the following variant of the preceding proposition.

\begin{proposition}
	\label{prop:uniform_bound}
	For any $m,l\in \bbR$ with $l<-1/2$ and $-1<m+2l$, the map $(\sigma,f)\mapsto (\sigma^2+\mathsf{Z}x)^{1/4}\tilde{R}_+(\sigma)f (\sigma^2+\mathsf{Z}x)^{1/4} $ defines a jointly continuous map 
	\begin{equation}
	[0,\Sigma)\times H_{\mathrm{b}}^{m,l+1}(X) \to H_{\mathrm{b}}^{m-\epsilon,l-\epsilon}(X),
	\end{equation}
	for any $\epsilon>0$ (with respect to the strong topologies on $H_{\mathrm{b}}^{m,l+1}(X)$ and $H_{\mathrm{b}}^{m-\epsilon,l-\epsilon}(X)$). 
	Consequently, $\{(\sigma^2+\mathsf{Z} x)^{1/4}\tilde{R}_+(\sigma) (\sigma^2+\mathsf{Z} x)^{1/4}\}_{\sigma\geq 0} \subset \calL(H_{\mathrm{b}}^{m,l+1}(X),H_{\mathrm{b}}^{m-\epsilon,l-\epsilon}(X))$ is continuous with respect to the uniform operator topology. 
\end{proposition}
\begin{proof}
	We mimic the proof of \Cref{prop:prelim_convergence}. By the metrizability of $[0,\Sigma)\times \smash{H_{\mathrm{b}}^{m,l+1}(X)}$ and $\smash{H_{\mathrm{b}}^{m-\epsilon,l-\epsilon}(X)}$, joint continuity follows from the claim that 
	\begin{itemize}
		\item given $f \in \smash{H_{\mathrm{b}}^{m,l+1}(X)}$ and $\sigma_\infty \in [0,\Sigma)$, for any and sequences $\{\sigma_k\}_{k\in \bbN}$, $\{f_k\}_{k\in \bbN} \subset H_{\mathrm{b}}^{m,l+1}(X)$ with $\sigma_k\to \sigma_\infty$ and $f_k \to f$ as $k\to\infty$, 
		\begin{equation}
			(\sigma^2_k+\mathsf{Z}x)^{1/4}\tilde{R}_+(\sigma_k) (\sigma^2_k+\mathsf{Z}x)^{1/4} f_k \to  (\sigma_\infty^2+\mathsf{Z}x)^{1/4}\tilde{R}_+(\sigma_\infty) (\sigma_\infty^2+\mathsf{Z}x)^{1/4}f
		\end{equation}
		strongly in $ H_{\mathrm{b}}^{m-\epsilon,l-\epsilon}(X)$.
	\end{itemize}
	It suffices to show that
	\begin{itemize}
		\item  given any $\{\sigma_k\}_{k\in \bbN} \subset [0,\Sigma)$, $\{f_k\}_{k\in \bbN} \subset \smash{H_{\mathrm{b}}^{m,l+1}(X)}$ with $\sigma_k\to \sigma_\infty$ and $f_k \to f$ as $k\to\infty$, there exists a subsequence $\{k_\kappa\}_{\kappa\in \bbN}\subset \{k\}_{k\in \bbN}$  such that 
		\begin{equation} 
			(\sigma^2_{k_\kappa}+\mathsf{Z}x)^{1/4}\tilde{R}_+(\sigma_{k_\kappa}) (\sigma^2_{k_\kappa}+\mathsf{Z}x)^{1/4}f_{k_\kappa} \to (\sigma^2_\infty+\mathsf{Z}x)^{1/4}\tilde{R}_+(\sigma_\infty)(\sigma^2_\infty +\mathsf{Z}x)^{1/4} f
		\end{equation} 
		strongly in $H_{\mathrm{b}}^{m-\epsilon,l-\epsilon}(X)$.
	\end{itemize}
	As before, we only consider the case of $\sigma_\infty = 0$, since the case $\sigma_\infty>0$ follows by a similar but even easier argument.

	By \Cref{prop:misc_estimate} and Banach--Alaoglu, we can find a subsequence $\{k_\kappa\}_{\kappa\in \bbN}\subset \{k\}_{k\in \bbN}$ such that  $\smash{(\sigma^2_{k_\kappa}+\mathsf{Z}x)^{1/4}\tilde{R}_+(\sigma_{k_\kappa})(\sigma^2_{k_\kappa}+\mathsf{Z}x)^{1/4}f_{k_\kappa}}$ converges weakly in  $H_{\mathrm{b}}^{m,l}(X)$. Let $w\in H_{\mathrm{b}}^{m,l}(X)$ denote the weak limit. We now show that  $w =  \mathsf{Z}^{1/2} x^{1/4}\tilde{R}_+(0) x^{1/4}f$. 
	\begin{itemize}
		\item We first check that $w$ solves the PDE $\tilde{P}(0)(\mathsf{Z}^{-1/4} x^{-1/4}w) = \mathsf{Z}^{1/4} x^{1/4} f$. As above, we use that, for any $m_0,l_0$, 
		\begin{equation}
			[0,\infty)_\sigma \times H_{\mathrm{b}}^{m_0,l_0}(X) \ni (\sigma, u)\mapsto \tilde{P}(\sigma) u \in \calS'(X) 
		\end{equation}
		is jointly continuous with respect to the strong topology on $H_{\mathrm{b}}^{m_0,l_0}(X)$.
		Since 
		\begin{equation} 
			\mathsf{Z}^{-1/4} x^{-1/4}(\sigma^2_{k_\kappa}+\mathsf{Z} x)^{1/4}\tilde{R}_+(\sigma_{k_\kappa}) (\sigma^2_{k_\kappa}+\mathsf{Z} x)^{1/4} f_{k_\kappa} \to \mathsf{Z}^{-1/4} x^{-1/4} w
		\end{equation} 
		weakly in $H_{\mathrm{b}}^{m,l-1/2}(X)$, 
		\begin{equation}
			\tilde{P}(\sigma_{k_\kappa})( \mathsf{Z}^{-1/4} x^{-1/4}(\sigma^2_{k_\kappa}+\mathsf{Z} x)^{1/4}\tilde{R}_+(\sigma_{k_\kappa}) (\sigma^2_{k_\kappa}+\mathsf{Z} x)^{1/4} f_{k_\kappa}) \to \tilde{P}(0)(\mathsf{Z}^{-1/4} x^{-1/4}w) 
		\end{equation}
		$\calS'(X)$. Moreover, it is not difficult to see that, as before, $\tilde{P}(\sigma_{k_\kappa})( \mathsf{Z}^{-1/4} x^{-1/4}(\sigma^2_{k_\kappa}+\mathsf{Z} x)^{1/4}\tilde{R}_+(\sigma_{k_\kappa}) (\sigma^2_{k_\kappa}+\mathsf{Z} x)^{1/4}f_{k_\kappa}) \to  \mathsf{Z}^{1/4} x^{1/4} f$ in $\calS'(X)$.

		Since $\calS'(X)$ is Hausdorff, it follows that $ \tilde{P}(0)(\mathsf{Z}^{-1/4} x^{-1/4} w)  = \mathsf{Z}^{1/4} x^{1/4}f$. 
		\item 
		Thus, we have 
		\begin{align}
			\begin{split} 
				\tilde{P}(0)(\mathsf{Z}^{-1/4} x^{-1/4} w)  &= \mathsf{Z}^{1/4} x^{1/4}f \\
				\tilde{P}(0)(\tilde{R}_+(0) (\mathsf{Z}^{1/4} x^{1/4}f)) &= \mathsf{Z}^{1/4} x^{1/4}f. 
			\end{split} 
			\label{eq:misc_pdf}
		\end{align}
		Set $l_0 = l-1/4$, so that $x^{1/4} f\in H_{\mathrm{b}}^{m,l+5/4}(X) = H_{\mathrm{b}}^{l_0+3/2}(X)$. By \Cref{prop:LA_at_0b}, 
		\begin{align}
			\begin{split} 
			\tilde{P}(0) &: \{u \in H_{\mathrm{b}}^{m,l_0}(X) : \tilde{P}(0) u \in H_{\mathrm{b}}^{m,l_0+3/2}(X) \} \to H_{\mathrm{b}}^{m,l_0+3/2}(X) \\
			&: \{u \in H_{\mathrm{b}}^{m,l-1/4}(X) : \tilde{P}(0) u \in H_{\mathrm{b}}^{m,l+5/4}(X) \} \to H_{\mathrm{b}}^{m,l+5/4}(X)
			\end{split}
			\label{eq:misc_cd1}
		\end{align}
		is invertible, and the inverse is $\tilde{R}_+(0)$. Thus, $u=\tilde{R}_+(0) \mathsf{Z}^{1/4} x^{1/4} f$ is the unique solution to $\tilde{P}(0) u = \mathsf{Z}^{1/4} x^{1/4} f$ in the domain of \cref{eq:misc_cd1}. But  $\mathsf{Z}^{-1/4} x^{-1/4} w$ is in the domain, and as we saw in \cref{eq:misc_pdf} solves this PDE. We conclude that
		\begin{equation} 
			\mathsf{Z}^{-1/4} x^{-1/4} w = \tilde{R}_+(0) \mathsf{Z}^{1/4} x^{1/4} f.
		\end{equation}  
	\end{itemize}
	Via the compactness of the inclusion $H_{\mathrm{b}}^{m,l}\hookrightarrow H_{\mathrm{b}}^{m-\epsilon,l-\epsilon}$, we conclude that 
	\begin{equation} 
	(\sigma^2_{k_\kappa}+\mathsf{Z}x)^{1/4}\tilde{R}_+(\sigma_{k_\kappa}) 	(\sigma^2_{k_\kappa}+\mathsf{Z}x)^{1/4} f_{k_\kappa} \to \mathsf{Z}^{1/4}x^{1/4} \tilde{R}_+(0) \mathsf{Z}^{1/4}x^{1/4} f
	\end{equation}
	strongly in $H_{\mathrm{b}}^{m-\epsilon,l-\epsilon}$. 
	
	Uniform continuity follows as in the proof of \Cref{prop:prelim_convergence}. 
\end{proof}

\begin{proposition}
	\label{prop:conormP}
	Fix $\psi \in C_{\mathrm{c}}^\infty(\bbR)$ supported sufficiently close to $(-\infty,\Sigma^2]$ and $\chi \in C_{\mathrm{c}}^\infty(X_{\mathrm{res}}^{\mathrm{sp}})$ supported away from $\mathrm{bf}$. 
	Then, for each $k,K\in \bbN$ with $k+K>0$, 
	\begin{equation} 
		\{ \psi(E)(E \partial_E)^k ( \chi x \partial_E)^K \tilde{P}(E^{1/2})\}_{E\geq 0} \in S\operatorname{Diff}_{\mathrm{leC}}^{1,0,-2,-1,-3}(X) \subseteq S\operatorname{Diff}_{\mathrm{b,leC}}^{1,-1,-3}(X).
		\label{eq:misc_yy1}
	\end{equation} 
\end{proposition}
\begin{proof}
	For $\sigma \in [0,\Sigma)$, $\tilde{P}(\sigma) = K(\sigma) + C$ for $K(\sigma) \in S\operatorname{Diff}_{\mathrm{leC}}^{1,0,-2,-1,-3}(X)$
	containing the $\sigma$-dependent part of $\tilde{P}(\sigma)$ and $C \in S\operatorname{Diff}_{\mathrm{scb}}^{2,0,-2}(X)$ constant in $\sigma$. Thus, for $k,K \in \bbN$ which are not both zero, 
	\begin{equation}
			\{ \psi(E)(E \partial_E)^k ( \chi x \partial_E)^K \tilde{P}(\sigma)\}_{\sigma\geq 0} \in
			(E \partial_E)^k ( \chi x \partial_E)^K S\operatorname{Diff}_{\mathrm{leC}}^{1,0,-2,-1,-3}(X).
	\end{equation}
	Since $E \partial_E$ lifts to a b-vector field on $X^{\mathrm{sp}}_{\mathrm{res}}$ and $\chi x \partial_E$ lifts to a smooth vector field which is b- at $\mathrm{tf}$ (and identically zero near $\mathrm{bf}$), 
	\begin{equation} 
		(E \partial_E)^k ( \chi x \partial_E)^K  S\operatorname{Diff}_{\mathrm{leC}}^{1,0,-2,-1,-3}(X) \subseteq S\operatorname{Diff}_{\mathrm{leC}}^{1,0,-2,-1,-3}(X), 
	\end{equation}	
	so \cref{eq:misc_yy1} follows.
\end{proof}

\begin{proposition}
	\label{prop:E>0_differentiation}
	For each $m,l\in \bbR$ with $m+l>-1/2>l$ and for each $f \in H_{\mathrm{b}}^{m,l+1}(X)$, the map 
	\begin{equation}
		\tilde{R}_+(\bullet)f:(0,\Sigma) \ni \sigma \mapsto \tilde{R}_+(\sigma) f \in \calS'(X) 
	\end{equation}
	is smooth as a map $(0,\Sigma)\to \calS'(X)$. Thus, $ \partial_\sigma^k (\tilde{R}_+(\sigma)\tilde{f}(-;\sigma)) : \bbR^+_\sigma \to \calS'(X)$ is well-defined for $k\in \bbN$ and $\tilde{f} \in C^\infty((0,\Sigma); H_{\mathrm{b}}^{m,l+1}(X))$ and is given by 
	\begin{equation}
		\Big[ \sum_{\kappa=0}^k \binom{k}{\kappa}  \partial_\sigma^\kappa \Big( \tilde{R}_+(\sigma) (\partial_{\sigma_0}^{k-\kappa} \tilde{f}(-;\sigma_0)) \Big) \Big] \Big|_{\sigma=\sigma_0}.
	\end{equation}

	Moreover, for $k\in \bbN$, it is the case that, for each $\sigma>0$, the map 
	\begin{align} 
		\partial_\sigma^k \tilde{R}_+(\sigma) = \partial_{\sigma_0}^k \tilde{R}_+(\sigma_0)|_{\sigma_0=\sigma}&: \bigcup_{\substack{ m,l\in \bbR \\ k-m-1/2<l<-1/2 }} H_{\mathrm{b}}^{m,l+1}(X)\to \calS'(X) \\
		&: f \mapsto 	 \partial_\sigma^k( \tilde{R}_+(\sigma)f)
	\end{align} 
	satisfies $\partial_\sigma^k \tilde{R}_+(\sigma) \in \calL(H_{\mathrm{b}}^{m,l+1}(X),H_{\mathrm{b}}^{m-k,l}(X))$. For each $k\in \bbN^+$, the identity 
	\begin{equation}
		 \partial_\sigma^k \tilde{R}_+(\sigma) = \Big[ \sum_{\{k_i\}_{i=1}^I\in \calI_k} c_{\{k_i\}_{i=1}^I} \prod_{i=1}^I (   \tilde{R}_+(\sigma)  \partial_\sigma^{k_i} \tilde{P}(\sigma)) \Big]  \tilde{R}_+(\sigma) 
		\label{eq:misc_ghj}
	\end{equation}
	holds for some $c_{\{k_i\}_{i=1}^I} \in \bbZ$, where $\calI_k$ is the set of finite sequences of positive integers summing to $k$. 
\end{proposition}
\begin{proof}
	The proposition holds for the usual conjugated resolvent family 
	\begin{equation} 
		\tilde{\tilde{R}}_+ = e^{ - i \sigma/x+ i\Phi} \tilde{R}_+ e^{+i \sigma/x-i\Phi}
	\end{equation} 
	in place of $\tilde{R}_+$ and $\tilde{\tilde{P}} =e^{-i\sigma/x} P e^{+i\sigma/x}$ in place of $\tilde{P}$ --- see \cite[Lemma 2.10]{HintzPrice}, \cite[Proposition 2.10]{HafnerHintzVasy}. 
	Since multiplication by $\partial_\sigma^\kappa \exp (\pm i (\Phi_0- \Phi))$ acts boundedly (but not uniformly so as $\sigma\to 0^+$) on the b-Sobolev spaces for each $\kappa\in \bbN$ --- see \cref{eq:osc_+} --- the mapping properties asserted in the proposition follow from those for the usual conjugated resolvent family. 
	The specific formula \cref{eq:misc_ghj} follows via inductively applying $  \partial_\sigma \tilde{R}_+(\sigma) f  = - \tilde{R}_+(\sigma)  (\partial_\sigma \tilde{P}(\sigma) )\tilde{R}_+(\sigma) f$. 
\end{proof}
\begin{remark}
	By the conjunction of 
	\begin{enumerate}
		\item $\tilde{R}_+(\sigma) : H_{\mathrm{b}}^{m,l+1}(X) \to H_{\mathrm{b}}^{m,l}(X)$ holding for all $m,l\in \bbR$ with $l<-1/2<m+l$ and $\sigma>0$, and 
		\item  $\partial_\sigma^k \tilde{P}(\sigma): H_{\mathrm{b}}^{m,l}(X) \to H_{\mathrm{b}}^{m-1,l+1}(X)$ holding for all $m,l\in \bbR$ and $\sigma>0$, 
	\end{enumerate}
it is the case that 
	\begin{equation}
	\tilde{R}_+(\sigma) \partial_\sigma^{k} \tilde{P}(\sigma)  : H_{\mathrm{b}}^{m,l}(X)\to H_{\mathrm{b}}^{m-1,l}(X)
	\end{equation}
	for each $\sigma>0$, $k\in \bbN$, and $m,l\in \bbR$ satisfying $l<-1/2<m-1+l$. 
	
	Consequently, for $\{i_i\}_{i=1}^I\in \calI_k$, 
	\begin{equation}
	\prod_{i=1}^I \Big[ \tilde{R}_+(\sigma)   \partial_\sigma^{k_i} \tilde{P}(\sigma) \Big]: H_{\mathrm{b}}^{m,l}(X)\to H_{\mathrm{b}}^{m-I,l}(X)
	\end{equation}
	is a well-defined composition whenever $l<-1/2<m+l-I$. 
	The right-hand side of \cref{eq:misc_ghj} is therefore a well-defined map $H_{\mathrm{b}}^{m,l+1}(X) \to H_{\mathrm{b}}^{m-k,l}(X)$ whenever $l<-1/2<m+l-k$. 
	
	The identity \cref{eq:misc_ghj} should therefore be read as stating that both sides agree as maps $H_{\mathrm{b}}^{m,l+1}(X) \to H_{\mathrm{b}}^{m-k,l}(X)$. 
\end{remark}

\begin{proposition}
	\label{prop:convergence_main}
	For each $k,K\in \bbN$ and $l<-1/2$, there exists an $m_0(l,k,K)\in \bbR$ such that for all $m>m_0$ and $\chi \in C_{\mathrm{c}}^\infty(X_{\mathrm{res}}^{\mathrm{sp}})$ supported away from $\mathrm{bf}$,   
	\begin{multline}
		\{ (E \partial_E)^k (\chi x \partial_E)^K ( (E+\mathsf{Z} x)^{1/4} \tilde{R}_+(E^{1/2})) \}_{E \in (0,\Sigma^2) } \\ \in  L^\infty \cap C^0 ((0,\Sigma) ; \calL(H_{\mathrm{b}}^{m,l+5/4}(X),H_{\mathrm{b}}^{m-k-K-\epsilon,l-\epsilon}(X) ) )
		\label{eq:misc_273}
	\end{multline}
	for all $\epsilon>0$. 
\end{proposition}
\begin{proof}
	We explicitly consider the $K=0$ case, and the $K\in \bbN^+$ case is similar (if a bit messier). 
	
	So, we want to prove that 
	\begin{equation}
		\{ (\sigma \partial_\sigma)^k ( (\sigma^2+\mathsf{Z} x)^{1/4} \tilde{R}_+(\sigma)) \}_{\sigma \in (0,\Sigma) } \in L^\infty\cap C^0 ((0,\Sigma) ; \calL(H_{\mathrm{b}}^{m,l+5/4}(X),H_{\mathrm{b}}^{m-k-\epsilon,l-\epsilon}(X) )).
		\label{eq:misc_272}
	\end{equation}
	
	By \cref{eq:misc_ghj}, 
	\begin{multline}
	(\sigma \partial_\sigma)^k ( (\sigma^2+\mathsf{Z} x)^{1/4} \tilde{R}_+(\sigma)) = \sum_{j=0}^k \binom{k}{j} (\sigma^2+\mathsf{Z} x)^{-1/4} ((\sigma \partial_\sigma)^j (\sigma^2+\mathsf{Z} x)^{1/4} ) \\  \Big[  \sum_{\{k_i\}_{i=1}^I\in \calI_{k-j}} c_{\{k_i\}_{i=1}^I} \prod_{i=1}^I (\sigma^2+\mathsf{Z}x)^{1/4} \tilde{R}_+(\sigma) (\sigma \partial_\sigma)^{k_i} \tilde{P}(\sigma) (\sigma^2+\mathsf{Z}x)^{-1/4} \Big]  \Big[ (\sigma^2+\mathsf{Z}x)^{1/4} \tilde{R}_+(\sigma)\Big] 
	\end{multline}
	for all $\sigma>0$. 
	By \Cref{prop:conormP}, for any $k\in \bbN^+$ we can write 
	\begin{equation}
		 (\sigma^2+\mathsf{Z} x)^{-1/4} ((\sigma \partial_\sigma )^{k} \tilde{P}(\sigma)) (\sigma^2+\mathsf{Z} x)^{-1/4} \in  L^\infty \cap C^0([0,\Sigma); S\operatorname{Diff}_{\mathrm{b}}^{1,-1+\epsilon}(X) )
	\end{equation}
	for any $\epsilon>0$. 
	On the other hand, by \Cref{prop:uniform_bound}, for all $\epsilon>0$, 
	\begin{equation}
		 (\sigma^2+\mathsf{Z}x)^{1/4} \tilde{R}_+(\sigma) (\sigma^2+\mathsf{Z} x)^{1/4} \in C^0([0,\Sigma); \calL(H_{\mathrm{b}}^{m-1,l+1-\epsilon/2}(X),H_{\mathrm{b}}^{m-1-\epsilon,l-\epsilon}(X)))
	\end{equation}
	for any $m,l\in \bbR$ satisfying $l<-1/2$ and $m>-2l$. 
	
	Combining these two observations, for any $k\in \bbN^+$ and $\epsilon>0$ we have
	that 
	\begin{multline}
		(\sigma^2+\mathsf{Z}x)^{\frac{1}{4}} \tilde{R}_+(\sigma) ((\sigma \partial_\sigma )^{k} \tilde{P}(\sigma)) (\sigma^2+\mathsf{Z}x)^{-\frac{1}{4}} \\ = (\sigma^2+\mathsf{Z}x)^{\frac{1}{4}} \tilde{R}_+(\sigma) (\sigma^2+\mathsf{Z}x)^{\frac{1}{4}}(\sigma^2+\mathsf{Z}x)^{-\frac{1}{4}} ((\sigma \partial_\sigma )^{k} \tilde{P}(\sigma)) (\sigma^2+\mathsf{Z}x)^{-\frac{1}{4}}
	\end{multline}
	lies in 
	\begin{equation} 
		C^0( [0,\Sigma) ; \calL(H_{\mathrm{b}}^{m-1,l+1 - \frac{\epsilon}{2}},H_{\mathrm{b}}^{m-1-\epsilon,l-\epsilon}) S\operatorname{Diff}_{\mathrm{b}}^{1,-1+\frac{\epsilon}{2}}(X) )  
	 \subseteq C^0( [0,\Sigma); \calL(H_{\mathrm{b}}^{m,l},H_{\mathrm{b}}^{m-1 - \epsilon,l - \epsilon}))
	\end{equation}
	for any $m,l\in \bbR$ satisfying $l<-1/2$ and $m>-2l$. Therefore,  by \Cref{prop:prelim_convergence}, 
	\begin{multline}
	 \Big[  \sum_{\{k_i\}_{i=1}^I\in \calI_{k-j}} c_{\{k_i\}_{i=1}^I} \prod_{i=1}^I (\sigma^2+\mathsf{Z}x)^{1/4} \tilde{R}_+(\sigma) (\sigma \partial_\sigma)^{k_i} \tilde{P}(\sigma) (\sigma^2+\mathsf{Z}x)^{-1/4} \Big] \Big[ (\sigma^2+\mathsf{Z}x)^{1/4} \tilde{R}_+(\sigma)\Big]\\ \in  C^0([0,\Sigma)_\sigma ; \calL(H_{\mathrm{b}}^{m,l+5/4} ,H_{\mathrm{b}}^{m-k-\epsilon,l-\epsilon} ) )
	\end{multline} 
	for any $\epsilon>0$ and for any $m,l\in \bbR$ satisfying $l < -1/2$  and $-2l+k<m$.

	Since $ (\sigma^2+\mathsf{Z}x)^{-1/4}(\sigma \partial_\sigma)^j (\sigma^2+\mathsf{Z} x)^{1/4} \in C^0([0,\infty)_\sigma ; S^{\epsilon}_{\mathrm{b}}(X ))$ (as checked in \Cref{lem:sigmaZx_conormal}), we conclude \cref{eq:misc_272}, \cref{eq:misc_273}. 
\end{proof}

The preceding results amount to:
\begin{proposition}
	\label{prop:grace}
	Given $\tilde{f}\in C^\infty(X_{\mathrm{res}}^{\mathrm{sp}})$  vanishing rapidly at $\mathrm{tf},\mathrm{bf}$ (i.e.\ $\tilde{f} \in \calA^{\infty,\infty,(0,0)}(X_{\mathrm{res}}^{\mathrm{sp}} )$), set  $u_{00,+}(\sigma) = \tilde{R}_+(\sigma) \tilde{f}(-;\sigma)$ for all $\sigma \in [0,\Sigma)$. 
	Defining $u_{0,+}(\sigma) = x^{-(n-1)/2} (\sigma^2+\mathsf{Z} x)^{1/4} u_{00,+}(\sigma)$, $u_{0,+}  =\{u_{0,+}(-;\sigma)\}_{\sigma \in [0,\Sigma)} \in \calA^{0-,0-,(0,0)}_{\mathrm{loc}}(X^{\mathrm{sp}}_{\mathrm{res}} \cap \{\sigma<\Sigma \})$. 
	
	Moreover, the mapping $\calA^{\infty,\infty,(0,0)}(X_{\mathrm{res}}^{\mathrm{sp}})\ni \tilde{f} \mapsto u_{0,+} \in \calA_{\mathrm{loc}}^{0-,0-,(0,0)}(X_{\mathrm{res}}^{\mathrm{sp}} \cap \{\sigma<\Sigma\} )$ is continuous. 
\end{proposition}
\begin{proof} Fix $\chi \in C_{\mathrm{c}}^\infty(X_{\mathrm{res}}^{\mathrm{sp}})$ supported away from $\mathrm{bf}$ and nonvanishing near $\mathrm{zf}$. 
	\begin{itemize}
		\item 
		We first show that $u_{0,+} =\{u_{0,+}(-;\sigma)\}_{\sigma \in (0,\Sigma)}\in \calA^{0-,0-,0}_{\mathrm{loc}}(X^{\mathrm{sp}}_{\mathrm{res}} \cap \{\sigma<\Sigma\})$. 
		
		For any $k,K\in \bbN$, $(E\partial_E)^k(\chi x \partial_E)^K \tilde{f} \in \calA^{\infty,\infty,(0,0)}(X_{\mathrm{res}}^{\mathrm{sp}})$.
		Thus, 
		\Cref{prop:E>0_differentiation} and \Cref{prop:convergence_main}  imply that, for all $m\in \bbR$,  $l<-1/2$, $K\in \bbN$,  
		\begin{multline}
			\{(\chi x \partial_E)^K u_{0,+}(-;E^{1/2}) \}_{E \in (0,\Sigma^2)} \\ =\Big\{\Big( \chi x \frac{\partial}{\partial E}\Big)^K\Big[ x^{-(n-1)/2} (E+\mathsf{Z}x)^{1/4}  \tilde{R}_+(E^{1/2}) \tilde{f}(-;E^{1/2})\Big]\Big\}_{E \in (0,\Sigma^2) } 
		\end{multline}
		is in $\calA^0_{\mathrm{loc}} ([0,\Sigma^2)_E; H_{\mathrm{b}}^{m,l-(n-1)/2}(X))$. 
		
		By \Cref{prop:conormality_inclusion},  we deduce that $(\chi x \partial_E)^K u_{0,+} \in \calA^{0-,0-,0}_{\mathrm{loc}}(X^{\mathrm{sp}}_{\mathrm{res}} \cap \{\sigma<\Sigma\})$.
		Moreover, \Cref{prop:convergence_main} shows that the map 
		\begin{equation} 
			\calA^{\infty,\infty,(0,0)}(X)\ni \tilde{f} \mapsto  (\chi x \partial_E)^K  u_{0,+}(-;E^{1/2}) \in \calA^0_{\mathrm{loc}} ([0,\Sigma^2)_E; H_{\mathrm{b}}^{m,l-(n-1)/2}(X))
		\end{equation} 
		is continuous, for each $K\in \bbN$, and the inclusion in \Cref{prop:conormality_inclusion} is continuous, so the map $\calA^{\infty,\infty,(0,0)}(X)\ni \tilde{f} \mapsto  (\chi x \partial_E)^K  u_{0,+}(-;E^{1/2}) \in \calA^{0-,0-,0}_{\mathrm{loc}}(X^{\mathrm{sp}}_{\mathrm{res}})$ is continuous as well.
		\item It now follows from \Cref{prop:smoothness_improver}  that $\{u_{0,+}(-;\sigma)\}_{\sigma>0}$ defines an element of the space $\smash{\calA_{\mathrm{loc}}^{0-,0-,(0,0)}(X_{\mathrm{res}}^{\mathrm{sp}} \cap \{\sigma < \Sigma\})}$. By the remark following \Cref{prop:smoothness_improver}, 
		\begin{equation}
			\calA^{\infty,\infty,(0,0)}(X_{\mathrm{res}}^{\mathrm{sp}})\ni\tilde{f} \mapsto u_{0,+} \in \calA^{0-,0-,(0,0)}_{\mathrm{loc}}(X^{\mathrm{sp}}_{\mathrm{res}}\cap \{\sigma < \Sigma\})
		\end{equation}
		is continuous.
		\item 
		Finally, we observe from \Cref{prop:conormP} that $\{u_{0,+}(-;\sigma)\}_{\sigma>0}$, considered as an element of $\calA_{\mathrm{loc}}^{0-,0-,(0,0)}(X_{\mathrm{res}}^{\mathrm{sp}})$, restricts to $x^{-(n-1)/2} (\mathsf{Z} x)^{1/4} \tilde{R}_+(0) \tilde{f}(-;0)$ at $\mathrm{zf}$.

	\end{itemize}
\end{proof}

\subsection{Asymptotics at $\mathrm{bf}$, $\mathrm{tf}$}
\label{subsec:smoothness2}

For this subsection, we do not assume that $P$ is the spectral family of an attractive Coulomb-like Schr\"odinger operator, only that $P$ satisfies the hypotheses in \S\ref{sec:introduction}. Let $\calE$ denote the index set 
\begin{equation}
	\calE = \{(k,\kappa)\in \bbN\times \bbN : \kappa \leq \lfloor k/2 \rfloor \}.
\end{equation}

Suppose that $P_1$ is classical to order $\beta_1>0$ and $P_2$ is classical to order $(\beta_2,\beta_3)$. 
Then, by \Cref{prop:normal_error}, 
\begin{equation} 
	\operatorname{N}(\tilde{P}) - \tilde{P} \in \operatorname{Diff}_{\mathrm{b,leC}}^{2,-2,-4}(X) + S\operatorname{Diff}_{\mathrm{b,leC}}^{2,-1-\delta_1,-3-2\delta_0}(X)
\end{equation}
holds for $\delta_1 = \min\{1+\beta_1,\beta_2,1/2+\beta_3\}$  and $\delta_0 = \min\{1/2+\beta_1,\beta_2,\beta_3\}$.

\begin{proposition}
	\label{prop:inductive_smoothness}
	Suppose that $P_1$ is classical to $\beta_1$th order and $P_2$ is classical to order $(\beta_2,\beta_3)$. 
	If $u_0\in \smash{\calA_{\mathrm{loc}}^{0-,0-,(0,0)}(X^{\mathrm{sp}}_{\mathrm{loc}})}$ and $\tilde{f} \in \calA_{\mathrm{loc}}^{\infty,\infty,(0,0)}(X^{\mathrm{sp}}_{\mathrm{res}})$, then, setting 
	\begin{equation} 
		u_{00}(-;\sigma)= x^{(n-1)/2}(\sigma^2+\mathsf{Z}x)^{-1/4}u_0(-;\sigma),
	\end{equation} 
	if $\tilde{P}u_{00} = \tilde{f}$, then
	\begin{equation} 
		u_0 \in \calA_{\mathrm{loc}}^{(0,0),\calE,(0,0)}(X^{\mathrm{sp}}_{\mathrm{res}}) + \calA_{\mathrm{loc}}^{((0,0),\delta_1),2\delta_0-,(0,0)}(X^{\mathrm{sp}}_{\mathrm{res}}), 
		\label{eq:misc_kzh}
	\end{equation}  
	and $(\calA_{\mathrm{loc}}^{0-,0-,(0,0)} \times \calA_{\mathrm{loc}}^{\infty,\infty,(0,0)})|_{\tilde{P}u_{00}=\tilde{f}} \ni (u_0,\tilde{f})\mapsto u_0 \in  \calA_{\mathrm{loc}}^{(0,0),\calE,(0,0)} + \calA_{\mathrm{loc}}^{((0,0),\delta_1),2\delta_0-,(0,0)}$ is continuous. 
	In particular, if $P$ is fully classical, then 
	\begin{equation} 
		u_{0}\in  \calA_{\mathrm{loc}}^{(0,0),\calE,(0,0)}(X^{\mathrm{sp}}_{\mathrm{res}}).
	\end{equation} 
\end{proposition} 
\begin{proof}
	We will prove via induction on $\alpha,\beta$ that, for all $\alpha\in (-\infty,\delta_1]$ and $\beta \in (-\infty,\delta_0)$, 
	\begin{equation}
		u_0 \in  \calA_{\mathrm{loc}}^{(0,0),\calE,(0,0)}(X^{\mathrm{sp}}_{\mathrm{res}}) + \calA_{\mathrm{loc}}^{((0,0),\alpha),2\beta ,(0,0)}(X^{\mathrm{sp}}_{\mathrm{loc}}).
		\label{eq:misc_ao1}
	\end{equation}
	The $\alpha,\beta < 0$ case of \cref{eq:misc_ao1} is the hypothesis of the proposition.

	Consider $\alpha,\beta$ such that \cref{eq:misc_ao1} holds.
	From $\tilde{P} u_{00} = \tilde{f}$, writing $\tilde{P} = \operatorname{N}(\tilde{P}) + (\tilde{P} - \operatorname{N}(\tilde{P}))$, we get 
	\begin{multline}
		\operatorname{N}(\tilde{P})  u_{00} = \tilde{f}  - (\tilde{P} - \operatorname{N}(\tilde{P})) u_{00} \in  \calA_{\mathrm{loc}}^{\infty,\infty,(0,0)}(X^{\mathrm{sp}}_{\mathrm{res}}) \\ + (\operatorname{Diff}_{\mathrm{b,leC}}^{2,-2,-4}(X) + S\operatorname{Diff}_{\mathrm{b,leC}}^{2,-1-\delta_1,-3-2\delta_0}(X))x^{(n-1)/2}\calA_{\mathrm{loc}}^{((0,0),\alpha),2\beta-1/2,(0,0)}(X^{\mathrm{sp}}_{\mathrm{res}}) \\
		+ (\operatorname{Diff}_{\mathrm{b,leC}}^{2,-2,-4}(X) + S\operatorname{Diff}_{\mathrm{b,leC}}^{2,-1-\delta_1,-3-2\delta_0}(X))x^{(n-1)/2}\calA_{\mathrm{loc}}^{(0,0),\calE-1/2,(0,0)}(X^{\mathrm{sp}}_{\mathrm{res}}). 
	\end{multline}
	The set on the right-hand side is $x^{(n-1)/2}$ times 
	\begin{multline}
		  \calA_{\mathrm{loc}}^{(2,0),\calE+7/2,(0,0)} (X_{\mathrm{res}}^{\mathrm{sp}} )  
		   +  \calA_{\mathrm{loc}}^{((2,0),2+\alpha ),2\beta+7/2,(0,0)} (X^{\mathrm{sp}}_{\mathrm{res}}) \\ +\calA_{\mathrm{loc}}^{\min\{1+\delta_1,1+\alpha+\delta_1\}, \min\{5/2+2\delta_0,5/2+2\delta_0+2\beta\},(0,0)}(X^{\mathrm{sp}}_{\mathrm{res}}).
	\end{multline}
	For $\alpha$ sufficiently close to zero or positive, we can apply \Cref{prop:N_mapping_properties} below to conclude that 
	\begin{multline}
	 u_0 \in  \calA^{(0,0),\calE,(0,0)}_{\mathrm{loc}}  + \calA_{\mathrm{loc}}^{((0,0),1+\alpha),1+2\beta-,(0,0)}   +\calA_{\mathrm{loc}}^{((0,0),\min\{\delta_1,\alpha+\delta_1\}), \min\{2\delta_0,2\delta_0+2\beta\}-,(0,0)} \\
	 \subseteq \calA^{(0,0),\calE,(0,0)}_{\mathrm{loc}}(X^{\mathrm{sp}}_{\mathrm{res}})  +\calA_{\mathrm{loc}}^{((0,0),\min\{1+\alpha,\delta_1,\alpha+\delta_1\}), \min\{1+2\beta,2\delta_0,2\delta_0+2\beta\}-,(0,0)}(X^{\mathrm{sp}}_{\mathrm{res}}).
	\end{multline}
	If $\delta_1 \leq \alpha+1 , \alpha+\delta_1$ and $2\delta_0 \leq 2\beta+1,2\beta+ 2\delta_0$, then 
	\begin{equation} 
		\calA_{\mathrm{loc}}^{((0,0),\min\{\alpha+1,\delta_1,\alpha+\delta_1\}), \min\{1+2\beta,2\delta_0,2\delta_0+2\beta\}-,(0,0)}(X^{\mathrm{sp}}_{\mathrm{res}}) = \calA_{\mathrm{loc}}^{((0,0),\delta_1),2\delta_0-,(0,0)}(X^{\mathrm{sp}}_{\mathrm{res}}),
	\end{equation} 
	so we have concluded that \cref{eq:misc_kzh} holds. Otherwise, if neither inequality holds,  
	\begin{equation}
		u_0 \in  \calA^{(0,0),\calE,(0,0)}_{\mathrm{loc}}(X^{\mathrm{sp}}_{\mathrm{res}}) + \calA_{\mathrm{loc}}^{((0,0),\alpha+\varepsilon_1),2\beta+2\varepsilon_0-}(X^{\mathrm{sp}}_{\mathrm{res}})
		\label{eq:misc_j12}
	\end{equation}
	for $\varepsilon_1 = \min\{1,\delta_1\}$ and $\varepsilon_0 = \min\{1/2,\delta_0\}$, so we have improved $\alpha,\beta$ by some amount bounded below. If one of the two inequalities holds, for example $\delta_1 \leq \alpha+1,\alpha+\delta_1$, then we get \cref{eq:misc_j12} with the corresponding one of $\varepsilon_0,\varepsilon_1=0$ and the corresponding of $\alpha,\beta$ equal to the corresponding of $\delta_1,\delta_0$. Either way, we have improved the other of $\alpha,\beta$ by some amount bounded below (and the other is already equal to the target value).
	The claim therefore follows by induction. 
	
	The continuity clause of the proposition can be proven using the same argument, keeping track of topologies. 
\end{proof}

\begin{proposition}
	\label{prop:N_mapping_properties}
	Suppose that $P$ satisfies the minimal hypotheses of \S\ref{sec:introduction}. 
	
	Suppose further that $u\in \calA_{\mathrm{loc}}^{-\infty,-\infty,(0,0)}(X^{\mathrm{sp}}_{\mathrm{res}}) =  \calA_{\mathrm{loc}}^{-\infty,(0,0)}([0,\infty)_{E}\times X)$ satisfies $\operatorname{N}(\tilde{P}) u =f$ for some 
	\begin{equation}
		f \in  x^{(n-1)/2}  \calA^{(2,0),\calE+7/2,(0,0)}_{\mathrm{loc}}(X^{\mathrm{sp}}_{\mathrm{res}} )  + x^{(n-1)/2}\calA_{\mathrm{loc}}^{((2,0),\alpha +1),\beta  + 5/2,(0,0)}(X^{\mathrm{sp}}_{\mathrm{res}})
		\label{eq:misc_ffa}
	\end{equation}
	for $\alpha,\beta \in \bbR^+$, $\beta\notin 2\bbN$. Then, setting $u_0 = x^{-(n-1)/2} (\sigma^2+\mathsf{Z}x)^{1/4} u$, 
	\begin{equation}
		u_0 \in \calA^{(0,0),\calE,(0,0)}_{\mathrm{loc}}(X^{\mathrm{sp}}_{\mathrm{res}}) + \calA_{\mathrm{loc}}^{((0,0),\alpha),\beta,(0,0)} (X^{\mathrm{sp}}_{\mathrm{res}})
		\label{eq:misc_u00}
	\end{equation} 
	holds.
\end{proposition}

\begin{proof}
	Now letting  $\tilde{\operatorname{N}}(\tilde{P})=x^{-(n-1)/2}(\sigma^2+\mathsf{Z}x)^{1/4}\operatorname{N}(\tilde{P})x^{(n-1)/2}(\sigma^2+\mathsf{Z}x)^{-1/4}$, $\tilde{\operatorname{N}}(\tilde{P})u_0 = f_0$ for $u_0 = x^{-(n-1)/2} (\sigma^2+\mathsf{Z}x)^{1/4} u$ and $f_0 = x^{-(n-1)/2} (\sigma^2+\mathsf{Z}x)^{1/4} f$. Below, it will be slightly more convenient to work with $f_1 =x^{-1} (\sigma^2+\mathsf{Z} x)^{-1/2} f_0$.
	\Cref{eq:misc_ffa} yields 
	\begin{equation}
			f_1 \in  x  (\sigma^2+\mathsf{Z}x)^{-1/2} \calA^{(0,0),\calE,(0,0)}_{\mathrm{loc}}(X^{\mathrm{sp}}_{\mathrm{res}} )  + \calA_{\mathrm{loc}}^{((1,0),\alpha),\beta,(0,0)}(X^{\mathrm{sp}}_{\mathrm{res}}).
	\end{equation}

	In order to prove the proposition, it suffices to restrict attention to $\hat{X} = [0,\bar{x})_x\times \partial X_y$.
	By \cref{eq:misc_i12} (with $k=-1/4$ and $l=-1/2$), we have 
	\begin{equation} 
		\tilde{\operatorname{N}}(\tilde{P}) = 2 i x^2(\sigma^2+\mathsf{Z}x)^{1/2} \partial_x. 
	\end{equation}	
	Thus, integrating $\tilde{\operatorname{N}}(\tilde{P}) u_0 =f_0$, we get
	\begin{equation}
		u_0(x,y;\sigma) = c(y;\sigma) + \frac{i}{2} \int_{x}^{\bar{x}} x_0^{-1} f_1(x_0,y;\sigma) \dd x _0 
		\label{eq:misc_999}
	\end{equation}
	for some $c(y;\sigma) \in \bbC$, for each $\sigma\geq 0$. Since $u_0 \in \calA_{\mathrm{loc}}^{-\infty,(0,0)}([0,\infty)_E\times \hat{X} )$, and since the same applies to the integral in \cref{eq:misc_999} (cut-off near $\partial X$), we deduce that $\smash{c \in \calA_{\mathrm{loc}}^{-\infty,(0,0)}( [0,\infty)_E \times \hat{X}) }$.
	Since $c(y,\sigma)$ does not depend on $x$, this implies that $c \in C^\infty([0,\infty)_E \times \partial X)$.
	
	\Cref{eq:misc_u00} then follows from the mapping properties of the integral in \cref{eq:misc_999}, which we record in \Cref{cor:normal_mapping_1} and \Cref{prop:normal_mapping_2} below.
\end{proof}

For the following proposition, we use $\hat{X}^{\mathrm{sp}}_{\mathrm{res}}$, defined as $X^{\mathrm{sp}}_{\mathrm{res}}$ with $\hat{X} = [0,\bar{x})\times \partial X$ in place of $X$. 

\begin{proposition}
	\label{prop:normal_mapping_1}
	Let $\calF\subset \bbN\times \bbN$ be some index set containing $(0,0)$.
	
	If $g \in  x (\sigma^2+\mathsf{Z}x)^{-1/2}\calA^{(0,0),\calF,(0,0)}_{\mathrm{loc}}(\hat{X}^{\mathrm{sp}}_{\mathrm{res}})= \calA^{(1,0),\calF+1,(0,0)}_{\mathrm{loc}}(\hat{X}^{\mathrm{sp}}_{\mathrm{res}})$, then 
	\begin{equation} 
		I=\int_x^{\bar{x}} x_0^{-1} g(x_0,y;\sigma) \dd x_0 \in \calA^{(0,0),\calF_+,(0,0)}_{\mathrm{loc}}(\hat{X}^{\mathrm{sp}}_{\mathrm{res}}),
	\end{equation} 
	where $\calF_+ = \calF\cup  \{ (k+1,\kappa+1) , (k+2,\kappa+1) :  (k,\kappa)\in \calF, k\equiv 1 \bmod 2 \}$.
\end{proposition}

\begin{proof}
	It suffices to consider the case when $g$ is supported in $[0,\bar{x}/2)_x$.
	Let us write $G(x,y;E)=x^{-1} (\sigma^2+\mathsf{Z}x)^{1/2} g$, so 
	\begin{equation}
		I =\int_x^{\bar{x}} (E+\mathsf{Z}x_0)^{-1/2} G(x_0,y;E) \dd x_0  = 2 \int_{x^{1/2}}^{\bar{x}^{1/2}} \Big( \frac{\rho_0^2}{E+\mathsf{Z} \rho_0^2} \Big)^{1/2} G(\rho_0^2,y;E) \dd \rho_0 .
		\label{eq:misc_hz2}
	\end{equation}
	This is evidently smooth away from $\mathrm{tf}$. 
	
	By the polyhomogeneity of $G$ at $\mathrm{zf}\cap\mathrm{tf}$, there exists a $G_0 \in \calA^{(0,0),\calF}_{\mathrm{loc}}([0,\infty)_{\hat{E}}\times [0,\infty)_\rho\times \partial X)$
	such that $G_0(\rho,y;\hat{E}) = G(x,y;E)$ when $x=\rho^2$ and $E=\hat{E}x$. That is, $G(x,y;E)=G_0(x^{1/2},y;E x^{-1})$.
	
	Then, as a function of $\rho,\rho_0,y$, and $\hat{E}=E/x$, $G(\rho_0^2,y;E)=G_0(\rho_0,y;E \rho_0^{-2}) = G_0(\rho_0,y;\hat{E}\rho^2 \rho_0^{-2})$. We can therefore write
	\begin{multline}
		I=I(\rho,y;\hat{E})  = 2\int_\rho^{\bar{x}^{1/2}} \frac{1}{(\hat{E} \rho^2 \rho_0^{-2}+\mathsf{Z})^{1/2}}  G_0\Big(\rho_0,y;\hat{E} \Big( \frac{\rho}{\rho_0} \Big)^2\Big) \dd \rho_0  \\ 
		= \int_\rho^{\bar{x}^{1/2}}   G_1\Big(\rho_0,y;\hat{E} \Big( \frac{\rho}{\rho_0} \Big)^2\Big) \dd \rho_0 
	\end{multline}
	for $G_1 \in \calA^{(0,0),\calF}_{\mathrm{loc}}([0,\infty)_{\hat{E}}\times [0,\infty)_\rho\times \partial X)$ defined by $G_1(\rho ,y;\hat{E}) =  2(\hat{E} +\mathsf{Z})^{-1/2} G_0(\rho,y;\hat{E})$.

	We now expand $G_1(\rho,y;\hat{E})$ polyhomogeneously around $\rho = 0$: there exist $G_1^{(k,\kappa)} \in  C^\infty([0,\infty)_{\hat{E}}\times \partial X)$ and $F_1^{(k)} \in \calA^{(0,0), \calF|_{\geq k}-k}_{\mathrm{loc}} ([0,\infty)_{\hat{E}}\times [0,\infty)_{\rho}\times \partial X)$ for $k\in \bbN$ such that
	\begin{equation}
		G_1\Big(\rho_0,y;\hat{E} \Big( \frac{\rho}{\rho_0} \Big)^2 \Big) =  \sum_{k=0}^K \sum_{\varkappa=0}^{\varkappa_k} \rho_0^k \log^\varkappa (\rho_0)  G_1^{(k,\varkappa)}\Big(y; \hat{E} \Big( \frac{\rho}{\rho_0} \Big)^2 \Big) + \rho_0^{K+1}  F_1^{(K+1)}\Big(\rho_0,y;\hat{E} \Big( \frac{\rho}{\rho_0} \Big)^2 \Big),
	\end{equation}
	where $\calF|_{\geq k}-k=\{(k'-k,\kappa) \in \calE : \Re k'\geq k \}$, $\varkappa_k = \max\{\varkappa: (k,\varkappa)\in \calF\}$,  and $G_1^{(k,\varkappa)}$ can be nonzero only if $(k,\varkappa)\in \calF$.

	Expand $G_1^{(k,\varkappa)}$ in the second slot: there exist $G_1^{(j,k,\varkappa)}\in C^\infty(\partial X)$ and $F_2^{(j,k,\varkappa)} \in C^\infty([0,\infty)_{\hat{E}} \times \partial X)$ such that 
	\begin{equation}
		G_1^{(k,\varkappa)}\Big(y; \hat{E} \Big( \frac{\rho}{\rho_0} \Big)^2 \Big) = \sum_{j=0}^J \rho^{2j} \rho_0^{-2j} \hat{E}^j G_1^{(j,k,\varkappa)}(y) + \rho^{2J+2}\rho_0^{-2J-2} \hat{E}^{J+1} F_2^{(J+1,k,\varkappa)}\Big(y; \hat{E}\Big(\frac{\rho}{\rho_0} \Big)^2\Big).
	\end{equation} 
	Let $\digamma_{k,\varkappa} = \int \rho^k \log^\varkappa(\rho) \dd \rho$ for $k\in \bbZ$ and $\varkappa \in \bbN$, with the additive constant chosen for convenience. Then, if $k\neq -1$, we can take 
	\begin{equation}
		\digamma_{k,\varkappa}(\rho) = \rho^{k+1} \sum_{\kappa=0}^\varkappa c_{k,\varkappa,\kappa} \log^\kappa(\rho)
	\end{equation} 
	for some $c_{k,\varkappa,\kappa}\in \bbR$. If instead $k=-1$, we have $\digamma_{-1,\varkappa}(\rho) = (\varkappa+1)^{-1} \log^{\varkappa+1}(\rho)$. Integrating $G_1$, we write 
	\begin{multline}
		I(\rho,y;\hat{E}) =  \sum_{k=0}^{K} \sum_{\varkappa=0}^{\varkappa_k}   \sum_{j=0}^{J}  \rho^{2j} \hat{E}^j G_1^{(j,k,\varkappa)}(y)\int_\rho^{\bar{x}^{1/2}}  \rho_0^{k-2j} \log^\varkappa (\rho_0) \dd \rho_0  \\ 
		+\sum_{k=0}^K  \sum_{\varkappa=0}^{\varkappa_k}  \rho^{2J+2} \hat{E}^{J+1} \int_{\rho}^{\bar{x}^{1/2}} \rho_0^{k-2J-2} \log^\varkappa (\rho_0) F_2^{(J+1,k,\varkappa)}\Big( y; \hat{E} \Big( \frac{\rho}{\rho_0} \Big)^2 \Big) \dd \rho_0  \\ 
		+ \int_{\rho}^{\bar{x}^{1/2}}  \rho_0^{K+1} F_1^{(K+1)}\Big(\rho_0,y;\hat{E} \Big( \frac{\rho}{\rho_0} \Big)^2 \Big) \dd \rho_0.
	\end{multline}
	We decompose this as $I=I_1+I_2+I_3$ for 
	\begin{align}
		\begin{split} 
		I_1 &=  \sum_{k=0}^{K} \sum_{\varkappa=0}^{\varkappa_k}   \sum_{j=0}^{J}  \rho^{2j} \hat{E}^j G_1^{(j,k,\varkappa)}(y)\int_\rho^{\bar{x}^{1/2}}  \rho_0^{k-2j} \log^\varkappa (\rho_0) \dd \rho_0 \\
		&=  \tilde{I}_1+  \sum_{k=0}^{K}\sum_{\varkappa=0}^{\varkappa_k}\sum_{\substack{ j=0 \\ k-2j\neq -1} }^{J} \rho^{2j}\hat{E}^j G_1^{(j,k,\varkappa)}(y) [\digamma_{k-2j,\varkappa}(\bar{x}^{1/2})  -\digamma_{k-2j,\varkappa}(\rho)  ] \\
		&=  \tilde{I}_1+  \sum_{k=0}^{K}\sum_{\varkappa=0}^{\varkappa_k}\sum_{\substack{ j=0 \\ k-2j\neq -1} }^{J}  \hat{E}^j G_1^{(j,k,\varkappa)}(y) \sum_{\kappa=0}^{\varkappa} c_{k-2j,\varkappa,\kappa} [ \rho^{2j}\bar{x}^{k-2j+1}\log^\varkappa(\bar{x}) - \rho^{k+1} \log^\varkappa(\rho) ], 
		\end{split} 
	\end{align} 
	where 
	\begin{align} 
		\begin{split}
		\tilde{I}_1 &=  
		\sum_{k=0}^{K}\sum_{\varkappa=0}^{\varkappa_k}  \sum_{\substack{ j=0 \\ 2j=  k+1} }^{J}  \rho^{2j} \hat{E}^j G_1^{(j,k,\varkappa)}(y) \Big[\digamma_{-1,\varkappa}(\bar{x}^{1/2}) - \digamma_{-1,\varkappa}(\rho) \Big]  \\ 
		&=\sum_{k=0}^{K}\sum_{\varkappa=0}^{\varkappa_k}  \sum_{\substack{ j=0 \\ 2j=  k+1} }^{J}  \rho^{k+1} \hat{E}^j G_1^{(j,k,\varkappa)}(y) \frac{1}{\varkappa+1} \Big[ \log^{\varkappa+1}(\bar{x}^{1/2})  -\log^{\varkappa+1}(\rho) \Big],
		\end{split}   
	\end{align} 
	and 
	\begin{align} 
		I_2 &= \sum_{k=0}^K  \sum_{\varkappa=0}^{\varkappa_k}  \rho^{2J+2} \hat{E}^{J+1} \int_{\rho}^{\bar{x}^{1/2}} \rho_0^{k-2J-2} \log^\varkappa (\rho_0) F_2^{(J+1,k,\varkappa)}\Big( y; \hat{E} \Big( \frac{\rho}{\rho_0} \Big)^2 \Big) \dd \rho_0 \\
		I_3 &= \int_{\rho}^{\bar{x}^{1/2}}  \rho_0^{K+1} F_1^{(K+1)}\Big(\rho_0,y;\hat{E} \Big( \frac{\rho}{\rho_0} \Big)^2 \Big) \dd \rho_0. 
	\end{align}
	Since $(0,0)\in \calF$,  
	\begin{equation}
		 \sum_{k=0}^{K}\sum_{\substack{ j=0 \\ k-2j\neq -1} }^{J}  \hat{E}^j G_1^{(j,k,\varkappa)}(y) \sum_{\kappa=0}^{\varkappa} c_{k-2j,\varkappa,\kappa}  \rho^{2j}\bar{x}^{k-2j+1}\log^\varkappa(\bar{x}) \in  \calA^{(0,0),\calF}_{\mathrm{loc}} ([0,\infty)_{\hat{E}}\times [0,\infty)_\rho\times \partial X). 
	\end{equation}
	Also, 
	\begin{equation}
		\sum_{k=0}^{K}\sum_{\substack{ j=0 \\ k-2j\neq -1} }^{J}  \hat{E}^j G_1^{(j,k,\varkappa)}(y) \sum_{\kappa=0}^{\varkappa} c_{k-2j,\varkappa,\kappa}  \rho^{k+1} \log^\varkappa(\rho) \in  \calA^{(0,0),\calF+1}_{\mathrm{loc}} ([0,\infty)_{\hat{E}}\times [0,\infty)_\rho\times \partial X). 
	\end{equation}
	Thus, $I_1-\tilde{I}_1  \in \calA^{(0,0),\calF,(0,0)}_{\mathrm{loc}} ([0,\infty)_{\hat{E}}\times [0,\infty)_\rho\times \partial X)$.  
	Moreover, $\tilde{I}_1 \in \hat{E} \calA^{(0,0),\calF_+,(0,0)}_{\mathrm{loc}} ([0,\infty)_{\hat{E}}\times [0,\infty)_\rho\times \partial X) $. 
	
	On the other hand, 
	 \begin{equation} 
	 I_3 \in \calA^{(0,0),((0,0),K+1) }_{\mathrm{loc}} ([0,\infty)_{\hat{E}} \times [0,\infty)_\rho \times \partial X ),
	 \end{equation} 
	 as follows from an $L^\infty$ estimate on $(\partial_{\hat{E}})^a (\rho\partial_\rho)^b (\partial_\rho)^c L I_3$ for $a,b,c\in \bbN$ with $c \leq K+1$ and $L\in \operatorname{Diff}(\partial X)$ (using that $\rho/\rho_0\leq 1$ in the domain of integration).

 	Additionally, we can write 
 	\begin{equation}
 	I_2 = \sum_{k=0}^K \sum_{\varkappa=0}^{\varkappa_k}\sum_{a=0}^\varkappa \rho^{k+1}  \hat{E}^{J+1}  \log^{\varkappa-a}(\rho) \int_{\rho/\bar{x}^{1/2}}^{1}  t^{2J-k} (\log t)^{a}  \tilde{F}_2^{(J+1,k,\varkappa),a}(y;\hat{E}t^2) \dd t 
 	\end{equation}
	for some $\tilde{F}_2^{(j,k,\kappa),a} \in C^\infty([0,\infty)_{\hat{E} }\times \partial X)$. We have
	\begin{equation}
	\int_{\rho/\bar{x}^{1/2}}^{1}  t^{2J-k} (\log t)^a \tilde{F}_2^{(J+1,k,\varkappa),a}(y;\hat{E}t^2) \dd t \in \calA_{\mathrm{loc}}^{(0,0),((0,0),2J-K+1-)}([0,\infty)_{\hat{E}}\times [0,\infty)_\rho \times \partial X ).
	\end{equation}
	So,
	\begin{equation}
	I_2  \in \calA^{(0,0), (\calF,2J-K+1-)}_{\mathrm{loc}} ([0,\infty)_{\hat{E}}\times [0,\infty)_\rho\times \partial X). 
	\end{equation}
	
	Combining all of the parts above, we have 
	\begin{equation}
		I \in \calA^{(0,0), (\calF_+, \min\{K+1, 2J-K+1\}-)}_{\mathrm{loc}} ([0,\infty)_{\hat{E}}\times [0,\infty)_\rho\times \partial X)
	\end{equation}
	By taking $J$ to be very large, $K+1$ and $2J-K+1$ can be made arbitrarily large simultaneously, so we conclude that $I$ defines an element of 
	\begin{equation}
		\bigcap_{L=1}^\infty \calA^{(0,0), (\calF_+, L)}_{\mathrm{loc}} ([0,\infty)_{\hat{E}}\times [0,\infty)_\rho\times \partial X) = \calA^{(0,0),\calF_+}_{\mathrm{loc}}([0,\infty)_{\hat{E}}\times [0,\infty)_\rho\times \partial X). 
	\end{equation}
	This shows -- in conjunction with the smoothness of $I$ away from $\mathrm{tf}$ -- that $I$ is locally in $\calA^{(0,0),\calF_+,(0,0)}(\hat{X}_{\mathrm{res}}^{\mathrm{sp}})$ everywhere except possibly the corner $\mathrm{bf} \cap \mathrm{tf}$, to which we now turn. 
	
	We consider the following two cases: 
	\begin{itemize}
		\item 
		Suppose that $G$ is supported in some set of the form $\{x/E>C\}$. Let $\varrho = x/E$. Near $\mathrm{bf}\cap \mathrm{tf}$,
		\begin{equation}
			I(x,y,E) = I(\varrho E,y,E) = I(C E,y,E).
			\label{eq:misc_ixy}
		\end{equation}
		Defining  $\rho = C^{1/2} \sigma$ and $\hat{E} = C^{-1}$, $I(C E,y,E)=I(\rho^2,y, \hat{E} \rho^2 ) $. Since we already know that $I(\rho^2,y,\hat{E} \rho^2)$ depends polyhomogeneously on $\rho,\hat{E},y$, with the desired index set $\calF_+$, we conclude that $I$ depends polyhomogeneously on $\sigma,y$ alone near $\mathrm{bf}\cap \mathrm{tf}$. 
		\item  
		On the other extreme, suppose that $G$ is supported in some set of the form $\{x/E<c\}$, for $c>0$.
		In order to study $I$ near $\mathrm{bf} \cap \mathrm{tf}$, we work with the coordinate $\varrho = x / E$.  By the polyhomogeneity of $G$, there exists a $G_2 \in \calA^{\calF,(0,0)}_{\mathrm{loc}} ([0,\infty)_\sigma \times [0,\infty)_\varrho \times \partial X )$ such that $G_2(\varrho,y;\sigma) = G(x,y;E)$ whenever $\varrho = x/E$ and $\sigma^2=E$. We can then write 
		\begin{equation}
			I = \sigma \int_{\varrho}^{\min\{\bar{x}/\sigma^2 ,c \}} (1+\mathsf{Z} \varrho_0)^{-1/2} G_2 (\varrho_0,y;\sigma)  \dd \varrho_0 =\sigma \int_{\varrho}^{\min\{\bar{x}/\sigma^2 ,c \}}  G_3 (\varrho_0,y;\sigma)  \dd \varrho_0 
			\label{eq:misc_o1g}
		\end{equation}
		for $G_3 \in  \calA^{\calF,(0,0)}_{\mathrm{loc}}([0,\infty)_\sigma \times [0,\infty)_\varrho \times \partial X )$ defined by $G_3(\varrho,y;\sigma) = (1+\mathsf{Z} \varrho)^{-1/2} G_2(\varrho,y;\sigma) $. The right-hand side of \cref{eq:misc_o1g} is in $\sigma \calA^{\calF,(0,0)}([0,\infty)_\sigma \times [0,\infty)_\varrho \times \partial X )$ for $\sigma < \bar{x}^{1/2} / c^{1/2}$.
	\end{itemize}
	Since any $G  \in \calA_{\mathrm{loc}}^{(0,0),\calF,(0,0)}(\hat{X}_{\mathrm{res}}^{\mathrm{sp}})$ can be decomposed $G=G_1+G_2$ into $G_1,G_2 \in \calA_{\mathrm{loc}}^{(0,0),\calF,(0,0)}(\hat{X}_{\mathrm{res}}^{\mathrm{sp}})$ with $G_1$ supported on $\{x/E>C\}$ and $G_2$ supported on $\{x/E<c\}$ for some $c,C>0$, and since $I=I[G]$ depends linearly on $G$, we can conclude that 
	\begin{equation} 
		I\in \calA_{\mathrm{loc}}^{(0,0),\calF_+,(0,0)}(\hat{X}_{\mathrm{res}}^{\mathrm{sp}}).
	\end{equation}  
\end{proof}
A simple example worth keeping mind is $g = x \sigma^2 (\sigma^2+\mathsf{Z} x)^{-1}$, for which 
\begin{equation}
	\int_{x}^{\bar{x}} x_0^{-1} g(x_0;\sigma) \dd x_0  = -\frac{\sigma^2}{\mathsf{Z}} \log \Big( \frac{\sigma^2+\mathsf{Z}x}{\sigma^2+\mathsf{Z} \bar{x} } \Big). 
\end{equation}
This shows that, even when $g\in x  C^\infty(\hat{X}_{\mathrm{res}}^{\mathrm{sp}})$, solving $x\partial_x u= g$ for $u$ may produce logarithms at $\mathrm{tf}$. To compare with \Cref{prop:normal_mapping_1}, 
\begin{equation} 
	(\sigma^2+\mathsf{Z}x)^{1/2}x^{-1} g \in \calA^{(0,0),(1,0),(1,0)}_{\mathrm{loc}}(\hat{X}_{\mathrm{res}}^{\mathrm{sp}} ) \subset \calA^{(0,0),(0,0),(0,0)}_{\mathrm{loc}}(\hat{X}_{\mathrm{res}}^{\mathrm{sp}} ),
\end{equation} 
while 
\begin{equation}
	-\frac{\sigma^2}{\mathsf{Z}} \log \Big( \frac{\sigma^2+\mathsf{Z}x}{\sigma^2+\mathsf{Z} \bar{x} } \Big) \in \calA_{\mathrm{loc}}^{(0,0), (2,1), (1,0)}(\hat{X}_{\mathrm{res}}^{\mathrm{sp}}) \subset \calA_{\mathrm{loc}}^{(0,0), (2,1), (0,0)}(\hat{X}_{\mathrm{res}}^{\mathrm{sp}}) .  
	\label{eq:mistake_example}
\end{equation}
Setting $\calF$ to be the index set generated by $(0,0)$, $\calF_+$ is generated by $(0,0)$ and $(2,1)$, so this is in accordance with \Cref{prop:normal_mapping_1}. 

Since the index set $\calE\subset \bbN\times \bbN$ contains $(0,0)$ and satisfies $\calE=\calE_+$:
\begin{corollary}
	If $g \in  x (\sigma^2+\mathsf{Z}x)^{-1/2}\calA^{(0,0),\calE,(0,0)}_{\mathrm{loc}}(\hat{X}^{\mathrm{sp}}_{\mathrm{res}})$, then 
	\begin{equation} 
		I=\int_x^{\bar{x}} x_0^{-1} g(x_0,y;\sigma) \dd x_0 \in \calA^{(0,0),\calE,(0,0)}_{\mathrm{loc}}(\hat{X}^{\mathrm{sp}}_{\mathrm{res}}).
	\end{equation}
	\label{cor:normal_mapping_1}
\end{corollary}

\begin{proposition}
	\label{prop:normal_mapping_2}
	If $g \in \calA_{\mathrm{loc}}^{((\alpha_0,0),\alpha),\beta,(0,0)}(\hat{X}_{\mathrm{res}}^{\mathrm{sp}})$ for $\alpha_0\in \bbN^+$, $\alpha,\beta>0$, $\beta\notin 2\bbN$, then  
	\begin{equation} 
		\int_x^{\bar{x}} x_0^{-1} g(x_0,y;\sigma) \dd x_0 \in C^\infty(\hat{X}^{\mathrm{sp}}_{\mathrm{res}}) + \calA_{\mathrm{loc}}^{((0,0),\alpha),\beta,(0,0)}(\hat{X}^{\mathrm{sp}}_{\mathrm{res}})
	\end{equation}  
	holds. 
\end{proposition}
\begin{proof}
	Let $I= \int_x^{\bar{x}} x_0^{-1} g(x_0,y;\sigma) \dd x_0 $. 
	It suffices to prove the following two claims:
	\begin{enumerate}[label=(\Roman*)]
		\item if $g$ is supported in $\{x/E>C\}$, then $I \in C^\infty(\hat{X}^{\mathrm{sp}}_{\mathrm{res}}) + \calA_{\mathrm{loc}}^{(0,0),\beta,(0,0)}(\hat{X}^{\mathrm{sp}}_{\mathrm{res}})$, 
		\item if $g$ is supported in $\{x/E<c\}$, then $I\in \calA_{\mathrm{loc}}^{((0,0),\alpha),\beta,(0,0)}(\hat{X}^{\mathrm{sp}}_{\mathrm{res}})$.
	\end{enumerate}

	In order to prove (II), we write $I$ in terms of $\varrho=x/E$:
	\begin{equation}
			I =  \int_{\varrho}^{\min\{\bar{x}/\sigma^2 ,c \}} \varrho_0^{-1}  G (\varrho_0,y;\sigma)  \dd \varrho_0, 
	\end{equation}
	where $G \in \calA_{\mathrm{loc}}^{\beta,((\alpha_0,0),\alpha)} ( [0,\infty)_\sigma \times [0,\infty)_\varrho \times \partial X)$. 
	Since $\alpha_0,\alpha>0$, we can write 
	\begin{equation}
		I = + \int_{0}^{c } \varrho_0^{-1}  G (\varrho_0,y;\sigma)  \dd \varrho_0 - \int_{0}^{\varrho} \varrho_0^{-1}  G (\varrho_0,y;\sigma)  \dd \varrho_0
	\end{equation}
	for $\sigma$ sufficiently small. The first term is in $\calA^\beta_{\mathrm{loc}} ([0,\infty)_\sigma \times \partial X)\subset \calA^{\beta,(0,0)}_{\mathrm{loc}} ([0,\infty)_\sigma\times [0,\infty)_\varrho \times \partial X)$, while the second is in $\calA_{\mathrm{loc}}^{\beta,((\alpha_0,0),\alpha)} ( [0,\infty)_\sigma \times [0,\infty)_\varrho \times \partial X)$. Thus, it is also the case that 
	\begin{equation} 
		I\in \calA_{\mathrm{loc}}^{((0,0),\alpha),\beta,(0,0)}(\hat{X}^{\mathrm{sp}}_{\mathrm{res}})
	\end{equation}
	away from $\mathrm{zf}$. 
	Since $I$ is identically zero in the set $\{Ec<x\}$ and therefore smooth in some neighborhood of $\mathrm{zf}$, this suffices to show that, globally, $I\in \calA_{\mathrm{loc}}^{((0,0),\alpha),\beta,(0,0)}(\hat{X}^{\mathrm{sp}}_{\mathrm{res}})$. 
	
	To prove (I), we write $I$ in terms of $\rho = x^{1/2}$ and $\hat{E} = E/x$:
	\begin{equation}
		I =  \int_{\rho }^{\bar{x}^{1/2}} \rho_0^{-1} G\Big(\rho_0,y; \hat{E} \Big( \frac{\rho}{\rho_0} \Big)^2 \Big) \dd \rho_0, 
		\label{eq:misc_o11}
	\end{equation}
	where $G(\rho,y;\hat{E})\in C^\infty([0,\infty)_{\hat{E}} ;  \calA^{\beta}_{\mathrm{loc}}( [0,\infty)_\rho\times \partial X)  )$ is supported in $\hat{E}<C^{-1}$. 
	
	We expand $G(\rho,y;\hat{E})$ in Taylor series around $\hat{E}=0$: there exist $G^{(k)} \in \smash{\calA^{\beta}_{\mathrm{loc}}}( [0,\infty)_\rho\times \partial X)$ and $F^{(k)} \in C^\infty([0,\infty)_{\hat{E}} ;  \calA^{\beta}_{\mathrm{loc}}( [0,\infty)_\rho\times \partial X))$ such that 
	\begin{equation}
		G(\rho,y;\hat{E}) = \sum_{k=0}^K  \hat{E}^k G^{(k)}(\rho,y) + \hat{E}^{K+1} F^{(K+1)}(\rho,y;\hat{E}). 
	\end{equation}
	Substituting this into \cref{eq:misc_o11}, 
	\begin{equation}
		I= \sum_{k=0}^K  \hat{E}^k \rho^{2k}\int_{\rho }^{\bar{x}^{1/2}} \rho_0^{-1-2k}  G^{(k)}(\rho_0,y) \dd \rho_0 + \hat{E}^{K+1} \rho^{2K+2} \int_\rho^{\bar{x}^{1/2}} \rho_0^{-3-2K} F^{(K+1)}\Big(\rho_0,y;\hat{E} \Big( \frac{\rho}{\rho_0} \Big)^2\Big) \dd \rho_0 
	\end{equation}
	for each $K\in \bbN$. We take $K=\lfloor \beta/2 \rfloor$, in which case we can write 
	\begin{equation}
		\int_{\rho }^{\bar{x}^{1/2}} \rho_0^{-1-2k}  G^{(k)}(\rho_0,y) \dd \rho_0  = \int_{0 }^{\bar{x}^{1/2}} \rho_0^{-1-2k}  G^{(k)}(\rho_0,y) \dd \rho_0  - \int_{0}^\rho \rho_0^{-1-2k}  G^{(k)}(\rho_0,y) \dd \rho_0, 
	\end{equation} 
	where the integrals on the right-hand side are well-defined because $\rho^{-2k} G^{(k)} \in  \calA^{0+}_{\mathrm{loc}}$ for $k=0,\ldots,K$ (owing to $\beta\notin 2\bbN$). 
	We split $I=I_0+I_1+I_2$, where 
	\begin{align}
		\begin{split} 
		I_0 &= \sum_{k=0}^K  \hat{E}^k \rho^{2k}\int_{0 }^{\bar{x}^{1/2}} \rho_0^{-1-2k}  G^{(k)}(\rho_0,y) \dd \rho_0  \in C^\infty([0,\infty)_{\hat{E}} \times [0,\infty)_\rho \times \partial X )\\
		I_1 &= - \sum_{k=0}^K  \hat{E}^k \rho^{2k}\int_{0}^{\rho} \rho_0^{-1-2k}  G^{(k)}(\rho_0,y) \dd \rho_0 \\
		I_2 &= \hat{E}^{K+1} \rho^{2K+2} \int_\rho^{\bar{x}^{1/2}} \rho_0^{-3-2K} F^{(K+1)}\Big(\rho_0,y;\hat{E} \Big( \frac{\rho}{\rho_0} \Big)^2\Big) \dd \rho_0.
		\end{split}
	\end{align}
	Since $\rho^{2k}\int_{0 }^{\rho} \rho_0^{-1-2k}  G^{(k)}(\rho_0,y) \dd \rho_0 \in \calA^\beta_{\mathrm{loc}}([0,\infty)_\rho \times \partial X)$, $I_1 \in C^\infty([0,\infty)_{\hat{E}} ;  \calA^{\beta}_{\mathrm{loc}}( [0,\infty)_\rho\times \partial X)  ) $. 
	The same holds for $I_2$. The $L^\infty$ case of this estimate is 
	\begin{align}
		\begin{split} 
		\Big|\rho^{2K+2} \int_\rho^{\bar{x}^{1/2}} \rho_0^{-3-2K} F^{(K+1)}\Big(\rho_0,y;\hat{E} \Big( \frac{\rho}{\rho_0} \Big)^2\Big) \dd \rho_0 \Big| &\preceq \rho^{2K+2} \int_\rho^{\bar{x}^{1/2}} \rho_0^{-3-2K + \beta} \dd \rho_0 \\
		&\preceq \rho^{2K+2} \Big[ (\bar{x}^{1/2})^{-2-2K+\beta} + \rho^{-2-2K+\beta} \Big] \\
		&\preceq \rho^{2\lfloor \beta /2\rfloor +2} + \rho^\beta =O( \rho^\beta),
		\end{split} 
	\end{align}
	and $\partial_{\hat{E}}^\kappa (\rho \partial_\rho)^{\varkappa} I_2$ is estimated similarly. 
	
	Thus, $I\in C^\infty([0,\infty)_{\hat{E}} \times [0,\infty)_\rho \times \partial X ) + C^\infty([0,\infty)_{\hat{E}} ;  \calA^{\beta}_{\mathrm{loc}}( [0,\infty)_\rho\times \partial X)  )$.
	Using \cref{eq:misc_ixy} as in the proof of the previous proposition, this suffices to show that 
	$I \in C^\infty(\hat{X}^{\mathrm{sp}}_{\mathrm{res}}) + \calA_{\mathrm{loc}}^{(0,0),\beta,(0,0)}(\hat{X}^{\mathrm{sp}}_{\mathrm{res}})$.
\end{proof}

\appendix

\section{The model ODE}
\label{ap:model} 

We record in this appendix some computations regarding the model ODE, \cref{eq:misc_aku}, now allowing a Schwarzschild-like subleading term and nonzero forcing: 
\begin{equation}
(1+x \mathsf{a} ) \Big(x^2 \frac{\partial}{\partial x} \Big)^2 u   + \sigma^2 u + \mathsf{Z} x u = -f, 
\label{eq:model_ODE}
\end{equation}
where $\mathsf{a} \neq 0$, $f\in C_{\mathrm{c}}^\infty(\bbR^+_x)$, and $u\in C^\infty(\bbR^+)$.
The case $\mathsf{a} <0$ is computationally similar to the case $\mathsf{a}\geq 0$, but in the former it is necessary to restrict $x$ to $(0,1 / |\mathsf{a}|)$ to avoid the second order term in \cref{eq:model_ODE} vanishing at $x=1/|\mathsf{a}|$. We will therefore only consider the case $\mathsf{a}\geq 0$.  
If we allow $\mathsf{Z}$ to depend on $\sigma$, the $\mathsf{a}>0$ case can be reduced to the $\mathsf{a} = 0$ case via a simple change of variables: let $x_0 = (\mathsf{a} + x^{-1} )^{-1} = x/(1+\mathsf{a}x)$, so that $r_0=1/x_0$ is given by $r_0=\mathsf{a}+r$, where $r=1/x$. Then $\partial_r = \partial_{r_0}$, and the ODE \cref{eq:model_ODE} is equivalent to 
\begin{equation}
		 \Big( x^2_0 \frac{\partial}{\partial x_0} \Big)^2 u + \sigma^2 u+  (\mathsf{Z} - \sigma^2 \mathsf{a}) x_0 u   = -f_0(x_0)
		\label{eq:model_ODE2}
\end{equation}
for $f_0(x_0)=(1- \mathsf{a} x_0) f(x)$. The interval $[0,\infty)_x$ becomes $[0,1/\mathsf{a})_{x_0}$, but we can analyze \cref{eq:model_ODE2} on the larger region $\smash{[0,\infty)_{x_0}}$, so that the analysis of \cref{eq:model_ODE} is reduced to the $\mathsf{a}=0$ case, with $\mathsf{Z}(\sigma)=\mathsf{Z} -\smash{\sigma^2} \mathsf{a}$. 
Since the results in this section mainly serve to illustrate the general features of the problem observed in the body of the paper, proofs are either sketched or omitted entirely when elementary. 
References for many of the elementary statements can be found in \cite{Bateman}\cite{Slater}\cite{SpecialFunctions}\cite{Olver}. We will mostly cite \cite{SpecialFunctions} when an explicit reference is desired.  

The (nonzero) solutions to the homogeneous ODE cannot be written in terms of elementary functions for any value of $\sigma\geq 0$. In fact, for $\sigma\neq 0$, the homogeneous ODE is essentially a special case of \emph{Whittaker's ODE} 
\begin{equation}
	\frac{d^2 W}{dz^2} + \Big( - \frac{1}{4} + \frac{\kappa}{z} + \frac{1/4-\mu^2}{z^2} \Big) W = 0,  
	\label{eq:whittaker_ODE}
\end{equation}
where $\kappa\in \bbC$, $\mu \in \bbC$ are parameters and $W \in \calD'(\bbR^+_z)$. When $\mu\notin - 2^{-1}\bbN^+$, there exist two named solutions to \cref{eq:whittaker_ODE}, and these extend from the nonnegative real axis to analytic functions 
\begin{equation}
	\operatorname{WhittM}_{\kappa,\mu}, \operatorname{WhittW}_{\kappa,\mu}: \bbC_z\backslash (-\infty,0] \to \bbC 
	\label{eq:misc_698}
\end{equation}
which solve Whittaker's ODE in the complex analytic sense. 
These functions are Whittaker's M- and W- functions. They can be written in terms of Kummer's and Tricomi's confluent hypergeometric functions \cite[\S13.1.32, \S13.1.33]{SpecialFunctions}. For certain values of $\mu$, the branch cuts in \cref{eq:misc_698} can be removed.

\begin{propositionp}
	For $\sigma>0$, the set of $u\in \calD'(\bbR^+_r)$ solving the original homogeneous ODE $\partial_r^2 u + \sigma^2 u + \mathsf{Z} r^{-1} u = 0$ is given by 
	\begin{equation}
		\calW_\sigma =\{c_1 \operatorname{WhittM}_{\kappa,1/2}(2i\sigma r) +c_2 \operatorname{WhittW}_{\kappa,1/2}(2i\sigma r) :c_1,c_2\in \bbC\}
	\end{equation}
	for $\kappa = - i\mathsf{Z}  / 2\sigma$. 
\end{propositionp}

The asymptotic expansions of the Whittaker M- and W-functions at large imaginary argument are due originally to Whittaker. For fixed $\sigma>0$, we have
\begin{align}
	\begin{split} 
	\operatorname{WhittM}_{\kappa,1/2}(2i\sigma r)  &= \Big[- \frac{(-2i \sigma)^{-i\mathsf{Z}/2\sigma} }{\Gamma(1+\kappa ) }e^{-i \sigma r - (i \mathsf{Z}/2\sigma) \log r}  \\ &\qquad\qquad\qquad + \frac{(+2i \sigma)^{i\mathsf{Z}/2\sigma } }{\Gamma(1-\kappa) }e^{+i \sigma r + (i\mathsf{Z}/2\sigma) \log r}  \Big] \Big(1+O_{\kappa,\sigma} \Big( \frac{1}{r} \Big) \Big) 
	\end{split} \\
	\operatorname{WhittW}_{\kappa,1/2}(2i\sigma r)  &= (2i \sigma)^{-i\mathsf{Z}/2\sigma } e^{-i \sigma r - (i\mathsf{Z}/2\sigma)\log r}  \Big(1+O_{\kappa,\sigma} \Big(\frac{1}{r}\Big)\Big) 
\end{align}
as $r\to\infty$. Here, and below, we leave the $\mathsf{Z}$ dependence of the estimates implicit.
We are using the principal branch of the logarithm in making sense of $ (2i \sigma)^{-i\mathsf{Z}/2\sigma}$. 
As we are concerned with $r\to\infty$ behavior, the Whittaker M- function (which is singled out of the space of all solutions to Whittaker's ODE by its behavior at small argument) is not one of the solutions of direct interest. For $\sigma>0$, 
\begin{align}
	\calW_{\sigma,-} &=  \operatorname{span}_\bbC\{ \operatorname{WhittW}_{\kappa,1/2}(2i\sigma r) \} \\
	\begin{split} 
		\calW_{\sigma,+} &= \operatorname{span}_\bbC\Big\{ \Big[ (2i \sigma)^{-i\mathsf{Z}/2\sigma }  \operatorname{WhittM}_{\kappa,1/2}(2i\sigma r)  \\
		&\qquad\qquad + (-2i\sigma)^{- (i\mathsf{Z}/2\sigma) }\Gamma(1+\kappa)^{-1} \operatorname{WhittW}_{\kappa,1/2}(2i\sigma r) \Big] \Big\}
	\end{split} 
\end{align}
are the spaces of ``incoming'' or ``outgoing'' solutions to the ODE, and $\calW_\sigma = \calW_{\sigma,-}\oplus \calW_{\sigma,+}$. It is the spaces $\calW_{\sigma,-},\calW_{\sigma,+}$ that concern us.

Set $w_-(r;\sigma) =  (2i \sigma)^{i\mathsf{Z}/2\sigma} \operatorname{WhittW}_{\kappa,1/2}(2i\sigma r)$ and 
\begin{equation}
	w_+(r;\sigma) = \Big[ (2i \sigma)^{-i\mathsf{Z}/2\sigma} \Gamma(1-\kappa) \operatorname{WhittM}_{\kappa,1/2}(2i\sigma r) + (-2i\sigma)^{-i\mathsf{Z}/2\sigma}\frac{\Gamma(1-\kappa)}{\Gamma(1+\kappa)} \operatorname{WhittW}_{\kappa,1/2}(2i\sigma r) \Big].
	\label{eq:misc_u++}
\end{equation}
Thus, $w_-(-;\sigma) \in \calW_{\sigma,-}$ and $w_+(-;\sigma)\in \calW_{\sigma,+}$, and $w_{\pm}(r;\sigma) \in C^\infty(\bbR^+_\sigma\times \bbR^+_r)$. 
These have normalized oscillatory behavior $\exp(\pm i \sigma r)$ as $r\to\infty$,
\begin{equation}
	w_\pm(r;\sigma) =  e^{\pm i \sigma r \pm (i\mathsf{Z}/2\sigma)\log r} \Big(1+O_{\sigma}\Big(\frac{1}{r}\Big)\Big).
	\label{eq:misc_asy}
\end{equation}
For each $\sigma>0$, $w_{\pm}$ are the unique solutions to the ODE satisfying  \cref{eq:misc_asy}. 
It follows from this and the fact that the ODE has real coefficients that $w_{-}(r;\sigma) = w_{+}(r;\sigma)^*$. 
Expanding to higher order \cite[\S13.5]{SpecialFunctions}: for each $\sigma>0$ and $K\in \bbN^+$, 
\begin{equation}
	w_{\pm}(r;\sigma) = e^{\pm i \sigma r} r^{\pm i\mathsf{Z}/2\sigma } \Big[1 + \mathsf{Z} \sum_{k=1}^{K-1} \frac{(\pm i)^k}{8^k k! \sigma^{3k} r^k} (\mathsf{Z} \pm 2k i \sigma ) \prod_{j=1}^{k-1}(\mathsf{Z} \pm 2j i \sigma )^2 + O_{\sigma}\Big( \frac{1}{r^K} \Big) \Big]
\end{equation}
as $r\to\infty$. 

For $\sigma=0$, the set of $u\in \calD'(\bbR^+_x)$ solving the ODE $(x^2 \partial_x)^2 u + \mathsf{Z} xu = 0$ is 
\begin{equation}
	 \{c_1 r^{1/2}  J_1(2 \mathsf{Z}^{1/2} r^{1/2} ) + c_2 r^{1/2} Y_1(2 \mathsf{Z}^{1/2} r^{1/2}) : c_1,c_2 \in \bbC\},
\end{equation}
where $J_1,Y_1$ denote the Bessel J- and Y- functions of order one. 
As $r\to\infty$, 
\begin{align}
	r^{1/2} J_1(2\mathsf{Z}^{1/2} r^{1/2}) &= \frac{r^{1/4}}{\pi^{\frac{1}{2}} \mathsf{Z}^{\frac{1}{4}}}\Big[-   \cos\Big(2 \sqrt{\mathsf{Z}r} + \frac{\pi}{4} \Big)  + \frac{3}{16} \frac{1}{\sqrt{\mathsf{Z}r}} \sin \Big(2\sqrt{\mathsf{Z}r} + \frac{\pi}{4} \Big)\Big] \Big(1+O\Big( \frac{1}{r}\Big)\Big) \\
	r^{1/2} Y_1(2\mathsf{Z}^{1/2} r^{1/2}) &= \frac{r^{1/4}}{\pi^{\frac{1}{2}} \mathsf{Z}^{\frac{1}{4}}}\Big[ -\sin\Big( 2  \sqrt{\mathsf{Z}r} + \frac{\pi}{4} \Big)-\frac{3}{16} \frac{1}{\sqrt{\mathsf{Z}r}} \cos \Big(2  \sqrt{\mathsf{Z}r} + \frac{\pi}{4}  \Big)  \Big] \Big(1+O\Big( \frac{1}{r}\Big)\Big), 
\end{align}
and we have a full expansion in powers of $r^{-1/2}$. 
For $\sigma=0$, set  
\begin{align}
	v_{-}(r;0) &=  r^{1/2} J_1(2\mathsf{Z}^{1/2} r^{1/2}) - i r^{1/2} Y_1(2\mathsf{Z}^{1/2} r^{1/2}) = r^{1/2} H_1^{(2)}(2\mathsf{Z}^{1/2} r^{1/2}) \label{eq:u-0}  \\
	v_{+}(r;0) &= r^{1/2} J_1(2\mathsf{Z}^{1/2} r^{1/2}) + i r^{1/2} Y_1(2\mathsf{Z}^{1/2} r^{1/2}) 
	= r^{1/2} H_1^{(1)}(2\mathsf{Z}^{1/2} r^{1/2}), \label{eq:u+0}
\end{align}
where $H_1^{(1)},H_1^{(2)}$ denote the Hankel functions of order one. These have the asymptotics
\begin{equation}
	v_{\pm}(r;\sigma) = e^{\pm 2 i \mathsf{Z}^{1/2} r^{1/2}\mp  3\pi i/4} \frac{r^{1/4}}{\pi^{1/2} \mathsf{Z}^{1/4}} \Big(1 + O\Big(\frac{1}{r^{1/2}} \Big) \Big)
	\label{eq:misc_asz}
\end{equation}
in the $r\to\infty$ limit. Expanding \cref{eq:misc_asz} to higher order \cite[\S9.2]{SpecialFunctions}: for each $K\in \bbN^+$, 
\begin{equation}
	v_{\pm}(r;0) = e^{\pm 2 i \mathsf{Z}^{1/2} r^{1/2} \mp 3\pi i/4} \frac{r^{1/4}}{\pi^{1/2} \mathsf{Z}^{1/4}} \sum_{k=0}^{K-1} (\pm i )^k \mathsf{Z}^{-k/2} \frac{(2k+1)! (2k)!}{64^k(k!)^3} \frac{1}{r^{k/2}} + O\Big( \frac{1}{r^{K/2 - 1/4}} \Big) 
\end{equation}
as $r\to\infty$. 

\begin{propositionp}
	$v_{\pm}$ are the unique solutions to the $\sigma=0$ ODE satisfying the asymptotic \cref{eq:misc_asz}. 
\end{propositionp}

\begin{propositionp}
	\label{prop:radialcont}
	Setting 
	\begin{equation}
		C_\pm(\sigma) = -\frac{\mathsf{Z}^{1/2}}{2\pi \sigma} (\mp 2i \sigma)^{\pm  i \mathsf{Z}/(2\sigma)  } \Gamma \Big( \mp\frac{i\mathsf{Z}}{2\sigma} \Big)
		\label{eq:CPM}
	\end{equation}
	and $v_{\pm}(r;\sigma) = C_\pm(\sigma) w_{\pm}(r;\sigma)$, the functions $v_{\pm}(r;E^{1/2}) :  [0,\infty)_E\times (0,\infty)_r\to \bbC$ are both smooth all the way down to $E = 0$. Consequently, if 
	\begin{align}
		C_{\pm,0}(\sigma) =   e^{\pm \pi i/4}\pi^{-1/2}\sigma^{-1/2} \exp\Big[\pm  \frac{\mathsf{Z} i}{\sigma} \Big(\log\Big(\frac{2\sigma}{ \mathsf{Z}^{1/2}} \Big) + \frac{1}{2} \Big) \Big] ,
		\label{eq:cpm}
	\end{align}
	then $C_{\pm,0}(\sigma) w_{\pm}(r;\sigma): [0,\infty)_\sigma \times (0,\infty)_r\to \bbC$ are smooth all the way down to $\sigma=0$, with restriction $v_{\pm}(-;0)$ to $\sigma = 0$.  
\end{propositionp}
\begin{remark*}
	The $C^0$ case of this proposition is similar to \cite[\S13.3.4, \S 13.3.5]{SpecialFunctions}, except that our $\kappa$ is purely imaginary rather than purely real. See also \cite{Taylor}\cite[\S6.13.3, Eq. (21) - (24)]{Bateman}.
	Of course, if we were to multiply $C_\pm(\sigma)$ by any element of $C^\infty[0,\infty)_E$, the resultant functions would also satisfy the proposition above, and likewise with $C_{\pm,0}(\sigma)$ in place of $C_\pm(\sigma)$ and $C^\infty[0,\infty)_\sigma$ in place of $C^\infty[0,\infty)_E$. 
	
\end{remark*}
\begin{remark*}
	Once we know that $C_\pm$ satisfies the first clause of the conclusion of \Cref{prop:radialcont}, the second clause of the proposition follows from the large argument asymptotics of $\log \Gamma$, which can be found in \cite[\S6.1.40]{SpecialFunctions}. 
	Observe, using \cite[\S6.1.40]{SpecialFunctions}, that the ratio $C_\pm / C_{\pm,0}: [0,\infty)_\sigma\to  \bbC$ is a smooth function of $\sigma$, not $E=\sigma^2$. This is related to the fact that, for $u_{0,\pm}$ as in \Cref{thm:main}, $u_{0,\pm}|_{\mathrm{bf}}:[0,\infty)_\sigma\times \partial X \to \bbC$ is only smooth with respect to $\sigma$, in contrast to $u_{0,\pm}|_{\{x=\varepsilon\}}: [0,\infty)_\sigma\times \partial X \to \bbC$ for $\varepsilon \in (0,\bar{x})$, which is smooth with respect to $E$.
\end{remark*}

\begin{figure}[t]
	\begin{center}
		\includegraphics[scale=.65]{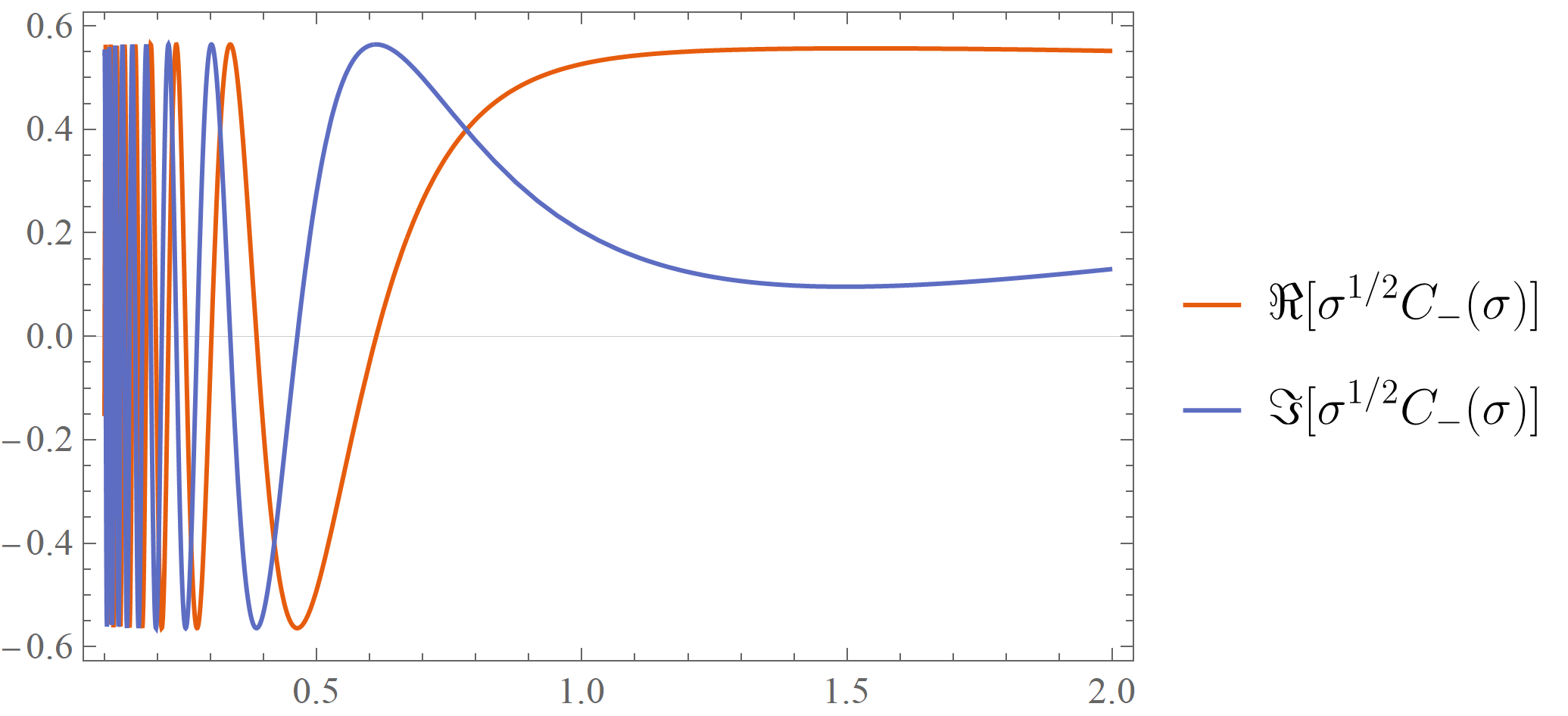}
	\end{center}
	\caption{The real and imaginary parts of the function $C_-(\sigma)$ defined by \cref{eq:CPM}, rescaled so that the oscillations have constant amplitude.}
\end{figure}

Up to the $\exp(\pm \pi i /4) \pi^{-1/2}$ in \cref{eq:cpm}, \cref{eq:cpm} is suggested by \Cref{thm:main}, as, for each $\sigma>0$, 
\begin{equation}
	\Phi(x;\sigma) = \frac{\sigma}{x} - \frac{\mathsf{Z}}{2\sigma} \log x + \frac{\mathsf{Z}}{2\sigma} + \frac{\mathsf{Z}}{\sigma} \log\Big( \frac{2\sigma}{\mathsf{Z}^{1/2}} \Big) + O_\sigma(x)
	\label{eq:misc_upi}
\end{equation}
as $x\to 0^+$, 
while the phase of the bracketed term in \cref{eq:misc_u++} consists only of the first two terms of \cref{eq:misc_upi}. The $\sigma^{-1/2}$ term in \cref{eq:cpm} is needed to match the $r^{1/2}$ in \cref{eq:u-0}, \cref{eq:u+0}.
This heuristic should be taken with a grain of salt, as the $O_\sigma(x)$ term blows up as $\sigma\to 0^+$.
Assuming that there exist some $C_\pm$ satisfying the conclusion of \Cref{prop:radialcont}, the sufficiency of the formula above can be seen from the $r\to 0^+$ asymptotics of the Whittaker functions. Indeed, the ODE can be written in the form 
\begin{equation}
	\Big( r \frac{\partial}{\partial r} \Big)^2 u -  \Big( r \frac{\partial}{\partial r} \Big) u  + ( \sigma^2 r^2 + \mathsf{Z} r)u = 0,
	\label{eq:misc_8y3}
\end{equation}
which has a regular singular point at $r=0$. The ``normal operator'' of $(r\partial_r)^2 - (r\partial_r) + \sigma^2 r^2 + \mathsf{Z} r$ at $r=0$ is $(r\partial_r)^2  - r \partial_r$, the indicial roots of which are $0$ and $1$. We can deduce that a family $\{u(-;\sigma)\}_{\sigma\geq 0} \subset \calD'(\bbR^+)$ of solutions to \cref{eq:misc_8y3} can be smooth at $\{r>0,\sigma=0\}$  only if it extends continuously to $[0,1)_\sigma\times [0,1)_r$ (we will not have smoothness at $r=0$ because of the presence of a $r\log r$ term), in which case $\lim_{\sigma\to 0^+} \lim_{r\to 0^+} u = \lim_{r\to 0^+}\lim_{\sigma\to 0^+}  u$ exists. Now observe that, as $r\to 0^+$,  
\begin{align}
	w_{\pm}(r;\sigma) &= \pm \frac{2i}{\mathsf{Z}} ( \mp 2i  \sigma)^{\mp i\mathsf{Z}/2\sigma} \frac{\sigma}{\Gamma(\mp i\mathsf{Z}/2\sigma)} + O_\sigma(r \log r), \\
	\lim_{r\to 0^+} w_{\pm}(r;\sigma) &= \pm \frac{2i}{\mathsf{Z}} (\mp 2i \sigma)^{\mp i\mathsf{Z}/2\sigma } \frac{\sigma}{\Gamma(\mp i\mathsf{Z}/2\sigma)}
	\label{eq:misc_9oo}
\end{align}
for each $\sigma>0$. We now compare \cref{eq:misc_9oo} with 
\begin{equation}
	\lim_{r\to 0^+} v_{\pm}(r;0) =  \mp \frac{i}{\pi \mathsf{Z}^{1/2}} + O(r \log r)
\end{equation}
\cite[\S9.1.9, \S9.1.11]{SpecialFunctions}.
This yields $C_\pm(\sigma) =  -(2 \pi \sigma)^{-1} \mathsf{Z}^{1/2} (\mp 2 i \sigma)^{\pm i \mathsf{Z}/2\sigma} \Gamma(\mp i \mathsf{Z}/2\sigma) $, which is \cref{eq:CPM}. A more thorough analysis of the ODE near $r=0$ suffices to prove \Cref{prop:radialcont} properly.

\begin{figure}
	\begin{center}
		\includegraphics[width =  \textwidth ]{w.png}
	\end{center}
	\caption{The function  $U(r;E)=v_{-}(r;E^{1/2})/v_{-}(r;0)$  (\textit{left}) and its first derivative $\partial_{E} U(r;E)$ (\textit{right}) evaluated at $r=5$, plotted against $E \in [0,2]$. The real parts  are plotted in orange and the imaginary parts are in blue. From the figures it appears that $U \in C^2[0,2)_E$ and $U(0)=1$, in accord with \Cref{prop:radialcont} (and with \Cref{cor:main1}).}
\end{figure}

Letting $\chi \in C^\infty_{\mathrm{c}}[0,\infty)$ be identically one in some neighborhood of the origin, $v_{\pm}^\circ=\chi(1/r) v_{\pm}(r;\sigma)$ satisfies 
\begin{equation}
	\frac{\partial^2 v_{\pm}^\circ}{\partial r^2} + \Big( \sigma^2+\frac{\mathsf{Z}}{r} \Big) v_{\pm}^\circ = f_\pm 
\end{equation}
for some $f_\pm\in \cap_{E_0>0} C_{\mathrm{c}}^\infty([0,E_0]_E\times \bbR^+ )$. 
We can deduce (e.g. by appealing to \Cref{thm:main} plus \Cref{rem:nonconstant} in the case of spherical symmetry) that $v_{\pm}^\circ$ has the form 
\begin{equation} 
	v_{\pm}^\circ  = \exp(\pm i \Phi(r^{-1};E^{1/2})) (E+\mathsf{Z}r^{-1})^{-1/4} v_{0,\pm}^\circ
\end{equation} 
for some $v_{0,\pm} \in \calA^{(0,0),\calE,(0,0)}_{\mathrm{loc}}(X_{\mathrm{res}}^{\mathrm{sp}})$, where $X=[0,\infty)_x$ and $X_{\mathrm{res}}^{\mathrm{sp}} = [[0,\infty)_E\times X ; \{0\}\times \{0\} ;1/2]$. Thus, $w_{\pm}(r,\sigma)$ satisfy 
\begin{equation}
	w_{\pm}(r;\sigma) \in C_\pm^{-1}(\sigma) e^{\pm i \Phi(r^{-1} ; E^{1/2}) } (E+\mathsf{Z} x)^{-1/4} \calA^{(0,0),\calE,(0,0)}_{\mathrm{loc}}(X_{\mathrm{res}}^{\mathrm{sp}}).  
\end{equation}
That is, in the `$-$' case: 
\begin{propositionp}
\label{prop:Whittakerfullasymptotics}
The Whittaker W-function $\operatorname{WhittW}_{\kappa,1/2}$ satisfies 
\begin{multline}
	\operatorname{WhittW}_{-i \mathsf{Z}/2\sigma,1/2}(2i\sigma r) \in  \sigma \Gamma \Big( \frac{i\mathsf{Z}}{2\sigma}\Big)^{-1} \Big(\sigma^2 + \frac{\mathsf{Z}}{r} \Big)^{-1/4} \\ \exp \Big( - ir \sqrt{\sigma^2 + \frac{\mathsf{Z}}{r} } - \frac{i\mathsf{Z}}{\sigma} \log \Big( \frac{\sigma r^{1/2}}{\mathsf{Z}^{1/2}} + \Big(1+\frac{\sigma^2 r}{\mathsf{Z}} \Big)^{1/2} \Big) \Big)  \calA^{(0,0),\calE,(0,0)}_{\mathrm{loc}}(X_{\mathrm{res}}^{\mathrm{sp}}). 
	\label{eq:Whittakerfullasymptotics}
\end{multline}
\end{propositionp}

\begin{remark*}
	It seems that there are no logarithmic terms in the expansion of the polyhomogeneous function in \cref{eq:Whittakerfullasymptotics}. In order to prove this, it should be possible to combine the previous proposition with a WKB type expansion at $\mathrm{tf}$, with $\sigma$ being the semiclassical parameter. The point here is that the Whittaker W-function has simple asymptotics at $\mathrm{bf}$, and the asymptotic expansion at $\mathrm{tf}$ without log terms can be concluded from this. The alternative argument above is indirect, utilizing the asymptotics at $\mathrm{tf}$, and from this perspective it is somewhat miraculous that the asymptotic expansion at $\mathrm{tf}$ is one-step polyhomogeneous. For general forcing $f\in C_{\mathrm{c}}^\infty(\bbR^+)$, we should not expect the outgoing solution to the forced ODE to have a one-step polyhomogeneous expansion at $\mathrm{tf}$. 
\end{remark*}

\begin{figure}[t]
	\begin{center}
		\includegraphics[width = .75\textwidth ]{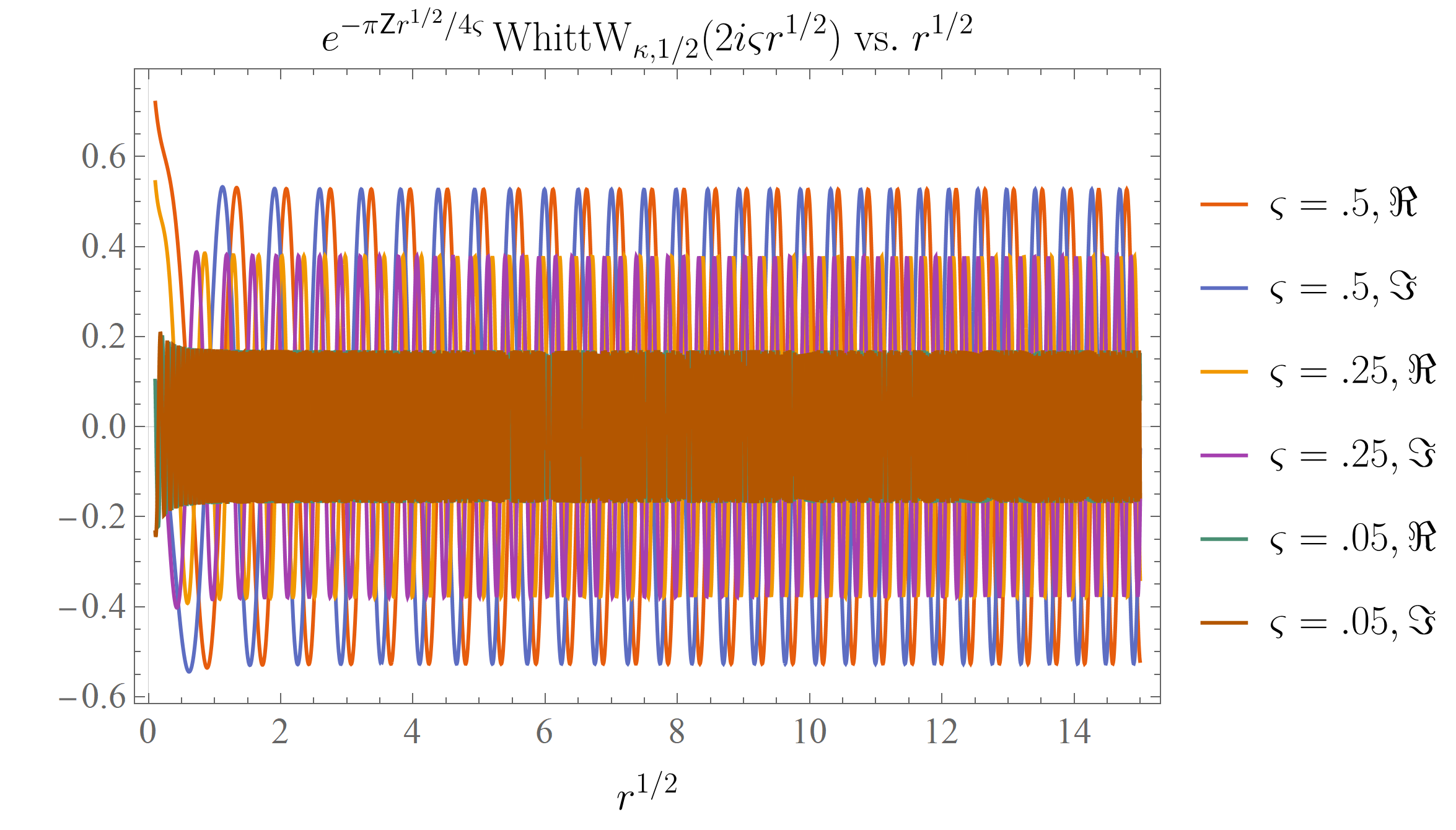}
	\end{center}
	\caption{The real (in red, orange, and green) and imaginary (in blue, purple, and brown) parts of the function $\exp(-\pi \mathsf{Z} r^{1/2} / 4 \varsigma) \operatorname{Whitt}_{\kappa,1/2}(2i\varsigma r^{1/2})$ as a function of $r^{1/2}$, for $\varsigma \in \{.05,.25,.5\}$ and $\mathsf{Z}=3$ fixed, where $\kappa = -i \mathsf{Z}r^{1/2}/2\varsigma$.}
	\label{fig:penultimate}
\end{figure}
\begin{figure}
	\begin{center}
		\includegraphics[width = .75 \textwidth ]{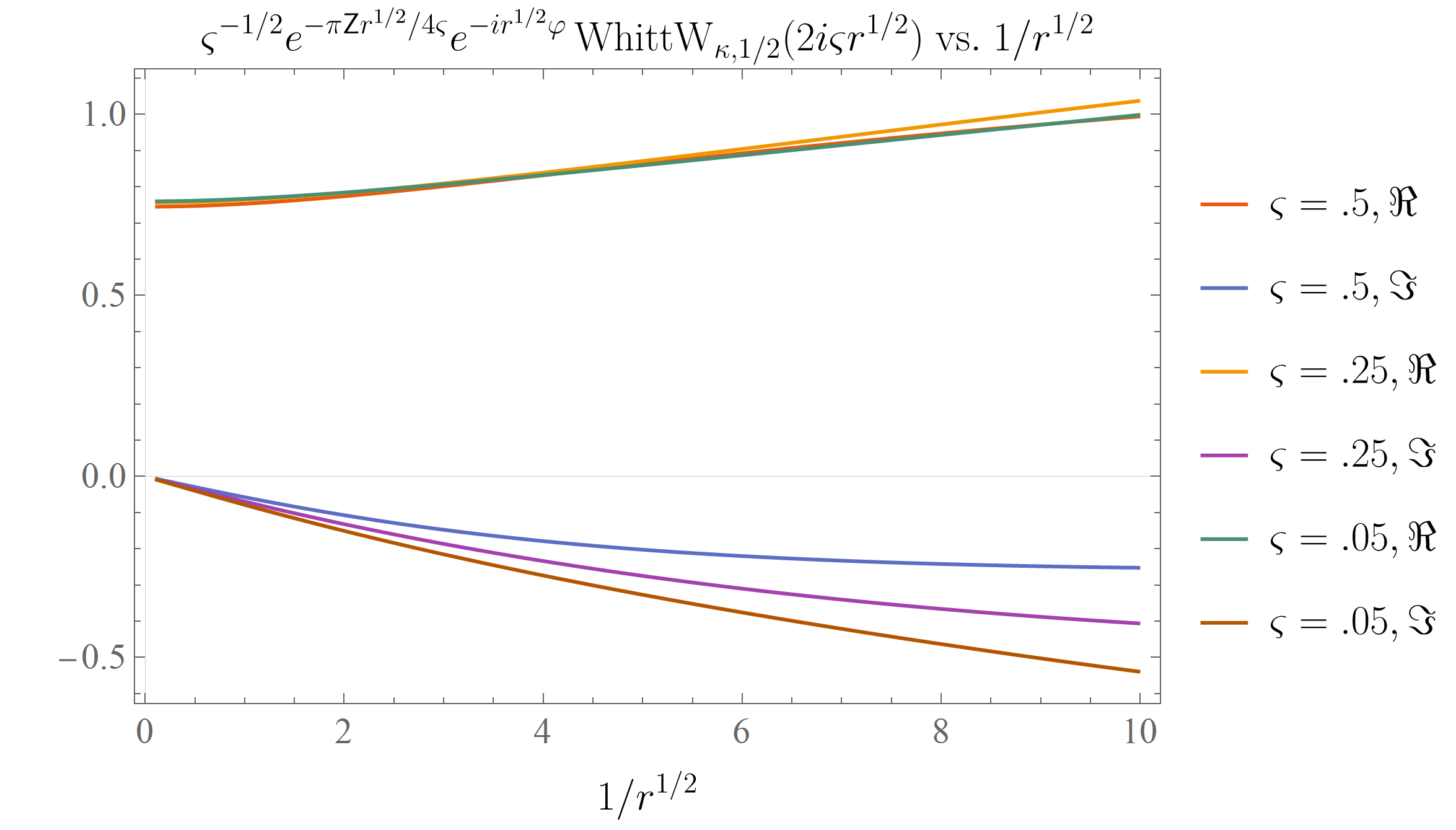}
	\end{center}
	\caption{The real and imaginary parts of the function $\varsigma^{-1/2}\exp(-\pi \mathsf{Z} r^{1/2} / 4 \varsigma) \exp(-i r^{1/2} \varphi) \operatorname{Whitt}_{\kappa,1/2}(2i\varsigma r^{1/2})$ (with the same color scheme as in \Cref{fig:penultimate}), now plotted as a function of $\rho=1/r^{1/2}$, for $\varsigma \in \{.05,.25,.5\}$ and $\mathsf{Z}=3$.}
	\label{fig:final}
\end{figure}

\Cref{prop:Whittakerfullasymptotics} can be strengthened by specifying the behavior of $\operatorname{WhittW}_{-i \mathsf{Z}/2\sigma,1/2}(2i\sigma r)$ at small $r$. Namely, if we replace $X$ by $\smash{\bar{X} =[0,\infty]_x}$, then, replacing $r$ by $\langle r \rangle$ in \cref{eq:Whittakerfullasymptotics}, the same statement holds, with $\calA^{(0,0),\calE,(0,0)}(X_{\mathrm{res}}^{\mathrm{sp}})$ replaced by the set of polyhomogeneous functions on $\smash{\bar{X}_{\mathrm{res}}^{\mathrm{sp}}}$ with index set $\bbN$ at $\mathrm{zf},\mathrm{tf},\mathrm{bf}$ and with some other index set (which can be specified) at the new face formed by the lift of $[0,\infty)_E\times \{\infty\}_x$ to $\smash{\bar{X}_{\mathrm{res}}^{\mathrm{sp}}}$.

The most interesting regime is $\mathrm{tf}$, which we can probe using the coordinates $\varsigma = \sigma r^{1/2}$ and $r^{-1/2}$. As a corollary of \Cref{prop:Whittakerfullasymptotics}, analogous to \Cref{cor:main3}, we deduce that 
\begin{multline}
	\operatorname{WhittW}_{-i\mathsf{Z} r^{1/2} /2\varsigma  ,1/2}(2i \varsigma r^{1/2}) \in \varsigma^{1/2} e^{ \frac{\pi \mathsf{Z}}{4\varsigma} r^{1/2}} \exp \Big( - i r^{1/2} \sqrt{\varsigma^2+\mathsf{Z}} + \frac{i\mathsf{Z}}{2\varsigma}r^{1/2} \Big)\\ 
	\exp\Big[ -  \frac{i\mathsf{Z}}{\varsigma}r^{1/2} \Big[\frac{1}{2}\log\Big( \frac{\mathsf{Z} r^{1/2}}{2\varsigma} \Big) + \log\Big( \frac{\varsigma}{\mathsf{Z}^{1/2}} + \Big(1+\frac{\varsigma^2}{\mathsf{Z}} \Big)^{1/2} \Big)  \Big]  \Big]
	 \calA^{(0,0),\calF}([0,\infty)_{\varsigma} \times [0,\infty)_{r^{-1/2}} )
\end{multline}
for some index set $\calF\subset \bbN\times \bbN$.
Here we used the large argument expansion of the $\Gamma$-function, i.e.\ Stirling's formula \cite[\S6.1.37]{SpecialFunctions}, which implies 
\begin{equation} 
	\Gamma(i \mathsf{Z} r^{1/2} / 2\varsigma) \in r^{-1/4} \varsigma^{1/2} e^{- \pi \mathsf{Z} r^{1/2} / 4 \varsigma } e^{-i \mathsf{Z} r^{1/2}/2\varsigma} ( \mathsf{Z} r^{1/2}/2\varsigma )^{i \mathsf{Z} r^{1/2}/2\varsigma}    C^\infty([0,\infty)_{\varsigma} \times [0,\infty)_{r^{-1/2}} ).
	\label{eq:Whittakerasymptoticstransitional}
\end{equation} 
Thus, if we let 
\begin{equation}
	\varphi =  - (\varsigma^2+\mathsf{Z})^{1/2}+ \frac{\mathsf{Z}}{2\varsigma} -    \frac{\mathsf{Z}}{\varsigma} \Big[\frac{1}{2}\log\Big( \frac{\mathsf{Z} r^{1/2}}{2\varsigma} \Big) + \log\Big( \frac{\varsigma}{\mathsf{Z}^{1/2}} + \Big(1+\frac{\varsigma^2}{\mathsf{Z}} \Big)^{1/2} \Big)  \Big], 
	\label{eq:varphi}
\end{equation}
then, for each $\varsigma>0$, $e^{-\pi \mathsf{Z} r^{1/2}/4\varsigma} e^{-i r^{1/2} \varphi} \smash{\operatorname{Whitt}_{\kappa,1/2}(2i \varsigma r^{1/2})}$ is a polyhomogeneous function of $r^{-1/2}$, all the way down to $r^{-1/2} = 0$. 
The convergence aspect of this result, in particular the fact that the multiplication by $\exp(-i r^{1/2} \varphi)$ kills off the oscillations of $\exp(-\pi \mathsf{Z} r^{1/2}/4\varsigma) \smash{\operatorname{Whitt}_{\kappa,1/2}(2i\varsigma r^{1/2})}$, is depicted in \Cref{fig:final} (and the contrast with \Cref{fig:penultimate}, where the $\exp(-i r^{1/2} \varphi)$ factor is missing).
The $C^0$ case of this is similar to \cite[\S6.13.3, Eq. (21) - (24)]{Bateman} (though we did not compute out the leading order term in the asymptotic expansion). 
\section*{Acknowledgements}

The author thanks Peter Hintz for guidance and comments on the manuscript, and in addition Nick Lohr for comments on the introduction. This work was supported by a Hertz fellowship. 

\printbibliography

\end{document}